\documentclass[submission]{jsfds}
\usepackage[latin1]{inputenc}
\usepackage[T1]{fontenc}

\usepackage{natbib}

\startlocaldefs
\newcommand{\pathfig}{.}

\usepackage{url}
\RequirePackage{dsfont}

\newcommand{\R}{\mathbb{R}}
\newcommand{\X}{\mathcal{X}}

\newcommand{\egaldef}{:=} % egalite definissant la quantite de gauche
\newcommand{\defegal}{=:} % egalite definissant la quantite de droite
\newcommand{\flapp}{\mapsto} % fleche d'application x->f(x) (elements)
\newcommand{\telque}{\, \text{ s.t. } \,} % tel que dans une definition d'ensemble
\newcommand{\grandO}{\ensuremath{\mathcal{O}}}
\newcommand{\petito}{\ensuremath{\mathrm{o}}}

\newcommand{\guil}[1]{``#1''}

\newcommand{\paren}[1]{\left(  #1 \right)}
\newcommand{\parenj}[1]{\mathopen{}\left( #1  \right) \mathclose{}}
\newcommand{\parens}[1]{( #1 )}
\newcommand{\parenb}[1]{\bigl( #1 \bigr)}
\newcommand{\parenB}[1]{\Bigl( #1  \Bigr)}

\newcommand{\carrej}[1]{\mathopen{}\left( #1  \right)^2 \mathclose{}}

\newcommand{\croch}[1]{\left[ #1 \right]}
\newcommand{\crochj}[1]{\mathopen{}\left[ #1 \right] \mathclose{}}
\newcommand{\crochs}[1]{[ #1 ]}
\newcommand{\crochb}[1]{\bigl[ #1 \bigr]}
\newcommand{\crochB}[1]{\Bigl[ #1 \Bigr]}
\newcommand{\crochbb}[1]{\biggl[ #1 \biggr]}

\newcommand{\set}[1]{\left\{ #1 \right\}}
\newcommand{\setj}[1]{\mathopen{}\left\{ #1 \right\} \mathclose{}}
\newcommand{\sets}[1]{\{ #1 \}}
\newcommand{\setb}[1]{\bigl\{ #1 \bigr\}}
\newcommand{\setB}[1]{\Bigl\{ #1  \Bigr\}}

\newcommand{\absj}[1]{\mathopen{} \left\lvert #1 \right\rvert \mathclose{}} %joli abs
\newcommand{\abss}[1]{\lvert #1 \rvert}
\newcommand{\absb}[1]{\bigl\lvert #1 \bigr\rvert}
\newcommand{\absB}[1]{\Bigl\lvert #1 \Bigr\rvert}

\newcommand{\prodscal}[2]{\left\langle #1 , \, #2 \right\rangle}
\newcommand{\prodscalj}[2]{\mathopen{} \left\langle #1 , \, #2 \right\rangle \mathclose{}}
\newcommand{\prodscals}[2]{\langle #1 , \, #2 \rangle}
\newcommand{\prodscalb}[2]{\bigl\langle #1 , \, #2 \bigr\rangle}

\newcommand{\norm}[1]{\left \lVert #1 \right\rVert}

\newcommand{\norms}[1]{\lVert #1 \rVert}
\newcommand{\normb}[1]{\bigl \lVert #1 \bigr\rVert}

\newcommand{\Proba}{\mathbb{P}} %probabilite
\newcommand{\E}{\mathbb{E}} %esperance
\DeclareMathOperator{\tr}{tr}
\DeclareMathOperator{\card}{card}
\DeclareMathOperator{\var}{var}
\DeclareMathOperator{\tmpargmin}{argmin}
\newcommand{\argmin}{\mathop{\tmpargmin}}
\DeclareMathOperator{\tmpargmax}{argmax}
\newcommand{\argmax}{\mathop{\tmpargmax}}
\newcommand{\Id}{I}
\DeclareMathOperator{\diag}{diag}
\newcommand{\vectun}{\ensuremath{\mathbf{1}}}

\newcommand{\Fh}{\widehat{F}}

\newcommand{\ih}{\widehat{\imath}}
\newcommand{\bayes}{\ensuremath{s^{\star}}}%estimateur de Bayes
\newcommand{\Set}{\mathbb{S}}
\newcommand{\Risk}{\mathcal{R}}%risque
\newcommand{\Remp}{\widehat{\mathcal{R}}_n}%risque empirique

\newcommand{\M}{\mathcal{M}}
\newcommand{\mM}{m \in \M}
\newcommand{\C}{\mathcal{C}}
\newcommand{\cE}{\mathcal{E}}
\newcommand{\cF}{\mathcal{F}}
\newcommand{\cG}{\mathcal{G}}

\newcommand{\Fhm}{\Fh_m}
\newcommand{\sh}{\widehat{s}}
\newcommand{\shp}{\sh^{\,\,\,\prime}}
\newcommand{\shm}{\sh_m}
\newcommand{\shmp}{\sh_{m^{\prime}}}
\newcommand{\mh}{\widehat{m}}

\newcommand{\mhmingal}{\mh_{\mathrm{min}}^{\mathrm{gal}}}
\newcommand{\mhgalzero}{\mh_{\mathrm{min}}^{(0)}}
\newcommand{\mhminlin}{\mh_{\mathrm{min}}^{\mathrm{lin}}}

\newcommand{\mhoptun}{\mh_{\mathrm{opt}}^{(1)}}
\newcommand{\mhslopeun}{\mh_1}% celui qui apparait dans Alg. 3
\newcommand{\mhAlgA}{\mh_{\mathrm{Alg. 1}}}
\newcommand{\mhAlgB}{\mh_{\mathrm{Alg. 2}}}
\newcommand{\mhAlgC}{\mh_{\mathrm{Alg. 3}}}
\newcommand{\mhAlgD}{\mh_{\mathrm{Alg. 4}}}
\newcommand{\mhAlgE}{\mh_{\mathrm{Alg. 5}}}
\newcommand{\mhAlgF}{\mh_{\mathrm{Alg. 6}}}
\newcommand{\mo}{m^{\star}}
\newcommand{\mominlin}{m_{\mathrm{min}}^{\star \, \mathrm{lin}}}
\newcommand{\momingal}{m_{\mathrm{min}}^{\mathrm{gal} \star}}
\newcommand{\critCgal}{\crit^{\mathrm{gal}}_C}
\newcommand{\Fhmh}{\widehat{F}_{\mh}}
\newcommand{\Ch}{\widehat{C}}
\newcommand{\Chjumpgal}{\Ch_{\mathrm{jump}}}% Cte correspondant a "un saut" (en general)
\newcommand{\Chmaxjump}{\Ch_{\mathrm{max\,j.}}}% Cas particulier du "max jump"
\newcommand{\Chwindow}{\Ch_{\mathrm{window}}}% Cas particulier "geometric window"
\newcommand{\Chwindowgal}{\widehat{\mathcal{E}}_{\mathrm{window}}^{\mathrm{gal}}}
\newcommand{\Chthr}{\Ch_{\mathrm{thr.}}}% Cas particulier "threshold"
\newcommand{\Chslope}{\Ch_{\mathrm{slope}}}% Cte obtenue via "la pente" (en general)
\newcommand{\Chcapushe}{\Ch_{\mathrm{CAPUSHE}}}
\newcommand{\Chslopecapushe}{\Ch_s}% Cte apparaissant au milieu de l'algo CAPUSHE
\newcommand{\mhcapushe}{\mh_{\mathrm{CAPUSHE}}}
\newcommand{\ChRoz}{\Ch_{\mathrm{SBJ}}}
\newcommand{\mhRoz}{\mh_{\mathrm{SBJ}}}
\newcommand{\un}{\mathds{1}}
\newcommand{\phW}{\widehat{p}^{\,W}}

\newcommand{\Dh}{\widehat{D}}
\newcommand{\N}{\mathbb{N}}
\newcommand{\sigh}{\widehat{\sigma}}

\DeclareMathOperator{\pen}{pen}%penalite
\DeclareMathOperator{\crit}{crit}%critere empirique
\newcommand{\critFPE}{\crit_{\mathrm{FPE}}}
\newcommand{\critGCV}{\crit_{\mathrm{GCV}}}
\newcommand{\penid}{\pen_{\mathrm{id}}} % penalite ideale
\newcommand{\penmin}{\pen_{\min}} % penalite minimale
\newcommand{\penminlin}{\pen_{\min}^{\mathrm{lin}}} % penalite minimale pour estimateurs lineaires
\newcommand{\penmingala}{\pen_{\min,0}^{\mathrm{gal}}} % penalite minimale dans le cas general, 1ere tentative
\newcommand{\penmingal}{\pen_{\min}^{\mathrm{gal}}} % penalite minimale dans le cas general
\newcommand{\opt}{\mathrm{opt}}
\newcommand{\penopt}{\pen_{\opt}}
\newcommand{\penoptgal}{\pen_{\opt}^{\mathrm{gal}}}
\newcommand{\penoptz}{\pen_{\opt,0}}
\newcommand{\penmult}{\pen^{\mathrm{mult}}}
\newcommand{\ovdelta}{\delta}% ou \overline{\delta}, choisir la notation appropriee

\newcommand{\etamoins}{\eta_n^{-}}
\newcommand{\etaplus}{\eta_n^{+}}

\newcommand{\risksmall}{h} %\mathcal{R}_{\mathrm{small}}
\newcommand{\biaismax}{\mathcal{B}}
\newcommand{\cteProPenminAbove}{K}% Cte apparaissant dans la proposition pbpenmin.gal.above
\newcommand{\cteThmOLS}{\kappa}% Cte apparaissant dans la preuve du Thm OLS

\newtheorem{algo}{Algorithm}
\newtheorem{theorem}{Theorem}
\newtheorem{proposition}{Proposition}
\newtheorem{lemma}{Lemma}
\newtheorem{remark}{Remark}

\newcommand{\algoinput}{\textup{\texttt{Input}}}

\newcommand{\algooutput}{\textup{\texttt{Output}}}

\newcommand{\Seasy}{S^{\mathrm{easy}}}
\newcommand{\Shard}{S^{\mathrm{hard}}}

\newcommand{\pbmult}{\ensuremath{\alpha}}
\newcommand{\pbmulttitre}{\texorpdfstring{\pbmult}{alpha}}% dans un titre de section, etc.
\newcommand{\pbpenmin}{\ensuremath{\beta}}
\newcommand{\pbpenmintitre}{\texorpdfstring{\pbpenmin}{beta}}% dans un titre de section, etc.
\newcommand{\pbpenminplustitre}{\texorpdfstring{\eqref{pb.penmin.Cgrd-Dpt}}{beta+}}% dans un titre de section, etc.
\newcommand{\pbpenminmoinstitre}{\texorpdfstring{\eqref{pb.penmin.Cpt-Dgrd}}{beta-}}% dans un titre de section, etc.
\newcommand{\pbpenmingap}{\ensuremath{\widetilde{\beta}}}
\newcommand{\pbpenminalt}{\ensuremath{\pbpenmin^{\prime}}}
\newcommand{\pbpenminaltgap}{\ensuremath{\widetilde{\pbpenmin}^{\, \prime}}}
\newcommand{\pbpenopt}{\ensuremath{\gamma}}
\newcommand{\pbpenopttitre}{\texorpdfstring{\pbpenopt}{gamma}}% dans un titre de section, etc.

\newcommand{\ch}{\phantom{0}}

\RequirePackage{amssymb}
\renewcommand{\leq}{\leqslant}
\renewcommand{\geq}{\geqslant}

\newcommand{\tabespvert}{\noalign{\vspace*{0.075cm}}}%%% Espacement vertical autour de \hline dans un tableau

\endlocaldefs

\arxiv{1901.07277} % If available

\setmainlanguageenglish
\firstpage{1} % to be set by the editor

\begin{document}

\begin{frontmatter}

  \title{Minimal penalties and the slope heuristics: a survey}
  \runtitle{Minimal penalties and the slope heuristics: a survey}
  \alttitle{P\'enalit\'es minimales et heuristique de pente}
  \begin{aug}
    \auteur{%
      \prenom{Sylvain} \nom{Arlot}%
      \thanksref{t1}%
%      \thanksref{t1,t2}%
      \contact[label=e1]{sylvain.arlot@u-psud.fr}%
    }%
\affiliation[t1]{%
Universit\'e Paris-Saclay, Univ. Paris-Sud, CNRS, Inria, Laboratoire de math\'ematiques d'Orsay, 91405, Orsay, France.
\printcontact{e1}% , \printcontact{u1}
}
    \runauthor{S. Arlot}
  \end{aug}

 \begin{abstract}
Birg\'e and Massart proposed in 2001 the slope heuristics as a way to choose optimally 
from data an unknown multiplicative constant in front of a penalty.
It is built upon the notion of minimal penalty, 
and it has been generalized since to some ``minimal-penalty algorithms''.
This article reviews the theoretical results obtained for such algorithms, 
with a self-contained proof in the simplest framework, precise proof ideas for further generalizations, 
and a few new results.
Explicit connections are made with residual-variance estimators 
---with an original contribution on this topic, showing that for this task 
the slope heuristics performs almost as well as a residual-based estimator 
with the best model choice--- and 
some classical algorithms such as L-curve or elbow heuristics, 
Mallows' $C_p\,$, and Akaike's FPE. 
Practical issues are also addressed, including two new practical definitions of minimal-penalty algorithms that are compared on synthetic data to previously-proposed definitions. 
Finally, several conjectures and open problems are suggested 
as future research directions.
  \end{abstract}


\begin{altabstract}
Birg\'e et Massart ont propos\'e en 2001 l'heuristique de pente, 
pour d\'eterminer \`a l'aide des donn\'ees 
une constante multiplicative optimale devant une p\'enalit\'e en s\'election de mod\`eles. 
Cette heuristique s'appuie sur la notion de p\'enalit\'e minimale, 
et elle a depuis \'et\'e g\'en\'eralis\'ee en 
``algorithmes \`a base de p\'enalit\'es minimales''. 
Cet article passe en revue les r\'esultats th\'eoriques obtenus sur ces algorithmes, 
avec une preuve compl\`ete dans le cadre le plus simple, 
des id\'ees de preuves pr\'ecises pour g\'en\'eraliser ce r\'esultat au-del\`a des cadres 
d\'ej\`a \'etudi\'es, 
et quelques r\'esultats nouveaux. 
Des liens sont faits avec les m\'ethodes d'estimation de la variance r\'esiduelle 
(avec une contribution originale sur ce th\`eme, qui d\'emontre que 
l'heuristique de pente produit un estimateur de la variance 
quasiment aussi bon qu'un estimateur fond\'e sur les r\'esidus d'un mod\`ele oracle) 
ainsi qu'avec plusieurs algorithmes classiques tels que les heuristiques de coude 
(ou de courbe en L), 
$C_p$ de Mallows et FPE d'Akaike. 
Les questions de mise en \oe uvre pratique sont \'egalement \'etudi\'ees, 
avec notamment la proposition de deux nouvelles d\'efinitions pratiques pour des algorithmes 
\`a base de p\'enalit\'es minimales 
et leur comparaison aux d\'efinitions pr\'ec\'edentes sur des donn\'ees simul\'ees. 
Enfin, des conjectures et probl\`emes ouverts sont propos\'es comme 
pistes de recherche pour l'avenir. 
\end{altabstract}

  \begin{altkeywords}
    \kwd{s\'election de mod\`eles}
    \kwd{s\'election d'estimateurs}%
    \kwd{p\'enalisation}%
    \kwd{heuristique de pente}%
    \kwd{p\'enalit\'e minimale}%
    \kwd{estimation de la variance r\'esiduelle}%
    \kwd{heuristique de courbe en L}%
    \kwd{heuristique de coude}%
    \kwd{test scree}%
    \kwd{surp\'enalisation}%
  \end{altkeywords}

  \begin{AMSclass}
    \kwd{62-02}%%62-02   	Statistics -- Research exposition (monographs, survey articles)
    \kwd{62G05}%% 62G05   		Nonparametric inference -- Estimation
    \kwd{62J05}%%62J05   	Linear regression
%%    \kwd{62G08}%%62G08   	Nonparametric inference -- Nonparametric regression
  \end{AMSclass}

\end{frontmatter}

\setcounter{footnote}{1} % to set manually (equal to the number of institutes minus one)

%%%%%%%%%%%%%%%%%%%%%%%%%%%%%%%%%%%%%%%%%%%%%%%%%%%%%%%%%%%%%%%%%%%%%%%%%%%%%%%%%%%%
%%%%%%%%%%%%%%%%%%%%%%%%%%%%%%%%%%%%%%%%%%%%%%%%%%%%%%%%%%%%%%%%%%%%%%%%%%%%%%%%%%%%

\tableofcontents

%%%%%%%%%%%%%%%%%%%%%%%%%%%%%%%%%%%%%%%%%%%%%%%%%%%%%%%%%%%%%%%%%%%%%%%%%%%%%%%%%%%%
%%%%%%%%%%%%%%%%%%%%%%%%%%%%%%%%%%%%%%%%%%%%%%%%%%%%%%%%%%%%%%%%%%%%%%%%%%%%%%%%%%%%

\section{Introduction} \label{sec.intro}

Model selection attracts much attention in statistics since more than forty years \citep{Aka:1973,Mal:1973,Bur_And:2002,Mas:2003:St-Flour}.
A related and crucial question for machine learning is the data-driven choice of hyperparameters of learning algorithms.
Both are particular instances of the estimator-selection problem: given a family of estimators, how to choose from data one among them whose risk is as small as possible?

One of the main strategies proposed for estimator (or model) selection is penalization, 
that is, choosing the estimator minimizing the sum of its empirical risk 
---how well it fits the data--- and some penalty term ---whose role is to avoid overfitting.
Optimal penalties often depend on at least one parameter whose data-driven choice is challenging.
In the early 2000s, \citet{Bir_Mas:2001,Bir_Mas:2006} pointed out two key facts 
leading to a novel approach for an optimal data-driven choice of multiplicative constants 
in front of penalties.
Birg\'e and Massart were considering a rather theoretical question: 
what is the \emph{minimal} amount of penalization needed for avoiding a strong overfitting?
For least-squares estimators in regression, they noticed that
(i) the minimal penalty is equal to half the optimal penalty, and
(ii) the minimal penalty is observable.
These two facts are called ``the slope heuristics''\footnote{In ``the slope heuristics'', the word ``heuristics'' is an \emph{uncountable} noun, following the Oxford Advanced Learner's Dictionary. One could also write ``the slope heuristic'', according to some other English dictionaries in which ``heuristic'' appears as a noun. We use the former spelling throughout this article, but some other articles make use of the latter spelling (without the final~s). }  
and lead to an algorithm for choosing multiplicative constants in front of penalties.

These ideas and the corresponding algorithm have been generalized since to several frameworks 
(see Section~\ref{sec.penmingal}--\ref{sec.theory} and \ref{sec.empirical}), 
with numerous applications in various fields such as 
biology \citep{Rey_Sch:2010,Aka:2011,Bon_Tou:2010,Rau_Mau_Mag_Cel:2015,Dev_Gal_Per:2017,Dev_Gal:2016}, 
energy \citep{Mic:2008:phd,Dev_Gou_Pog:2015:journal}, 
or text analysis \citep{Der_LeP:2017};  
Section~\ref{sec.empirical.conjectures} provides more examples of applications. 

In particular, for linear estimators in regression, 
the original slope heuristics does not work directly 
and can be modified successfully into a more general ``minimal-penalty algorithm'' 
\citep{Arl_Bac:2009:minikernel_nips,Arl_Bac:2009:minikernel_long_v2} 
detailed in Section~\ref{sec.penmingal}.

For least-squares regression with projection or linear estimators, 
the slope heuristics also provides a residual-variance estimator with nice properties 
(Section~\ref{sec.related.variance}). 
In the general setting, 
the slope heuristics can also be seen as a way to give proper mathematical grounds to 
``L-curve'' or ``elbow-heuristics'' algorithms that are used 
for choosing regularization parameters in ill-posed problems \citep{Han_OLe:1993}, 
as explained in Sections~\ref{sec.related.elbow}--\ref{sec.related.scree}.

\paragraph{Goals} 
The goals of this survey are the following:
\begin{enumerate}
\item 
to review recent theoretical results about the slope heuristics, 
and more generally about all minimal-penalty algorithms (Sections~\ref{sec.slopeOLS}--\ref{sec.theory}); 
\item 
to help identifying how ---and under which assumptions--- 
such results could be generalized to other settings, 
possibly with new algorithms, by giving a precise account of existing proofs 
(Sections~\ref{sec.slopeOLS.proofs}, \ref{sec.theory.approach}, and~\ref{sec.hints}),  
and by identifying several conjectures and open problems suggested by experimental results 
(Section~\ref{sec.empirical}); 
\item 
to make connections between minimal penalties and other classical procedures 
for residual-variance estimation and for model or estimator selection (Section~\ref{sec.related}).
\end{enumerate}
Practical issues are only briefly mentioned in Section~\ref{sec.practical}, 
since more details can be found on these in the survey by \citet{Bau_Mau_Mic:2010}.

There is currently no final answer to the question of generalizing minimal-penalty algorithms 
as much as possible, but we hope that this survey will motivate 
further theoretical and empirical work in this direction, 
which could have a great practical impact in statistics, machine learning, 
and data science in general.

\paragraph{Contributions} 
Let us finally point out some original results appearing in this article. 
In the framework of least-squares fixed-design regression with projection estimators 
and Gaussian noise, 
Theorem~\ref{thm.OLS} validates the slope heuristics in a stronger sense 
compared to previous results \citep{Bir_Mas:2006};  
it is inspired by \citet{Arl_Bac:2009:minikernel_long_v2} but makes weaker assumptions. 
Its extension to sub-Gaussian noise (Remark~\ref{rk.thm.OLS.subgaussian}  in Section~\ref{sec.slopeOLS.math}) is original. 
As a corollary, Proposition~\ref{pro.variance-estim} in Section~\ref{sec.related.variance} 
is the first precise statement on a slope-heuristics-based residual-variance estimator 
---more precise than the result that can be derived from \citet{Arl_Bac:2009:minikernel_long_v2}---, 
showing that it is minimax optimal (up to $\log(n)$ factors) under mild assumptions. 
Proposition~\ref{pro.variance-estim} provides non-asymptotic bounds 
(in expectation and with high probability) on this residual-variance estimator, 
that can be seen as some kind of oracle inequality for residual-variance estimation, 
which is interesting independently from the slope heuristics. 

In the general framework, Propositions~\ref{pro.pbpenmin.gal.below}--\ref{pro.pbpenmin.gal.above} 
in Section~\ref{sec.theory.hint.penmin} propose two general approaches 
for justifying minimal-penalty algorithms. 
These approaches were previously proposed in specific settings \citep{Ler_Tak:2011,Gar_Ler:2011}, 
%%% je cite d'abord Ler_Tak:2011 car le preprint date de 06/2011, tandis que Gar_Ler:2011 date de 11/2011
but their generalization to the setting of Section~\ref{sec.penmingal.framework} is new. 
For instance, the application of Proposition~\ref{pro.pbpenmin.gal.below} 
to general minimum-contrast estimators with a bounded contrast is new, 
to the best of our knowledge. 

On the practical side, as a complement to the survey by \citet{Bau_Mau_Mic:2010}, 
Section~\ref{sec.practical} shows original numerical experiments on synthetic data, 
assessing the performance of the slope heuristics in the least-squares 
regression framework, 
for both residual-variance estimation and model selection. 
Two new practical definitions of the slope heuristics 
(called `median' and `consensus') are proposed and compared to the classical ones. 
An efficient implementation of one previously-proposed definition 
is also provided and proved (Algorithm~\ref{algo.window} and Proposition~\ref{pro.algo.window} 
in Appendix~\ref{app.algos.window}).

%%%%%%%%%%%%%%%%%%%%%%%%%%%%%%%%%%%%%%%%%%%%%%%%%%%%%%%%%%%%%%%%%%%%%%%%%%%%%%%%%%%%
%%%%%%%%%%%%%%%%%%%%%%%%%%%%%%%%%%%%%%%%%%%%%%%%%%%%%%%%%%%%%%%%%%%%%%%%%%%%%%%%%%%%

\section{The slope heuristics} \label{sec.slopeOLS}
This section presents the original ``slope heuristics'' \citep{Bir_Mas:2001,Bir_Mas:2006} in the framework of fixed-design regression, with the least-squares risk and projection estimators.
By focusing on this framework, we get most of the flavor of the slope heuristics while keeping the exposition simple.

\subsection{Framework} \label{sec.slopeOLS.framework}
The framework considered in Section~\ref{sec.slopeOLS} is the following.
We observe
\begin{equation}
\label{eq.OLS.model}
Y = F + \varepsilon \in \R^n
\end{equation}
where $\varepsilon_1, \ldots, \varepsilon_n$ are independent and identically distributed with mean 0 and variance $\sigma^2$,
and $F \in \R^n$ is some (deterministic) signal of interest.
For instance, $F$ can be equal to $(f(x_i))_{1\leq i \leq n}$ for some deterministic design points $x_1, \ldots, x_n \in \X$ and $f$ some unknown measurable function $\X \mapsto \R$, with no assumption on the set $\X$.

The goal is to reconstruct $F$ from $Y$, that is, to find some $t \in \R^n$ such that its quadratic risk
\[ \frac{1}{n} \norm{t-F}^2 \]
is small, where for every $u \in \R^n$, $\norm{u}^2=\sum_{i=1}^n u_i^2$.
To this end, for every linear subspace $S$ of $\R^n$, the \emph{projection estimator} or \emph{least-squares estimator} on $S$ is defined as
\[ \Fh_S \in \argmin_{t \in S} \set{ \frac{1}{n} \norm{t-Y}^2 } \]
where $n^{-1} \norm{t-Y}^2$  is called the empirical risk of $t$.
Since $S$ is a linear subspace, $\Fh_S$ exists and is unique:
$\Fh_S = \Pi_S Y$
where $\Pi_S : \R^n \to \R^n$ denotes the orthogonal projection onto $S$.
In the following, any linear subspace $S$ of $\R^n$ is called a \emph{model}.

%%% Model selection
Let $(S_m)_{\mM}$ be some collection of models, and for every $\mM$, let
\[ \Fhm = \Fh_{S_m} = \Pi_{S_m} Y \quad \text{and} \quad  \Pi_m = \Pi_{S_m} \, . \]
In this survey, we assume that the goal of model selection is to choose from data some $\mh \in \M$ such that the quadratic risk of $\Fhmh$ is minimal.
The best choice would be the \emph{oracle}:
\[ \mo \in \argmin_{\mM} \set{\frac{1}{n} \norm{\Fhm - F}^2 } \, , \]
which cannot be used since it depends on the unknown signal $F$.
Therefore, the goal is to define a data-driven $\mh$ satisfying an \emph{oracle inequality}
\begin{equation} \label{eq.oracle}
\frac{1}{n} \norm{\Fhmh - F}^2 \leq K_n \inf_{\mM} \set{ \frac{1}{n} \norm{\Fhm - F}^2 } + R_n 
\end{equation}
with large probability, where the leading constant $K_n$ should be close to~1 
---at least for large $n$--- and the remainder term $R_n$ should be small 
compared to the oracle risk $\inf_{\mM} \sets{ n^{-1} \norms{\Fhm - F}^2 }$.

\subsection{Optimal penalty} \label{sec.slopeOLS.optimal}
Many classical selection methods are built upon the ``unbiased risk estimation'' heuristics: If $\mh$ minimizes a criterion $\crit(m)$ such that
\[ \forall \mM, \qquad \E\crochb{ \crit(m) } \approx \E\crochj{ \frac{1}{n} \norm{\Fhm -F}^2 } \, , \]
then $\mh$ satisfies with large probability an oracle inequality such as Eq.~\eqref{eq.oracle} with an optimal constant $K_n = 1+\petito(1)$.
This can be proved by showing a concentration inequality for $\norms{ \Pi_m \varepsilon}^2$ and $\prodscal{\varepsilon}{(\Id_n - \Pi_m) F}$ around their expectations for all $\mM$, where $\Id_n$ denotes the identity matrix of $\R^n$, see Section~\ref{sec.slopeOLS.proofs}.
For instance, %
cross-validation \citep{All:1974,Sto:1974} and generalized cross-validation \citep[GCV;][]{Cra_Wah:1979}
are built upon this heuristics.

One way of implementing this heuristics is penalization, which consists of minimizing the sum of the empirical risk and a penalty term, that is, using a criterion of the form:
\begin{equation}
\label{eq.crit}
 \crit(m) = \frac{1}{n} \norm{ \Fhm - Y }^2 + \pen(m) \, .
\end{equation}

The unbiased risk estimation heuristics, also called Mallows' heuristics, then leads to the \emph{optimal \textup{(}deterministic\textup{)} penalty}
\begin{equation}
\label{eq.penoptz}
\penoptz(m) \egaldef \E\croch{ \frac{1}{n} \norm{ \Fhm - F }^2}  - \E \croch{ \frac{1}{n} \norm{ \Fhm - Y }^2 }
\, .
\end{equation}
When $\Fhm = \Pi_m Y$, we have 
\begin{align} \label{eq.riskFhm}
  \norm{ \Fhm - F }^2 &= \normb{ (\Id_n - \Pi_m) F}^2 + \norms{ \Pi_m \varepsilon}^2  \\ \label{eq.riskempFhm}
\text{and} \qquad 
\norm{ \Fhm - Y }^2 &= \norm{ \Fhm - F }^2 + \norms{\varepsilon}^2 - 2 \prodscals{\varepsilon}{\Pi_m \varepsilon} + 2 \prodscalb{\varepsilon}{(\Id_n - \Pi_m) F} \, ,
  \end{align}
where $\forall t,u \in \R^n$, $\prodscals{t}{u} = \sum_{i=1}^n t_i u_i\,$.
Since the $\varepsilon_i$ are independent, centered, with variance $\sigma^2$,
Eq.~\eqref{eq.riskFhm} and Eq.~\eqref{eq.riskempFhm} imply that
\begin{align}
\label{eq.EriskFhm}
\E\croch{\frac{1}{n} \norm{ \Fhm - F }^2} &= \frac{1}{n} \normb{ (\Id_n - \Pi_m) F}^2 + \frac{\sigma^2 D_m} {n} \, , 
\\
\label{eq.EriskempFhm}
\E\croch{\frac{1}{n} \norm{ \Fhm - Y }^2} &= \frac{1}{n} \normb{ (\Id_n - \Pi_m) F}^2 + \frac{\sigma^2 \paren{n-D_m}} {n} \, , 
 \\
\label{eq.penid}
\text{and} \qquad 
\penoptz(m) + \sigma^2
&= \frac{ 2 \sigma^2 D_m} {n} 
=: \penopt(m) \, ,
\end{align}
where $D_m \egaldef \dim(S_m)$.
Note that the optimal penalties \eqref{eq.penid} and \eqref{eq.penoptz} differ by an additive constant $\sigma^2$, which does not change the argmin of the penalized criterion \eqref{eq.crit}; this choice simplifies formulas involving $\penopt\,$.

Eq.~\eqref{eq.EriskFhm} is classically known as a \emph{bias-variance decomposition of the risk}: the first term ---called \emph{approximation error} or bias--- decreases when $S_m$ gets larger, while the second term ---called \emph{estimation error} or variance--- increases when $S_m$ gets larger, see Figure~\ref{fig.OLS.Erisk+critC} left.
Eq.~\eqref{eq.EriskempFhm} shows that the expectation of the empirical risk decreases when $S_m$ gets larger, as expected since $\Fhm$ is defined as a minimizer of the empirical risk, see Figure~\ref{fig.OLS.Erisk+critC} left.

\begin{figure}
\begin{center}
\includegraphics[width=0.455\textwidth]{\pathfig/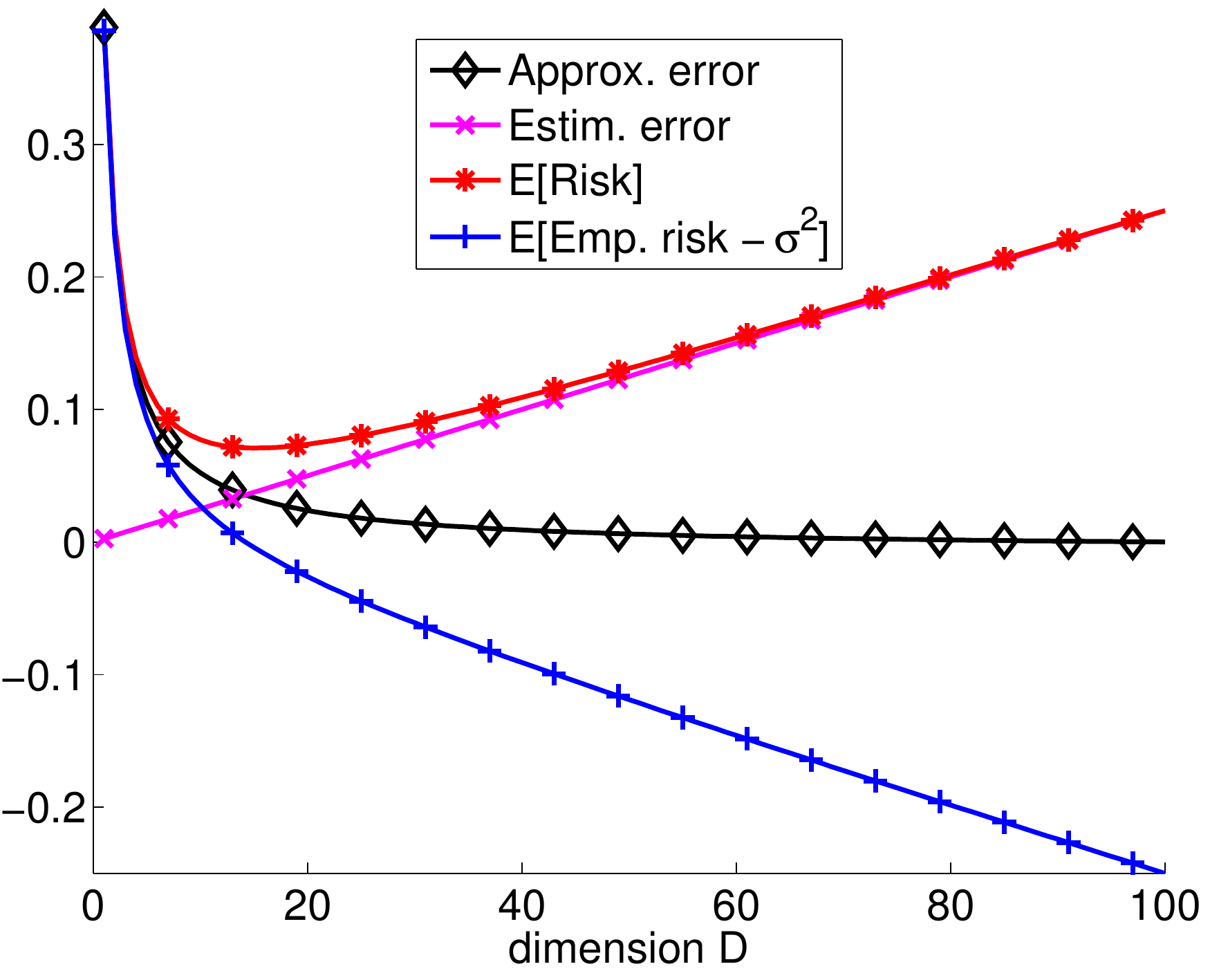}
%%% resolution 504x408
%
\includegraphics[width=0.505\textwidth]{\pathfig/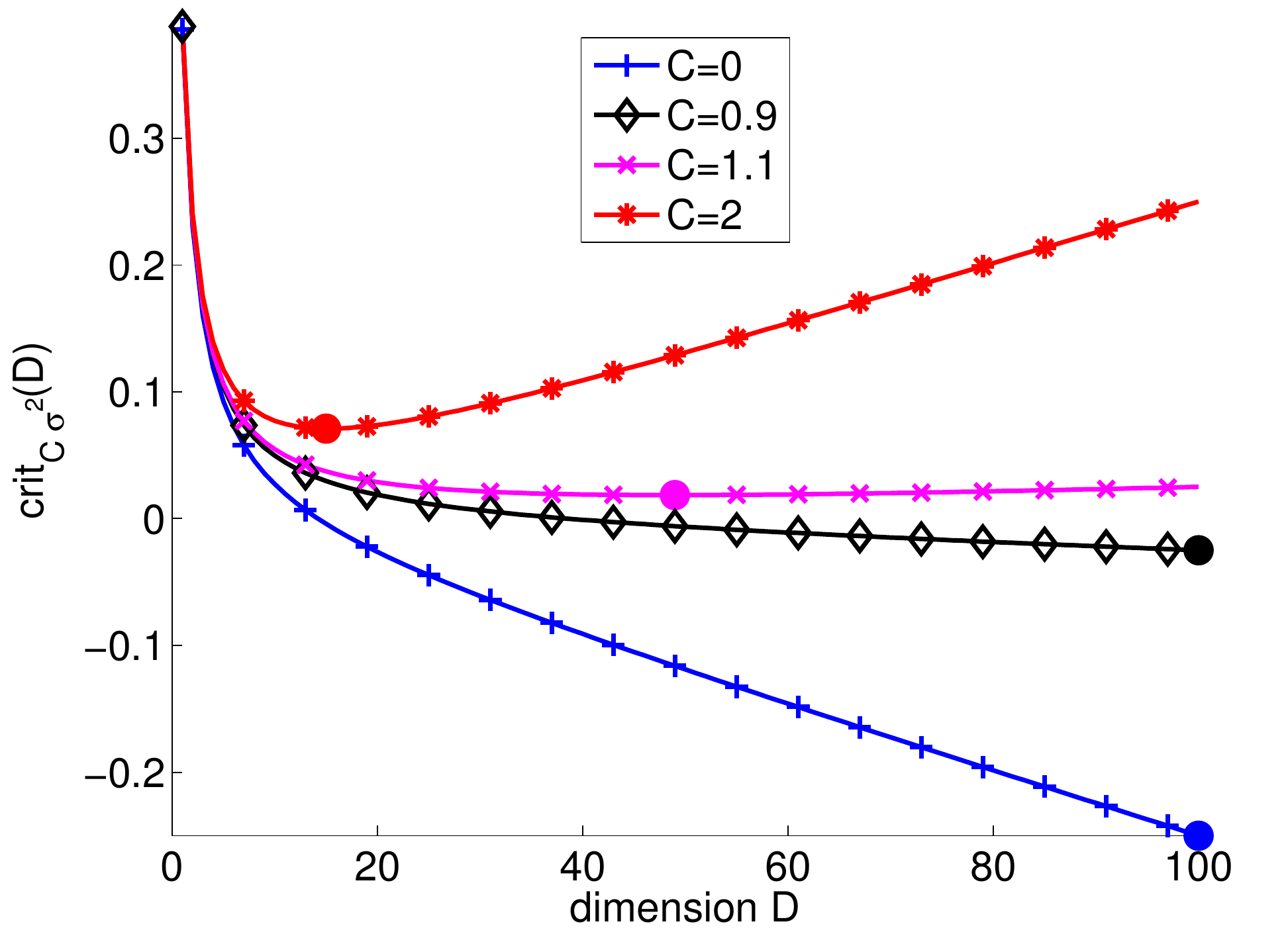}
%%% resolution 560x408 => semble moins large si on met les 2 graphes a .48\textwidth.
%
\caption{\label{fig.OLS.Erisk+critC}
Left\textup{:} Expectations of the risk and empirical risk, bias-variance decomposition of the risk.
Right\textup{:} $\crit_{C\sigma^2}(m)$, defined by Eq.~\eqref{eq.critC}, for $C \in \set{0, 0.9, 1.1, 2}$\textup{;}
its minimal value at $\mo(C\sigma^2)$ is shown by a plain dot.
`Easy setting', see Appendix~\ref{app.details-simus} for detailed information.
%Framework: $F_i \propto i^{-1}$, $n^{-1} \norm{F}^2 = 1$, $\sigma^2 = 0.25$. %(P1G2)
}
\end{center}
\end{figure}

The expression  of the optimal penalty in Eq.~\eqref{eq.penid} leads to Mallows' $C_p$ \citep{Mal:1973}, where $\sigma^2$ is replaced by some estimator $\widehat{\sigma^2}$.
Several approaches exist for estimating $\sigma^2$, see Section~\ref{sec.related.variance}.
The slope heuristics provides a data-driven estimation of the unknown constant $\sigma^2$ in front of the penalty shape $D_m/n$ thanks to the notion of minimal penalty.

\subsection{Minimal penalty and the slope heuristics} \label{sec.slopeOLS.penmin}

Eq.~\eqref{eq.penid} shows that the shape $\pen_1(m)=D_m/n$ of the optimal penalty is known, even when $\sigma^2$ is unknown.
A natural question is to determine the minimal value of the constant that should be put in front of $\pen_1(m)$. More precisely, if for every $C \geq 0$ 
\begin{equation} \label{eq.mhC}
\mh(C) \in \argmin_{\mM} \setj{ \frac{1}{n} \norm{\Fhm - Y}^2 + C \frac{D_m}{n} }
\, ,
\end{equation}
what is the minimal value of $C$ such that $\mh(C)$ stays a \guil{reasonable} choice, that is, avoids strong overfitting, or equivalently, satisfies an oracle inequality like Eq.~\eqref{eq.oracle} with $K_n = \grandO(1)$ as $n$ tends to infinity?

In order to understand how $\mh(C)$ behaves as a function of $C$, let us consider, for every $C \geq 0$,
\begin{gather}
\mo(C) \in \argmin_{\mM} \setj{ \E\crochj{ \frac{1}{n} \norm{\Fhm - Y}^2 + C \frac{D_m}{n} } }
= \argmin_{\mM} \setb{ \crit_C(m)  }
\notag
\\
\text{with} \quad \crit_C(m) \egaldef \frac{1}{n} \crochj{ \normb{ (\Id_n - \Pi_m) F}^2 + ( C - \sigma^2) D_m }
\, ,
\label{eq.critC}
\end{gather}
by Eq.~\eqref{eq.EriskempFhm}.
Provided that we can prove some uniform concentration inequalities for $\norms{\Fhm - Y}^2$, $m \in \M$, we can expect $\mo(C)$ to be close to $\mh(C)$. 
Let us assume that for $D_m$ large enough the approximation error $n^{-1} \norms{ (\Id_n - \Pi_m) F}^2 $ is almost constant. 
For simplicity, let us also assume that the approximation error is a decreasing function of $D_m$ ---which holds for instance if the $S_m$ are nested. 
Then, two cases can be distinguished with respect to $C\,$:
\begin{itemize}
\item if $C < \sigma^2$, then $\crit_C(m)$ is a decreasing function of $D_m\,$, and $D_{\mo(C)}$ is huge: $\mo(C)$ overfits.
\item if $C > \sigma^2$, then $\crit_C(m)$ increases with $D_m$ for $D_m$ large enough, so $D_{\mo(C)}$ is much smaller. 
\end{itemize}
This behavior is illustrated on the right part of Figure~\ref{fig.OLS.Erisk+critC}.
In other words, 
\begin{equation}
\label{eq.penmin}
\penmin(m) \egaldef \frac{\sigma^2 D_m}{n}
\end{equation} 
seems to be the minimal amount of penalization needed so that a minimizer $\mh$ of the penalized criterion \eqref{eq.crit} does not clearly overfit.
The above arguments are made rigorous in Section~\ref{sec.slopeOLS.math}, showing that
$\sigma^2 D_m / n$ is indeed a minimal penalty in the current framework. 

\medbreak

We can now summarize the slope heuristics into two major facts.
First, from Eq.~\eqref{eq.penid} and~\eqref{eq.penmin}, we get a \emph{relationship between the optimal and minimal penalties}:
\begin{equation} \label{eq.OLS.slope}
\penopt(m) = 2 \penmin(m)
\, . \end{equation}
Second, \emph{the minimal penalty is observable}, since $D_{\mh(C)}$ decreases \guil{smoothly} as a function of $C$ everywhere except around $C=\sigma^2$ where it jumps.

\subsection{Data-driven penalty algorithm} \label{sec.slopeOLS.algo}
The two major facts of the slope heuristics described above directly lead to a data-driven penaltization algorithm, which can be formalized in two ways.
\subsubsection{Dimension jump} \label{sec.slopeOLS.algo.jump}
First, we can estimate the minimal penalty by looking for a jump of $C \mapsto D_{\mh(C)}\,$, and make use of Eq.~\eqref{eq.OLS.slope} to get an estimator of the optimal penalty.
\begin{algo}[Slope-heuristics algorithm, jump formulation] \label{algo.OLS.jump}
\algoinput\textup{:} $\parenb{ \norms{ \Fhm - Y }^2 }_{\mM}\,$.
\begin{enumerate}
\item Compute $(\mh(C))_{C \geq 0}\,$, where $\mh(C)$ is defined by Eq.~\eqref{eq.mhC}.
\item Find $\Chjumpgal>0$ corresponding to the ``unique large jump'' of $C \mapsto D_{\mh(C)}\,$.
\item Select $\mhAlgA \in \argmin_{\mM} \setb{ n^{-1} \norms{ \Fhm - Y }^2 + 2 \Chjumpgal D_m / n }$.
\end{enumerate}
\algooutput\textup{:} $\mhAlgA\,$.
\end{algo}
The left part of Figure~\ref{fig.OLS.algo} shows one instance of the plot of $C \mapsto D_{\mh(C)}\,$, with one clear jump corresponding to $\Chjumpgal\,$.
Computational issues are discussed in Section~\ref{sec.practical.cost}; in particular, step~1 of Algorithm~\ref{algo.OLS.jump} can be done efficiently, see Appendix~\ref{app.algos.path}.
Step~2 of Algorithm~\ref{algo.OLS.jump} can be done in several ways, see Section~\ref{sec.practical.jump-vs-slope}.
The practical problems arising with step~2 of Algorithm~\ref{algo.OLS.jump} can motivate the use of an alternative algorithm that we detail below.

\begin{figure}
\begin{center}
\includegraphics[width=0.49\textwidth]{\pathfig/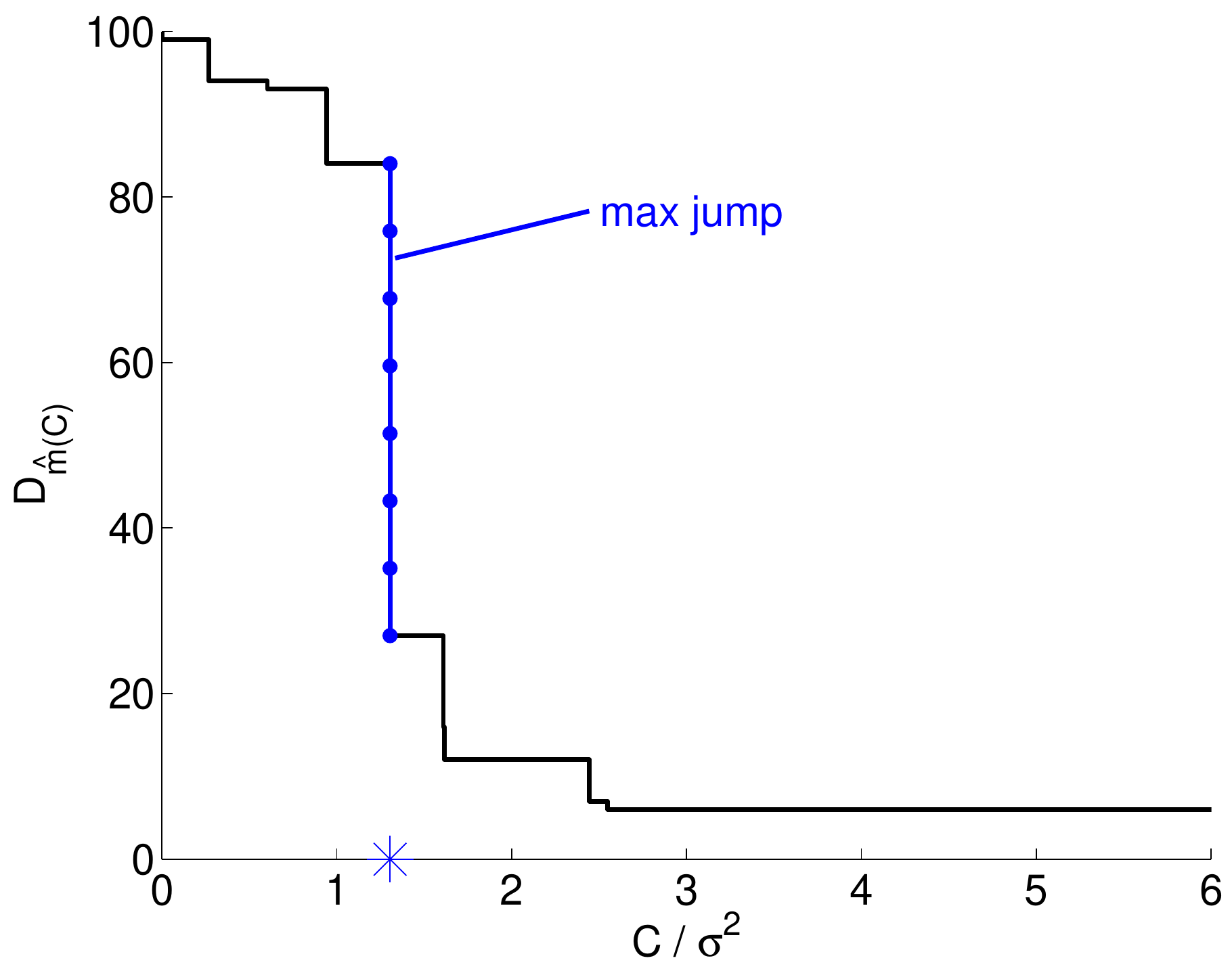}
\includegraphics[width=0.49\textwidth]{\pathfig/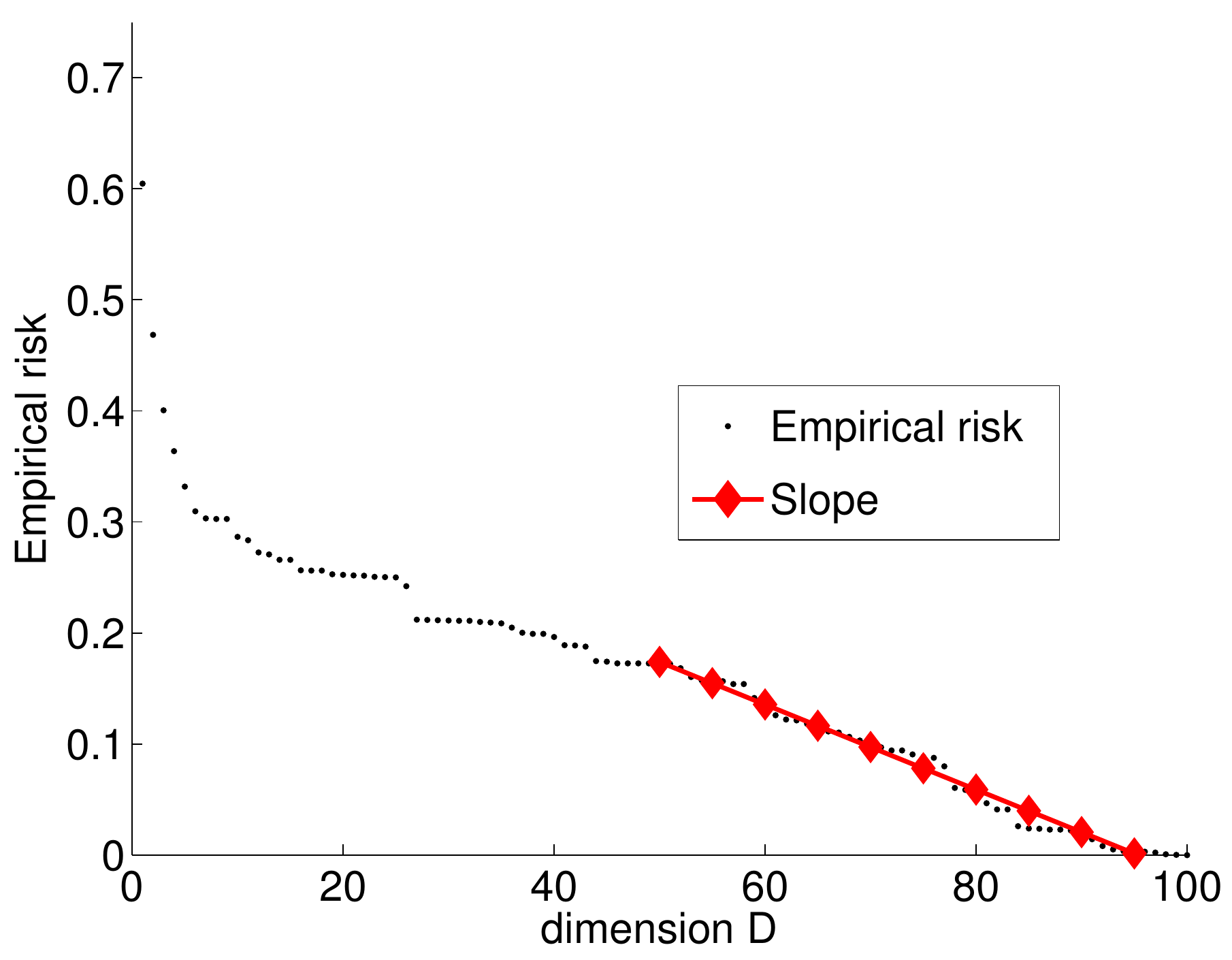}
\caption{\label{fig.OLS.algo}
Illustration of Algorithms~\ref{algo.OLS.jump} and~\ref{algo.OLS.slope} on the same sample
\textup{(}`easy' setting,
see Appendix~\ref{app.details-simus} for details\textup{)}.
Left\textup{:} Plot of $C \mapsto D_{\mh(C)}$ and visualization of $\Chjumpgal\,$.
Right\textup{:} Plot of $D_m \mapsto n^{-1}\norms{Y-\Fh_m}^2$ and visualization of $-\Chslope/n$.
}
\end{center}
\end{figure}

\subsubsection{Slope estimation} \label{sec.slopeOLS.algo.slope}
As explained in Section~\ref{sec.slopeOLS.penmin}, the reason why $D_{\mh(C)}$ jumps around $C \approx \sigma^2$ is that by Eq.~\eqref{eq.EriskempFhm},
\[ \frac{1}{n} \E\croch{\norm{ \Fhm - Y }^2} = \frac{a(m) - \sigma^2  D_m}{n} \]
where
$a(m) \egaldef  \norms{ (\Id_n - \Pi_m) F}^2 + n \sigma^2$. 
Let us assume that $a(m)$ ---or equivalently, the approximation error--- is 
almost constant for all $m$ such that $D_m$ is large enough.
Then, considering only models with a large dimension, the empirical risk approximately has a linear behavior as a function of $D_m\,$, with slope $-\sigma^2/n$.
Since the empirical risk is observable, one can estimate this slope in order to get an estimator of $\sigma^2$, and plug it in the optimal penalty given by Eq.~\eqref{eq.penid}.

\begin{algo}[Slope-heuristics algorithm, slope formulation] \label{algo.OLS.slope}
\algoinput\textup{:} $(\norms{ \Fhm - Y }^2)_{\mM}\,$.
\begin{enumerate}
\item Estimate the slope $\widehat{S}$ of $\norms{ \Fhm - Y }^2$ as a function of $D_m$ for all $\mM$ with $D_m$ ``large enough'', for instance by \textup{(}robust\textup{)} linear regression, and define $\Chslope = - n \widehat{S}$.
\item Select $\mhAlgB \in \argmin_{\mM} \sets{ n^{-1} \norms{ \Fhm - Y }^2 + 2 \Chslope D_m / n }$.
\end{enumerate}
\algooutput\textup{:} $\mhAlgB\,$.
\end{algo}
The right part of Figure~\ref{fig.OLS.algo} shows an instance of the plot of $n^{-1} \norms{ \Fhm - Y }^2$ as a function of~$D_m\,$.
Algorithm~\ref{algo.OLS.slope} relies on the choice of what is a ``large enough'' dimension, on how the slope $\Chslope$ is estimated, 
and on the assumption that the approximation error is almost constant among large models ---otherwise it can fail strongly, as shown in Section~\ref{sec.practical.jump-vs-slope}. 
Therefore, Algorithms~\ref{algo.OLS.jump} and~\ref{algo.OLS.slope} both have pros and cons, 
and there is no universal choice between them.
The links between Algorithms~\ref{algo.OLS.jump} and~\ref{algo.OLS.slope}, as well as their differences, 
are discussed in Section~\ref{sec.practical.jump-vs-slope}. 

\subsection{What can be proved mathematically} \label{sec.slopeOLS.math}
A major interest of the slope heuristics is that it can be made rigorous.
For instance, we prove in Section~\ref{sec.slopeOLS.proofs} the next theorem.
\begin{theorem} \label{thm.OLS}
In the framework described in Section~\ref{sec.slopeOLS.framework}, assume that $\M$ is finite, 
contains at least one model of dimension at most $n/20$, and that
\begin{gather}
\label{hyp.thm.OLS.Id}
\tag{\ensuremath{\mathbf{HId}}}
\exists m_1 \in \M \, , \quad S_{m_1} = \R^n
\\
\label{hyp.thm.OLS.Gauss}
\tag{\ensuremath{\mathbf{HG}}}
\text{and} \qquad \varepsilon \sim \mathcal{N}(0,\sigma^2 \Id_n)
\, .
\end{gather}
Recall that for every $C \geq 0$, $\mh(C)$ is defined by Eq.~\eqref{eq.mhC}.
Then, for every $\gamma \geq 0$, some $n_0(\gamma)$ exists such that if $n \geq n_0(\gamma)$, with probability at least $1 - 4 \card(\M) n^{-\gamma}$, the following inequalities hold simultaneously\textup{:}
\begin{eqnarray}
\label{eq.thm.OLS.Cpt-Dgrd}
&\hspace{-1cm}\forall C \leq \parenb{ 1 - \etamoins } \sigma^2
\, ,
\ 
&\hspace{1.5cm}D_{\mh(C)} \geq \frac{9n}{10} 
\, , 
\\
\label{eq.thm.OLS.Cpt-Rgrd}
&\hspace{-1cm}\forall C \leq \parenb{ 1 - \etamoins } \sigma^2
\, ,
\ 
&\frac{1}{n} \norm{\Fh_{\mh(C)} - F}^2 \geq \frac{7 \sigma^2}{8}
\, ,
\\
\label{eq.thm.OLS.Cgrd-Dpt}
&\hspace{-1cm}\forall C \geq \parenb{ 1 + \etaplus } \sigma^2
\, ,
\ 
&\hspace{1.5cm}D_{\mh(C)} \leq \frac{n}{10}
\, ,
\\
\label{eq.thm.OLS.Cgrd-Rpt}
&\hspace{-1cm}\forall C > \sigma^2  \, ,
\ 
&\frac{1}{n} \norm{\Fh_{\mh(C)} - F}^2 \leq \risksmall\parenj{\frac{C}{\sigma^2}} 
\crochj{ \inf_{\mM} \setj{ \frac{1}{n} \norm{\Fhm - F}^2 }
+ \frac{20 \sigma^2 \gamma \log(n)}{n} }
\, ,
\end{eqnarray}
and for every $\eta \in (0,1/2]$ and $C \in \crochb{ (2-\eta) \sigma^2 , (2+\eta) \sigma^2 }$,
\begin{gather}
\label{eq.thm.OLS.oracle}
\frac{1}{n} \norm{ \Fh_{\mh(C)} - F}^2 \leq (1+3\eta) \inf_{\mM} \set{ \frac{1}{n} \norm{ \Fhm - F}^2 } + \frac{880 \sigma^2 \gamma \log(n)}{\eta n}
\, ,
\\
\notag
\text{where} \qquad 
\sigma^2 \etaplus = 40 \inf_{\mM \,/\, D_m \leq n/20} \setj{\frac{1}{n} \normb{ (\Id_n - \Pi_m) F}^2 } + 82 \sigma^2 \sqrt{\frac{\gamma \log(n)}{n}}
\, , 
\\ 
\notag 
\etamoins = 41 \sqrt{\frac{\gamma \log(n)}{n}}
\, ,
\qquad 
\text{and}
\qquad
\forall u > 1 \, , \,
\risksmall\parens{u}  =  \frac{10}{(u-1)^4} \un_{u \in (1,2)} + u^3 \un_{u \geq 2} 
\, .
\end{gather}
\end{theorem}
Theorem~\ref{thm.OLS} revisits results first obtained by \citet{Bir_Mas:2006}, 
formulating them similarly to \citet{Arl_Bac:2009:minikernel_long_v2} but with milder assumptions. 

\paragraph{What Theorem~\ref{thm.OLS} proves about Algorithms~\ref{algo.OLS.jump}--~\ref{algo.OLS.slope}}
Eq.~\eqref{eq.thm.OLS.Cpt-Dgrd} and~\eqref{eq.thm.OLS.Cgrd-Dpt} do not show exactly that there is a single large jump in $C \mapsto D_{\mh(C)}\,$, as in the heuristic reasoning of Section~\ref{sec.slopeOLS.penmin}.
We cannot hope to prove it since numerical experiments show that the global jump of $D_{\mh(C)}$ can be split into several small jumps within a small interval of values of $C$, see Figure~\ref{fig.DmhC.easy.ech5} in Section~\ref{sec.practical.jump-vs-slope}.
Nevertheless, Eq.~\eqref{eq.thm.OLS.Cpt-Dgrd} and~\eqref{eq.thm.OLS.Cgrd-Dpt} imply that the variation of $D_{\mh(C)}$ over a geometric window of $C$ is extremely strong around $\sigma^2$: if $\Chjumpgal$ in Algorithm~\ref{algo.OLS.jump} is defined as
\begin{equation} 
\label{def.Chwin} 
\Chwindow = \Chwindow(\eta) \in \argmax_{C > 0} \set{ D_{\mh(C/[1+\eta])} - D_{\mh(C[1+\eta])} } 
\end{equation} 
with $\eta = \max\sets{\etamoins,\etaplus}$,
then $\Chwindow$ is close to $\sigma^2$ ---see 
Proposition~\ref{pro.variance-estim} in Section~\ref{sec.related.variance} 
for a precise statement---, and Eq.~\eqref{eq.thm.OLS.oracle} implies a first-order optimal 
oracle inequality for the model-selection procedure of Algorithm~\ref{algo.OLS.jump}.
Note that $\Chwindow$ can be computed efficiently, see Section~\ref{sec.practical.cost} and Appendix~\ref{app.algos.window}.
In addition, Eq.~\eqref{eq.thm.OLS.Cpt-Dgrd} and~\eqref{eq.thm.OLS.Cgrd-Dpt} imply that
\begin{equation} 
\label{def.Chthr} 
\Chthr = \Chthr(T_n) \egaldef \inf \setj{ C \geq 0 \, / \, D_{\mh(C)} \leq T_n } 
\end{equation} 
is close to $\sigma^2$ when $T_n \in [n/10 , 9n/10]$ 
---precise statements are provided by Proposition~\ref{pro.variance-estim} in Section~\ref{sec.related.variance}---, 
and Eq.~\eqref{eq.thm.OLS.oracle} implies a first-order optimal 
oracle inequality for the corresponding model-selection procedure.
See Section~\ref{sec.practical.jump-vs-slope} for practical comments about these variants of  Algorithm~\ref{algo.OLS.jump}.

Theorem~\ref{thm.OLS} does not prove that  Algorithm~\ref{algo.OLS.slope} works, and it seems difficult to prove such a result without adding some assumptions.
Indeed, the key heuristics behind Algorithm~\ref{algo.OLS.slope} is a linear behavior of the empirical risk as a function of the dimension, at least for large models.
In the proof of Theorem~\ref{thm.OLS}, we control the deviations of the empirical risk around its expectation,
but this is not sufficient for justifying Algorithm~\ref{algo.OLS.slope} without a \emph{strong uniform control on the approximation errors} of the models, an assumption much stronger than the ones of Theorem~\ref{thm.OLS}. 

Note finally that Eq.~\eqref{eq.thm.OLS.Cpt-Rgrd} and~\eqref{eq.thm.OLS.Cgrd-Rpt} are not necessary for justifying Algorithm~\ref{algo.OLS.jump}, but they are interesting for theory since they justify the term ``minimal penalty''.
Eq.~\eqref{eq.thm.OLS.Cpt-Rgrd} is a straightforward consequence of Eq.~\eqref{eq.thm.OLS.Cpt-Dgrd}, and results like Eq.~\eqref{eq.thm.OLS.Cgrd-Rpt} are easier to obtain than Eq.~\eqref{eq.thm.OLS.oracle}, see Section~\ref{sec.theory.pbpenminalt}.

\paragraph{Variant of Theorem~\ref{thm.OLS}}
If $\M$ contains at least one model of dimension at most $c_n \in [0,n)$, 
on the event defined in Theorem~\ref{thm.OLS}, we can actually prove that more results hold true:
we can change Eq.~\eqref{eq.thm.OLS.Cpt-Dgrd} and~\eqref{eq.thm.OLS.Cgrd-Dpt} respectively into
\begin{align}
\label{eq.thm.OLS.Cpt-Dgrd.alt}
\forall a_n < n \, ,
\quad
\forall C \leq \crochb{  1 - \etamoins(a_n) } \sigma^2  \, ,
\quad
D_{\mh(C)} &\geq a_n
\\
\label{eq.thm.OLS.Cgrd-Dpt.alt}
\forall b_n > c_n \, ,
\quad
\forall C \geq  \crochb{ 1 + \etaplus(b_n,c_n) } \sigma^2 \, ,
\quad
D_{\mh(C)} &\leq b_n
\end{align}
where
\begin{gather*}%\\
\sigma^2 \etaplus(b_n,c_n) \egaldef  \frac{n}{b_n - c_n} \parenj{ 2 \biaismax (c_n) + 4.1 \sigma^2 \sqrt{\frac{ \gamma \log(n)}{n} } }
\, , 
\\ 
\biaismax (c_n) \egaldef \inf_{\mM \,/\, D_m \leq c_n} \setj{\frac{1}{n} \normb{ (\Id_n - \Pi_m) F}^2 } 
\, , 
\quad \text{and} \quad 
\etamoins(a_n) \egaldef 4.1 \paren{1-\frac{a_n}{n}}^{-1} \sqrt{\frac{ \gamma \log(n)}{n} } 
\, .
\end{gather*}
In particular, under the assumptions of Theorem~\ref{thm.OLS}, 
taking $a_n \in (9n/10,n)$, $b_n \in ( n/20 , n/10)$ and $c_n = n/20$, 
we get a larger jump of $D_{\mh(C)}$ ---hence easier to detect--- by considering a larger window of values of $C$, hence reducing the precision of the estimation of $\sigma^2$.

\paragraph{Relaxation of the noise assumption}
Assumption~\eqref{hyp.thm.OLS.Gauss} is a classical noise model for proving non-asymptotic oracle inequalities. 
In Theorem~\ref{thm.OLS}, it is only used for proving some concentration inequalities at the beginning of the proof 
---Eq.~\eqref{eq.conc1}--\eqref{eq.conc2} in Section~\ref{sec.slopeOLS.proofs}---, 
so it could be changed into any noise assumption ensuring that similar concentration inequalities hold true.
For instance, Theorem~\ref{thm.OLS} can be generalized to the case of sub-Gaussian noise, 
as formalized below. 

\begin{remark}[Generalization of Theorem~\ref{thm.OLS} to sub-Gaussian noise] 
\label{rk.thm.OLS.subgaussian} 
%% Remark~\ref{rk.thm.OLS.subgaussian}  in Section~\ref{sec.slopeOLS.math}
Assume that the $(\varepsilon_i)_{1 \leq i \leq n}$ are centered, independent, and $(\phi^2 \sigma^2)$-sub-Gaussian for some $\phi >0$  
---with any definition of sub-Gaussianity among the classical ones since they are all equivalent 
up to numerical constants \citep[Section~2.3]{Bou_Lug_Mas:2011:livre}. 
Then, by the Cram\'er-Chernoff method \citep[Section~2.2]{Bou_Lug_Mas:2011:livre},  
Eq.~\eqref{eq.conc2} holds true with probability at least 
$1-2\exp(x/\phi^2)$. 
In addition, \citet[Theorem~3]{Bel:2019:quadform} 
shows that Eq.~\eqref{eq.conc1} holds true with probability at least 
$1-2\exp[ x/ (L \phi^2) ]$ for some numerical constant $L$. 
%%% L=256 under assumption 1 of \citet{Bel:2019:quadform} 
%
Therefore, the event $\Omega_{L \phi^2 x}$ defined in the proof of Theorem~\ref{thm.OLS} 
has a probability at least $1 - 4 \card(\M) \mathrm{e}^{-x}$. 
So, the result of Theorem~\ref{thm.OLS} holds true 
with $x$ \textup{(}resp. $\gamma$\textup{)} replaced by $L \phi^2 x$ \textup{(}resp. $L\phi^2 \gamma$\textup{)}  
in $n_0\,$, $\etamoins\,$, $\etaplus\,$, 
and in the risk bounds \eqref{eq.thm.OLS.Cgrd-Rpt}--\eqref{eq.thm.OLS.oracle}. 
The same generalization holds for Eq.~\eqref{eq.thm.OLS.Cpt-Dgrd.alt}--\eqref{eq.thm.OLS.Cgrd-Dpt.alt} 
and for consequences of Theorem~\ref{thm.OLS} such as Proposition~\ref{pro.variance-estim}. 
\end{remark}

\paragraph{Comments on the assumptions on $\M$}
Assumption~\eqref{hyp.thm.OLS.Id} is barely an assumption since we can always add such a model to the collection considered (and it will never be selected by the procedure). 
It is used in the proof of Eq.~\eqref{eq.thm.OLS.Cpt-Dgrd} where we need to make sure that a model of large dimension and small bias exists.

Theorem~\ref{thm.OLS} implicitly assumes that  
$\M$ contains a model of dimension at most $n/20$ 
with a small approximation error. 
This is much milder than the assumption of the corresponding results of \citet{Arl_Bac:2009:minikernel_nips,Arl_Bac:2009:minikernel_long_v2} and \citet{Arl_Mas:2009:pente}, 
which is that 
$\M$ contains a model of dimension at most $\sqrt{n}$ with an approximation error upper bounded by $\sigma^2 \sqrt{\log(n)/n}$. 
Here, having a model of dimension $n/20$ with approximation error $\sigma^2/\log(n)$ is sufficient 
to get a consistent estimation of $\sigma^2$ and a first-order optimal model-selection procedure. 
Note however that such an assumption seems almost \emph{necessary} for Algorithm~\ref{algo.OLS.jump} to work: 
if the approximation error never vanishes for a (not too) large model, 
and if it is not almost constant among large models 
---which can happen in practice---, 
we conjecture that the slope heuristics fails. 

Finally, $\M$ is assumed to be finite, but Theorem~\ref{thm.OLS} implicitly assumes a little more, since the event on which the result holds has a large probability only if $\card(\M) n^{-\gamma}$ is small, which requires to take $\gamma$ large enough.
Since $\gamma$ appears in all the bounds, assuming that it can be chosen fixed as $n$ grows is equivalent to assuming that $\card(\M)$ grows at most like a power of $n$, which excludes model collections of exponential complexity ---that is, $\card(\M) \propto a^n$ for some $a>0$.
The case of exponential collections is discussed in Sections \ref{sec.theory.rich} and~\ref{sec.empirical.conjectures.rich}.

\subsection{Bibliographical remarks} \label{sec.slopeOLS.history}
\paragraph{Algorithms}
The slope heuristics and the corresponding data-driven penalty were first proposed by Birg\'e and Massart in a preprint \citep{Bir_Mas:2001} and the subsequent article \citep{Bir_Mas:2006}.
They are also exposed by \citet{Mas:2005}, \citet[Section~2]{Bla_Mas:2006}, \citet[Section~8.5.2]{Mas:2003:St-Flour} and \citet{Mas:2008}.

The term ``slope'' corresponds to the linear behavior of the empirical risk as a function of the dimension, as Algorithm~\ref{algo.OLS.slope} exploits.

The first implementation of data-driven penalties built upon the slope heuristics was expressed as a slope estimation, as in Algorithm~\ref{algo.OLS.slope}; it was done by \citet[Section~A.4]{Let:2000} for penalized maximum likelihood, inspired by a preliminary version of the preprint by \citet{Bir_Mas:2001}.

Several practical issues with Algorithm~\ref{algo.OLS.slope} were underlined in the context of change-point detection by \citet[Chapter~4]{Leb:2002}, who then suggested to prefer the ``dimension jump'' formulation of Algorithm~\ref{algo.OLS.jump} which was present in the final version of the preprint by \citet{Bir_Mas:2001}, as well as in the articles by \citet{Mas:2005} and \citet{Bir_Mas:2006}. 
The drawbacks of Algorithm~\ref{algo.OLS.jump} were also underlined by \citet{Leb:2002,Leb:2005} where some automatic ways to detect the dimension jump were proposed and tested on some synthetic data.
Later on, \citet{Bau_Mau_Mic:2010} studied more deeply the practical use of Algorithms~\ref{algo.OLS.jump} and~\ref{algo.OLS.slope}, with several variants (see also Section~\ref{sec.practical}).
The first proposition of detecting a jump over some sliding window was made by \citet{Bon_Tou:2010}, who considered only a finite set of values of~$C$; to the best of our knowledge, the continuous formulation for $\Chwindow$ is new, as well as the corresponding algorithm in Appendix~\ref{app.algos.window}.

\paragraph{Theory} The first theoretical results about the slope heuristics were proved in the setting of the present section, that is, regression on a fixed design with the least-squares risk and projection (least-squares) estimators.
In the articles by \citet{Bir_Mas:2001,Bir_Mas:2006}, the first results obtained were similar to Eq.~\eqref{eq.thm.OLS.Cpt-Dgrd}, \eqref{eq.thm.OLS.Cpt-Rgrd}, \eqref{eq.thm.OLS.Cgrd-Rpt}, and~\eqref{eq.thm.OLS.oracle}, 
making slightly stronger assumptions.
A result similar to Eq.~\eqref{eq.thm.OLS.Cpt-Rgrd} was even published previously by \citet{Bir_Mas:2002}, but only in the restrictive case $F=0$. 

The first result showing the existence of a jump ---that is, Eq.~\eqref{eq.thm.OLS.Cpt-Dgrd} and~\eqref{eq.thm.OLS.Cgrd-Dpt} holding simultaneously for all $C$ on the same large-probability event--- was obtained for least-squares regression on a random design with regressogram estimators \citep{Arl_Mas:2009:pente}.
It was then proved in the fixed-design setting with more general estimators including projection estimators \citep{Arl_Bac:2009:minikernel_nips,Arl_Bac:2009:minikernel_long_v2}.

Eq.~\eqref{eq.thm.OLS.Cgrd-Rpt} is a corollary of a classical non-asymptotic oracle inequality for $C_p$-like penalties; 
similar results were known before the introduction of the slope heuristics 
\citep[see for instance][]{Bar_Bir_Mas:1999}.
Eq.~\eqref{eq.thm.OLS.oracle} is more precise because of the constant $1+\petito(1)$ in front of the oracle risk, which was first obtained by \citet{Bir_Mas:2001,Bir_Mas:2006}.

The extension of Theorem~\ref{thm.OLS} to sub-Gaussian noise (Remark~\ref{rk.thm.OLS.subgaussian}  in Section~\ref{sec.slopeOLS.math}) is new, to the best of our knowledge. 

\subsection{Proof of Theorem~\ref{thm.OLS}} \label{sec.slopeOLS.proofs}
The proof mixes ideas from 
\citet{Bir_Mas:2006} and \citet{Arl_Bac:2009:minikernel_long_v2}. 
%%\citet{Bir_Mas:2006} and \citet{Arl_Bac:2009:minikernel_long_v2}. 
We split it into three main steps, the last two ones being split themselves into several substeps: 
(1) using concentration inequalities, 
(2) proving the existence of a dimension jump (Eq.~\eqref{eq.thm.OLS.Cpt-Dgrd}--\eqref{eq.thm.OLS.Cgrd-Dpt}), 
and (3) proving risk bounds thanks to a general oracle inequality (Eq.~\eqref{eq.thm.OLS.Cgrd-Rpt}--\eqref{eq.thm.OLS.oracle}).

We define $n_0(\gamma)$ as the smallest integer such that
$\gamma \log(n)/n \leq 1/80^2$ for every $n \geq n_0(\gamma)$. 
At various places in the proof (in steps 2.3, 3.2, and 3.3), we make use of the inequality: 
for all $a,b,\theta >0$, 
$2 \sqrt{a b} \leq \theta a + \theta^{-1} b$. 
%(or $2 a b \leq \theta a^2 + \theta^{-1} b^2$)

\paragraph{Step 1: concentration inequalities}
%
%%%\detail{This step uses \eqref{hyp.thm.OLS.Gauss}.}
As explained in Section~\ref{sec.slopeOLS.penmin}, the slope heuristics relies on the fact that $\norms{\Fhm-Y}^2$ is close to its expectation.
Let $x \geq 0$ be fixed.
Given Eq.~\eqref{eq.riskFhm}--\eqref{eq.riskempFhm}, for every $\mM$, we consider the event $\Omega_{m,x}$ on which the following two inequalities hold simultaneously:
\begin{align}
\label{eq.conc1}
\absj{ \prodscal{\varepsilon}{\Pi_m \varepsilon} - \sigma^2 D_m} &\leq 2 \sigma^2 \sqrt{x D_m} + 2 x \sigma^2
\\
\label{eq.conc2}
\absj{\prodscalb{\varepsilon}{(\Id_n - \Pi_m) F} } &\leq \sigma \sqrt{2x} \normb{ (\Id_n - \Pi_m) F }
\, .
\end{align}
Under \eqref{hyp.thm.OLS.Gauss}, by standard Gaussian concentration results 
\citep[for instance,][Propositions 4 and~6]{Arl_Bac:2009:minikernel_long_v2}---, we have
\[ \Proba\paren{\Omega_{m,x} } \geq 1 - 4 \mathrm{e}^{-x} \, . \]
Then, defining $\Omega_x \egaldef \bigcap_{\mM} \Omega_{m,x}\,$, the union bound gives
\[ \Proba\paren{\Omega_{x} } \geq 1 - 4 \card(\M) \mathrm{e}^{-x} \]
and it is sufficient to prove that Eq.~\eqref{eq.thm.OLS.Cpt-Dgrd}--\eqref{eq.thm.OLS.oracle} hold true on $\Omega_x$ with $x = \gamma \log(n)$.

From now on, we restrict ourselves to the event $\Omega_x\,$.
%

%%%%%%%%%%%%%%%%%%%%%%%%%%%%%%%%%%%%%%%%%%%%%%%%%%%%%%%%%%%%%%%%%%%%%%%%%%%%%%%%%%%%%%%%%%%%%%%%
\paragraph{Step 2: existence of a dimension jump}
For proving Eq.~\eqref{eq.thm.OLS.Cpt-Dgrd} and~\eqref{eq.thm.OLS.Cgrd-Dpt}, we show
that $\mh(C)$ minimizes a quantity $G_C(m)$ close to $\crit_C(m)$,
and then we show that $G_C(m_1)$ (resp. $G_C(m_2)$, for some well-chosen $m_2 \in \M$) is smaller than
$G_C(m)$ for any model $m$ with $D_m < 9n/10$ (resp. $D_m > n/10$).

%%%%%%%%%%%%%%%%%%%%%%%%%%%%%%%%%%%%%%%%%%%%%%%%
\paragraph{Step 2.1: control of the difference between $\crit_C(m)$ and the quantity minimized by $\mh(C)$}
Let $C \geq 0$.
By Eq.~\eqref{eq.mhC}, \eqref{eq.riskFhm}, and \eqref{eq.riskempFhm}, since $\norms{\varepsilon}^2$ does not depend from $m$, $\mh(C)$ minimizes over $\M$ 
the function $G_C: \M \to \R$ defined by 
\begin{align*}
\forall m \in \M, \qquad 
G_C(m) &\egaldef \frac{1}{n} \norm{\Fhm - Y}^2 + C \frac{D_m}{n}  - \frac{1}{n} \norm{\varepsilon}^2
\notag
\\
&= \frac{1}{n} \normb{ (\Id_n - \Pi_m) F}^2 - \frac{1}{n} \prodscal{\varepsilon}{\Pi_m \varepsilon} + C \frac{D_m}{n} + \frac{2}{n} \prodscalb{\varepsilon}{(\Id_n- \Pi_m) F}
\notag
\\
&= \crit_C(m) - \paren{ \frac{1}{n} \prodscal{\varepsilon}{\Pi_m \varepsilon} - \sigma^2 D_m } + \frac{2}{n} \prodscalb{\varepsilon}{(\Id_n- \Pi_m) F}
%\, .
%\label{eq.GCM}
\end{align*}
where $\crit_C$ is defined by Eq.~\eqref{eq.critC}.
Therefore, by Eq.~\eqref{eq.conc1}--\eqref{eq.conc2} and using $D_m \leq n$, for every $\mM$,
\begin{align}
\label{eq.GCM-critC}
\absb{G_C(m) - \crit_C(m) }
&\leq 2 \sigma^2 \parenj{ \sqrt{ \frac{x}{n} } + \frac{x}{n} }  + \frac{2 \sigma \sqrt{2x}}{n} \normb{(\Id_n- \Pi_m) F}
\, .
\end{align}

%%%%%%%%%%%%%%%%%%%%%%%%%%%%%%%%%%%%%%%%%%%%%%%%
\paragraph{Step 2.2: lower bound on $D_{\mh(C)}$ when $C$ is too small (proof of Eq.~\eqref{eq.thm.OLS.Cpt-Dgrd})}
%%%\detail{This step uses $x/n \leq 1/60^2$. }
%
Let $C \in [0,\sigma^2)$. 
Since $\mh(C)$ minimizes $G_C(m)$ over $\mM$, it is sufficient to prove that if $C \leq \parens{ 1 - \etamoins } \sigma^2$,
\begin{equation}
\label{eq.pr.thm.OLS.Cpt-Dgrd.but}
G_C(m_1) < \inf_{\mM , \, D_m < 9n / 10} \setb{G_C(m)}
\end{equation}
where $m_1$ is given by \eqref{hyp.thm.OLS.Id}.
On the one hand, by Eq.~\eqref{eq.GCM-critC},
\begin{align}
G_C(m_1)
&\leq \crit_C(m_1) + 2 \sigma^2 \parenj{ \sqrt{ \frac{x}{n} } + \frac{x}{n} }
= C-\sigma^2 + 2 \sigma^2 \parenj{ \sqrt{ \frac{x}{n} } + \frac{x}{n} }
\label{eq.pr.thm.OLS.Cpt-Dgrd.1}
\, .
\end{align}
On the other hand, by Eq.~\eqref{eq.GCM-critC}, for any $\mM$ such that $D_m < 9n/10$,
\begin{align}
G_C(m)
&\geq
\frac{(C-\sigma^2) D_m}{n}
- 2 \sigma^2 \parenj{ \sqrt{ \frac{x}{n} } + \frac{x}{n} }
+ \frac{1}{n} \normb{(\Id_n - \Pi_m) F}^2 - \frac{2\sigma \sqrt{2x}}{n} \normb{(\Id_n - \Pi_m) F}
\notag
\\
&> \frac{9}{10} (C-\sigma^2)
- 2 \sigma^2 \parenj{ \sqrt{ \frac{x}{n} } + \frac{2 x}{n} }
\, .
\label{eq.pr.thm.OLS.Cpt-Dgrd.2}
\end{align}
To conclude, the upper bound in Eq.~\eqref{eq.pr.thm.OLS.Cpt-Dgrd.1} is smaller than the lower bound in Eq.~\eqref{eq.pr.thm.OLS.Cpt-Dgrd.2} when
\begin{equation}
\label{eq.pr.thm.OLS.Cpt-Dgrd.3}
C \leq \sigma^2 \parenj{ 1 - 40 \sqrt{\frac{x}{n}} - 60 \frac{x}{n} }
\defegal \widetilde{C}_1(x) \, .
\end{equation}
Taking $x=\gamma \log(n)$, for $n \geq n_0(\gamma)$, 
we have $\widetilde{C}_1(x) \geq \sigma^2 (1 - \etamoins)$ hence Eq.~\eqref{eq.thm.OLS.Cpt-Dgrd}.

Remark that the same reasoning with $9n/10$ replaced by any $a_n \in [0, n)$ proves that 
$D_{\mh(C)} \geq a_n$ for every 
\begin{equation}
\label{eq.pr.thm.OLS.Cpt-Dgrd.4}
C \leq \sigma^2 \parenj{ 1 - \frac{4 \sqrt{\frac{x}{n}} + 6 \frac{x}{n}}{1 - \frac{a_n}{n}} }
\defegal C_1(x; a_n) \, . 
\end{equation}
We get Eq.~\eqref{eq.thm.OLS.Cpt-Dgrd.alt} by taking $x = \gamma \log(n)$ 
and using that $x/n \leq 1/60^2$ since $n \geq n_0(\gamma)$.

%%%%%%%%%%%%%%%%%%%%%%%%%%%%%%%%%%%%%%%%%%%%%%%%
\paragraph{Step 2.3: upper bound on $D_{\mh(C)}$ when $C$ is large enough (proof of Eq.~\eqref{eq.thm.OLS.Cgrd-Dpt})}
%%%\detail{This step uses $x/n \leq 1/80^2$. }
%
Let $C>\sigma^2$.
Similarly to the proof of Eq.~\eqref{eq.thm.OLS.Cpt-Dgrd}, it is sufficient to prove that 
if $C \geq \parens{ 1 + \etaplus } \sigma^2$,
\begin{equation}
\label{eq.pr.thm.OLS.Cpt-Dpt.but.alt2}
G_C(m_2) < \inf_{\mM \, , \, D_m > n / 10} \set{G_C(m)}
\end{equation}
where $m_2 \in \argmin_{\mM \,/\, D_m \leq n/20} \setb{ \norms{(\Id_n - \Pi_m) F}^2 }$ exists by assumption. 
For any $c_n \in [0,n]$, let us define 
\[ 
\biaismax (c_n) \egaldef \inf_{\mM \,/\, D_m \leq c_n} \setj{\frac{1}{n} \normb{ (\Id_n - \Pi_m) F}^2 } 
\, , 
\]
so that $m_2$ has an approximation error equal to $\biaismax (n/20)$. 
On the one hand, by Eq.~\eqref{eq.GCM-critC}, 
\begin{align}
G_C(m_2)
&\leq \crit_C(m_2)
+ 2 \sigma^2 \parenj{ \sqrt{ \frac{x}{n} } + \frac{x}{n} } + \frac{2 \sigma \sqrt{2x}}{n} \normb{(\Id_n- \Pi_{m_2}) F}
\notag
\\
&\leq \frac{2}{n}\normb{(\Id_n- \Pi_{m_2}) F}^2 + \frac{(C-\sigma^2) D_{m_2}}{n}
+ 2 \sigma^2 \parenj{ \sqrt{ \frac{x}{n} } + \frac{2x}{n} }
\notag
\\
&\leq  2 \biaismax  \parenj{\frac{n}{20}} + \parens{ C- \sigma^2 } \frac{n/20}{n} 
+ 2 \sigma^2 \parenj{ \sqrt{ \frac{x}{n} } + \frac{2x}{n} }
\label{eq.pr.thm.OLS.Cpt-Dpt.1.alt2}
\, .
\end{align}
On the other hand, by Eq.~\eqref{eq.GCM-critC}, for any $\mM$ such that $D_m > n/10$,
\begin{align}
G_C(m)
%% NOTE: je supprime une etape de plus car c'est vraiment strictement identique au cas precedent (la seule difference est qu'ici C-\sigma^2 est positif, donc c'est bien une borne inf sur D_m qui nous sauve)
&> \frac{n/10}{n} (C-\sigma^2) - 2 \sigma^2 \parenj{ \sqrt{ \frac{x}{n} } + \frac{2x}{n} }
\, .
\label{eq.pr.thm.OLS.Cpt-Dpt.2.alt2}
\end{align}
To conclude, 
the upper bound in Eq.~\eqref{eq.pr.thm.OLS.Cpt-Dpt.1.alt2}
is  smaller than the lower bound in Eq.~\eqref{eq.pr.thm.OLS.Cpt-Dpt.2.alt2}
when
\begin{equation}
\label{eq.pr.thm.OLS.Cpt-Dpt.3}
C  \geq \sigma^2 \crochj{ 1 + 80 \parenj{ \sqrt{\frac{x}{n}} + \frac{2 x}{n} } } + 40 \biaismax \parenj{\frac{n}{20}}
\defegal \widetilde{C}_2(x)
\, .
\end{equation}
Taking $x = \gamma \log(n)$, for $n \geq n_0(\gamma)$, 
we have $\widetilde{C}_2(x) \leq \sigma^2 (1+\etaplus)$
hence Eq.~\eqref{eq.thm.OLS.Cgrd-Dpt}. 

Remark that if $\M$ contains a model of dimension at most $c_n \in [0,n)$, 
the same reasoning with $n/10$ replaced by any $b_n \in (c_n,n]$ 
and $n/20$ replaced by $c_n$ proves that 
$D_{\mh(C)} \leq b_n$ for every 
\begin{equation} 
\label{eq.pr.thm.OLS.Cpt-Dpt.4}
C \geq \sigma^2 \crochj{ 1 + \frac{4n}{b_n - c_n} \parenj{ \sqrt{\frac{x}{n}} + \frac{2 x}{n} } } 
+ \frac{2n}{b_n - c_n} \biaismax  (c_n) 
\defegal C_2 (x ; b_n ; c_n)
\, . 
\end{equation}
We get Eq.~\eqref{eq.thm.OLS.Cgrd-Dpt.alt} by taking $x = \gamma \log(n)$ 
and using that $x/n \leq 1/80^2$ since $n \geq n_0(\gamma)$. 

Until the end of the proof, we fix $x=\gamma \log(n)$.

%%%%%%%%%%%%%%%%%%%%%%%%%%%%%%%%%%%%%%%%%%%%%%%%
\paragraph{Step 2.4: lower bound on the risk of large models (proof of Eq.~\eqref{eq.thm.OLS.Cpt-Rgrd})}
%
%%%\detail{This step uses $x/n \leq 1/77^2$. }
%
This is a straightforward consequence of Eq.~\eqref{eq.thm.OLS.Cpt-Dgrd}.
Indeed, on $\Omega_x\,$, for any $\mM$ such that $D_m \geq 9n/10$,
\begin{align*}
   \frac{1}{n} \norm{F - \Fhm}^2
&= \frac{1}{n} \normb{(\Id_n - \Pi_m) F}^2 + \frac{1}{n} \prodscal{\varepsilon}{\Pi_m \varepsilon}
\\
&\geq \frac{\sigma^2}{n} \parenb{ D_m - 2 \sqrt{x D_m} - 2 x}
= \frac{\sigma^2}{n} \crochB{ \parenb{ \sqrt{D_m} - \sqrt{x}}^2  - 3 x}
\geq \frac{7 \sigma^2}{8}
\, , 
\end{align*}
where we use that $x/n \leq 1/77^2$ since $n \geq n_0(\gamma)$. 

%%%%%%%%%%%%%%%%%%%%%%%%%%%%%%%%%%%%%%%%%%%%%%%%%%%%%%%%%%%%%%%%%%%%%%%%%%%%%%%%%%%%%%%%%%%%%%%%
\paragraph{Step 3: upper bounds on the risk}
For proving Eq.~\eqref{eq.thm.OLS.Cgrd-Rpt}--\eqref{eq.thm.OLS.oracle}, we prove a slightly more general oracle inequality ---Eq.~\eqref{eq.pr.thm.OLS.oracle-general}--- using the classical approach used for instance by \citet{Bir_Mas:2002}, \citet{Mas:2003:St-Flour} and \citet{Arl_Bac:2009:minikernel_long_v2}.

%%%%%%%%%%%%%%%%%%%%%%%%%%%%%%%%%%%%%%%%%%%%%%%%
\paragraph{Step 3.1: general approach for proving an oracle inequality}
Following Section~\ref{sec.slopeOLS.optimal}, an ideal penalty is
\[ \penid(m) \egaldef \frac{1}{n} \norm{\Fhm - F}^2 - \frac{1}{n} \norm{\Fhm - Y}^2 + \norm{\varepsilon}^2 \]
which has expectation $2\sigma^2 D_m / n = \penopt(m)$.
A key argument for getting an oracle inequality is that $\penid(m)$ concentrates around its expectation.
Indeed, let us define
\begin{equation}
\label{eq.eq.pr.thm.OLS.Delta}
 \Delta(m) \egaldef \penid(m) - \frac{2\sigma^2 D_m}{n}
 = \frac{2}{n} \paren{ \prodscal{\varepsilon}{\Pi_m \varepsilon} - \sigma^2 D_m } - \frac{2}{n} \prodscalb{\varepsilon}{(\Id_n - \Pi_m) F}
 \, ,
\end{equation}
where the second formulation is a consequence of Eq.~\eqref{eq.riskempFhm}.
Then, by Eq.~\eqref{eq.mhC}, for any $C \geq 0$ and $\mM$,
\[
\frac{1}{n} \norm{\Fh_{\mh(C)} - Y}^2 + \frac{C D_{\mh(C)}}{n} \leq \frac{1}{n} \norm{\Fhm - Y}^2 + \frac{C D_m}{n}
\]
which is equivalent to
\begin{equation}
\label{eq.eq.pr.thm.OLS.oracle.1}
\frac{1}{n} \norm{\Fh_{\mh(C)} - F}^2 - \Delta\parenb{ \mh(C) } + \frac{(C - 2\sigma^2) D_{\mh(C)}}{n} \leq \frac{1}{n} \norm{\Fhm - F}^2 - \Delta(m) + \frac{(C - 2\sigma^2) D_m}{n} \, .
\end{equation}
It remains to show that $\Delta(m)$ and $(C-2\sigma^2)D_m/n$ are small compared to $n^{-1} \norms{\Fhm - F}^2$ for all $\mM$. 
Recall that we restrict ourselves to the event $\Omega_x$ until the end of the proof, 
with $x = \gamma \log(n)$. 

%%%%%%%%%%%%%%%%%%%%%%%%%%%%%%%%%%%%%%%%%%%%%%%%
\paragraph{Step 3.2: control of $\Delta(m)$}
By Eq.~\eqref{eq.conc1}, \eqref{eq.conc2}, and \eqref{eq.eq.pr.thm.OLS.Delta}, for every $\mM$ and $\theta>0$,
\begin{align}
\absb{\Delta(m)}
&\leq
\frac{2}{n} \croch{ 2 \sigma^2 \sqrt{x D_m} + 2 \sigma^2 x + \sigma \sqrt{2 x} \normb{ (\Id_n - \Pi_m) F} }
\notag
\\
&\leq
2 \theta \E\croch{ \frac{1}{n} \norm{\Fhm-F}^2 } + \frac{\sigma^2 x}{n} \paren{ 3 \theta^{-1} + 4 }
\, .
\label{eq.eq.pr.thm.OLS.oracle.2}
\end{align}

%%%%%%%%%%%%%%%%%%%%%%%%%%%%%%%%%%%%%%%%%%%%%%%%
\paragraph{Step 3.3: upper bound on the expected risk in terms of risk}
By Eq.~\eqref{eq.riskFhm} and \eqref{eq.conc2}, for every $\mM$ and $\theta^{\prime}>0$,
\begin{align*}
\norm{\Fhm - F}^2
&=
\E\croch{ \norm{\Fhm - F}^2 } + \prodscal{\varepsilon}{\Pi_m \varepsilon} - \sigma^2 D_m
\\
&\geq
\E\croch{ \norm{\Fhm - F}^2 } - \sigma^2 \paren{ 2 \sqrt{ x D_m } + 2x }
\\
&\geq
\paren{ 1 - \theta^{\prime} } \E\croch{ \norm{\Fhm - F}^2 } - x \sigma^2 \paren{ 2 + \theta^{\prime -1} }
\end{align*}
so that, for every $\theta^{\prime} \in (0,1)$,
\begin{equation}
\E\croch{ \norm{\Fhm - F}^2 } \leq \frac{1}{1- \theta^{\prime}} \norm{\Fhm - F}^2 + \cteThmOLS(\theta^{\prime}) x \sigma^2
\qquad
\text{with} \qquad
\cteThmOLS(\theta^{\prime}) \egaldef \frac{2 + \frac{1}{\theta^{\prime}}}{1-\theta^{\prime}}
\, .
\label{eq.eq.pr.thm.OLS.oracle.3}
\end{equation}

%%%%%%%%%%%%%%%%%%%%%%%%%%%%%%%%%%%%%%%%%%%%%%%%
\paragraph{Step 3.4: control of the remainder terms appearing in Eq.~\eqref{eq.eq.pr.thm.OLS.oracle.1}}
Combining Eq.~\eqref{eq.eq.pr.thm.OLS.oracle.2} and~\eqref{eq.eq.pr.thm.OLS.oracle.3}, we get on the one hand that for every $\mM$, $\theta>0$, $\theta^{\prime} \in (0,1)$,
\begin{align}
& \quad \Delta(m) + \frac{(2\sigma^2 - C)D_m}{n}
\notag
\\
&\leq
\crochj{ 2 \theta + \paren{ 2 - \frac{C}{\sigma^2}}_+ } \E\crochj{ \frac{1}{n} \norm{\Fhm-F}^2 } + \frac{\sigma^2 x}{n} \parenj{ \frac{3}{\theta} + 4 }
\notag
\\
&\leq
\frac{2 \theta + \paren{ 2 - \frac{C}{\sigma^2}}_+}{1- \theta^{\prime}}  \frac{1}{n} \norm{\Fhm-F}^2
+ \frac{\sigma^2 x}{n} \parenj{ \frac{3}{\theta} + 4 + \cteThmOLS(\theta^{\prime}) \crochj{ 2 \theta + \paren{ 2 - \frac{C}{\sigma^2}}_+ } }
\, .
\label{eq.eq.pr.thm.OLS.oracle.4}
\end{align}
On the other hand, similarly, for every $\mM$, $\theta>0$, $\theta^{\prime} \in (0,1)$,
\begin{align}
& \quad - \Delta(m) + \frac{(C - 2\sigma^2)D_m}{n}
\notag
\\
& \leq
\frac{2 \theta + \paren{ \frac{C}{\sigma^2} - 2 }_+}{1- \theta^{\prime}}  \frac{1}{n} \norm{\Fhm-F}^2
+ \frac{\sigma^2 x}{n} \parenj{ \frac{3}{\theta} + 4 + \cteThmOLS(\theta^{\prime}) \crochj{ 2 \theta + \paren{ \frac{C}{\sigma^2} - 2}_+ } }
\, .
\label{eq.eq.pr.thm.OLS.oracle.5}
\end{align}

%%%%%%%%%%%%%%%%%%%%%%%%%%%%%%%%%%%%%%%%%%%%%%%%
\paragraph{Step 3.5: proof of a general oracle inequality}
Combining Eq.~\eqref{eq.eq.pr.thm.OLS.oracle.1}, \eqref{eq.eq.pr.thm.OLS.oracle.4}, and \eqref{eq.eq.pr.thm.OLS.oracle.5} 
above with $\theta^{\prime}=2\theta \in (0,1)$,
we get that for every $\theta \in (0,1/2)$ and $\mM$,
\begin{align}
\notag 
&\hspace{-1cm} 
\croch{ 1 - \frac{ 2 \theta + \paren{2 - \frac{C}{\sigma^2} }_+}{ 1 - 2 \theta } } \frac{1}{n} \norm{\Fh_{\mh(C)} - F}^2
\\ 
\label{eq.eq.pr.thm.OLS.oracle.6}
&\leq \croch{ 1 + \frac{ 2 \theta + \paren{\frac{C}{\sigma^2} -2}_+}{ 1 - 2 \theta }} \frac{1}{n} \norm{\Fhm - F}^2 + \frac{\sigma^2 x}{n} R_1(\theta,C\sigma^{-2})
\\
\notag
\text{with} \qquad
R_1(\theta,C\sigma^{-2}) &\egaldef \frac{6}{\theta} + 8 + \cteThmOLS( 2 \theta) \paren{ 4 \theta + \absj{ \frac{C}{\sigma^2} - 2} }
\, .
\end{align}

Let us assume $C>\sigma^2$. 
For any $\delta \in (0,1]$, we choose
\[
\theta = \theta^{\star}(\delta , C\sigma^{-2}) 
\egaldef \frac{\delta}{4} 
\frac{\croch{1 - \paren{2 - \frac{C}{\sigma^2} }_+}^2}{1 + \paren{\frac{C}{\sigma^2} -2}_+ + \delta \crochj{ 1 - \paren{2 - \frac{C}{\sigma^2} }_+}}  
< \frac{\delta}{4} 
\leq \frac{1}{4} \, . 
\]
So, if $C \geq (1+\delta) \sigma^2$, 
we have $C > (1+4\theta) \sigma^2$ 
hence we can divide both sides of Eq.~\eqref{eq.eq.pr.thm.OLS.oracle.6} by 
\[ 
1 - \frac{2\theta + \paren{2 - \frac{C}{\sigma^2}}_+}{1-2\theta} > 0 
\, . 
\]
Remark that 
\[ 
\crochj{ 1 + \frac{ 2 \theta + \paren{\frac{C}{\sigma^2} -2}_+}{ 1 - 2 \theta }}
\times 
\croch{ 1 - \frac{2\theta + \paren{2 - \frac{C}{\sigma^2}}_+}{1-2\theta} }^{-1} 
= \frac{ 1 + \paren{\frac{C}{\sigma^2} -2}_+ }{ 1 - 4\theta - \paren{2 - \frac{C}{\sigma^2} }_+ } 
= \frac{ 1 + \paren{\frac{C}{\sigma^2} -2}_+ }{ 1 - \paren{2 - \frac{C}{\sigma^2} }_+ } + \delta 
\]
where the last equality uses $\theta = \theta^{\star}(\delta , C\sigma^{-2})$. 
So, if $C \geq (1+\delta) \sigma^2$, Eq.~\eqref{eq.eq.pr.thm.OLS.oracle.6} leads to 
\begin{equation}
\label{eq.pr.thm.OLS.oracle-general}
\frac{1}{n} \norm{\Fh_{\mh(C)} - F}^2
\leq \paren{ \frac{1+ \paren{\frac{C}{\sigma^2} -2}_+}{1 -  \paren{2 - \frac{C}{\sigma^2}}_+} + \delta} \croch{ \inf_{\mM} \set{ \frac{1}{n} \norm{\Fhm - F}^2 }
+ \frac{\sigma^2 x}{n} R_2 \paren{ \delta,\frac{C}{\sigma^{2}} }
}
\end{equation}
where for every $\delta\in(0,1]$ and $u \in (1,+\infty)$,
\begin{align*} 
R_2(\delta ,u) 
= R_1 \parenb{ \theta^{\star}(\delta , u) , u } 
%\\ &
\leq \parenb{ 10+ 2 \abss{u-2} } \theta^{\star}(\delta , u)^{-1} 
\, .
\end{align*}
Therefore, for every $C \geq (1+\delta) \sigma^2$, 
\begin{align*} 
R_2 \parenj{ \delta , \frac{C}{\sigma^{2}} }
&\leq \frac{8}{\delta} \parenb{ 5 + \abss{C \sigma^{-2}-2}} \max\set{ 2+(C \sigma^{-2}-2)_+ \, , \, \frac{2}{\crochb{1 - (2 - C \sigma^{-2})_+}^2} } 
\, . 
\end{align*}

%%%%%%%%%%%%%%%%%%%%%%%%%%%%%%%%%%%%%%%%%%%%%%%%
\paragraph{Step 3.6: risk bound for $\mh(C)$ when $C$ is large enough (proof of Eq.~\eqref{eq.thm.OLS.Cgrd-Rpt})}
%
%%%\detail{This step uses $\gamma \geq 1$.}
%
In this step, we assume $C > \sigma^2$.
When $C / \sigma^2 \in (1,2]$, Eq.~\eqref{eq.pr.thm.OLS.oracle-general} with $\delta = C\sigma^{-2} - 1 \in (0,1]$ 
%%%\detail{(which guarantees $C \geq (1+\delta)\sigma^2$) } 
yields 
\begin{align*}
\frac{1}{n} \norm{\Fh_{\mh(C)} - F}^2
&\leq 2 \paren{\frac{C}{\sigma^2} - 1}^{-1} 
\crochj{ \inf_{\mM} \setj{ \frac{1}{n} \norm{\Fhm - F}^2 }
+ \frac{96 \sigma^2 x}{n} \paren{\frac{C}{\sigma^2} - 1}^{-3} }
\, .
\end{align*}
When $C / \sigma^2 \geq 2$, Eq.~\eqref{eq.pr.thm.OLS.oracle-general} with $\delta =  1$ yields 
\begin{align*}
\frac{1}{n} \norm{\Fh_{\mh(C)} - F}^2
&\leq \frac{C}{\sigma^2} 
\crochj{ \inf_{\mM} \setj{ \frac{1}{n} \norm{\Fhm - F}^2 }
+ \frac{20 \sigma^2 x}{n} \paren{\frac{C}{\sigma^2} }^{2} }
\, .
\end{align*}
Eq.~\eqref{eq.thm.OLS.Cgrd-Rpt} follows. 

%%%%%%%%%%%%%%%%%%%%%%%%%%%%%%%%%%%%%%%%%%%%%%%%
\paragraph{Step 3.7: oracle inequality for $\mh(C)$ when $C$ is close to $2\sigma^2$ (proof of Eq.~\eqref{eq.thm.OLS.oracle})}
Now, we assume $C / \sigma^2 \in [2-\eta,2+\eta]$ with $\eta \in [0,1/2]$.
Taking $\delta = \eta$ in Eq.~\eqref{eq.pr.thm.OLS.oracle-general} yields
\begin{align*}
\frac{1}{n} \norm{\Fh_{\mh(C)} - F}^2
&\leq
\paren{ \max\set{ 1+\eta \, , \, \frac{1}{1-\eta} } + \eta} 
\\ 
&\qquad \times \crochj{ \inf_{\mM} \set{ \frac{1}{n} \norm{\Fhm - F}^2 }
+ \frac{\sigma^2 x}{n} \frac{8}{\eta} \paren{5+\eta} \max\set{ 2 + \eta \, , \, \frac{2}{(1-\eta)^2} } }
\\
&\leq
\paren{ 1 + 3 \eta} \inf_{\mM} \set{ \frac{1}{n} \norm{\Fhm - F}^2 }
\\ 
&\qquad + \frac{\sigma^2 x}{n} \frac{8}{\eta} \paren{ 1 + 3 \eta} \paren{5+\eta} \max\set{ 2 + \eta \, , \, \frac{2}{(1-\eta)^2} }
\\
&\leq
\paren{ 1 + 3 \eta} \inf_{\mM} \set{ \frac{1}{n} \norm{\Fhm - F}^2 }
+ \frac{880 \sigma^2 x}{\eta n} 
\, , 
\end{align*}
using that $1/(1-\eta) \leq 1+2\eta$ for every $\eta \in [0,1/2]$. 
Eq.~\eqref{eq.thm.OLS.oracle} follows. 
\qed

%%%%%%%%%%%%%%%%%%%%%%%%%%%%%%%%%%%%%%%%%%%%%%%%%%%%%%%%%%%%%%%%%%%%%%%%%%%%%%%%%%%%
%%%%%%%%%%%%%%%%%%%%%%%%%%%%%%%%%%%%%%%%%%%%%%%%%%%%%%%%%%%%%%%%%%%%%%%%%%%%%%%%%%%%

\section{Generalizing the slope heuristics} \label{sec.penmingal}
The slope heuristics has first been formulated and theoretically validated in the framework of Section~\ref{sec.slopeOLS}.
Then, it rapidly became a more general heuristics for building data-driven optimal penalties.
This section discusses two possible formulations for its generalization.

\subsection{General framework} \label{sec.penmingal.framework}
Before going any further, we need to introduce a general model/estimator-selection framework.
Let $\Set$ be some set, $\Risk:\Set\mapsto [0,+\infty)$ be some risk function, 
and assume that our goal is to build from data some estimator $\sh \in \Set$ 
such that $\Risk(\sh\,)$ is as small as possible.
Let $\parens{\shm}_{\mM}$ be a collection of estimators.
The goal of estimator selection is to choose from data some $\mh\in\M$ such that 
the risk of $\sh_{\mh}$ is as small as possible, that is, satisfying an oracle inequality
\begin{equation} \label{eq.oracle.gal}
\Risk\paren{\sh_{\mh}} - \Risk\parens{\bayes} 
\leq K_n \inf_{\mM} \setb{ \Risk\parens{\shm} - \Risk\parens{\bayes} } + R_n
\end{equation}
with large probability, where $\Risk\parens{\bayes} \egaldef \inf_{t \in \Set} \Risk\parens{t}$.
In the following, $K_n$ is called ``the leading constant'' of 
the oracle inequality \eqref{eq.oracle.gal}.
Let $\Remp:\Set\mapsto [0,+\infty)$ be the empirical risk 
associated with~$\Risk$, that is, we assume throughout Sections~\ref{sec.penmingal}--\ref{sec.hints} 
that $\forall t \in \Set$, $\E\crochs{\Remp\parens{t}} = \Risk\parens{t}$.

This framework includes the one of Section~\ref{sec.slopeOLS} by taking $\Set=\R^n$, 
$\Risk(t)=n^{-1}\norms{t-F}^2 + \sigma^2$, $\shm=\Fhm$ the projection estimator associated with  
some model $S_m$ for every $\mM$, and the $L^2$ empirical risk $\Remp(t)=n^{-1}\norms{t-Y}^2$.
Many other classical settings also fit into this framework, 
such as density estimation with the Kullback risk or the $L^2$ risk, 
random-design regression with the $L^2$ risk,  
and classification with the 0--1 risk 
\citep[see][Section~1, for details]{Arl_Cel:2010:surveyCV}. 

\subsection{Penalties known up to some constant factor} \label{sec.penmingal.apriori}
The most natural extension of the slope heuristics is to generalize it to all frameworks where a penalty is known up to some multiplicative constant \citep{Mas:2005,Bla_Mas:2006}, that is, if theoretical results show that a good penalty is $C^{\star} \pen_1(m)$ with $\pen_1$ known but $C^{\star}$ unknown.
Penalties known up to a constant factor appear in several frameworks, for four main reasons:
\begin{enumerate}
\item A penalty satisfying an optimal oracle inequality ---that is, an oracle inequality with leading constant $1+\petito(1)$--- is theoretically known, but involves \emph{unknown quantities} in practice, such as the noise-level $\sigma^2$ for Mallows' $C_p$ and $C_L$ \citep{Mal:1973}, see Sections~\ref{sec.slopeOLS} and~\ref{sec.penmingal.fails}.
\item An optimal penalty $\pen_1$ is known theoretically and in practice, but only \emph{asymptotically}, that is, the (unknown) non-asymptotic optimal penalty is $C^{\star}_n \pen_1$ with $C^{\star}_n \to 1$ as the sample size $n$ tends to infinity, but $C^{\star}_n$ is unknown 
and can be far from $1$ for finite sample sizes.
For instance, AIC \citep{Aka:1973} and BIC \citep{Sch:1978} penalties for maximum likelihood rely on asymptotic computations.
Section~\ref{sec.empirical.overpenalization} explains why such a problem can arise in 
almost any framework. 
\item An optimal penalty is obtained by resampling, hence depending on some multiplicative factor that might depend on unknown quantities or be correct only for $n$ large enough, see Remark~\ref{rk.reech} in Section~\ref{sec.theory.complete} and the article by \citet{Arl:2009:RP}.
\item A penalty $C\pen_1$ satisfying an oracle inequality with a leading constant $\grandO(1)$ when $C$ is well chosen is known theoretically, but theoretical results are not precise enough to specify the optimal value $C^{\star}$ of $C$.
This occurs for instance for change-point detection \citep{Com_Roz:2004,Leb:2005}, density estimation with Gaussian mixtures \citep{Mau_Mic:2008}, 
and local Rademacher complexities in classification \citep{Bar_Bou_Men:2005,Kol:2006}.
In some frameworks, some partial information is available about the optimal value of the constant:
in binary classification, global Rademacher complexities differ by a factor $2$ between theory \citep{Kol:2001} and practice \citep{Loz:2000}.
Note that in such cases, it might happen that $C^{\star}\pen_1$ is not exactly an optimal penalty, so that no oracle inequality with leading constant $1+\petito(1)$ can be obtained; nevertheless, choosing the constant $C$ in the penalty $C \pen_1$ remains an important practical problem.
\end{enumerate}

Then, if for every $\mM$, $\C_m$ measures the ``complexity'' of $\shm\,$, the slope heuristics suggests to generalize Algorithm~\ref{algo.OLS.jump} into the following.
\begin{algo}[Slope-heuristics algorithm, jump formulation, general setting] \label{algo.gal.slope.naif}
\hfill \newline \algoinput\textup{:} $\parenb{\Remp\parens{\shm}}_{\mM}\,$, $\parenb{\pen_1(m)}_{\mM}\,$, and $\parens{\C_m}_{\mM}\,$.
\begin{enumerate}
\item Compute $\parenb{ \mhslopeun(C) }_{C \geq 0}\,$, where for every $C \geq 0$,
\begin{equation} \label{eq.mhC.gal.naif}
\mhslopeun(C) \in \argmin_{\mM} \set{ \Remp\parens{\shm} + C \pen_1(m) } \, .
\end{equation}
\item Find $\Chjumpgal>0$ corresponding to the ``unique large jump'' of $C \mapsto \C_{\mhslopeun(C)}\,$.
\item Select $\mhAlgC \in \argmin_{\mM} \sets{ \Remp\parens{ \shm } + 2 \Chjumpgal \pen_1(m) }$.
\end{enumerate}
\algooutput\textup{:} $\mhAlgC\,$.
\end{algo}
Algorithm~\ref{algo.gal.slope.naif} relies on two ideas: (i) Eq.~\eqref{eq.OLS.slope}, that is, $\penopt = 2 \penmin$, is valid in a more general framework than least-squares regression and projection estimators, and (ii) if a proper complexity measure $\C_m$ is used instead of the dimension $D_m$ of the models, the minimal penalty can be characterized empirically by a jump of $\C_{\mhslopeun(C)}\,$.

\subsection{Algorithm~\ref{algo.gal.slope.naif} fails for linear-estimator selection} \label{sec.penmingal.fails}
We now illustrate on an example why Algorithm~\ref{algo.gal.slope.naif} can fail, before showing how to correct it in Section~\ref{sec.penmingal.linear}.
Let us consider the fixed-design regression framework of Section~\ref{sec.slopeOLS.framework} with linear estimators instead of projection estimators, that is, for every $\mM$,
\[ \Fhm = A_m Y \] for some deterministic linear mapping $A_m: \R^n \to \R^n\,$.
For instance, projection estimators are linear estimators since the orthogonal projection $A_m=\Pi_m$ onto a linear space $S_m$ is linear.
Other examples include kernel ridge regression or spline smoothing, nearest-neighbor regression, 
and Nadaraya-Watson estimators \citep[provide more examples and references]{Arl_Bac:2009:minikernel_long_v2}.

Similarly to Section~\ref{sec.slopeOLS.optimal}, expectations of the risk and empirical risk of a linear estimator can be computed as follows:
\begin{align}
\label{eq.EriskFhm.linear}
\E\crochj{\frac{1}{n} \norm{ \Fhm - F }^2} &= \frac{1}{n} \normb{ (\Id_n - A_m) F}^2 + \frac{\sigma^2 \tr\parenj{ A_m^{\top} A_m}} {n}
\, , 
\\
\label{eq.EriskempFhm.linear}
\E\crochj{\frac{1}{n} \norm{ \Fhm - Y }^2} &= \frac{1}{n} \normb{ (\Id_n - A_m) F}^2 + \frac{\sigma^2 \crochj{n + \tr\parenj{ A_m^{\top} A_m} - 2 \tr(A_m)}} {n}
\, ,
 \\
\label{eq.penid.linear}
\text{and} \qquad \penopt(m)
&= \E\crochj{\frac{1}{n} \norm{ \Fhm - F }^2} - \E\crochj{\frac{1}{n} \norm{ \Fhm - Y }^2} + \sigma^2
= \frac{ 2 \sigma^2 \tr(A_m)} {n}
\, .
\end{align}
Eq.~\eqref{eq.EriskFhm.linear} can be interpreted as a bias-variance decomposition similarly to Eq.~\eqref{eq.EriskFhm}.
The optimal penalty given by Eq.~\eqref{eq.penid.linear} has been called $C_L$ by \citet{Mal:1973} and is similar to $C_p\,$, with the dimension $D_m$ replaced by the \emph{degrees of freedom} $\tr(A_m)$.
It also depends on $\sigma^2$ which is unknown, so one could think of using Algorithm~\ref{algo.gal.slope.naif} with $\sh_m=\Fhm\,$, $\Remp\parens{t}=n^{-1}\norms{t-Y}^2$, $\pen_1(m)= \tr(A_m) / n$, and $\C_m = \tr(A_m)$.
Then, plotting $\C_{\mhslopeun(C)}$ as a function of~$C$, what we typically get is shown in Figure~\ref{fig.linear.jump} (black curve / diamonds): no clear jump of the complexity is observed around $\sigma^2$, contrary to what Algorithm~\ref{algo.gal.slope.naif} predicts.

\begin{figure}
\begin{center}
\includegraphics[width=0.49\textwidth]{\pathfig/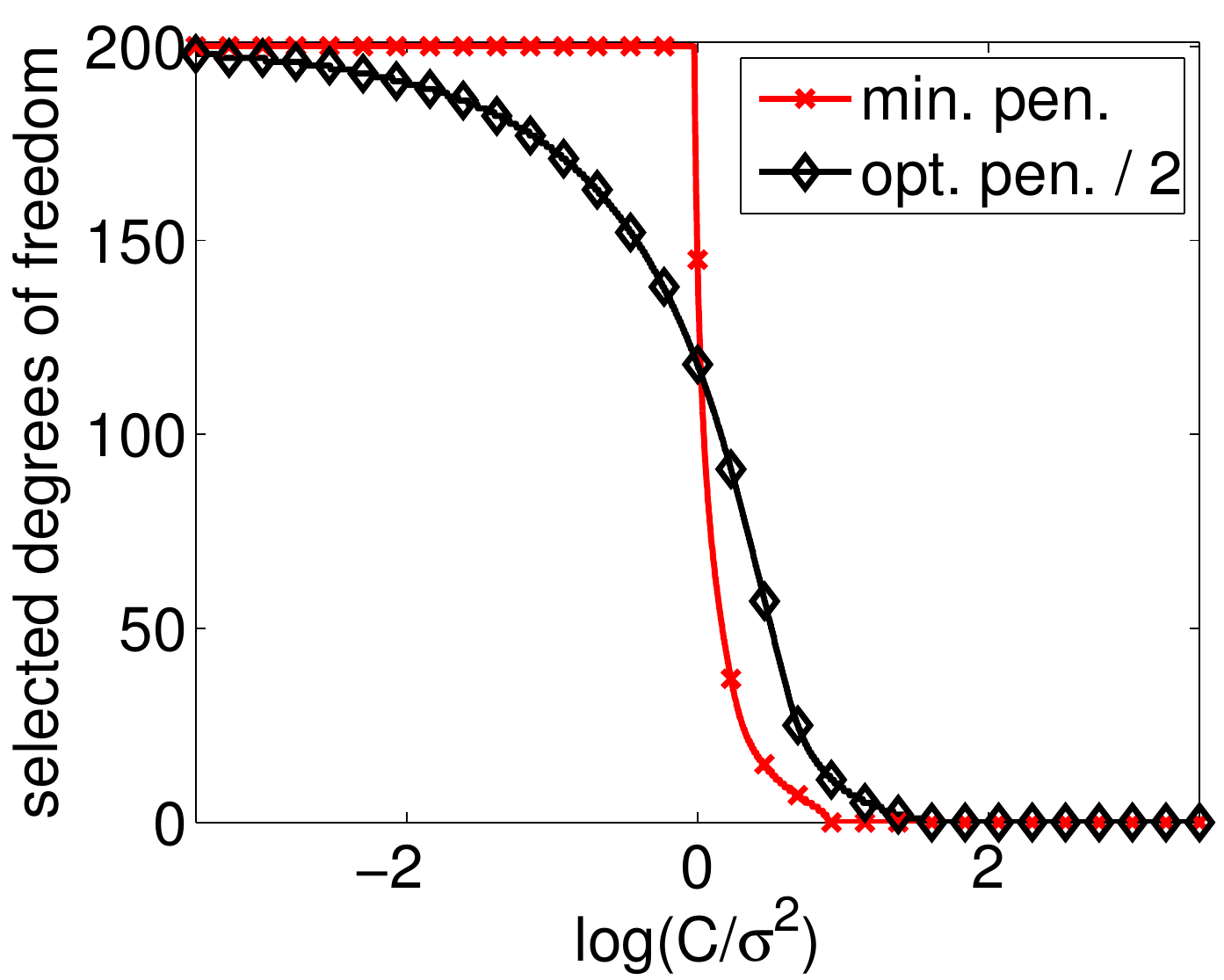}
\caption{\label{fig.linear.jump}
The minimal penalty is not proportional to $\tr(A_m)$ for kernel ridge estimators \textup{(}Figure taken from the article by \citet{Arl_Bac:2009:minikernel_long_v2}, `kernel ridge' framework, see Appendix~\ref{app.details-simus} for details\textup{)}\textup{:}
$C \mapsto \C_{\mhslopeun(C)}$ for Algorithm~\ref{algo.gal.slope.naif} with $\pen_1(m)=\tr(A_m) / n$ and $\C_m = \tr(A_m)$ \textup{(}black curve / diamonds), and $C \mapsto \tr(A_{\mhminlin(C)})$ for Algorithm~\ref{algo.penmin.linear} \textup{(}red curve / crosses\textup{)} with linear estimators \textup{(}kernel ridge\textup{)}.
% Y_i = \sin(25 \pi X_i^3) + \varepsilon_i$
}
\end{center}
\end{figure}

\subsection{Minimal-penalty heuristics for linear estimators} \label{sec.penmingal.linear}
Following \citet{Arl_Bac:2009:minikernel_nips,Arl_Bac:2009:minikernel_long_v2}, the correct minimal penalty in the linear-estimators framework is
\[
\penminlin(m)
\egaldef
\frac{\sigma^2 \crochj{2 \tr(A_m) - \tr\parenj{ A_m^{\top} A_m}}} {n}
\, .
\]
Indeed, as in Section~\ref{sec.slopeOLS.penmin}, let us consider, for every $C \geq 0$,
\begin{align}
\label{eq.mhC.linear}
\mhminlin(C) &\in \argmin_{\mM} \set{ \frac{1}{n} \norm{\Fhm - Y}^2 + C \frac{2 \tr(A_m) - \tr\paren{ A_m^{\top} A_m}} {n} }
\, ,
\\
\text{and}
\quad
\mominlin(C) &\in \argmin_{\mM} \set{ \E\croch{ \frac{1}{n} \norm{\Fhm - Y}^2 + C \frac{2 \tr(A_m) - \tr\parenj{ A_m^{\top} A_m}} {n} } }
\notag
\\
\notag
&= \argmin_{\mM} \set{ \crit_C^{\mathrm{lin}}(m)  }
\\
\notag
\text{with} \quad
\crit_C^{\mathrm{lin}}(m) &\egaldef \frac{1}{n} \parenB{ \normb{ (\Id_n - A_m) F}^2 + ( C - \sigma^2) \crochb{2 \tr(A_m) - \tr\parens{ A_m^{\top} A_m}} }
\, ,
%\label{eq.critC}
\end{align}
by Eq.~\eqref{eq.EriskempFhm.linear}.
Let us assume that the approximation error term $n^{-1} \norms{ (\Id_n - A_m) F}^2$, which appears in Eq.~\eqref{eq.EriskFhm.linear}, is a decreasing function of the degrees of freedom $\C_m=\tr(A_m)$; 
let us also assume for simplicity that 
$0 \leq \tr\parens{ A_m^{\top} A_m} \leq \tr(A_m)$ for every $m \in \M$. 
Then, we can distinguish two cases:
\begin{itemize}
\item if $C < \sigma^2$, then $\crit_C^{\mathrm{lin}}(m)$ is a decreasing function of $\C_m\,$, and $\C_{\mominlin(C)}$ is huge: $\mominlin(C)$ overfits.
\item if $C > \sigma^2$, then $\crit_C^{\mathrm{lin}}(m)$ increases with $\C_m$ for $\C_m$ large enough, so $\C_{\mominlin(C)}$ is much smaller than when $C<\sigma^2$.
\end{itemize}
This behavior is also the one of $\mhminlin(C)$, as illustrated in Figure~\ref{fig.linear.jump} (red curve / crosses), which leads to the following algorithm.
\begin{algo}[Minimal-penalty algorithm for linear estimators] \label{algo.penmin.linear} 
\hfill \newline 
\algoinput\textup{:} $\parenb{ \norms{ \Fhm - Y }^2 }_{\mM}\,$, $\parenb{\tr \parens{A_m}}_{\mM}\,$, and $\parenb{\tr \parens{A_m^{\top} A_m}}_{\mM}\,$.
\begin{enumerate}
\item Compute $\parenb{ \mhminlin(C) }_{C \geq 0}\,$, where $\mhminlin(C)$ is defined by Eq.~\eqref{eq.mhC.linear}.
\item Find $\Chjumpgal>0$ corresponding to the ``unique large jump'' of $C \mapsto \tr \parenb{ A_{\mhminlin(C)} }$.
\item Select $\mhAlgD \in \argmin_{\mM} \setb{ n^{-1} \norms{ \Fhm - Y }^2 + 2 \Chjumpgal \tr(A_m) / n }$.
\end{enumerate}
\algooutput\textup{:} $\mhAlgD\,$.
\end{algo}
Theorem~\ref{thm.OLS} can be extended to Algorithm~\ref{algo.penmin.linear}, up to some minor changes in the assumptions and results \citep{Arl_Bac:2009:minikernel_nips,Arl_Bac:2009:minikernel_long_v2}.
For kernel ridge regression, Algorithm~\ref{algo.penmin.linear} is proved to work also for choosing over a continuous set $\M$ \citep{Arl_Bac:2009:minikernel_long_v2}, provided the kernel is fixed.

Note that when $\tr\parens{ A_m^{\top} A_m} =  \tr(A_m)$, 
$\penopt(m) = 2\penminlin(m)$. 
This occurs for least-squares estimators ---for which $A_m = A_m^{\top} A_m = \Pi_m\,$, 
and we recover the setting of Section~\ref{sec.slopeOLS} 
and Algorithm~\ref{algo.OLS.jump}--- 
and for $k$-nearest neighbors estimators.

\subsection{General minimal-penalty algorithm} \label{sec.penmingal.algo}
We now go back to the general setting of Section~\ref{sec.penmingal.framework}, and propose a generalization of Algorithms~\ref{algo.OLS.jump} and \ref{algo.penmin.linear}.
Here, we suggest to take $\Chjumpgal=\Chwindow\,$, but any other formal definition of $\Chjumpgal$ could be used instead.
\begin{algo}[General minimal-penalty algorithm, jump formulation] \label{algo.penmingal}
\hfill \newline 
\algoinput\textup{:} $\parenb{\Remp\parens{\shm}}_{\mM}\,$, $\parenb{\pen_0(m)}_{\mM}\,$, $\parenb{\pen_1(m)}_{\mM}\,$, $\parens{\C_m}_{\mM}\,$, and $\eta \geq 0$.
\begin{enumerate}
\item Compute $\parenb{ \mhgalzero(C) }_{C \geq 0}\,$, where for every $C \geq 0$,
\begin{equation} \label{eq.mhC.gal}
\mhgalzero(C) \in \argmin_{\mM} \set{ \Remp\parens{\shm} + C \pen_0(m) } \, .
\end{equation}
\item Find $\Chjumpgal>0$ corresponding to the ``unique large jump'' of $C \mapsto \C_{\mhgalzero(C)}\,$,
for instance,
\[
\Chwindow \in \argmax_{C>0} \setj{ \C_{\mhgalzero(C /[1+\eta])} -  \C_{\mhgalzero(C [1+\eta])} }
\, .
\]
\item Select $\mhAlgE \in \argmin_{\mM} \setb{ \Remp\parens{ \shm } + \Chjumpgal \pen_1(m) }$.
\end{enumerate}
\algooutput\textup{:} $\mhAlgE\,$.
\end{algo}
Algorithm~\ref{algo.penmingal} implicitly assumes that the minimal and the optimal penalty are respectively equal to $C^{\star} \pen_0$ and $C^{\star}\pen_1\,$, with $\pen_0$ and $\pen_1$ known, but $C^{\star}$ unknown.
We refer to Section~\ref{sec.practical.jump-vs-slope} for practical remarks about the choice of $\Chjumpgal\,$. 
Computational issues are discussed in Section~\ref{sec.practical.cost}. 

In the ``slope heuristics'' setting (Section~\ref{sec.slopeOLS}), $\pen_1=2\pen_0\,$, 
and Algorithm~\ref{algo.penmingal} reduces to Algorithm~\ref{algo.gal.slope.naif}.

Similarly to Algorithm~\ref{algo.OLS.slope}, we can also propose a ``slope'' formulation of Algorithm~\ref{algo.penmingal}.
\begin{algo}[General minimal-penalty algorithm, slope formulation] \label{algo.penmingal.slope} 
\hfill \newline 
\algoinput\textup{:} $\parenb{\Remp\parens{\shm}}_{\mM}\,$, $\parenb{\pen_0(m)}_{\mM}\,$, $\parenb{\pen_1(m)}_{\mM}\,$, and $\parens{\C_m}_{\mM}\,$.
\begin{enumerate}
\item Estimate the slope $\Chslope$ of $-\Remp\parens{\shm}$ as a function of $\pen_0(m)$ for all $\mM$ with $\C_m$ ``large enough'', for instance by \textup{(}robust\textup{)} linear regression.
\item Select $\mhAlgF \in \argmin_{\mM} \setb{ \Remp\parens{ \shm } + \Chslope \pen_1(m) }$.
\end{enumerate}
\algooutput\textup{:} $\mhAlgF\,$.
\end{algo}

What remains now is to identify natural candidates for being a minimal or an optimal penalty in the general setting.

\subsection{Optimal and minimal penalties} \label{sec.penmingal.penoptmin}
In the general setting, the unbiased risk estimation heuristics \citep{Aka:1969,Ste:1981} suggests the following optimal (deterministic) penalty
\begin{equation}
\label{eq.penopt.gal}
\penoptgal(m) \egaldef \E\crochb{ \Risk\parens{\shm} - \Remp\parens{\shm} }
\end{equation}
which generalizes formula~\eqref{eq.penoptz}.
If $\Risk\parens{\shm} - \Remp\parens{\shm}$ is concentrated around its expectation uniformly over $\mM$ ---which excludes too large collections $\M$---,
one can prove an oracle inequality for the penalty \eqref{eq.penopt.gal},
or for any penalty which differs from Eq.~\eqref{eq.penopt.gal}
by an additive term independent from~$m$,
as in Eq.~\eqref{eq.penid} and~\eqref{eq.penid.linear}.

Building a minimal penalty in the general setting is more difficult.
For every $\mM$, let $\C_m$ be some ``complexity measure'' associated with $\shm\,$, 
that is, we assume that the empirical risk $\Remp\parens{\shm}$ (or its expectation) is (approximately) a decreasing function of $\C_m\,$.
Algorithm~\ref{algo.penmingal.slope} suggests that the minimal penalty is a quantity which exactly compensates this decreasing trend, such as 
\begin{equation}
\label{eq.penmingala}
\penmingala(m) \egaldef -\E\crochb{ \Remp\parens{\shm} }
\, .
\end{equation}
Nevertheless, in most cases including least-squares and linear estimators, $\penmingala$ is unknown, even up to a multiplicative factor, so that we need another candidate for being a minimal penalty.

For every $\mM$, let $\bayes_m$ be a well-chosen element of $\Set$ 
---see Remark~\ref{rk.penmingal.bayesm} below--- such that
$\Risk\parens{\bayes_m} - \inf_{t \in \Set} \Risk\parens{t}  \text{ decreases to zero as }
\C_m \to \infty$. 
As argued below, a natural choice for the minimal penalty is
\begin{equation}
\label{eq.penmingal}
\penmingal(m) \egaldef \E\crochb{ \Remp\parens{\bayes_m} - \Remp\parens{\shm} }
\, .
\end{equation}
Indeed, for every $C \geq 0$ let
\begin{align}
\critCgal (m) \egaldef \E\croch{\Remp\parens{\shm}} + C \penmingal(m)
= C \Risk\parens{\bayes_m} + (1-C) \E\croch{ \Remp\parens{\shm} }
\label{eq.gal.critC}
\end{align}
so that 
$\momingal(C) \in \argmin_{\mM} \sets{ \critCgal (m) } $
is a proxy for
\[
\mhmingal(C) \in \argmin_{\mM} \setb{ \Remp\parens{\shm} + C \penmingal(m) }
\, .
\]
Let us assume for simplicity that $\Remp\parens{\shm}$ is a decreasing function of $\C_m\,$.
Then, when $C<1$, $\critCgal(m)$ is a decreasing function of $\C_m\,$, so that $\C_{\momingal(C)} \approx \max_{\mM} \C_m$ which corresponds to overfitting.
On the contrary, when $C>1$, $(1-C) \E\crochs{ \Remp\parens{\shm} } $ 
is an increasing function of $\C_m$ while $C \Risk\parens{\bayes_m}$ is approximately constant for $\C_m$ large enough, 
so that $\C_{\momingal(C)} \ll \max_{\mM} \C_m\,$.
Therefore, if concentration inequalities show that $\mhmingal(C)$ behaves likes $\momingal(C)$, $\penmingal$ is a minimal penalty.

Let us emphasize that for making use of the fact that $\penmingal$ is a minimal penalty, 
we must assume that $\penmingal \approx C^{\star} \pen_0$  and $\penoptgal \approx C^{\star} \pen_1$ 
for some unknown $C^{\star}>0$ and some \emph{known} penalty shapes $\pen_0$ and $\pen_1\,$.
Remark that we could generalize this assumption to the existence of some known function $f$ such that $\penmingal \approx C^{\star} \pen_0$  and $\penoptgal \approx f(C^{\star}) \pen_1$, but such a generalization has not been proved useful yet.
\begin{remark}[Choice of $\bayes_m$] \label{rk.penmingal.bayesm} 
Overall, the above heuristics makes two assumptions on $\bayes_m \in \Set$. 
First, $\Risk\parens{\bayes_m} - \inf_{t \in \Set} \Risk\parens{t}$ is 
small when $\C_m$ is large. 
Second, $\penmingal(m)$ is known up to a multiplicative constant, 
whose value can be used for deriving an optimal penalty. 
When $\shm \in \argmin_{t \in S_m} \Remp\parens{t}$ is an empirical risk minimizer over some model $S_m \subset \Set$, a natural choice is $\bayes_m \in \argmin_{t \in S_m} \Risk\parens{t}$,
so that
$\Risk\parens{\bayes_m} - \inf_{t \in \Set} \Risk\parens{t}$ is the approximation error.
For linear estimators, the decomposition \eqref{eq.EriskFhm.linear} of the risk suggests to take $\bayes_m = A_m F$.
By analogy, we call $\Risk\parens{\bayes_m} - \inf_{t \in \Set} \Risk\parens{t}$ 
the approximation error associated with $\shm$ in the general case. 
Choosing $\bayes_m$ might be difficult in general; 
when this makes sense, an option is the expectation of~$\shm\,$. 
\end{remark}

\subsection{Bibliographical remarks} \label{sec.penmingal.history}
\paragraph{Algorithms}
%%% Algos 1: Prehistoire: propositions de l'algo (version pente) et de la penalite minimale generale
%
The slope-heuristics algorithm for calibrating penalties was first proposed in the Gaussian least-squares regression setting of Section~\ref{sec.slopeOLS} with a penalty proportional to the dimension \citep{Bir_Mas:2001}, as in Algorithms~\ref{algo.OLS.jump}--\ref{algo.OLS.slope}.
Then, it was generalized to a penalty function of the dimension \citep{Bir_Mas:2006}, and to a general penalty shape \citep{Mas:2005,Bla_Mas:2006,Mas:2003:St-Flour}, as in Algorithm~\ref{algo.gal.slope.naif}.

%%% Algos 2: Prehistoire: implementations
%
The first implementations of the slope heuristics were done directly with Algorithm~\ref{algo.gal.slope.naif} 
(or its ``slope'' version) with $\C_m=D_m\,$, instead of Algorithms~\ref{algo.OLS.jump}--\ref{algo.OLS.slope}, 
since they were outside the setting of Section~\ref{sec.slopeOLS}: maximum-likelihood estimators 
\citep[Section~A.4]{Let:2000}, and change-point detection \citep[Chapter~4]{Leb:2002}.

%%% Algos 3: utilisation d'une mesure de complexite generale
%
The proposition of using a general complexity measure $\C_m$ instead of a dimension $D_m$ (as in Algorithms~\ref{algo.gal.slope.naif}, \ref{algo.penmingal}--\ref{algo.penmingal.slope}) was first made in density estimation \citep{Ler:2009:phd}, with the suggestion of estimating $\C_m$ by resampling if necessary.

%%% Algos 4: Historique: generalisations hors du cadre "pente" pur et dur

The failure of Algorithm~\ref{algo.gal.slope.naif} for linear estimators in regression was noticed by \citet{Arl_Bac:2009:minikernel_nips}, where Algorithm~\ref{algo.penmin.linear} was proposed and theoretically justified.
The general Algorithm~\ref{algo.penmingal} has only been formalized by \citet[Section~2.5]{Arl:2011:cours_Peccot}, while its ``slope estimation'' version (Algorithm~\ref{algo.penmingal.slope}) is new, even if the (approximate) equivalence between ``jump'' and ``slope'' algorithms is not.
Up to now, the general formulation of Algorithms~\ref{algo.penmingal}--\ref{algo.penmingal.slope} has only been proved useful in the case of linear estimators
in regression \citep{Arl_Bac:2009:minikernel_nips,Arl_Bac:2009:minikernel_long_v2} 
and in density estimation \citep{Mag:2015,Ler_Mag_Rey:2016},  
with different shapes for $\pen_0$ and $\pen_1\,$. 
It can also be useful in a few other settings where $\pen_1$ is proportional to $\pen_0$ 
but the ratio between optimal and minimal penalty 
might be different from~$2$, 
for selecting among a rich collection of models or estimators. 
%
%%% Algos 4-1: une remarque sur la version "maximal plateau"
%
For instance, for pruning a decision tree, 
the ``max'' variant considered by \citet[Section~5.4]{Bar_Gey_Pog:2018} 
is equivalent to Algorithm~\ref{algo.penmingal} with $\pen_1 = \pen_0\,$, 
hence selecting the estimator ``just after'' the maximal jump.

\paragraph{Theory}
%
%%% Theorie 1: 1er resultat hors du cadre de la section 2
%
The first (partial) theoretical result proved outside the setting of Section~\ref{sec.slopeOLS} 
was for maximum-likelihood estimators (histograms) in density estimation, 
assuming that the true density $\bayes$ is the uniform density over $[0,1]$ \citep{Cas:1999}. 
Other theoretical results outside the setting of Section~\ref{sec.slopeOLS} 
are reviewed in Section~\ref{sec.theory}.

%%% Theorie 2: Vrais resultats avec des formes de penalites plus generales (sans la dimension)
%
The first theoretical result proved for Algorithm~\ref{algo.gal.slope.naif} with a penalty shape $\pen_1(m)$ \emph{not} function of a dimension $D_m$ was obtained in heteroscedastic least-squares regression \citep{Arl_Mas:2009:pente}, where the penalty shape can be estimated by resampling \citep{Arl:2009:RP}.

%%% Theorie 3: elements theoriques heuristiques generaux
%
The general heuristics ``$\penopt \approx 2 \penmin$'' underlying Algorithm~\ref{algo.gal.slope.naif} was formulated by \citet[Section~2]{Bla_Mas:2006} and \citet[Section~8.5.2]{Mas:2003:St-Flour}, together with a heuristic argument for suggesting
\[
p_2(m) \egaldef \Remp\parens{\bayes_m} - \Remp\parens{\shm}
\]
as a minimal penalty, when $\shm$ is an empirical risk minimizer and $\bayes_m$ is defined according to  Remark~\ref{rk.penmingal.bayesm}.
In these papers, $p_2(m)$ is called $\widehat{v}_m$  since it can be interpretated as a variance.
Here, the general minimal penalty that we propose is $\penmingal(m) = \E\crochs{p_2(m)}$, as in the PhD dissertation of \citet[Chapter~3]{Arl:2007:phd} for instance.
Another formulation of the heuristics behind Algorithm~\ref{algo.gal.slope.naif} 
is ``$p_1(m) \approx p_2(m)$'', where
\[
p_1(m) \egaldef \Risk\parens{\shm} - \Risk\parens{\bayes_m}
\, ,
\]
as for instance written in a binary classification framework by \citet[Section~6.4.3]{Zwa:2005:phd}, 
together with ``$p_2(m) \propto D_m$ for $D_m$ large enough''.

%%%%%%%%%%%%%%%%%%%%%%%%%%%%%%%%%%%%%%%%%%%%%%%%%%%%%%%%%%%%%%%%%%%%%%%%%%%%%%%%%%%%
%%%%%%%%%%%%%%%%%%%%%%%%%%%%%%%%%%%%%%%%%%%%%%%%%%%%%%%%%%%%%%%%%%%%%%%%%%%%%%%%%%%%

\section{Theoretical results: a review} \label{sec.theory}
This section collects all theoretical results that are directly related to minimal-penalty algorithms, to the best of our knowledge.
First, the proof of Algorithm~\ref{algo.penmingal} is split into several subproblems (Section~\ref{sec.theory.approach}).
Then, we present full proofs of Algorithm~\ref{algo.penmingal} (Section~\ref{sec.theory.complete}) and partial results (Sections~\ref{sec.theory.penopt.weak}--\ref{sec.theory.rich}). 
Note that some related results outside the setting 
of Section~\ref{sec.penmingal} are reported in conclusion 
(Sections \ref{sec.empirical.conjectures.identif} and~\ref{sec.empirical.outside}).

In this section, all partial or full proofs of Algorithm~\ref{algo.penmingal} that we present define $\Chjumpgal$ as
$\Chthr (T_n)$ or $\Chwindow (\eta)$ for some well-chosen $T_n$ or~$\eta$.
For the sake of simplicity, we do not discuss anymore the exact definition chosen for $\Chjumpgal\,$, until we tackle this question in Section~\ref{sec.practical.jump-vs-slope}.

\subsection{General approach for proving Algorithm~\ref{algo.penmingal}} \label{sec.theory.approach}
Following Theorem~\ref{thm.OLS} and its proof, let us suggest a general approach towards a theoretical justification of Algorithm~\ref{algo.penmingal}, that we split into several subproblems.
\begin{itemize}
\item[(\pbmult)] \textbf{The minimal and optimal penalties are known up to some common multiplicative factor:}
Find two penalty functions $\pen_0\,$, $\pen_1\,$, and a complexity measure $(\C_m)_{\mM}\,$, such that for some (unknown) $C^{\star}>0$, $C^{\star} \pen_0$ is a minimal penalty and $C^{\star} \pen_1$ is an optimal penalty.
\item[(\pbpenmin)] \textbf{$C^{\star} \pen_0$ is actually a minimal penalty:}
$\etamoins, \etaplus>0$ exist such that, 
on a large-probability event, 
\begin{align}
\tag{\ensuremath{\pbpenmin^-}}
\label{pb.penmin.Cpt-Dgrd}
&\forall C < \paren{ 1 - \etamoins } C^{\star}
\, ,
\qquad
\C_{\mhgalzero(C)} \geq \C_{\mathrm{overfit}} \propto \max_{\mM} \C_m
\\
\tag{\ensuremath{\pbpenmin^+}}
\label{pb.penmin.Cgrd-Dpt}
&\forall C > \paren{ 1 + \etaplus } C^{\star}
\, ,
\qquad
\C_{\mhgalzero(C)} \leq \C_{\mathrm{small}} \ll \max_{\mM} \C_m
\\
\notag
\text{where}
\qquad
&\forall C \geq 0 \, ,
\quad
\mhgalzero(C) \in \argmin_{\mM} \set{ \Remp\parens{\shm} + C \pen_0(m) }
\, .
\end{align}
The above statements about $\C_{\mhgalzero(C)}$ are vague on purpose, 
since the range of $(\C_m)_{\mM}$ is not specified.
When $\C_m$ is the dimension $D_m \in [1,n]$ of some model, 
one can specify $\C_{\mathrm{overfit}}$ and $\C_{\mathrm{small}}$ 
similarly to Eq.~\eqref{eq.thm.OLS.Cpt-Dgrd} and~\eqref{eq.thm.OLS.Cgrd-Dpt}, respectively.

\item[(\pbpenopt)] \textbf{$C^{\star} \pen_1$ is actually an optimal penalty:} 
there exists $\eta>0$ such that, 
on a large-probability event, 
for every $C \in ([1-\eta] C^{\star} \, , \, [1+\eta] C^{\star})$,
\begin{gather}
\notag
\forall \mhoptun(C) \in \argmin_{\mM} \set{ \Remp\parens{\shm} + C \pen_1(m) } \, ,
\\
\label{pb.penopt}
\tag{\pbpenopt}
\Risk\parenB{ \sh_{\mhoptun(C)}} - \Risk\parens{\bayes} \leq \parenb{ 1+\varepsilon_n(\eta) } \inf_{\mM} \setb{\Risk\parens{\shm} - \Risk\parens{\bayes}} + R_n(\eta)
\end{gather}
where $\lim_{\eta \to 0, \, n \to + \infty } \varepsilon_n(\eta) = 0$ and the remainder term $R_n(\eta)$ 
is negligible in front of the oracle risk $\inf_{\mM} \sets{\Risk\parens{\shm} - \Risk\parens{\bayes}}$.
\end{itemize}

As in the PhD dissertation of \citet{Arl:2007:phd}, we use in this section the following notation:
\begin{align}
\label{def.p1}
p_1(m) &\egaldef \Risk\parens{\shm} - \Risk\parens{\bayes_m} 
\, , 
\\
\label{def.p2}
p_2(m) &\egaldef \Remp\parens{\bayes_m} - \Remp\parens{\shm}
\, , 
\\
\label{def.delta}
\text{and} \qquad 
\ovdelta(m) &\egaldef \Risk\parens{\bayes_m} - \Remp\parens{\bayes_m}
\, .
\end{align}
In particular, with the notation of Section~\ref{sec.penmingal.penoptmin},
\[
\penoptgal(m) = \E\crochb{ p_1(m) + \ovdelta(m) + p_2(m)}
\qquad \text{and} \qquad
\penmingal(m) = \E\crochb{p_2(m)}
\, .
\]

\subsection{Full proofs of Algorithm~\ref{algo.penmingal}} \label{sec.theory.complete}
%
%(\pbmult)
%\eqref{pb.pen-known}
%
%(\pbpenmin)
%\eqref{pb.penmin.Cpt-Dgrd}
%\eqref{pb.penmin.Cgrd-Dpt}
%
%(\pbpenminalt)
%\eqref{pb.penmin.Cpt-Rgrd}
%\eqref{pb.penmin.Cgrd-Rpt}
%
%(\pbpenopt)
%\eqref{pb.penopt}
%\eqref{pb.penopt.weak}
%
%
Few settings exist where a full proof of Algorithm~\ref{algo.penmingal} is available, that is, 
a proof that \eqref{pb.penmin.Cpt-Dgrd}, \eqref{pb.penmin.Cgrd-Dpt}, and \eqref{pb.penopt} 
hold true on a large-probability event for some known $\pen_0\,$, $\pen_1\,$, $(\C_m)_{\mM}$ and 
some (unknown) $C^{\star}$.
In this article, $\M$ is always assumed to be finite with $\card(\M) \leq L_1 n^{L_2}$ for some $L_1, L_2>0$, except in Section~\ref{sec.theory.rich}. 

We first collect results assuming that $\pen_1 = 2 \pen_0\,$,
so that Algorithm~\ref{algo.penmingal} reduces to Algorithm~\ref{algo.gal.slope.naif}.
Without explicit mention of the contrary, for all results reviewed in the list below,
the noise is assumed independent and identically distributed,
$\shm \in \argmin_{t \in S_m} \Remp\parens{t}$ is an empirical risk minimizer,
so we take $\bayes_m \in \argmin_{t \in S_m} \Risk\parens{t}$ for defining $p_2(m)$,
and the complexity used is $\C_m=D_m$ the dimension of $S_m\,$.
Full proofs of Algorithm~\ref{algo.penmingal} exist in the following settings:
\begin{itemize}
\item \emph{Regression on a fixed design, homoscedastic (sub-)Gaussian noise, least-squares risk and estimators}:
\citet{Bir_Mas:2006} and Theorem~\ref{thm.OLS} 
(and Remark~\ref{rk.thm.OLS.subgaussian} for the sub-Gaussian case) 
prove it 
with $\pen_0(m)=D_m/n$ and $C^{\star}=\sigma^2$ the (constant) noise level.
Note that $p_1(m)=p_2(m)=n^{-1} \norms{\Pi_m \varepsilon}^2$ in this setting. 
\item Regression on a \emph{random design}, \emph{heteroscedastic noise} (not necessarily Gaussian), least-squares risk, with various least-squares estimators:
regressograms with moment assumptions on the noise \citep{Arl_Mas:2009:pente},
piecewise polynomials with bounded noise 
(\citealp{Sau:2010:Reg}, with key concentration results for $p_1$ and $p_2$ proved by \citealp{Sau:2010:conc-p1-p2}),
or more general models satisfying a ``strongly-localized basis'' assumption with bounded noise \citep{Sau:2010:supnorm,Sau_Nav:2017}.
Contrary to the previous setting, $\E\crochs{p_1(m)} \approx \E\crochs{p_2(m)}$ 
holds true only for most models and for $n$ large enough. 
\\
The penalty shape $\pen_0(m)=\E\crochs{p_2(m)}$ is unknown in general and $C^{\star}=1$.
For regressograms, the results remain true when $\pen_0(m)$ is a resampling-based estimation of $\E\crochs{p_2(m)}$ \citep[see][]{Arl:2009:RP}.
For piecewise polynomials, the same holds when $\pen_0(m)$ is a hold-out estimation of $\E\crochs{p_2(m)}$ \citep[see][]{Sau:2010:Reg}. 
For strongly localized bases, the (approximate) closed-form formula 
for $\E\crochs{p_1(m)}$ and $\E\crochs{p_2(m)}$ provided by \citet[Theorem~6.3]{Sau_Nav:2017} 
might be used for estimating $\pen_0(m)$ without resampling; 
another option is $V$-fold penalization \citep[Section~5]{Sau_Nav:2017}. 
\item \emph{Density estimation}, least-squares risk and estimators, i.i.d.\@ \citep{Ler:2010:iid} or \emph{mixing data} \citep{Ler:2010:mixing}.
The penalty shape $\pen_0(m)=\E\crochs{p_2(m)}$ is approximately known for some specific models (regular histograms), in general it can be estimated by resampling as previously.
In this setting, the complexity $\C_m$ can either be the dimension of $S_m$ or the resampling-based estimator of $\E\crochs{p_2(m)}$ itself. 
Note that in least-squares density estimation, we have $p_1(m)=p_2(m)$ almost surely.
\item Density estimation, \emph{Kullback risk and maximum-likelihood estimators}, histogram models \citep{Sau:2010:MLE}.
This result is the first one obtained without the least-squares risk.
The penalty shape $\pen_0(m)=D_m/(2n)$ is known, $C^{\star}=1$, and the optimal penalty is AIC. 
A partial result, for the uniform density over $[0,1]$ only, has previously been proved by \citet{Cas:1999}. 
\item \emph{Specification probabilities in general random fields} (that is, graphical models), 
least-squares or Kullback risks, estimators that are empirical distributions conditionally to the values observed on a subset $m$ of the field \citep{Ler_Tak:2011}. 
The shape of the penalty and the complexity $\pen_0(m) = \C_m = p_2(m)$ are unknown.
The authors suggest to use instead the shape of a theoretical upper bound on $\E\croch{p_2(m)}$, dropping off pessimistic constants, with convincing experimental results.
\end{itemize}
The above results for least-squares regression on a random design, 
least-squares density estimation (i.i.d.\@ case), and maximum-likelihood density estimation 
can all be recovered (sometimes up to minor differences) as a corollary of a general result 
which holds for all ``regular estimators'' \citep[Chapters~7--8]{Sau:2010:phd}.

\medbreak

Full proofs of Algorithm~\ref{algo.penmingal} (or a slight modification of it) also exist
in two settings where $\pen_1 \neq 2 \pen_0$ in general:
\begin{itemize}
\item Regression on a fixed design, independent and identically distributed (homoscedastic) Gaussian noise, least-squares risk,
\emph{linear estimators}:
\citet{Arl_Bac:2009:minikernel_nips,Arl_Bac:2009:minikernel_long_v2} prove that Algorithm~\ref{algo.penmin.linear} works,
while Algorithm~\ref{algo.gal.slope.naif} fails in general,
as detailed in Sections~\ref{sec.penmingal.apriori}--\ref{sec.penmingal.fails}.
\item Density estimation, independent and identically distributed data,
least-squares risk, \emph{linear estimators} (for instance, Parzen density estimators and weighted least-squares estimators):
\citet{Ler_Mag_Rey:2016} ---after a preliminary version in the PhD dissertation of \citet[Chapter~2]{Mag:2015}--- 
define some theoretical quantities $\pen_0(m) \approx \E[p_2(m)]$ and $\C_m \approx \E[p_1(m)]$ 
---easy to estimate in general, and known for several examples such as Parzen density estimators--- 
such that 
\[ 
\mh(C) \in \argmin_{\mM} \set{ \Remp\parens{\shm} + \pen_0(m) + C \C_m }
\] 
overfits for $C<0$ 
and satisfies an oracle inequality for all $C>0$, 
first-order optimal when $C=C^{\star}=1$. 
In other words, \citet{Ler_Mag_Rey:2016} almost prove that 
Algorithm~\ref{algo.penmingal} works 
with $\pen_0(m)=\E[p_2(m)]$, 
$\pen_1(m)=\E[p_1(m)] + \E[p_2(m)]$, 
$\C_m = \E[p_1(m)]$, and $C^{\star}=1$. 
This result implies the one of \citet{Ler:2010:iid} for least-squares estimators.
A noticeable fact in this framework is that
$\penmin$ ---and sometimes even $\penopt$, surprisingly--- can be negative, 
making the terminology ``minimal'' penalty questionable 
\citep[Sections~4.3 and~5]{Ler_Mag_Rey:2016}. 
Note that $\penopt$ here can be negative because $\sh_m$ is not an empirical 
risk minimizer, hence $\Remp\parens{\shm}$ is not necessarily biased downards 
as an estimator of $\Risk\parens{\shm}$. 
Nevertheless, for most usual estimators, $\penmin$ and $\penopt$ are always positive.
Theoretical results for choosing among Parzen density estimators with slightly different 
minimal-penalty algorithms ---closer to Goldenshluger-Lepski's method--- are discussed 
in Section~\ref{sec.empirical.outside}. 
\end{itemize}

\medbreak

\begin{remark}[Minimal penalties with resampling-based estimators of $\E\crochs{p_2(m)}$]
\label{rk.reech}
In several papers mentioned above, theoretical results validate Algorithm~\ref{algo.penmingal} with $\pen_0(m) = \C_m = p_2(m)$ or $\E\croch{p_2(m)}$, $\pen_1(m) = 2 \pen_0(m)$, and $C^{\star}=1$.
Such results might seem useless since
(i) $\pen_0$ is unknown, and
(ii) $C^{\star}$ is known,
that is, the exact opposite of the motivation for Algorithm~\ref{algo.penmingal} 
exposed in Section~\ref{sec.penmingal.apriori}.
Nevertheless, problem (i) can be solved by taking $\pen_0(m)=\C_m$ equal to 
a resampling-based estimator $\phW_2(m)$ of $\E\croch{p_2(m)}$ \citep[for instance]{Ler:2010:iid}.
Then, as usual with resampling, it remains to find the constant $C_W$ such that $C_W \E\croch{\phW_2(m)} \approx \E\croch{p_2(m)}$ for all $\mM$.
When such a constant $C_W$ exists, it usually depends on the resampling scheme $W$, the sample size, and the particular setting considered \citep{Arl:2009:RP}.
As a consequence, we recover a setting where $\pen_0(m)=\C_m=\phW_2(m)$ is known and $C^{\star}=C_W$ is unknown, for which Algorithm~\ref{algo.penmingal} can be useful.
Note that we here propose to take $\C_m = \phW_2(m)$, which does not estimate $\E\croch{p_2(m)}$ but $\E\croch{p_2(m)}/C_W\,$; 
this is not a problem since the complexity jump is independent from the rescaling by $C_W\,$. 
\\ 
According to simulation experiments, 
the above strategy of combining resampling penalties with the slope heuristics 
can be better 
\citep[least-squares density estimation]{Ler:2010:iid:v2} 
or worse \citep[context-tree estimation]{Gar_Ler:2011}  
compared to using the penalty $\widetilde{C}_W \phW_2(m)$, 
where $\widetilde{C}_W$ derives from asymptotic theoretical results 
and does not depend on any unknown quantity in the settings of these two articles. 
\end{remark}

\subsection{Partial proofs: uncertainty on the optimal penalty} \label{sec.theory.penopt.weak}
An optimal oracle inequality like \eqref{pb.penopt} in Section~\ref{sec.theory.approach} 
has not been proved in many frameworks, and it is quite difficult to obtain a leading constant $1+\varepsilon_n(\eta) = 1 +\petito(1)$ 
while keeping the remainder term negligible in front of the oracle risk.
A much more usual result in the model-selection literature is the following weakened version of \eqref{pb.penopt}: 
on a large-probability event, 
for every $C \in ( [1-\eta] C^{\star} \, , \, [1+\eta] C^{\star})$ with $\eta>0$ small enough,
\begin{gather}
\label{pb.penopt.weak}
\tag{\ensuremath{\widetilde{\pbpenopt}\,}}
\Risk\parenB{ \sh_{\mhoptun(C)}} - \Risk\parens{\bayes} \leq K_n(\eta) \inf_{\mM} \setb{\Risk\parens{\shm} - \Risk\parens{\bayes}} + R_n(\eta)
\end{gather}
for some $K_n(\eta),R_n(\eta)<\infty$.
Note that \eqref{pb.penopt} with $\varepsilon_n=0$ and 
$R_n \geq \inf_{\mM} \sets{\Risk\parens{\shm} - \Risk\parens{\bayes}}$ 
should be understood as \eqref{pb.penopt.weak} with $K_n\geq 2$. % and $R_n=0$.
Similarly, a classical way to write an oracle-type inequality is
\begin{gather}
\label{pb.penopt.weak2}
\tag{\ensuremath{\widetilde{\pbpenopt}^{\, \prime}}}
     \Risk\parenB{ \sh_{\mhoptun(C)}} - \Risk\parens{\bayes} 
\leq \inf_{\mM} \setb{\Risk\parens{\shm} - \Risk\parens{\bayes} + R_n(m)}
\, .
\end{gather}
When $R_n(m)$ is comparable to $\Risk\parens{\shm} - \Risk\parens{\bayes}$, or even larger, 
(for instance, $R_n(m) \geq \pen(m)$), \eqref{pb.penopt.weak2} 
should be understood as \eqref{pb.penopt.weak} with $K_n\geq 2$.

Proving only \eqref{pb.penopt.weak} instead of \eqref{pb.penopt} is a significant limitation: 
\eqref{pb.penopt.weak} does not show that $C^{\star} \pen_1$ is an optimal penalty 
if we cannot prove that $K_n(\eta)$ is first-order optimal, 
which is very difficult to prove unless $K_n(\eta) = 1 + \petito(1)$ as in \eqref{pb.penopt}.
As a consequence, in such cases, $C^{\star} \pen_1$ might not be optimal, 
and the optimal penalty might be $C^{\prime} \pen_1$ with $C^{\star} \neq C^{\prime}$, 
or even have a completely different shape than $\pen_1\,$.
For instance, in the setting of Section~\ref{sec.penmingal.linear}, 
the optimal penalty is $2 \sigma^2 \tr(A_m) / n$,
but taking $2 \sigma^2 [\tr(A_m) + \tr(A_m^{\top} A_m)] / n$ as a penalty,
we could have an oracle inequality \eqref{pb.penopt.weak} with a penalty having a suboptimal shape.

\medbreak

\paragraph{Results}
Nevertheless, proving (\pbpenmin) and \eqref{pb.penopt.weak} still shows that
Algorithm~\ref{algo.penmingal} provides a data-driven estimator satisfying an oracle inequality.
Such a result exists for context-tree estimation with the Kullback risk, $\phi$-mixing processes, and maximum-likelihood estimators \citep{Gar_Ler:2011},
with $\pen_0(m) = \C_m = p_2(m)$ and $\pen_1(m) = 2 \pen_0(m)$.
Simulation experiments suggest that $p_2(m)$ can be replaced by a BIC-type penalty or a resampling-based estimator of $\E\croch{p_2(m)}$, 
see Remark~\ref{rk.reech} in Section~\ref{sec.theory.complete}.
What is missing to get a proof of \eqref{pb.penopt} 
is a tight concentration inequality for $\delta(m)-\delta(m^{\prime})$, that is, to have Eq.~\eqref{eq.altproof.controldelta} satisfied with $\varepsilon_{\delta} = \petito(1)$ as required in Proposition~\ref{pro.pbpenmin.gal.above}, see Section~\ref{sec.theory.hint.penmin.above}.
Simulation experiments suggest that $\pen_1=2\pen_0$ is indeed an optimal choice here.

\subsection{Minimal penalty in terms of risk: \texorpdfstring{(\pbpenminalt)}{(beta')}} \label{sec.theory.pbpenminalt}
Another way to define a minimal penalty is in terms of the risk of $\sh_{\mhgalzero(C)}\,$, which is theoretically interesting but does not prove the presence of a complexity jump as expected by Algorithm~\ref{algo.penmingal}:
\begin{align}
\tag{\ensuremath{\pbpenmin^{\prime -}}}
\label{pb.penmin.Cpt-Rgrd}
&\hspace{-1cm}\forall C < \parens{ 1 - \etamoins } C^{\star}
\, ,
%\quad
\,
\Risk\parenB{\sh_{\mhgalzero(C)} } - \Risk\parens{\bayes} \geq \kappa \max_{\mM} \setb{\Risk\parens{\shm} - \Risk\parens{\bayes} }
\\
\tag{\ensuremath{\pbpenmin^{\prime +}}}
\label{pb.penmin.Cgrd-Rpt}
&\hspace{-1cm}\forall C > \parens{ 1 + \etaplus } C^{\star}
\, ,
%\quad
\,
\Risk\parenB{\sh_{\mhgalzero(C)}}  - \Risk\parens{\bayes} 
\leq K\parenj{\frac{C}{C^{\star}}} 
    \inf_{\mM} \setb{ \Risk\parens{\shm} - \Risk\parens{\bayes} } 
    + R_n\parenj{\frac{C}{C^{\star}}}
%\, ,
\end{align}
where for every $x>1$, $K(x) \in [1,\infty)$ and  
$R_n(x) \ll \inf_{\mM} \sets{\Risk\parens{\shm} - \Risk\parens{\bayes}}$ in general, 
and $\kappa>0$ is an absolute constant.

When (\pbpenmin), which is defined in Section~\ref{sec.theory.approach}, 
is replaced by (\pbpenminalt) above, 
the justification of Algorithm~\ref{algo.penmingal} is far from being complete, 
since there might be no complexity jump as required in the definition of $\Chwindow\,$.
Nevertheless, once (\pbpenminalt) is proved, one can reasonably conjecture that 
(\pbpenmin) holds true under similar assumptions, provided that $\C_m$ is well chosen.
Moreover, (\pbpenmin) and (\pbpenminalt) are closely related if
\begin{equation}
\label{eq.hyp-lien-pbpenmin-alt}
\forall x > 0 \, , \quad
\inf_{\mM \, / \, \C_m \geq x} \setb{\Risk\parens{\shm} - \Risk\parens{\bayes}}
\geq
g(x) > 0
\end{equation}
for some increasing function~$g$.
Indeed, assuming \eqref{eq.hyp-lien-pbpenmin-alt}, 
\eqref{pb.penmin.Cpt-Dgrd} implies \eqref{pb.penmin.Cpt-Rgrd} with
\[
\kappa= \frac{g(\C_{\mathrm{overfit}})}{\max_{\mM} \setj{\Risk\parens{\shm} - \Risk\parens{\bayes}}}
\, ,
\]
and
\eqref{pb.penmin.Cgrd-Rpt} implies \eqref{pb.penmin.Cgrd-Dpt} with
\[
\C_{\mathrm{small}} = g^{-1} \parenj{ K\parenj{\frac{C}{C^{\star}}} \inf_{\mM} \setb{\Risk\parens{\shm} - \Risk\parens{\bayes}} } + R_n\parenj{\frac{C}{C^{\star}}}
\, .
\]
Note that Eq.~\eqref{eq.hyp-lien-pbpenmin-alt} holds true with $g(x) = x / \alpha$ if $\C_m \approx \alpha p_1(m)$; 
for instance, for least-squares estimators and risk, 
$\alpha = \sigma^2 / n$ since $\C_m = D_m$ and $p_1(m) \approx \sigma^2 D_m / n$.
Let us remark finally that the proof of \eqref{pb.penmin.Cpt-Rgrd} usually relies on a proof of \eqref{pb.penmin.Cpt-Dgrd}, sometimes hidden by technical details.

\medbreak

\paragraph{Results}
To the best of our knowledge, a full proof of (\pbpenminalt) currently exists only in settings where (\pbpenmin) is proved to hold, except one result that we report in Section~\ref{sec.theory.rich}.
%%for instance \citet[Proposition~2]{Bir_Mas:2006} for the setting of Section~\ref{sec.slopeOLS}.
%
Some partial proofs of (\pbpenminalt) are reviewed in the next subsections.

\subsection{Partial proofs: uncertainty on the minimal penalty} \label{sec.theory.pbpenmingap}
A result weaker than (\pbpenmin) can be proved about the complexity jump: 
for some $C^{\star}_1 < C^{\star}_2$ (that remain distinct even when $n \to +\infty$), 
\begin{equation}
\tag{\ensuremath{\pbpenmingap}}
\label{pb.penmin.gap}
\forall C < C^{\star}_1 \, ,
\quad
\text{ \eqref{pb.penmin.Cpt-Dgrd} holds true,} 
\qquad 
\text{and}
\qquad
\forall C > C^{\star}_2 \, ,
\quad
\text{\eqref{pb.penmin.Cgrd-Dpt} holds true.}
\end{equation}
In other words, $C^{\star}_1 \pen_0$ is a too small penalty, 
while $C^{\star}_2 \pen_0$ is a sufficiently large penalty. 

From the theoretical point of view, proving (\pbpenmingap) instead of (\pbpenmin) is a serious limitation:
for reasons similar to the ones explained in Section~\ref{sec.theory.penopt.weak} for
\eqref{pb.penopt.weak}, it can happen that $\pen_0$ is not the shape of a minimal penalty.
For instance, in the setting of Section~\ref{sec.penmingal.fails},
(\pbpenmingap) holds true with $\pen_0(m) = \tr(A_m)$ although this quantity is not
always proportional to the minimal penalty, as shown by Figure~\ref{fig.linear.jump}.

Nevertheless, from the practical point of view, one can still derive from \eqref{pb.penmin.gap} a way to get from data some $\Ch \in [C^{\star}_1 , C^{\star}_2]$, for instance by taking a large $\eta$
in the definition of~$\Chwindow(\eta)$.
If $(C^{\star}_2 / C^{\star}_1)$ is not too large and if \eqref{pb.penopt.weak} holds true for some $C^{\star} \in [C^{\star}_1 , C^{\star}_2]$, this leads to an estimator satisfying an oracle inequality.

\medbreak

\paragraph{Results}
One full proof of \eqref{pb.penmin.gap} and \eqref{pb.penopt.weak} is available for prediction 
in a Gaussian graphical model via neighborhood selection, 
with conditional least-squares risk and estimators, 
a minimal-penalty shape $\pen_0(m) = \C_m  = D_m\,$, and 
an optimal penalty shape $\pen_1$ proportional to $\pen_0$ \citep{Ver:2010}. 
The proof of \eqref{pb.penmin.gap} assumes in addition that the graph is a square lattice.
Simulation experiments suggest that there is indeed a jump around $C^{\star}$ and
that Algorithm~\ref{algo.penmingal} works well.

\subsection{Partial proofs: for some specific \texorpdfstring{$\bayes$}{settings} only} \label{sec.theory.partial-specific}
The weakest partial proofs of (\pbpenmin) are the ones only valid for some particular $\bayes$, which often is $\bayes=0$.
Then, although $C^{\star} \pen_0$ is a minimal penalty for this particular $\bayes$, the general shape of the minimal penalty can differ from $\pen_0\,$.
For instance, in the Lasso case, an empirical study shows that the shape of $\E\croch{p_2(m)}$ depends on $\bayes$ and on some other features of the distribution of the data \citep{Con:2011:phd}.

Nevertheless, such weak results still are a good way to guess $\pen_0$ for a practical use of Algorithm~\ref{algo.penmingal}, and they can be a first step towards a full theoretical justification.

\medbreak

\paragraph{Results}
Such partial proofs exist in the case of multiplicative penalties, an apparently different setting that can still be cast into the framework of Algorithm~\ref{algo.penmingal}.
The principle, as exposed by \citet{Bar_Gir_Hue:2007} for least-squares regression, is to replace the penalized criterion \eqref{eq.crit} by the \emph{product} of the empirical risk by some penalty term, that is, choosing
\begin{equation}
\label{def.penmult}
 \mh \in \argmin_{\mM} \setj{ \Remp\parens{\shm} \parenj{ 1 + \frac{\penmult(m)}{n - D_m} } } \, .
 \end{equation}
This can actually be seen as an additive penalization method as in Eq.~\eqref{eq.crit},
with the penalty
\[ \pen(m) \egaldef \Remp\parens{\shm} \frac{\penmult(m)}{n - D_m} \, . \]
So, choosing a multiplicative factor in front of $\penmult(m)$ is equivalent to choosing a multiplicative factor in front of an additive penalty of a particular form.
For fixed-design regression with least-squares risk and estimators, \citet{Bar_Gir_Hue:2007}
prove that \eqref{pb.penmin.Cpt-Dgrd} holds true if $\bayes=0$, while \eqref{pb.penmin.Cgrd-Rpt}
and \eqref{pb.penopt.weak} hold true in general, with $C^{\star} \penmult_0(m) = D_m$ and
$\penmult_1 \propto  \penmult_0\,$.

In addition, in the setting of multivariate regression on a fixed design with the least-squares risk and low-rank least-squares estimators, \eqref{pb.penopt.weak} and \eqref{pb.penmin.Cgrd-Rpt} are proved in a general case, while \eqref{pb.penmin.Cpt-Dgrd} is proved only for $\bayes=0$ \citep{Gir:2010};
remark that \eqref{pb.penmin.Cgrd-Dpt} can certainly be proved in a general case, although its proof is not written in the article by \citet{Gir:2010}.
Note also that these results are valid both for additive penalties and for multiplicative penalties as the ones of \citet{Bar_Gir_Hue:2007}.

\subsection{Partial proofs: richer collections of models} \label{sec.theory.rich}
Throughout the article, we assume (at least implicitly) that $\M$ is not too large, that is, $\card(\M)$ grows at most polynomially with the sample size $n$, or $\M$ can be well approximated by such a polynomial set of estimators ---e.g., kernel ridge regression with one continuous parameter $\lambda$ \citep{Arl_Bac:2009:minikernel_long_v2}.
Nevertheless, the case where $\M$ is larger deserves attention, and we review in this subsection the partial results about minimal penalties in such settings.
Note that each of them suffers from some of the limitations emphasized in Sections~\ref{sec.theory.penopt.weak}--\ref{sec.theory.partial-specific}.

Let us consider the fixed-design regression setting, with least-squares risk and estimators on finite-dimensional vector spaces $S_m\,$.
Assuming as \citet{Bir_Mas:2006} that the penalty is a function of the dimension,
the selected estimator
\[
\sh_{\mh} \qquad \text{with} \qquad
\mh \in \argmin_{\mM} \setb{\Remp\parens{\shm} + \pen(D_m)}
\]
can be rewritten as
$\shp_{\Dh}$ where
\begin{gather*}
\forall D \in \N, \quad
\shp_D \in \argmin_{t \in S^{\prime}_D} \setb{ \Remp\parens{t} }
\, ,
\quad
S^{\prime}_D \egaldef \mathop{\bigcup_{\mM}}_{D_m = D} S_m
\, ,
\quad
\text{and} \quad
\Dh \in \argmin_{D \in \N} \setb{\Remp\parens{\shp_D} + \pen(D)}
\, .
\end{gather*}
Then, discarding all models of dimension $D > n$, $\sh_{\mh} = \shp_{\Dh}$ is a penalized empirical risk minimizer over a collection $(S^{\prime}_D)_{0 \leq D \leq n}$ of cardinality at most $n+1$.
The difference with the initial formulation is that the models $S^{\prime}_D$ are not vector spaces (in general), and the complexity of $S^{\prime}_D$ strongly depends on $f(D) \egaldef \card\sets{\mM \, / \, D_m = D}$.
Three cases can be distinguished, following \citet{Bir_Mas:2006}:
\begin{enumerate}
\item[(i)] $\M$ is ``small'' or ``polynomial'' when $f(D) \leq C D^{\omega}$ for some $C, \omega>0$.
Then, $\card(\M)$ grows polynomially with $n$ (since models of dimension $D_m>n$ can safely be discarded), and the complexity of $S^{\prime}_D$ is essentially the same as the one of a $D$-dimensional vector space.
\item[(ii)] $\M$ is ``large'' or ``exponential'' when $f(D)$ grows much faster ---typically of order $\binom{n}{D}$---, which implies in particular that $\card(\M)$ grows exponentially with~$n$. 
Then, $S^{\prime}_D$ is much more complex than a $D$-dimensional vector space.
A typical example is (full) variable selection among $p \geq n$ variables, 
for which $f(D) = \binom{p}{D} \geq \binom{n}{D}$. 
\item[(iii)] $\M$ is ``moderate'' in the intermediate situation, when $D^{-1} \log f(D)$ stays bounded away from 0 and $\infty$ for $n \gg D \gg 1$. 
\end{enumerate}
The current subsection focuses on cases (ii) and (iii); all other results mentioned in this article correspond to case (i).

\subsubsection*{Results for case (ii): large number of models}
In fixed-design regression with least-squares risks and estimators, 
(\pbpenminalt) and \eqref{pb.penopt.weak} are proved by \citet{Bir_Mas:2006} for the 
(full) variable-selection problem with $p$ orthonormal variables, assuming that the noise is Gaussian, 
and defining 
\begin{equation*} %\label{eq.penmin.complete-var-selec}
 \pen_0(m) = \frac{D_m}{n} \parenj{ 1 + 2 \log {\frac{p}{D_m}}}
\, ,
\qquad 
C^{\star} = \sigma^2 
   \, ,
\end{equation*}
and $\pen_1(m) = C \pen_0$ with any $C>1$.
Note that \eqref{pb.penmin.Cpt-Dgrd} can be derived from the proof by \citet[Proposition~2]{Bir_Mas:2006}, 
but it is not written in the article by \citet{Bir_Mas:2006}.

Similar theoretical results are proved by \citet[Chapter~8]{Sor:2017} 
for Gaussian variable selection with $p=n$ and a more general collection of models, 
that can be smaller than full variable selection but still exponentially large. 
Formally, \citet[Section~8.1]{Sor:2017} assumes that $\M$ satisfies a ``completion rule'',  
which holds for instance for the collection of regressograms over a partition 
whose cells are hyperrectangles of $\R^d$. 
Then, \eqref{pb.penmin.Cpt-Rgrd} and \eqref{pb.penopt.weak} hold true, 
with $\pen_0(m)$ proportional to $\frac{D_m}{n} \log \frac{\mathrm{e} n}{D_m} $ 
---showing that a $\log(n)$ factor is still necessary here--- 
and \eqref{pb.penmin.Cpt-Dgrd} is proved when the target signal is null. 
Compared to the results of \citet{Bir_Mas:2006}, a gap of a multiplicative constant remains between 
minimal and sufficient penalties. 

\citet[Chapter~9]{Sor:2017} proves similar theoretical results about histogram selection for 
density estimation by penalized log-likelihood, with the Kullback risk, 
for any large collection of subpartitions of a regular partition of $[0,1]$ into $N+1$ pieces, 
assuming $N \leq n / (\log n)^2$. 
A sufficient penalty $\propto \frac{D_m}{n} \log(N)$ satisfies \eqref{pb.penopt.weak}. 
If the target density is uniform over $[0,1]$, 
\eqref{pb.penmin.Cpt-Rgrd} holds true with a minimal penalty level of the same order of magnitude. 
As a consequence, when $\log(N) \sim \log(n)$, this proves that a $\log(n)$ factor must be 
added to the penalty compared to the case of a polynomial collection~$\M$. 

Two partial results are available with multiplicative penalties,
which are introduced in Section~\ref{sec.theory.partial-specific}.
In the same setting as \citet{Bir_Mas:2006}, \eqref{pb.penmin.Cpt-Dgrd} 
---assuming $\bayes=0$ and a specific ``exponential'' collection $\M$ with $f(D) \approx \binom{n}{D}$---, 
\eqref{pb.penmin.Cgrd-Rpt}, and \eqref{pb.penopt.weak} are proved by \citet{Bar_Gir_Hue:2007}, 
with $\penmult_0(m) = 2 D_m \log(n)$.
For estimation of a Gaussian graph ---that is, in a Gaussian graphical model, predict the value at each vertex of the graph given its neighbors, by linear regression---, 
with least-squares risk and estimators, \eqref{pb.penmin.Cpt-Dgrd} ---assuming that $\bayes=0$ and $\M$ contains some specific ``exponential'' collection with $f(D) \approx \binom{p}{D}$ for some $p>n$---, 
\eqref{pb.penmin.Cgrd-Rpt}, and \eqref{pb.penopt.weak} are proved by \citet{Gir:2008}, with $\penmult_0(m) = 2 D_m \log(p)$.
In both articles by \citet{Bar_Gir_Hue:2007} and \citet{Gir:2008}, \eqref{pb.penopt.weak} holds with
$\penmult_1 = C \penmult_0$ for any $C>1$.

Finally, several other arguments can be found for the necessity of a penalty larger than 
\[ 
\sigma^2 \frac{D_m}{n} \parenj{ 1 + \log { \frac{n}{D_m} } }
\] 
---up to a numerical constant--- for change-point detection, which is an instance of variable selection 
with $p=n-1$. 
Minimax lower bounds \citep[Theorem~2]{Dur_Leb_Toc:2009} 
and general oracle inequalities \citep{Bir_Mas:2006} 
prove that for the true model $m^{\star}$, 
\[ 
\pen(m^{\star}) \geq \kappa \sigma^2 \frac{D_{m^{\star}}}{n} \parenj{ 1 + \log { \frac{n}{D_{m^{\star}}} } }
\] 
is necessary for some constant $\kappa>0$. 
\citet[Section~1.9]{Abr_etal:2006} provide several other reasons why the optimal penalty 
should be close to $2 \sigma^2 \frac{D_m}{n} \log \frac{n}{D_m}$.

\subsubsection*{Results for case (iii): moderate number of models}
In fixed-design regression with least-squares risks and estimators, \eqref{pb.penopt.weak} holds true in general, and (\pbpenminaltgap) is proved assuming that $\bayes=0$ and 
all models of the same dimension $D$ are orthogonal \citep[Proposition~3]{Bir_Mas:2006}, with
\[
\pen_0(m)
= \lambda \frac{D_m}{n} \crochj{ 1 + 2 \sqrt{f(D_m)} + 2 f(D_m)}
\, ,
\qquad
f(D) = a + \frac{b \log(D+1)}{D}
\, ,
\quad \text{and} \quad 
C^{\star} = \sigma^2
\]
under some condition on the constants $\lambda,a,b$.
By (\pbpenminaltgap), we mean (\pbpenmingap) with a jump in the risk instead of the complexity; here, the gap in (\pbpenminaltgap) is $C^{\star}_2/C^{\star}_1 = 6/5$.

Results of the same flavor exist 
for a toy problem close to the above setting %%\citet[Proposition~3]{Bir_Mas:2006} 
\citep[Chapter~10]{Sor:2017}, 
and for a Gaussian linear process 
and a $b$-ary tree partition collection 
---with no assumption on $\bayes$ for proving \eqref{pb.penmin.Cpt-Dgrd} and \eqref{pb.penmin.Cpt-Rgrd}--- 
\citep[Chapter~6]{Sor:2017}.
All these results show that ``intermediate'' collections of models can require a penalty 
strictly larger than the minimal penalty $\frac{\sigma^2 D_m}{n}$ of ``polynomial'' collections. 

\section{Towards new theoretical results on minimal penalties} \label{sec.hints}
%
%%hints for solving each subproblem in other settings (Sections~\ref{sec.theory.pbmult}--\ref{sec.theory.hint.penopt})
%
We now describe some strategies for proving that Algorithm~\ref{algo.penmingal}
works in other settings.
This section is a bit more abstract and technical than the rest of the article,
so it can be skipped at first reading.
As in Section~\ref{sec.theory.approach},
whose notation is used throughout the section,
we consider separately subproblems (\pbmult), (\pbpenmin), and \eqref{pb.penopt}.

\subsection{Hints for (\pbmulttitre): how to find \texorpdfstring{$\pen_0$}{pen0}, \texorpdfstring{$\pen_1$}{pen1}, and \texorpdfstring{$\C_m$}{Cm}?} \label{sec.theory.pbmult}
Using the notation defined by Eq.~\eqref{def.p1}, \eqref{def.p2}, and~\eqref{def.delta}, Section~\ref{sec.penmingal.penoptmin} suggests
that $p_1(m) + p_2(m)$ or its expectation $\penoptgal(m)$ should be an optimal penalty,
and $p_2(m)$ or its expectation $\penmingal(m)$ should be a minimal penalty.

In both cases, computing (approximately) $\E\crochs{p_i(m)}$, $i=1,2$, or deriving an asymptotic expansion of $p_i(m)$, $i=1,2$, at least for $\C_m$ large enough, can lead to formulas for $\pen_0(m)$ and $\pen_1(m)$.
For instance, for fixed-design regression with the least-squares risk,
exact formulas for $\E\crochs{p_i(m)}$ lead to Algorithm~\ref{algo.OLS.jump} for least-squares estimators, and to Algorithm~\ref{algo.penmin.linear} for linear estimators.
The main difficulty here is to have no unknown quantity inside $\pen_0$ or $\pen_1\,$. 
\\
For general estimators in the fixed-design regression setting, an exact formula for $\penoptgal(m)$ is given by covariance penalties \citep{Efr:2004}, which can be expressed using the degrees of freedom when the noise is Gaussian and the loss is quadratic.
For maximum-likelihood estimators and risk, a partial asymptotic solution is given by the formula of the AIC criterion \citep{Aka:1973}, which derives from some version of the Wilks phenomenon (see also Section~\ref{sec.theory.hint.penmin}).
In both cases, only a formula for $\pen_1$ is available, and $\pen_0$ remains unknown, even if one can sometimes conjecture that $\pen_0=\pen_1 / 2$. 
\\
For random-design regression with the quadratic risk, \citet[Theorem~6.3]{Sau_Nav:2017} 
provide an (approximate) closed-form formula 
for $\E\crochs{p_1(m)}$ and $\E\crochs{p_2(m)}$ 
---by proving that $p_1(m)$ and $p_2(m)$ concentrate around some deterministic 
quantity, which is not necessarily equal to their expectation--- 
that is not directly useful because it depends on the unknown distribution of the $(X_i,Y_i)$. 

Another option is to define $\pen_0(m)$, resp. $\pen_1(m)$, as some resampling-based estimator of $\E\crochs{p_2(m)}$, resp. $\E\crochs{p_2(m)+p_1(m)}$, and to use Algorithm~\ref{algo.penmingal} for estimating the common (unknown) multiplicative factor $C^{\star}$ such that
\[
C^{\star} \E\crochb{ \pen_0(m) } \approx \E\crochb{p_2(m)}
\qquad \text{and} \qquad
C^{\star} \E\crochb{ \pen_1(m) } \approx \E\crochb{p_1(m) + p_2(m)}
\, ,
\]
see Remark~\ref{rk.reech} in Section~\ref{sec.theory.complete}.
In addition to the papers mentioned in Section~\ref{sec.theory.complete}, let us mention here
that a concentration result for the resampling estimate of $\E\crochs{p_2(m)}$ is proved by
\citet[Chapter~7]{Arl:2007:phd}, for empirical risk minimizers and a general bounded risk.

If no natural quantity arises as a complexity measure $\C_m\,$, such as the number of parameters in regression, $\E\croch{p_2(m)}$ (or a resampling-based estimator of it) can be a good guess for $\C_m\,$, see the article by \citet{Ler:2010:iid} and Remark~\ref{rk.reech}.

\subsection{Hints for (\pbpenmintitre): how to prove that \texorpdfstring{$C^{\star} \pen_0$}{Cstar pen0} is a minimal penalty?} \label{sec.theory.hint.penmin}
Following the results mentioned in the previous subsections, two general approaches can be used for proving (\pbpenmin), assuming either that
$C^{\star} \pen_0(m) = \E\croch{p_2(m)}$ as in Theorem~\ref{thm.OLS}, or that
$\pen_0(m) = \C_m = p_2(m)$ as done by \citet{Ler_Tak:2011}.
This section details these two approaches for proving \eqref{pb.penmin.Cpt-Dgrd} and \eqref{pb.penmin.Cgrd-Dpt}, before focusing on the concentration inequalities they both require.
Recall that 
%%for every $C \geq 0$, $\mhgalzero(C)$ is defined by
%Eq.~\eqref{eq.mhC.gal} in Section~\ref{sec.penmingal.algo}.
\[ \forall C \geq 0 , \qquad
\mhgalzero(C) \in \argmin_{\mM} \setj{ \Remp\parens{\shm} + C \pen_0(m) }
\, . \]

\subsubsection{Below the minimal penalty: \pbpenminmoinstitre} \label{sec.theory.hint.penmin.below}
\paragraph{Following the proof of Theorem~\ref{thm.OLS}}  
When $C^{\star} \pen_0(m) = \E\crochs{p_2(m)}$, similarly to the proof of Eq.~\eqref{eq.thm.OLS.Cpt-Dgrd} in Theorem~\ref{thm.OLS},
one can prove \eqref{pb.penmin.Cpt-Dgrd} by showing that for some well-chosen $m_1 \in \M$,
\begin{gather*}
\forall C < (1-\etamoins) C^{\star} \, , \qquad 
\inf_{\mM \, / \, \C_m < \C_{\mathrm{overfit}}} G_C(m) > G_C(m_1)
\\
\text{where} \quad
G_C(m) \egaldef \Remp\parens{\shm} + C \pen_0(m)
%= \Remp\parens{\shm} + \frac{C}{C^{\star}} \penmingal(m)
= \Risk\parens{\bayes_m} - \ovdelta(m) - p_2(m) + \frac{C}{C^{\star}} \E\crochb{p_2(m)}
\end{gather*}
is the quantity minimized by $\mhgalzero(C)$.
Then, in addition to the arguments sketched in Section~\ref{sec.penmingal.algo}, 
we only need here tight concentration inequalities for $\ovdelta(m)-\ovdelta(m_1)$ and for $p_2(m)$, 
see Section~\ref{sec.theory.hint.penmin.conc}. 
A natural choice for $m_1$ is a minimizer of the approximation error 
$\Risk\parens{\bayes_m} - \Risk(\bayes)$ over $\mM$. 

\medbreak

\paragraph{Generalizing the strategy of \citet{Ler_Tak:2011} and \citet{Gar_Ler:2011}} 

When $\pen_0(m) = \C_m \propto p_2(m)$, the approach\footnote{It should be noticed here that the article by \citet{Ler_Tak:2011} was prepublished on June 2011 (arXiv preprint number 1106.2467), a few months \emph{before} the preprint by \citet{Gar_Ler:2011}, despite what the publication date of the former paper suggests. } used by \citet{Ler_Tak:2011} and \citet{Gar_Ler:2011} can be summarized into the following proposition.
\begin{proposition}
\label{pro.pbpenmin.gal.below}
Let us consider the general framework of Section~\ref{sec.penmingal.framework} 
and use the notation of Section~\ref{sec.theory.approach}. 
Let $\varepsilon_{\delta} \in [0,1]$, $\varepsilon^{\prime}_{\delta} \geq 0$,
and assume that for every $\mM$, $\pen_0(m) = p_2(m)$ and
\begin{gather}
\label{eq.altproof.controldelta}
\forall m,m' \in \M, \qquad
\absj{\ovdelta(m) - \ovdelta(m^{\prime})}
\leq \varepsilon_{\delta} \crochb{ \Risk\parens{\bayes_m} - \Risk\parens{\bayes}  + \Risk\parens{\bayes_{m^{\prime}}} - \Risk\parens{\bayes} } +
\varepsilon^{\prime}_{\delta}
\, .
\end{gather}
Then, for every $C \in [0,1)$,
\begin{equation}
\label{eq.proofalt.penmin.Cpt-Dgrd}
p_2 \parenb{\mhgalzero(C)} \geq
\sup_{\mM} \setj{ p_2(m) - \frac{2}{1-C} \crochb{ \Risk\parens{\bayes_{m}}-\Risk\parens{\bayes} } }
- \frac{\varepsilon^{\prime}_{\delta}}{1-C}
\, ,
\end{equation}
and if $\Risk\parens{\bayes_{m_1}} = \Risk\parens{\bayes}$ for some $m_1 \in \M$ with $p_2(m_1)>0$,
for any $\alpha \in (0,1)$,
\begin{equation}
\label{eq.proofalt.penmin.Cpt-Dgrd.2}
\forall C \leq 1 - \eta_{\alpha} \, , \quad
p_2\parenb{\mhgalzero(C)} \geq
(1-\alpha) p_2(m_1)
\quad \text{with} \quad
\eta_{\alpha} = \frac{\varepsilon^{\prime}_{\delta}}{\alpha p_2(m_1)}
\, .
\end{equation}

Assume in addition that the data $\xi_1, \ldots, \xi_n \in \X$ are i.i.d.\@ and 
some contrast function $\gamma: \Xi \times \Set \to \R $ and constants $A,L>0$ exist such that
\begin{align}
\label{hyp.gal-controldelta.contrast}
&\forall t \in \Set \, , \quad \Remp\parens{t} = \frac{1}{n} \sum_{i=1}^n \gamma(\xi_i,t)
\quad \text{and} \quad
\Risk\parens{t} = \E\crochb{\Remp\parens{t}} = \E\crochb{ \gamma(\xi_1,t) }
\, ,
\\
\label{hyp.gal-controldelta.bounded}
&\forall t \in \Set \, , \quad \absb{\gamma(\xi_1 , t )} \leq A \quad \text{a.s.}
\, ,
\\
\label{hyp.gal-controldelta.margin}
\text{and}
\qquad
&\forall \mM \, , \quad
\var\parenb{\gamma\parens{\xi_1,\bayes_m} - \gamma\parens{\xi_1,\bayes} } 
\leq L \crochb{ \Risk\parens{\bayes_m} - \Risk\parens{\bayes}}
\, .
\end{align}
Then, for every $x \geq 0$, with probability at least $1 - 2 \card(\M) \mathrm{e}^{-x}$, 
for any $\theta>0$,
Eq.~\eqref{eq.altproof.controldelta} holds true with
\[ 
\varepsilon_{\delta} = \theta 
\quad \text{and} \quad 
\varepsilon^{\prime}_{\delta} = \paren{ \frac{L}{\theta} + \frac{4 A}{3}} \frac{x}{n} 
\, .  
\]
\end{proposition}
Proposition~\ref{pro.pbpenmin.gal.below} is proved in Appendix~\ref{app.proof.pro.pbpenmin.gal.below}.
%
%% Ce que prouve cette proposition
%%
Eq.~\eqref{eq.proofalt.penmin.Cpt-Dgrd.2} proves that \eqref{pb.penmin.Cpt-Dgrd} holds true if $p_2(m_1)$ is close to $\sup_{\mM} p_2(m)$, which is a reasonable assumption.
For instance, in the setting of Section~\ref{sec.slopeOLS}, $m_1$ is given by assumption \eqref{hyp.thm.OLS.Id}  leading to $\sh_{m_1}=Y$, hence
\[
p_2(m_1) = \frac{1}{n} \parenb{ \norms{Y - F_{m_1}}^2 - \norms{Y-\Fh_{m_1}}^2}
= \frac{1}{n} \norms{\varepsilon}^2 \approx \sigma^2
\, . \]

%% Une variante
%%
The proof of Proposition~\ref{pro.pbpenmin.gal.below} also works when assuming only that 
$p_2(m) \geq 0$ for every $\mM$ and 
\[
(1-\varepsilon_0) p_2(m) \leq \pen_0(m) \leq (1+\varepsilon_0) p_2(m)
\, ,
\]
which can be used when $\pen_0(m)$ is a resampling estimate of $\E\crochs{p_2(m)}$ for instance.
Then, we loose a factor in the ``rate'' $\eta_{\alpha}$ of estimation of $C^{\star}$ in Eq.~\eqref{eq.proofalt.penmin.Cpt-Dgrd.2}, see Appendix~\ref{app.proof.pro.pbpenmin.gal.below}.

%% Commentaire sur les hypotheses
%%
Proposition~\ref{pro.pbpenmin.gal.below} is new ---apart from the fact that its proof relies 
heavily on the proof technique proposed by \citet{Ler_Tak:2011}--- but rather abstract in its general form.
Under the additional conditions \eqref{hyp.gal-controldelta.contrast}--\eqref{hyp.gal-controldelta.margin}, 
it can be used for minimum-contrast estimators
\[ \shm \in \argmin_{t \in S_m} \setb{ \Remp\parens{t} } \]
with a bounded contrast $\gamma$, so that one automatically has $p_2(m) \geq 0$ and
Eq.~\eqref{hyp.gal-controldelta.contrast}--\eqref{hyp.gal-controldelta.bounded} as requested.
Then, Eq.~\eqref{hyp.gal-controldelta.margin} is a classical assumption \citep{Mas_Ned:2003} 
which holds for bounded regression with the least-squares contrast 
---with $L = A = 8 M^2$ if data are bounded by~$M$, according to \citet{Arl_Mas:2009:pente}---, 
and for binary classification with the 0--1 loss under 
the margin condition \citep{Mam_Tsy:1999,Mas_Ned:2003}.
Let us emphasize that Algorithm~\ref{algo.penmingal} has never been justified 
for binary classification with the 0--1 loss up to now, 
so Proposition~\ref{pro.pbpenmin.gal.below} is of significant interest 
even if it only provides a partial justification 
---with \eqref{pb.penmin.Cpt-Dgrd} only, in a rather abstract form.

\subsubsection{Above the minimal penalty: \pbpenminplustitre} \label{sec.theory.hint.penmin.above}
Before detailing two approaches for proving \eqref{pb.penmin.Cgrd-Dpt}, let us recall that 
if an oracle inequality like \eqref{pb.penmin.Cgrd-Rpt} is available, a simple way to prove \eqref{pb.penmin.Cgrd-Dpt} is to use the connection from \eqref{pb.penmin.Cgrd-Rpt} to \eqref{pb.penmin.Cgrd-Dpt} explained in Section~\ref{sec.theory.pbpenminalt}.

\medbreak

\paragraph{Following the proof of Theorem~\ref{thm.OLS}} 
When $C^{\star} \pen_0(m) = \E\croch{p_2(m)}$, following the proof of
Eq.~\eqref{eq.thm.OLS.Cgrd-Dpt} in Theorem~\ref{thm.OLS},
\eqref{pb.penmin.Cgrd-Dpt} can be proved by showing that for some well-chosen $m_2 \in \M$,
\[ 
\forall C > (1+\etaplus) C^{\star} \, , \qquad 
\inf_{\mM \, / \, \C_m > \C_{\mathrm{small}}} G_C(m) > G_C(m_2) 
\, , \]
which requires concentration inequalities for $\ovdelta(m)-\ovdelta(m_2)$ and for $p_2(m)$, see Section~\ref{sec.theory.hint.penmin.conc}. 
Two natural choices are 
\[
m_2 \in \argmin_{\mM \, / \, \C_m \leq \C_{\mathrm{small}}} \setb{ \Risk\parens{\bayes_m} - \Risk\parens{\bayes}} 
\qquad \text{and} \qquad 
m_2 \in \argmin_{\mM} \setB{ \E\crochb{ \Risk\parens{\shm} - \Risk\parens{\bayes}}}
\, .
\]

\medbreak

\paragraph{Generalizing the strategy of \citet{Ler_Tak:2011} and \citet{Gar_Ler:2011}} 
When $\pen_0(m) = \C_m \propto p_2(m)$, the approach used by \citet{Ler_Tak:2011} and \citet{Gar_Ler:2011} can be summarized into the following proposition.
\begin{proposition}
\label{pro.pbpenmin.gal.above}
%%Proposition~\ref{pro.pbpenmin.gal.above} in Section~\ref{sec.theory.hint.penmin.above}
%
Let us consider the general framework of Section~\ref{sec.penmingal.framework} 
and use the notation of Section~\ref{sec.theory.approach}. 
Let $\varepsilon_{\delta}, \varepsilon_p  \in [0,1) $, $\varepsilon_{\delta}^{\prime} \geq 0$,
assume that
Eq.~\eqref{eq.altproof.controldelta} holds true
and that for every $m \in \M$,
$\pen_0(m) = p_2(m)$ and
\begin{gather}
\label{eq.altproof.controlp1p2}
\absb{p_1(m) - p_2(m)} \leq \varepsilon_p p_1(m)
\, .
\end{gather}
Then, for every $C>1$, we have
\begin{align}
\label{eq.proofalt.penmin.Cgrd-Rpt}
\Risk\parenB{\sh_{\mhgalzero(C)}} - \Risk\parens{\bayes}
&\leq
\cteProPenminAbove(C)
\inf_{\mM} \setb{ \Risk\parens{\sh_m} - \Risk\parens{\bayes} }
+ \cteProPenminAbove^{\prime}(C)
\\
\text{and} \qquad
\label{eq.proofalt.penmin.Cgrd-Dpt}
p_2\parenB{\mhgalzero(C)}
&\leq
\cteProPenminAbove(C) (1+\varepsilon_p)
\inf_{\mM} \setb{ \Risk\parens{\sh_m} - \Risk\parens{\bayes} }
+ \cteProPenminAbove^{\prime}(C) (1+\varepsilon_p)
\, ,
\\
\text{where} \qquad
\cteProPenminAbove(C)
&\egaldef
\frac{ \max\setb{ (C-1) (1+\varepsilon_p) \, , \, 1 + \varepsilon_{\delta} }}{ \min\setb{  (C-1) (1-\varepsilon_p)  \, , \, 1 - \varepsilon_{\delta} }}
\notag
\\
\text{and} \qquad
\cteProPenminAbove^{\prime}(C)
&\egaldef
\frac{\varepsilon^{\prime}_{\delta}}{\min\setb{  (C-1) (1-\varepsilon_p)  \, , \, 1 - \varepsilon_{\delta} } }
\, .
\notag
\end{align}
\end{proposition}
Proposition~\ref{pro.pbpenmin.gal.above} is proved in Appendix~\ref{app.proof.pro.pbpenmin.gal.above}.

%% What proves Proposition~\ref{pro.pbpenmin.gal.above}
%%
Now, assume that Eq.~\eqref{eq.altproof.controldelta} and~\eqref{eq.altproof.controlp1p2} hold on a large-probability event.
Then, taking  $C = 1+\eta$ with $\eta>0$, we get
$\cteProPenminAbove(C) \leq \grandO(1) / \eta$
if $\eta$ is small enough,
and Eq.~\eqref{eq.proofalt.penmin.Cgrd-Rpt} implies \eqref{pb.penmin.Cgrd-Rpt}.
If in addition $p_2(m_1)$ stays bounded away from zero as $n$ tends to infinity, 
Eq.~\eqref{eq.proofalt.penmin.Cpt-Dgrd.2} implies \eqref{pb.penmin.Cpt-Dgrd}. 
If moreover the oracle risk tends to zero and  $\varepsilon^{\prime}_{\delta} = \petito(1)$,
then, Eq.~\eqref{eq.proofalt.penmin.Cgrd-Dpt} implies \eqref{pb.penmin.Cgrd-Dpt}.
Assuming also that $\max\sets{ \varepsilon_{\delta},\varepsilon_{p} } = \petito(1)$, 
then $\cteProPenminAbove(2) = 1+\petito(1)$ hence Eq.~\eqref{eq.proofalt.penmin.Cgrd-Rpt} with $C=2$ 
implies a first-order optimal oracle inequality \eqref{pb.penopt} with $\pen_1 = 2 \pen_0\,$.

%% Variantes de la Proposition~\ref{pro.pbpenmin.gal.above}
%%
The conditions of Proposition~\ref{pro.pbpenmin.gal.above} can be relaxed.
First, $\pen_0(m)=p_2(m)$ can be replaced by $(1-\varepsilon_0) p_2(m) \leq \pen_0(m) \leq (1+\varepsilon_0) p_2(m)$ for some $\varepsilon_0 \geq 0$ with $C(1-\varepsilon_0)>1$, which can be used when $\pen_0(m)$ is a resampling estimate of $\E\croch{p_2(m)}$ for instance.
Second, Eq.~\eqref{eq.altproof.controlp1p2} can be replaced by $\forall m \in \M$,
$-\varepsilon_p^{\prime} + \varepsilon_p^- p_1(m) \leq p_2(m) \leq  \varepsilon_p^+ p_1(m) + \varepsilon_p^{\prime}$ for some $\varepsilon_p^-,\varepsilon_p^+>0$ and $\varepsilon_p^{\prime} \geq 0$.
Then, $\cteProPenminAbove(C)$ and $\cteProPenminAbove^{\prime}(C)$ are slightly enlarged, as well as the bound in Eq.~\eqref{eq.proofalt.penmin.Cgrd-Dpt}, see Appendix~\ref{app.proof.pro.pbpenmin.gal.above}.
In particular, Proposition~\ref{pro.pbpenmin.gal.above} can justify \eqref{pb.penopt} with $\pen_1 = (1+\alpha^{-1}) \pen_0$ provided that
both $\varepsilon_p^-$ and $\varepsilon_p^+$ converge to $\alpha>0$, and $\varepsilon_p^{\prime}$ is small enough.

%% Commentaires sur les hypotheses
%%
Assumption~\eqref{eq.altproof.controlp1p2} is strong and we do not expect that it can be proved as generally as
assumption~\eqref{eq.altproof.controldelta} in Proposition~\ref{pro.pbpenmin.gal.below}.
Nevertheless, it holds when $p_1(m)$ and $p_2(m)$ both concentrate around $\E\croch{p_1(m)} \approx \E\croch{p_2(m)}$ for every $\mM$, and it can be satisfied in other cases.
For instance, for least-squares estimators in least-squares fixed-design regression (Section~\ref{sec.slopeOLS}) 
or in least-squares density estimation \citep{Ler:2010:iid}, $p_1(m) = p_2(m)$ almost surely.
Bounding $\abss{p_1(m)-p_2(m)}$ also turns out to be easier to get than a concentration inequality for $p_1$ and $p_2$ separately in some settings where $p_1(m) \neq p_2(m)$ in general \citep{Gar_Ler:2011}.

\medbreak

Whatever the proof technique used ---through \eqref{pb.penmin.Cgrd-Rpt}, 
as in the proof of Theorem~\ref{thm.OLS} or as in the article by \citet{Ler_Tak:2011}---, 
proving that the upper bound on the complexity in \eqref{pb.penmin.Cgrd-Dpt} is much smaller than the lower bound in \eqref{pb.penmin.Cpt-Dgrd}
is done by assuming that the oracle risk tends to zero as $n \to \infty$, 
or at least that some estimator of ``not too large'' complexity has a small approximation error 
(as in Theorem~\ref{thm.OLS}). 
We conjecture that such an assumption is unavoidable in general.

\subsubsection{Concentration inequalities} \label{sec.theory.hint.penmin.conc}
The proof techniques summarized in Sections~\ref{sec.theory.hint.penmin.below}--\ref{sec.theory.hint.penmin.above} require some concentration inequalities for $\ovdelta(m) - \ovdelta(m^{\prime})$ and $p_2(m)$,  
and some deviation inequalities for $\absj{p_1(m)-p_2(m)}/p_1(m)$. 
This section provides some ways to obtain such results. 

\paragraph{Concentration of $\ovdelta(m) - \ovdelta(m^{\prime})$} 
As explained in the proof of Proposition~\ref{pro.pbpenmin.gal.below}, $\ovdelta(m) - \ovdelta(m^{\prime})$ is a sum of independent and identically distributed random variables, 
so it can be concentrated with Bernstein's inequality \citep[Theorem~2.10]{Bou_Lug_Mas:2011:livre}, 
leading to a result like Eq.~\eqref{eq.altproof.controldelta} if a boundedness assumption 
\eqref{hyp.gal-controldelta.bounded} and some margin-type condition \eqref{hyp.gal-controldelta.margin} 
are satisfied. 
\\ 
In the unbounded case, other kinds of concentration inequalities 
can be used. 
For instance, for density estimation with the Kullback risk 
and maximum-likelihood estimators on histogram models 
---that is, the setting of \citet{Sau:2010:MLE} without the boundedness 
assumption on the target density---, 
\citet[Sections~4.2--4.3]{Sau_Nav:2018:arXivv4} use a modified version of 
Bernstein's inequality for controlling $\ovdelta(m) - \ovdelta(m^{\prime})$. 
Note that \citet{Sau_Nav:2018:arXivv4} only prove \eqref{pb.penmin.Cgrd-Dpt} and a first-order optimal oracle inequality \eqref{pb.penopt}, 
but the proofs of \citet{Sau_Nav:2018:arXivv4} can be adapted to get 
a full proof of the slope heuristics, 
that is, a result similar to the one of \citet{Sau:2010:MLE} when the density 
can be unbounded.

\paragraph{Concentration of $p_2(m)$} 
The problem is much harder for $p_2(m)$. 
It can be seen as proving a non-asymptotic version of the Wilks phenomenon \citep{Wil:1938} in a nonparametric setting with model misspecification \citep{Bou_Mas:2004},  
which makes this problem interesting beyond minimal-penalty algorithms. 
In addition to the settings mentioned in Section~\ref{sec.theory.complete}, concentration results for $p_2$ are available in two cases. 
\\
For bounded-contrast minimizers, a concentration inequality is proved in a general setting 
including bounded regression and classification with Vapnik-Chervonenkis classes \citep{Bou_Mas:2004}.
This result can be used for proving that Algorithm~\ref{algo.penmingal} works 
with regressogram estimators \citep{Arl_Mas:2009:pente}. \\
For maximum-likelihood estimators, in a parametric setting \citep{Spo:2012a}, 
in a semiparametric setting \citep{And_Spo:2013:journal}
and in a nonparametric setting with a quadratic penalty \citep{Spo:2012b:journal}, 
$p_2$ is close to some quadratic form with high probability, 
and this quadratic form itself satisfies some concentration properties.
Nevertheless, these results have not been used yet for proving that Algorithm~\ref{algo.penmingal} works.
\\
Note also that a concentration inequality for $p_2(m)$, 
with histogram (maximum-likelihood) density estimators and the Kullback risk, 
have been obtained by \citet{Sau_Nav:2018:arXivv4}, improving previous results by 
\citet{Sau:2010:MLE}. 

\paragraph{Proof of Eq.~\eqref{eq.altproof.controlp1p2}} 
Apart from the specific approaches mentioned in Section~\ref{sec.theory.hint.penmin.above}, 
we do not know any result for bounding directly $\abss{p_1(m)-p_2(m)}/p_1(m)$ 
as required in Eq.~\eqref{eq.altproof.controlp1p2}.

\subsection{Hints for (\pbpenopttitre): how to prove that \texorpdfstring{$C^{\star} \pen_1$}{Cstar pen1} is an optimal penalty?} \label{sec.theory.hint.penopt}
%Eq.~\eqref{eq.thm.OLS.oracle} in Theorem~\ref{thm.OLS}
%
When $C^{\star} \pen_1(m) = \E\croch{ p_1(m) + p_2(m)}$, oracle inequalities \eqref{pb.penopt} or \eqref{pb.penopt.weak} usually rely on some concentration inequality for the ideal penalty $\Risk\parens{\shm} - \Remp\parens{\shm} = p_1(m) + \ovdelta(m) + p_2(m)$, as in step~3 of the proof of Theorem~\ref{thm.OLS}.
One actually needs only concentration for
\[
\crochb{ \Risk\parens{\shm} - \Remp\parens{\shm} }
 - \crochb{ \Risk\parens{\shmp} - \Remp\parens{\shmp} }
\]
 for all $m, m^{\prime} \in \M$.
Concentration results for $\ovdelta(m) - \ovdelta(m^{\prime})$ and for $p_2(m)$ are reviewed in Section~\ref{sec.theory.hint.penmin} since they are usually required for proving (\pbpenmin). 
Therefore, we now focus on~$p_1(m)$. 

\paragraph{Concentration of the excess risk $p_1(m)$} 
What remains is to concentrate $p_1(m)$ 
---or equivalently $\Risk\parens{\shm} - \Risk\parens{\bayes}$ or $\Risk\parens{\shm}$--- 
around its expectation, a difficult problem that has not been solved except in a few settings: 
the ones for which a full proof of the slope heuristics exists ---see Section~\ref{sec.theory.complete}---, 
and the ones listed below. 

Several papers recently tackled the case of fixed-design linear regression with the least-squares risk, 
when $\shm$ minimizes a (penalized) least-squares criterion over a convex set, 
assuming that the penalty $\Omega$ is convex. 
Concentration inequalities for 
\[ 
\sqrt{\Risk\parens{\shm} - \Risk\parens{\bayes}} 
\qquad \text{or} \qquad 
\sqrt{\Risk\parens{\shm} - \Risk\parens{\bayes} + \Omega(\shm)} 
\] 
are available under different assumptions on the noise (Gaussian or not) and on $\Omega$ 
\citep{Cha:2014,Bel:2017,Bel_Tsy:2017,Mur_vdG:2018}. 
They apply to various examples such as 
the Lasso (in its constrained formulation) and isotonic regression \citep{Cha:2014}, 
the Lasso and the group Lasso in their usual regularization formulation \citep{Bel:2017,Bel_Tsy:2017}, 
splines and total-variation regularization \citep{Mur_vdG:2018}. 
When $\Omega$ is a semi-norm, \citet{Bel:2018:v4} proves 
upper and lower bounds on $\E[\Risk\parens{\shm}]$. 
The article by \citet{Che_Gun_Zha:2017:journal} also is related to this topic. 

For general losses, high-probability upper and lower bounds on 
$\Risk\parens{\shm} - \Risk\parens{\bayes}$ 
---sometimes plus a regularization term $\Omega(\shm)$--- are proved  
by \citet{Bar_Men:2006a} for general empirical minimizers ---with a rather abstract result---, 
by \citet{Sau:2010:phd} for ``regular'' estimators and losses, 
and by \citet{vdG_Wai:2017} for regularized empirical risk minimizers 
---with precise applications provided for ``linear losses'' such as linearized least-squares regression, 
maximum-likelihood estimators on an exponential model,  
and log-linear regression. 
Note that the general approaches of \citet{Bar_Men:2006a}, \citet{Sau:2010:phd} and \citet{vdG_Wai:2017} are closely related; 
\citet{Cha:2014}, \citet{Bel:2017}, \citet{Bel_Tsy:2017} and \citet{Mur_vdG:2018}, which are mentioned above for linear regression, 
use a similar technique that is exposed clearly by \citet[Section~2]{Bel:2017} for instance. 

\citet{Sau:2017} proves a concentration inequality for the quadratic risk $\Risk\parens{\shm}$ 
of a least-squares estimator over a convex set, in the heteroscedastic random-design regression setting; 
this result requires to handle specifically the quadratic part of the empirical process, 
which cannot be concentrated tightly with the general approach of \citet{vdG_Wai:2017} for instance. 
Note also that the result obtained by \citet{Sau:2017} applies to more general models than 
the ones of \citet{Sau_Nav:2017} for which a full proof of the slope heuristics exist. 

For histogram (maximum-likelihood) estimators and the Kullback risk in density estimation, 
in addition to the result obtained by \citet{Sau:2010:MLE} and mentioned in Section~\ref{sec.theory.complete}, 
concentration inequalities for $\Risk\parens{\shm}$ have been obtained 
by \citet{Cas:1999}, 
and \citet{Sau_Nav:2018:arXivv4} have recently improved them. 
 
\medbreak

Let us also recall that Proposition~\ref{pro.pbpenmin.gal.above} in Section~\ref{sec.theory.hint.penmin.above} provides an alternative approach for proving \eqref{pb.penopt} when $\pen_1(m) \propto p_2(m)$.

%%%%%%%%%%%%%%%%%%%%%%%%%%%%%%%%%%%%%%%%%%%%%%%%%%%%%%%%%%%%%%%%%%%%%%%%%%%%%%%%%%%%
%%%%%%%%%%%%%%%%%%%%%%%%%%%%%%%%%%%%%%%%%%%%%%%%%%%%%%%%%%%%%%%%%%%%%%%%%%%%%%%%%%%%

\section{Related procedures} \label{sec.related}
Minimal-penalty algorithms are primarily made for model/estimator selection, but in
the fixed-design regression setting (Algorithms~\ref{algo.OLS.jump} and~\ref{algo.penmin.linear}) they also provide an estimator $\Chjumpgal$ of the noise variance $\sigma^2$.
This section compares minimal penalties to its main alternatives for both tasks, 
starting by residual-(co)variance estimation. 

\subsection{Residual-variance estimation} \label{sec.related.variance}
Let us consider the fixed-design regression setting of Sections~\ref{sec.slopeOLS} and~\ref{sec.penmingal.fails}--\ref{sec.penmingal.linear} and their notation.
An example of interest is when
\begin{equation}
\label{eq.ex.Fi=f(xi)}
%%\tag{\ensuremath{\star}}
\forall i \in \set{1, \ldots, n} \, , \quad F_i = f(x_i)
\quad \text{for some smooth $f$ and some design points $x_i \in \R^d$.}
\end{equation}

%%% 1. Description generale de la litterature sur le sujet

\paragraph{Literature on residual-variance estimation}
Many estimators exist for the residual variance $\sigma^2$ in nonparametric regression.
An exhaustive list is beyond the scope of this article; for more references, we refer to the articles by \citet{Hal_Kay_Tit:1990}, \citet{Det_Mun_Wag:1998}, \citet{Spo:2002}, \citet{Mul_Sch_Wef:2003}, \citet{Lii_Ver_Cor_Len:2009} and \citet{Ram_Pau:2018:journal}, 
and to the articles by \citet{Bel_Che_Wan:2014}, \citet{Cha:2015b}, \citet{Rei_Tib_Fri:2013:journal} and \citet{Gia_etal:2017} for the high-dimensional variable-selection case. 
A related problem is noise-variance estimation in heteroscedastic regression 
(see \citealp{Bro_Lev:2007}, \citealp{Gen:2008}, and references therein).

This section focuses on minimal-penalty based estimators and on estimators that are quadratic forms of the data vector $Y \in \R^n$, that is,
\begin{equation}
\label{eq.varestim.gal}
\sigh^2_B \egaldef \frac{\prodscal{Y}{BY}}{\tr(B)}
\end{equation}
where $B$ is some $n \times n$ symmetric matrix.
Eq.~\eqref{eq.varestim.gal} actually covers several, if not all, classical residual-variance estimators 
---in particular the ones suggested in the context of model/estimator selection with $C_p$ or $C_L\,$---,  
which allows their common non-asymptotic analysis \citep{Det_Mun_Wag:1998}.

%%% 2. L'estimateur le plus classique: \sigh^2_{m_0}
%%%
\paragraph{Residual-based estimators}
%%%
%%% 2a. Definition/ref classiques
%%%
The most classical variance estimators are based upon the residuals ---through the empirical risk--- 
on some model $S_{m_0}\,$:
\begin{equation}
%\notag
\label{def.sighm0}
\sigh^2_{m_0} \egaldef \frac{1}{n-D_{m_0}} \normb{Y-\Fh_{m_0}}^2
\, .
\end{equation}
Remark that $\sigh^2_{m_0} = \sigh^2_B$ with $B = \Id_n - \Pi_{m_0}\,$.
When $\sigma^2$ must be estimated in the formula of the $C_p$ penalty, the classical suggestions are of the form $\sigh^2_{m_0}$ \citep{Mal:1973,Efr:1986,Bar:2000}.

%%% 2b. Analyse biais/variance/MSE
%%%
The bias of $\sigh^2_{m_0}$ as an estimator of $\sigma^2$ can be derived from Eq.~\eqref{eq.EriskempFhm}:
\begin{equation}
\label{eq.sighm0.biais}
\E\crochb{ \sigh^2_{m_0} } - \sigma^2=  \frac{1}{n-D_{m_0}} \normb{(\Id_n - \Pi_{m_0}) F}^2
\, .
\end{equation}
If $F \in S_{m_0}\,$, then $\sigh^2_{m_0}$ is unbiased, and otherwise it suffers some upward bias, depending on the approximation error and on the dimension of $S_{m_0}\,$.
Proposition~\ref{pro.var-sighm0} in Appendix~\ref{app.var-sighm0} 
provides a general formula for the variance and MSE of $\sigh^2_{m_0}\,$. 
For instance, assuming for simplicity that the noise is Gaussian, 
\begin{equation}
\label{eq.sighm0.MSE.gauss}
\E \crochj{ \parenb{ \sigh^2_{m_0} - \sigma^2 }^2 } 
= \frac{2 \sigma^4}{n-D_{m_0}} 
+ \frac{4 \sigma^2 \normb{(\Id_n - \Pi_{m_0}) F}^2}{\parens{n-D_{m_0}}^2}
+ \frac{\normb{(\Id_n - \Pi_{m_0}) F}^4}{\parens{n-D_{m_0}}^2} 
\, .
\end{equation}

%%% 2c. Choix de m0, discussion generale
%%%
Choosing the model $S_{m_0}$ without prior knowledge is a difficult question.
For minimizing the MSE, one must trade off terms depending on $1/(n-D_{m_0})$ 
and on the approximation error, that vary differently as functions of $S_{m_0}\,$; 
this can be as difficult as the model-selection problem.
When some unbiased model $S_{m_0}$ is known with $D_{m_0} = \petito(n)$, taking it for the estimation of $\sigma^2$ is a reasonable choice.
This matches the suggestion of \citet{Mal:1973} and \citet{Efr:1986}
in the variable-selection setting with a full model of dimension $p = \petito(n)$.
In the setting given by Eq.~\eqref{eq.ex.Fi=f(xi)}, when contiguous $x_i$ are close enough,
a natural choice for $S_{m_0}$ is the linear span of $(e_{2i} + e_{2i-1})_{1 \leq i \leq n/2}$ where $(e_1, \ldots, e_n)$ denotes the canonical basis of $\R^n$.
Then, assuming for simplicity that $n$ is even,
\begin{equation}
\label{eq.sighm0.ex-fsmooth-n/2}
\sigh_{m_0}^2 = \frac{1}{n} \sum_{i=1}^{n/2} (Y_{2i}-Y_{2i-1})^2
\end{equation}
is a consistent estimator of $\sigma^2$ if $f$ is uniformly continuous 
and $\max_{1 \leq i \leq n} \norms{x_i -x_{i+1}} = \petito(1)$.

Note that for high-dimensional variable selection, several residual-based estimators on a data-driven model $m_0$ 
---for instance, chosen by cross-validation--- are available, 
but theoretical guarantees are still lacking for several of them \citep{Rei_Tib_Fri:2013:journal,Gia_etal:2017}.

%%% 3. Recap sur l'estimateur 'penmin'
%
\paragraph{Variance estimation with minimal penalties} 
The problem of choosing $m_0$ for $\sigh^2_{m_0}$ can be solved (bypassed in fact) 
by using the slope heuristics. 
Let us state non-asymptotic risk bounds for $\Chthr$ and $\Chwindow$ as estimators of $\sigma^2$, 
that derive from Theorem~\ref{thm.OLS} and its proof in Section~\ref{sec.slopeOLS}. 
\begin{proposition} \label{pro.variance-estim} 
%% Proposition~\ref{pro.variance-estim} in Section~\ref{sec.related.variance}
In the framework described in Section~\ref{sec.slopeOLS.framework}, assume that $\M$ is finite, 
contains at least one model of dimension at most $c_n \in [0,n/3)$, and that 
\eqref{hyp.thm.OLS.Id} and \eqref{hyp.thm.OLS.Gauss} hold true ---see Section~\ref{sec.slopeOLS.math}. 
Let $\Chthr(T_n)$ be defined by Eq.~\eqref{def.Chthr} with $T_n \in (c_n , n)$, 
and $\Chwindow(\eta)$ be defined by Eq.~\eqref{def.Chwin} with $\eta>0$. 
For any $x \geq 0$ and $T \in (c_n , n)$, let us define 
\begin{align*} 
C_1(x ; T) 
&\egaldef 
\sigma^2 \parenj{ 1 - \frac{4 \sqrt{\frac{x}{n}} + 6 \frac{x}{n}}{1 - \frac{T}{n}} }
\\ \text{and} \qquad 
C_2(x ; T ; c_n) 
&\egaldef 
\sigma^2 \crochj{ 1 + \frac{4 n}{T - c_n} \parenj{ \sqrt{\frac{x}{n}} + \frac{2 x}{n} } } 
+ \frac{2 n}{T - c_n} \biaismax (c_n)
\\ 
\text{where} \qquad 
\biaismax (c_n) &\egaldef \inf_{\mM \,/\, D_m \leq c_n} \setj{\frac{1}{n} \normb{ (\Id_n - \Pi_m) F}^2 } 
\, . 
\end{align*}
Then, an event $\Omega_x$ of probability at least $1-4\card(\M) \mathrm{e}^{-x}$ exists on which 
\begin{gather}
\label{eq.bound-Chwindow}
\hspace{-0.3cm} 
\frac{C_1 \parenj{ x; \frac{2n}{3}} }{1+\eta} \leq \Chwindow(\eta) \leq C_2\parenj{ x ; \frac{n}{3}}  (1+\eta)
\qquad 
\text{if } \eta > \sqrt{\frac{C_2\parenj{ x ; \frac{n}{3}}}{C_1 \parenj{ x; \frac{2n}{3}}}} - 1 
\  \text{and} \ 
x \in \crochj{ 0, \frac{n}{180} }
\, , 
\\ 
\label{eq.bound-Chthr}
\text{and} \qquad 
C_1(x ; T_n) \leq \Chthr(T_n) \leq C_2(x ; T_n ; c_n)
\, .
\end{gather}
If we assume in addition that $c_n \leq T_n / 2$, then, 
\begin{equation}
\label{eq.MSE-Chthr}
\begin{split}
\E\crochB{ \parenb{\Chthr - \sigma^2}^2 } 
&\leq 739 \max\setj{ \paren{1 - \frac{T_n}{n}}^{-2} \,,\, \paren{\frac{T_n}{2n}}^{-2} } 
\\ 
&\times \crochj{ \carrej{ \biaismax  \parenj{\frac{T_n}{2}} } + \frac{\sigma^4 \log \parens{ 4\card \M }}{n} + \sigma^4 \carrej{ \frac{\log \parens{ 4\card \M }}{n} } }
\, .
\end{split}
\end{equation}
\end{proposition}
Proposition~\ref{pro.variance-estim} is proved in Appendix~\ref{app.MSE-Chthr}. 
Note that the constant $739$ in Eq.~\eqref{eq.MSE-Chthr} 
can be strongly reduced under mild additional assumptions, see Appendix~\ref{app.MSE-Chthr}. 
Proposition~\ref{pro.variance-estim} can also be extended to $(\phi^2 \sigma^2)$-sub-Gaussian noise, 
at the price of replacing $x$ by $L \phi^2 x$ in Eq.~\eqref{eq.bound-Chwindow}--\eqref{eq.bound-Chthr} 
and $\log \parens{ 4\card \M }$ by $L \phi^2 \log \parens{ 4\card \M }$ 
in Eq.~\eqref{eq.MSE-Chthr}, 
where $L$ is a numerical constant; 
see Remark~\ref{rk.thm.OLS.subgaussian}  in Section~\ref{sec.slopeOLS.math}. 

If $\biaismax (c_n)$ tends to $0$ as $n$ tends to $+\infty$ 
---which is a mild assumption---, 
by Proposition~\ref{pro.variance-estim} with $x = 2 \log(n) + \log(4 \card \M )$, 
we get that 
%%If $\biaismax (c_n) \to 0$ as $n \to +\infty$ 
%%---which is a mild assumption---, 
%%taking $x = 2 \log(n) + \log(4 \card(\M))$ in Eq.~\eqref{eq.bound-Chwindow}--\eqref{eq.bound-Chthr} 
%%shows that 
$\Chthr (T_n)$ and $\Chwindow (\eta)$ estimate consistently $\sigma^2$, 
with deviation bounds of order 
\[ 
\biaismax (c_n) + \sigma^2 \sqrt{\frac{\log(n) + \log\parens{\card \M}}{n}} 
\, , 
\]
provided that $c_n \leq \tau n$ with $\tau < 1/3$, 
\[ 
T_n = \rho n  \quad \text{with}  \quad 
\rho \in (0,1) \, , 
\qquad \text{and} \qquad 
\sigma^2 \eta \propto \biaismax  (c_n) + \sigma^2 \sqrt{\frac{\log(n) + \log\parens{\card \M }}{n}}
\, . 
\]
These deviation bounds for $\Chthr$ and $\Chwindow$ can be interpreted as an oracle inequality, 
since they coincide with the best possible risk of $\sigh^2_{m_0}$ with $D_{m_0} \leq c_n\,$, 
without any prior knowledge except the choice of $(S_m)_{\mM}\,$. 
Indeed, Eq.~\eqref{eq.sighm0.MSE.gauss} shows that when $D_{m_0} \leq c_n \leq \tau n$ with $\tau <1$, 
up to constants depending on $\tau$ only, 
\[ 
\sqrt{ 
\E\crochj{ \parens{ \sigh^2_{m_0} - \sigma^2 }^2 } } 
\gtrsim 
\frac{1}{n} \normb{(\Id_n - \Pi_{m_0}) F}^2 
+ \frac{\sigma^2}{\sqrt{n}} 
\geq \biaismax (c_n) + \frac{\sigma^2}{\sqrt{n}} 
\, . 
\]

The bound \eqref{eq.MSE-Chthr} on the mean squared error (MSE) of $\Chthr$ can be compared to the minimax optimal rate $\propto n^{-\min\sets{1,8/d}}$ for the MSE in the setting of Eq.~\eqref{eq.ex.Fi=f(xi)} 
when $f$ has a bounded second-order derivative \citep{Spo:2002}. 
Assuming that $\log\parens{\card \M} = \petito(n^{1/9})$, 
that the approximation error  $\biaismax (c_n)$ is minimax optimal hence of order $c_n^{-4/d}$, 
and that $c_n \leq \tau n$ with $\tau < 1$, 
we get that $\Chthr$ is optimal when $d \geq 9$ up to constants and within a factor $\log(\card \M)$ of the minimax risk when $d \leq 8$. 
Similar risk bounds can easily be obtained from Proposition~\ref{pro.variance-estim} 
under different assumptions on the signal.  
For instance, when Eq.~\eqref{eq.ex.Fi=f(xi)} holds true with $f$ that is $\alpha$-H\"olderian for some $\alpha>0$, 
an approximation error of order $D_m^{-2\alpha / d}$ can be obtained with local polynomials 
of maximal degree $\lfloor \alpha \rfloor$. 
We conjecture that these residual-variance estimation bounds are minimax-optimal up to logarithmic factors 
provided that $(S_m)_{\mM}$ has a cardinality at most polynomial in $n$ 
and achieves the minimax approximation error bounds. 
These consequences of Proposition~\ref{pro.variance-estim} have the flavor of adaptive risk 
bounds derived from oracle inequalities \citep{Bir_Mas:1997}, 
which is new for residual-variance estimation to the best of our knowledge. 
Therefore, the additional $\log(\card \M)$ factor 
---coming from the union bound over $\mM$ and typically of order $\log(n)$--- seems a mild price 
for the versatility of $\Chthr$ and $\Chwindow\,$.

%%% 4. Comparaison $\sigh^2_{m_0}$ vs pente: complement (formule et simus)
%%%
\paragraph{Residual-based estimators vs. the slope heuristics}  
In addition to the risk bounds comparison above, 
we can compare the definition of $\sigh^2_{m_0}$ 
with the one of $\Chslope$ in Algorithm~\ref{algo.OLS.slope}. 
On the one hand, $\sigh^2_{m_0}$ estimates the asymptotic slope of $-\norms{Y - \Fhm}^2$ as a function of $D_m$ from two points: 
$m_0$ and $m_1$ such that $S_{m_1}=\R^n$. 
On the other hand, Algorithm~\ref{algo.OLS.slope} makes a (robust) linear regression over, 
say, all $m$ such that $D_m \in [n/2,n]$. 
An illustration is provided by Figure~\ref{fig.OLS.slope-vs-residuals.easy-hard} in Appendix~\ref{app.morefig}. 
This confirms that minimal penalties ---here, the slope heuristics--- allow to avoid 
the choice of a single $m_0 \in \M$ by making use of the full collection $(S_m)_{\mM}$ for estimating the residual variance. 
Intuitively, this difference makes Algorithm~\ref{algo.OLS.slope} more stable and less dependent on some strong assumption about~$S_{m_0}\,$. 
The numerical experiments of Figure~\ref{fig.OLS.dist-Ch}b in Section~\ref{sec.practical.jump-vs-slope} 
and Figure~\ref{fig.OLS.slope-vs-residuals.easy-hard} in Appendix~\ref{app.morefig} 
indeed show that when $S_{m_0}$ happens to be a bad model, 
$\sigh_{m_0}^2$ can suffer from a large error, 
whereas $\Chthr$ is much more robust. 

\paragraph{Residuals of linear estimators}  
%%% 3e. Residus sur un modele lineaire
%%%
Residual-based estimators have also been proposed with several other linear estimators $\Fh_{m_0} = A_{m_0} Y$ of $F$, that is, defined as $\sigh^2_B$ in Eq.~\eqref{eq.varestim.gal} 
with a matrix $B = (\Id_n - A_{m_0})^{\top} (\Id_n - A_{m_0})$.
For instance, $A_{m_0}$ can correspond to some Nadaraya-Watson fit \citep[also known as kernel-based estimator;][]{Hal_Mar:1990} or to spline smoothing \citep{Car_Eag:1992}; the article by \citet{Det_Mun_Wag:1998} provides more references.
In the setting given by Eq.~\eqref{eq.ex.Fi=f(xi)}, the challenging case $d >1$ can be tackled with $k$-nearest neighbors \citep{Lii_Cor_Len:2010} or a local linear fit of~$f$ \citep{Spo:2002}.
All these estimators suffer from the same drawback as $\sigh^2_{m_0}\,$, that is, they rely on the choice of a \emph{single} matrix $A_{m_0}\,$, hence requiring to specify the regularization parameter $m_0\,$.
On the contrary, the minimal-penalty approach of Algorithm~\ref{algo.penmin.linear} avoids this choice in a principled way.

%%% 4. Difference-based estimators, pour le cas particulier: F_i = f(x_i) avec f reguliere
%%%
\paragraph{Difference-based estimators}
%%%
%%% 4a. Difference-based: Definition et biblio
%%%
Difference-based estimators are an important family of residual-variance estimators, which are designed for the setting of Eq.~\eqref{eq.ex.Fi=f(xi)} when $\norms{x_i-x_{i+1}}=\petito(1)$ and $f$ is smooth, often assuming~$d=1$.
%%% FAIT: clarifier la suite a ce sujet aussi (qui traite $d=1$ et qui traite $d$ quelconque)
The first example has been proposed by \citet{Ric:1984}, 
\[ 
\sigh^2_{\mathrm{Rice}} \egaldef \frac{1}{2 (n-1)} \sum_{i=1}^{n-1} (Y_{i+1}-Y_i)^2 
\, , 
\]
which is close to the residual-based estimator defined by Eq.~\eqref{eq.sighm0.ex-fsmooth-n/2}.
More generally, difference-sequence estimators of order $m \geq 1$ are defined by
\[ \sigh^2_{(d_0, \ldots, d_m)} = \frac{1}{n-m} \sum_{i=1}^{n-m} \carrej{ \sum_{j=0}^m d_j Y_{i+j} }
\quad \text{where} \quad
\sum_{j=0}^m d_j = 0 \quad \text{and} \quad \sum_{j=0}^m d_j^2 = 1
\, . \]
The only admissible sequence $(d_j)_{j=0, 1}$ for $m=1$ leads to $\sigh^2_{\mathrm{Rice}}\,$.
For a general order $m$, when $x_i \in \R$, the optimal sequence $(d_j)_{j=0, \ldots, m}$ in terms of MSE does not depend on $f$ asymptotically and can be computed explicitly \citep{Hal_Kay_Tit:1990}, although the picture can be quite different in a non-asymptotic setting \citep{Det_Mun_Wag:1998}.

%%% 4b. Difference-based: Defaut d'optimalite et manieres de le corriger
%%%
%%%
Assuming $x_i \in \R$, difference-based estimators of order $m \geq 1$ are suboptimal by a constant factor $1+1/(2m)$ in terms of MSE for normal data \citep{Hal_Kay_Tit:1990}, while for instance the residual-based estimator of \citet{Hal_Mar:1990} attains the optimal rate $\var(\varepsilon_1^2)/n$.
This issue is corrected for instance with covariate matching \citep{Mul_Sch_Wef:2003,Du_Sch:2009}, which in the case of order $m=1$ consists in replacing  $\sigh^2_{\mathrm{Rice}}$ by
\[ \frac{1}{2 n(n-1)} \sum_{i \neq j} W_{i,j} (Y_i - Y_j)^2 \]
with some well-chosen weights $W_{i,j} \geq 0$ \citep{Mul_Sch_Wef:2003}.
Another variant of difference-based estimators, which is asymptotically optimal, is studied by \citet{Ton_Ma_Wan:2013}.

%%% 4c. Difference-based: Autres defauts
%%%
Choosing $m$ and the sequence $(d_j)_{j=0, \ldots, m}$ is also a difficult problem with no prior knowledge \citep{Det_Mun_Wag:1998}.
But the main drawback of such estimators is that, when $f$ is not continuous, 
they can be severely biased in an unpredictable way.
An empirical method for detecting whether the bias is small enough is proposed by \citet{Buc_Eag:1989} and might be useful.

\subsection{Estimation of the residual covariance matrix}
\label{sec.related.covmat}
Assume now that several regression problems such as \eqref{eq.OLS.model} must be solved simultaneously, 
a framework known as ``multi-task regression'', ``multivariate regression'', ``multiple linear regression'', 
and ``seemingly unrelated regression''; 
the article by \citet{Sol_Arl_Bac:2011:multitask} and the PhD dissertation of \citet{Sol:2013:phd} provide references on this topic.  
One observes $Y^j = F^j + \varepsilon^j \in \R^n$ for $j=1,\ldots, p$, 
assuming that the noise vectors $\mathcal{E}_i \egaldef (\varepsilon^j_i)_{j=1, \ldots, p} \in \R^p$ are independent and identically distributed, with zero mean and covariance matrix $\Sigma \in \mathcal{M}_p(\R)$.
Then, a natural extension of the residual-variance estimation problem is the estimation of $\Sigma$ with as few prior knowledge on the $F^j$ as possible, which is often required for the multi-task problem of estimating $(F^j)_{j=1,\ldots,p}\,$.
For instance, \citet{Sol_Arl_Bac:2011:multitask} make use of the prior knowledge that the $F^j$ are close,
in combination with kernel ridge regression,
and selects regularization parameters with a penalty generalizing $C_L$ which depends on the full matrix $\Sigma$.

An estimator $\widehat{\Sigma}$ of $\Sigma$ based upon minimal penalties is proposed by \citet{Sol_Arl_Bac:2011:multitask}.
It satisfies $(1-\eta) \Sigma \preceq \widehat{\Sigma} \preceq (1+\eta) \Sigma$ with large probability, with $\eta \propto p \sqrt{\log(n)/n} c(\Sigma)^2$ where $c(\Sigma)$ is the condition number of $\Sigma$.
The construction of $\widehat{\Sigma}$ goes as follows:
\begin{itemize}
\item[(i)] For every $j \in \sets{1, \ldots, p}$, apply Algorithm~\ref{algo.penmin.linear} to the one-dimensional regression problem $Y^j = F^j + \varepsilon^j$, and store $\widehat{a}_j = \Chjumpgal$ which estimates $a_j \egaldef \Sigma_{j,j}\,$.
\item[(ii)] For every $i \neq j \in \sets{1, \ldots, p}$, apply Algorithm~\ref{algo.penmin.linear} to the one-dimensional regression problem $(Y^i + Y^j) = (F^i + F^j) + (\varepsilon^i + \varepsilon^j)$, and store $\widehat{a}_{i,j} = \Chjumpgal$ which is an estimator of $a_{i,j} \egaldef \Sigma_{i,i} + \Sigma_{j,j} + 2 \Sigma_{i,j}\,$.
\item[(iii)] Denote by $J$ the linear map on $\R^{p(p+1)/2}$ that sends $((a_j)_{1 \leq j \leq p}, (a_{i,j})_{1 \leq i \neq j \leq p})$ to $\Sigma$, 
and define $\widehat{\Sigma} = J((\widehat{a}_j)_{1 \leq j \leq p}, (\widehat{a}_{i,j})_{1 \leq i \neq j \leq p})$.
\end{itemize}
This construction can actually be generalized to any other one-dimensional residual-variance estimator $\sigh^2$ that satisfies $(1-\eta_0) \sigma^2 \leq \sigh^2 \leq (1+\eta_0) \sigma^2$ with large probability for all the above one-dimensional problems, leading to a similar result with $\eta \propto c(\Sigma) p \eta_0\,$.
The remarkable property of the minimal-penalty-based estimator $\widehat{\Sigma}$ 
of \citet{Sol_Arl_Bac:2011:multitask} is that it suffices to assume 
that an estimator of small complexity has a small approximation error 
---with slightly stronger constraints compared to Theorem~\ref{thm.OLS}, 
see the exact assumptions of \citet{Arl_Bac:2009:minikernel_long_v2}--- 
for each one-dimensional problem $Y^j$ to get this assumption automatically satisfied for all the  $Y^i + Y^j$, $i \neq j$.

\subsection{Model/estimator-selection procedures based on \texorpdfstring{$C_p$/$C_L$}{Cp/CL}}
\label{sec.related.Cp}
Let us go back to the model/estimator-selection problem in the fixed-design regression setting.
For selecting among linear estimators with the least-squares risk, a popular penalization approach is Mallows' $C_L$ \citep{Mal:1973}, as described in Section~\ref{sec.penmingal.fails}:
\begin{equation}
\label{def.CL}
\mh_{C_L} \in \argmin_{\mM} \set{ \frac{1}{n} \norm{ \Fhm - Y }^2 + \pen_{C_L} (\sigma^2, m)  }
\quad \text{with} \quad
\pen_{C_L} (\sigma^2, m) \egaldef \frac{2 \sigma^2 \tr(A_m)}{n}
\, .
\end{equation}
In the particular case of projection estimators, $\tr(A_m) = D_m$ the dimension of the corresponding model,
and $C_L$ reduces to $C_p$ which is described in Section~\ref{sec.slopeOLS.optimal}.
Both $C_p$ and $C_L$ penalties assume that the noise-level $\sigma^2$ is known, so in general it must be replaced by some data-driven estimation of it.
Minimal-penalty algorithms provide an estimator of $\sigma^2$ specially built for this estimator-selection task (Algorithms~\ref{algo.OLS.jump} and~\ref{algo.penmin.linear}), which can be plugged into Eq.~\eqref{def.CL} and for which theoretical guarantees can be proved, 
as shown by Theorem~\ref{thm.OLS} and \citet{Arl_Bac:2009:minikernel_long_v2}.
This subsection reviews some classical ways to estimate $\sigma^2$ inside Eq.~\eqref{def.CL}, as well as other estimator-selection procedures that are closely related, in the framework of Section~\ref{sec.penmingal.fails} (linear estimator selection) and with its notation.

\paragraph{Fixed variance estimator}
A first option is to replace $\sigma^2$ by some fixed estimator $\sigh^2$ of this quantity, for instance chosen among the estimators described in Section~\ref{sec.related.variance}.
The most classical choice is to take a residual-based estimator $\sigh^2_{m_0}$ as defined by Eq.~\eqref{def.sighm0}, for some $m_0 \in \M$ \citep{Mal:1973,Efr:1986,Bar:2000}.
For projection estimators, this option is often named ``$C_p$'' and called ``$C_p(\mathcal{L}_0,\mathcal{L})$'' by \citet[Table~4]{Efr:1986}.
As discussed in Section~\ref{sec.related.variance}, choosing $m_0$ can then be as difficult as the original estimator-selection problem. 

When using the penalty $\pen_{C_L} (\sigh^2, m)$, theoretical guarantees can be obtained if one can prove that
$\sigh^2$ is close to $\sigma^2$ with large probability, in combination with Theorem~\ref{thm.OLS} or its analogous for linear estimators.
For instance, \citet[Theorem~6.1]{Bar:2000} does it for projection estimators with $\sigh^2_{m_0}$ such that $D_{m_0}=n/2$; 
the approximation error of $m_0$ then appears as an additive term in the right-hand side of the oracle inequality.

Nevertheless, even if such guarantees imply some asymptotic-optimality result ---provided that
the approximation error of $m_0$ tends to zero---, they might not help for choosing the best possible estimator $\sigh^2$ in terms of estimator selection, since these only are \emph{upper bounds}.
Indeed, the best bound is obtained when $\sigh^2$ is well concentrated around $\sigma^2$, but it is known that overpenalizing a bit ---that is, taking $\sigh^2$ slightly larger than $\sigma^2$--- empirically improves the estimator-selection performance (see Section~\ref{sec.empirical.overpenalization}).
Residual-based estimators do overpenalize, because of the approximation error of $m_0\,$, but the overpenalization factor is unknown in practice and cannot be controlled without strong assumptions on the target $F$;
the consequences of a bad choice of $m_0$ with $\sigh^2_{m_0}$ are illustrated in the
numerical experiments of Section~\ref{sec.practical.jump-vs-slope}, see Figures~\ref{fig.OLS.dist-Ch}--\ref{fig.OLS.dist-risk-ratio}.

On the contrary, minimal-penalty algorithms are more than a simple ``plug in'' of an estimator of $\sigma^2$ 
---independent of the estimator-selection problem--- inside $\pen_{C_L}\,$. 
As detailed in Section~\ref{sec.empirical.overpenalization}, 
minimal-penalty algorithms seem to overpenalize slightly, by design, 
but a formal proof of this phenomenon remains an open problem.

\paragraph{Variance estimator depending on $m$}
Another approach is to plug into Eq.~\eqref{def.CL} a different variance estimator for each $\mM$, by considering the residuals on the model $m$ for which the penalty is computed.
In other words, assuming that $D_m<n$ for all $\mM$,
$\mh$ is chosen by penalization with the penalty $\pen_{C_L} (\sigh^2_m, m)$.
Let us consider projection estimators for simplicity.
Then $\mh$ minimizes over $\mM$ the criterion
\begin{align}
\critFPE(m) \egaldef
\frac{1}{n} \norm{ \Fhm - Y }^2 + \frac{2 \sigh^2_m D_m}{n}
&=
\frac{1}{n} \norm{ \Fhm - Y }^2  \paren{ 1 + \frac{2 D_m}{n - D_m} }
\label{eq.critFPE}
\end{align}
which has been proposed by \citet{Aka:1969a,Aka:1969} under the name FPE ---final prediction error--- and is called ``naive $C_p$'' or ``$C_p (\mathcal{L}_0,\mathcal{L}_0)$'' by \citet[Table~4]{Efr:1986}.
The FPE criterion \eqref{eq.critFPE} actually is a particular case of the multiplicative penalties defined by Eq.~\eqref{def.penmult} in Section~\ref{sec.theory.partial-specific}. 
More references and non-asymptotic oracle inequalities satisfied by such multiplicative penalties can be found in the article by \citet{Bar_Gir_Hue:2007}, which explains in particular how the FPE criterion \eqref{eq.critFPE} should be enlarged, depending on the size of the collection $\M$.
The main drawback of multiplicative penalties is probably that they need to deal carefully with the largest models.
For instance, for FPE, the results of \citet[Theorem~1]{Bar_Gir_Hue:2007} assume that for all $\mM$, $D_m \leq 0.39 (n+2) - 1$,  an assumption that can be weakened into $D_m \leq \rho n$ for some $\rho<1$ when considering a modified multiplicative penalty \citep[Corollary~1]{Bar_Gir_Hue:2007}.

\paragraph{Generalized cross-validation}
For choosing the regularization parameter of some smoothing methods, \citet{Wah:1977} proposed the criterion called ``generalized cross-validation'' (GCV) defined as a rotationally invariant form of the cross-validation estimate, that is,
\begin{align}
\critGCV(m) \egaldef
\frac{1}{n} \norm{ \Fhm - Y }^2 \paren{ \frac{1}{n} \tr(\Id_n - A_m) }^{-2}
\, .
%\notag %
\label{eq.critGCV}
\end{align}
GCV can also be seen as a reweighted cross-validation estimate,
which takes into account the asymmetry of the design \citep{Cra_Wah:1979}.
Nevertheless, 
as remarked by \citet[Remark~W]{Efr:1986}, 
``Despite its name, GCV is (nearly) a member of the $C_p$ family of estimates''.
Indeed, considering projection estimators only,
\begin{align}
\critGCV(m) =
\frac{1}{n} \norm{ \Fhm - Y }^2 \paren{ \frac{n}{n-D_m} }^2
\approx \frac{1}{n} \norm{ \Fhm - Y }^2 \frac{n+D_m}{n-D_m} = \critFPE(m)
\label{eq.critGCV.approx}
\end{align}
where the approximation holds true when $D_m \ll n$.
Theoretical guarantees for GCV are available in various settings \citep{KCLi:1985,KCLi:1987,Cao_Gol:2006}, with the same limitation as the ones of FPE and other multiplicative penalties.
For instance, \citet[Theorem~2]{Cao_Gol:2006} consider a truncated version of GCV where all $\mM$ such that $\tr(A_m) > \sqrt{n}$ are discarded, and some examples exist where GCV is not asymptotically optimal \citep{KCLi:1986}.
Let us finally mention that an empirical comparison of GCV and minimal penalties (Algorithm~\ref{algo.penmin.linear}) is done by \citet{Arl_Bac:2009:minikernel_nips,Arl_Bac:2009:minikernel_long_v2} for several kinds of linear estimators, showing that either Algorithm~\ref{algo.penmin.linear} clearly outperforms GCV or the two methods perform similarly, depending on the setting.

\subsection{L-curve, corner, and elbow heuristics} \label{sec.related.elbow}
Minimal-penalty algorithms, in particular Algorithms~\ref{algo.OLS.jump}--\ref{algo.gal.slope.naif}, 
can be related to some ``L-curve'', ``corner'' or ``elbow'' heuristics, which are often used
for choosing hyperparameters in the statistics and machine-learning communities.

%%% 1. Principe de l'elbow heuristics
%
%%% L-curve
The L-curve is defined as a plot of the size of the residuals versus the size 
or the estimator complexity. % (or conversely, as in the article by \citet{Han_OLe:1993}). 
Using the notation of Section~\ref{sec.penmingal}, when the goal is to select
an estimator among $(\shm)_{\mM}\,$, the L-curve can be defined as $(\C_m , \Remp\parens{\shm})_{\mM}\,$.
For instance, the right part of Figure~\ref{fig.OLS.algo} in Section~\ref{sec.slopeOLS.algo} shows an L-curve (the black dots); 
Figure~\ref{fig.ridge.Lcurve} in Appendix~\ref{app.morefig} provides another illustration.
%
%%% Usage precis de la L-curve
%
The practical use of the L-curve has been suggested by \citet{Mil:1970} and \citet[Chapters~25--26]{Law_Han:1974}.
Some precise heuristic choice of a regularization parameter ---often called ``the L-curve method''--- has first been proposed by \citet{Han:1992} and \citet{Han_OLe:1993} for some inverse problem with Tikhonov regularization.
The main idea is that the L-curve has three main parts:
\begin{itemize}
\item[(i)] a straight part where the residuals $\Remp\parens{\shm}$ decrease fastly while $\C_m$ is almost constant, where the regularization is too strong,
\item[(ii)] a flat part where $\C_m$ increases much while the residuals $\Remp\parens{\shm}$ decrease slowly, where some overfitting occurs,
and
\item[(iii)] in between, a ``\emph{corner}'' or ``\emph{elbow}'', where the regularization parameter is of the correct order.
\end{itemize}
Therefore, the L-curve is L-shaped ---hence its name, given by \citet{Han_OLe:1993}--- and the L-curve method suggests to choose $m$ corresponding to the corner of the~``L''.

%%% Definition exacte (courbure, etc.) et moyens de la calculer
%
Several definitions of the corner can be proposed, as well as several measures of ``size'' and ``complexity'' can be considered when plotting the L-curve \citep{Han_OLe:1993}.
The most common choice is to define the corner as the location where the L-curve has a maximal curvature \citep{Han_OLe:1993} ---hence the name ``maximum-curvature criterion'' \citep{Gro_Wol:2009} often used for this heuristics--- and to look at the L-curve in log-log scale \citep{Han_OLe:1993}, although several variants exist \citep{Reg:1996}. 
Additional practical problems need also to be solved, especially when $\M$ is discrete (how to define the curvature of a finite set of points?) and when some computational issues arise, for instance because computing every single point of the L-curve is expensive 
\citep{Han_OLe:1993,Cas_Gom_Gue:2002,Han_Jen_Rod:2007,Hen_Lu_Mha_Per:2010}.
Another option, proposed for a change-point detection problem \citep{Lun_Lev_Cap:2015}, 
relies on performing two linear regressions on the L-curve in order to identify its parts (i) and (ii); 
it can therefore be directly related with the ``slope'' formulation of 
minimal-penalty algorithms.

%%% Resultats theoriques (ou empiriques) positifs ou negatifs 
%%% Inverse problems
%
Empirical or theoretical studies of L-curve algorithms are available mostly for inverse problems 
with Tikhonov regularization \citep{Han:1992}, truncated SVD \citep{Han_OLe:1993,Rei_Rod:2013}, 
or conditional-gradient regularization \citep{Cas_Gom_Gue:2002}, showing reasonably good empirical performance. 
Several of the papers mentioned in this subsection show that L-curve algorithms compare favorably to generalized cross-validation (GCV) on simulated examples; for instance, \citet{Han:1992} show some similarity between GCV and L-curve algorithms, and report a tendency of GCV to overregularize.
Nevertheless, the L-curve method is proved to be not consistent in several families of realistic examples \citep{Eng_Gre:1994,Vog:1996,Han:1996}, when the noise tends to zero or when the sample size tends to infinity.
According to \citet{Han:1996}, the reason for this inconsistency is that the corner seems to correspond to the minimal level of regularization ---the minimal penalty, with the words of the present survey--- more than to the optimal level; hence choosing $m$ at the corner leads to some overfitting.
%

%%% Choice of number of clusters in clustering?
%
The L-curve can also be used similarly in unsupervised learning for choosing the number of clusters, where it is defined as a plot of the within-cluster dispersion as a function of the number of clusters.
Indeed, as written by \citet{Tib_Wal_Has:2001}, ``Statistical folklore has it that the location of such an `elbow' indicates the appropriate number of clusters''.
Various methods actually use a similar idea \citep{Tib_Wal_Has:2001,Sug_Jam:2003,Mat_Mie:2017}, although they are not straightforward applications of the method of \citet{Han_OLe:1993}.
Remark that procedures choosing the number of clusters using the slope heuristics 
show good experimental results, 
according to \citet[Section~4.4]{Bau:2009:phd}, 
\citet{Mau_Mic:2010}, \citet{Bon_Tou:2010} and \citet{Bau:2012:ejs}.

\medbreak

\paragraph{Comparison with minimal-penalty algorithms} 
%%% 3. Parallele elbow/pente
%
Let us start with their common points. 
%
%% 3.a Points communs
%
Both corner/elbow heuristics and minimal-penalty algorithms are based on the L-curve: 
directly in Algorithm~\ref{algo.OLS.slope}, indirectly in Algorithm~\ref{algo.OLS.jump} since $(D_{\mh(C)})_{C \geq 0}$ can be seen as a reparametrization of the convex hull of the L-curve.
Both rely on the idea of detecting a sharp variation of an observable quantity 
(the curvature / the selected dimension in $D_{\mh(C)}$) in some region of interest 
(the optimal value of regularization parameters / the minimal value of the constant in front of the penalty).
For both methods, a visual check (of the presence of an elbow / a jump) is possible, and strongly encouraged \citep{Han_OLe:1993,Bau_Mau_Mic:2010}.
The strength of the connection between elbow heuristics and minimal penalties is emphasized in the following three works.
For choosing the constant in front of the penalty for change-point detection, 
\citet[Remark~2]{Lav:2005} suggests an algorithm close to (but slightly different from) 
the slope-heuristics algorithm of \citet{Leb:2005}, which can be formulated as a maximal-curvature criterion on the L-curve.
For Hawkes-process estimation via model selection, when the model collection is large, 
\citet{Rey_Sch:2010} remark that the true model generally corresponds to a sharp angle of the L-curve, 
hence propose an algorithm between the slope and elbow heuristics, 
which consists in choosing $\mh(\Ch)$ with $\pen(m) = \C_m$ 
and $\Ch$ equal to the opposite of the slope of the segment joining the first and the last point of the L-curve. 
For choosing the bandwidth of a Gaussian kernel for quantile estimation with one-class support vector machines, 
\citet[Section~6.2.2]{Vert:2006:phd} points out an ``elbow effect'' and locates the elbow region 
with a ``maximal jump'' procedure similar to Algorithm~\ref{algo.OLS.jump}.

%% 3.b Differences
%
Nevertheless, several important differences must be pointed out between the two approaches.
%
%%% 1ere difference
%
First, the elbow heuristics tries to localize directly the optimal $m$ whereas the slope heuristics localizes it in two steps: 
first, it estimates the minimal penalty, then, it uses an estimated optimal penalty for selecting some~$m$.
%
%%% 2eme difference
%
Second, the assumptions made on the shape of the L-curve are different.
The L-curve must be \emph{exactly} L-shaped for elbow heuristics, 
since otherwise the curvature can be large far from the ``true'' elbow.
On the contrary, the slope heuristics assumes a linear behavior of the empirical risk as a function of $D_m$ (or $\C_m$) 
only for large models (Algorithms~\ref{algo.OLS.slope} and~\ref{algo.penmingal.slope}), and makes an even milder assumption with its jump formulation 
(Algorithms~\ref{algo.OLS.jump} and~\ref{algo.penmingal}; see Theorem~\ref{thm.OLS}).
%
%%% 3eme difference
%
Third, theoretical grounds are much stronger for the slope heuristics (with strong optimality results like Theorem~\ref{thm.OLS} in several settings) than for the elbow heuristics which is even proved inconsistent in some realistic cases  \citep{Eng_Gre:1994,Vog:1996,Han:1996}.

Overall, we consider the slope heuristics and its generalization (minimal penalties) as 
a simple and principled way to localize an elbow on the L-curve (when there is one), 
and to make use of it for optimal model/estimator selection.
In particular, a natural answer to the problem of choosing the scale at which the L-curve should be considered on the $x$-axis is given by Section~\ref{sec.penmingal.penoptmin}: it should be (the shape of) the minimal penalty.

\subsection{Scree test and related methods} \label{sec.related.scree}

For choosing the number of factors in factor analysis,
or the number of components in principal components analysis,
some classical methods can be related to minimal-penalty and L-curve algorithms.

\paragraph{Scree test} 
The most popular one ---named the scree test--- has been proposed by Cattell for factor analysis \citep{Cat:1966,Cat_Vog:1977}.
It is based upon the \guil{scree plot}, that is, a plot of the eigenvalues
versus their rank (in decreasing order), which can be seen as an L-curve for factor analysis.
The key remark made by \citet{Cat:1966} is that the scree plot ends with a linear part ---a scree---,
and that the beginning of the linear part corresponds to the ``correct'' number of factors.
Overall, the scree test chooses a number of factors equal to the rank of the
starting point of the linear part at the end of the scree plot.

This idea is close to the ``slope'' formulation of minimal-penalty algorithms
(Algorithms~\ref{algo.OLS.slope} and~\ref{algo.penmingal.slope}).
By analogy, we can say that the starting point of the linear part in the scree plot
is a ``minimal regularization level'' ---an upper bound on the number of factors that
should be kept at the end.
This fits well the goal of the initial paper by \citet{Cat:1966},
which is not to find the exact true number of factors ---a quantity which might be impossible to define formally---,
but only to keep a number of factors which explain 95\% to 99\% of the ``substantive variance''.
Nevertheless, the scree test seems to be often used for estimating the ``true number of factors''
itself \citep{Jac:1993}.

Similarly to minimal-penalty and L-curve algorithms,
making use of the scree test requires to overcome practical issues:
Cattell remarks that \guil{even a test as simple as this requires the acquisition of \emph{some} art in administering it} \citep{Cat:1966}.
For instance, it seems important to normalize the data \citep{Cat:1966}, %%% normalize before computing eigenvalues
and sometimes the scree plot ends with two or three linear parts ---then, one should
cut at the beginning of the first linear part \citep{Cat:1966}.

\medbreak

\paragraph{Variants} 
Several variants of the scree test exist.
The number of factors can for instance be given by the intersection of the scree plot
with some reference curve, which corresponds to some ``average scree plot'' obtained with
``pure noise'', as proposed by \citet{Hor:1965} and \citet{Hor_Eng:1979}, and by Frontier's broken stick method \citep{Fro:1976}.

Another variant exists for estimating the intrinsic dimension of some data set,
inside a classification procedure \citep{Bou_Fau_Gir:2014}.
Given a decreasing sequence of eigenvalues $(\lambda_{(j)})_{1 \leq j \leq n}$,
the estimated intrisic dimension is the smallest $j$ such that
$\lambda_{(j)}-\lambda_{(j+1)} \leq T$  for some threshold $T$,
which can be chosen by cross-validation in the article by \citet{Bou_Fau_Gir:2014}.
The underlying assumption is that the scree plot is L-shaped,
and that the point where the discrete derivative goes below $T$
corresponds to the ``elbow'', or to the beginning of the linear part.

\medbreak

\paragraph{Results} 
All these methods are only validated by numerical experiments \citep{Cat_Vog:1977}.
For instance, for principal component analysis, \citet{Jac:1993}
concludes that the broken stick method is one of the two best methods
for choosing the number of components.
The scree test tends to overestimate by one the number of components according to
\citet{Jac:1993},
which is consistent with our remark above that it corresponds to a
``minimal regularization level'' ---not an optimal one.

Note however that some theoretical results are proved for the closely related problem
of low-rank matrix recovery from noisy data by hard-thresholding of singular values.
In an asymptotic framework,
when the goal is to minimize the asymptotic mean squared error in some specific asymptotic regime,
\citet{Gav_Don:2014} show that the optimal hard threshold can be written
$\lambda_{\star}(m/n) \sqrt{n} \sigma$ when the matrix to recover is of size $m \times n$.
Interestingly, the \guil{minimal threshold}, which corresponds to the largest singular value obtained from pure noise,
is asymptotically equal to $\parens{  1 +  \sqrt{m/n} } \sqrt{n} \sigma$.
In the framework of \citet{Gav_Don:2014}, the scree test would correspond to
using the \guil{minimal hard threshold},
and it seems indeed reasonable to use it for estimating the rank (the number of factors).
On the contrary, when the goal is to minimize some quadratic error, the optimal threshold
is a bit larger:
\citet[Figure~4]{Gav_Don:2014} show that
\[
\forall m \leq n, \qquad
\lambda_{\star} \parenj{ \frac{m}{n} } > 1 +  \sqrt{\frac{m}{n}}
\, .
\]

\subsection{Thresholding under the null} \label{sec.related.thresh-null}
A related approach, for choosing the threshold $\lambda$ of thresholding estimators, 
starts by considering the minimal value $\widehat{\lambda}_{\min}$ of the threshold 
such that the estimator is equal to zero. 
Under the null hypothesis ---that is, when the true signal is zero---, 
$\widehat{\lambda}_{\min} = \widehat{\lambda}_{\min}^{\mathrm{null}}$ corresponds to the \emph{minimal} thresholding level, 
and any good threshold must be larger than $\widehat{\lambda}_{\min}^{\mathrm{null}}$. 
For instance, in the setting of the previous paragraph ---that is, singular-values hard thresholding---, 
$\widehat{\lambda}_{\min}^{\mathrm{null}}$ is of order $(1 +  \sqrt{m/n}) \sqrt{n} \sigma$,  
and \citet{Gav_Don:2014} provide an explicit formula for the optimal threshold, 
which can be written $c(m/n) \widehat{\lambda}_{\min}^{\mathrm{null}}$ for some $c(m/n)>1$. 

In a general setting, \citet{Gia_etal:2017} define the quantile universal threshold (QUT) 
$\lambda^{\mathrm{QUT}}$ as the $(1-\alpha)$-quantile of $\widehat{\lambda}_{\min}^{\mathrm{null}}$ 
for some $\alpha \in (0,1)$. 
It turns out that QUT corresponds to the universal threshold %%\sqrt{2 \log(p)} \sigma 
proposed for wavelet thresholding \citep{Don_etal:1995}, 
and it can be used much more generally, beyond hard or soft thresholding. 
For instance, for choosing the regularization parameter $\lambda$ of the Lasso and related procedures 
in high-dimensional regression, 
good performance can be obtained with $\lambda = c \lambda^{\mathrm{QUT}}$ for some $c>1$ 
\citep[Section~4.2]{Gia_etal:2017}. 

\medbreak

QUT and Algorithm~\ref{algo.OLS.jump} have common points: 
they both start by identifying a minimal value for the parameter of interest ($\lambda$ or $C$), 
then multiply it by a constant factor to get an optimal value of the parameter. 
Their main difference lies in the definition of the minimal parameter value: 
it is obtained from data under the null-hypothesis for QUT,  
hence requiring to know ---or at least to approximate--- the null-hypothesis distribution, 
while Algorithm~\ref{algo.OLS.jump} defines it directly from the data, whatever their distribution. 

Note that Section~\ref{sec.practical.variants-chpt} details a procedure by \citet{Roz:2012}, 
that is a variant of minimal-penalty algorithms for change-point detection, 
and can also be seen as a null-hypothesis based calibration procedure, hence similar to QUT.

\subsection{Other model/estimator-selection procedures}
Many other model/estimator-selection procedures exist and are studied. 
A detailed account on these is far beyond the scope of this survey.
This subsection only mentions a few of them, that are of interest in relation with minimal-penalty algorithms.

\paragraph{Unknown variance}
First, in addition to the procedures based upon $C_p$ and $C_L$ that are listed in Section~\ref{sec.related.Cp}, some procedures are specially built for dealing with the problem of not knowing the noise variance in regression, which is also what minimal-penalty algorithms do in the regression case.
For instance, \citet{Bar:2011} provides an abstract general-purpose estimator-selection procedure, 
which leads to the model-selection procedure of \citet{Bar_Gir_Hue:2010} 
for Gaussian model selection with unknown variance. 
We refer to the article by \citet{Gir_Hue_Ver:2011} for a detailed survey on high-dimensional variable-selection methods when the variance is unknown.

\paragraph{Cross-validation and resampling}
An important family of general-purpose estimator-selection procedures is 
cross-validation \citep{Arl_Cel:2010:surveyCV}, and more generally all resampling-based selection procedures 
---e.g., resampling-based penalties \citep[see][and references therein]{Arl:2009:RP}.
Comparing them to minimal-penalty algorithms is interesting at least in two distinct situations.

First, when Algorithm~\ref{algo.penmingal} works with $\pen_0$ and $\pen_1$ known but $C^{\star}$ is
unknown ---for instance, for linear estimators in regression with the least-squares risk--- resampling-based
procedures are natural competitors, that can be used either for choosing the constant in front
of $\pen_1\,$, or directly for the initial estimator-selection problem.
Then, minimal-penalty algorithms have a clear advantage over resampling, because of their much
smaller computational cost (see Section~\ref{sec.practical.cost}), while they have comparable or better
statistical performance according to both theoretical and experimental results, as shown for instance
by \citet{Arl_Bac:2009:minikernel_long_v2}.

Second, when Algorithm~\ref{algo.penmingal} works with some unknown $\pen_0$ and/or $\pen_1\,$, an option mentioned in Remark~\ref{rk.reech} in Section~\ref{sec.theory.complete} is to estimate them by resampling.
Then, the computational cost of Algorithm~\ref{algo.penmingal} is comparable to that of cross-validation and other resampling strategies applied to the initial estimator-selection problem.
In such cases, the interest of using minimal penalties is the precise non-asymptotic calibration of the constant in front of the resampling-based penalty, which is not guaranteed when using the theoretical value for this constant, since it is often based upon asymptotic considerations.
In addition, the conjecture detailed in Section~\ref{sec.empirical.overpenalization} suggests another reason for combining resampling and minimal penalties in such frameworks.

%%%%%%%%%%%%%%%%%%%%%%%%%%%%%%%%%%%%%%%%%%%%%%%%%%%%%%%%%%%%%%%%%%%%%%%%%%%%%%%%%%%%
%%%%%%%%%%%%%%%%%%%%%%%%%%%%%%%%%%%%%%%%%%%%%%%%%%%%%%%%%%%%%%%%%%%%%%%%%%%%%%%%%%%%

\section{Some practical remarks} \label{sec.practical}
This section discusses several practical questions about the use of minimal-penalty algorithms.
A more detailed study of some of them can be found in the survey by \citet{Bau_Mau_Mic:2010}.

\subsection{Several definitions for \texorpdfstring{$\Ch$}{hat(C)}} \label{sec.practical.jump-vs-slope}
%
%%% 1. Petit recapitulatif des differentes definitions de \Ch
%
The minimal-penalty estimator $\Ch$ of the constant that should be put in front of the penalty $\pen_1$ can be defined in several ways, which leads to the practical issue of choosing one among these definitions.
Two main approaches are proposed in the previous sections. 

\paragraph{Jump approach} 
First, $\Ch$ can be defined 
as the position $\Chjumpgal$ of ``the unique large jump'' of $C \mapsto \C_{\mhgalzero(C)}\,$, as in Algorithms \ref{algo.OLS.jump}, \ref{algo.gal.slope.naif}, \ref{algo.penmin.linear}, and~\ref{algo.penmingal}. 
Section~\ref{sec.slopeOLS.math} suggests two ways to formally define $\Chjumpgal\,$: 
choosing the maximal jump $\Chwindow (\eta)$ over a geometric window $[C/(1+\eta) ,C(1+\eta)]$, 
and choosing the value $\Chthr (T_n)$ of $C$ for which $\C_{\mhgalzero(C)}$ goes under some threshold $T_n\,$.
Another natural option is to choose the position of the maximal jump
\[
\Chmaxjump \in \argmax_{C \geq 0} \setj{ \C_{\mhgalzero(C^-)}  - \C_{\mhgalzero(C^+)}}
\, ,
\]
that is, taking $\lim_{\eta \to 0} \Chwindow (\eta)$.

\paragraph{Slope approach} 
Second, $\Ch$ can be defined as $\Chslope\,$, the opposite of the estimated value of the slope of the empirical risk as a function of $\pen_0\,$, as in Algorithms \ref{algo.OLS.slope} and~\ref{algo.penmingal.slope}.
This approach can be formalized in several ways, using ordinary or robust linear regression, either over a fixed range $[p_{\min},p_{\max}]$ of values of $\pen_0\,$, or with the method $\mhcapushe$ proposed 
by \citet[Section~4.2]{Bau_Mau_Mic:2010}, which is based upon a stability study of the selected estimator and depends on some parameter $pct \in (0,1]$.

\medbreak

Note that $\Chwindow\,$, $\Chthr\,$, and $\Chslope$ all depend on some hyperparameter ($\eta$, $T_n\,$, $p_{\min}$ and $p_{\max}\,$, $pct$).
We refer to Appendix~\ref{app.details-simus.proc} for more details on each definition of $\Ch$
considered in this section. 

\begin{figure}
\begin{center}
\includegraphics[width=.68\textwidth]{\pathfig/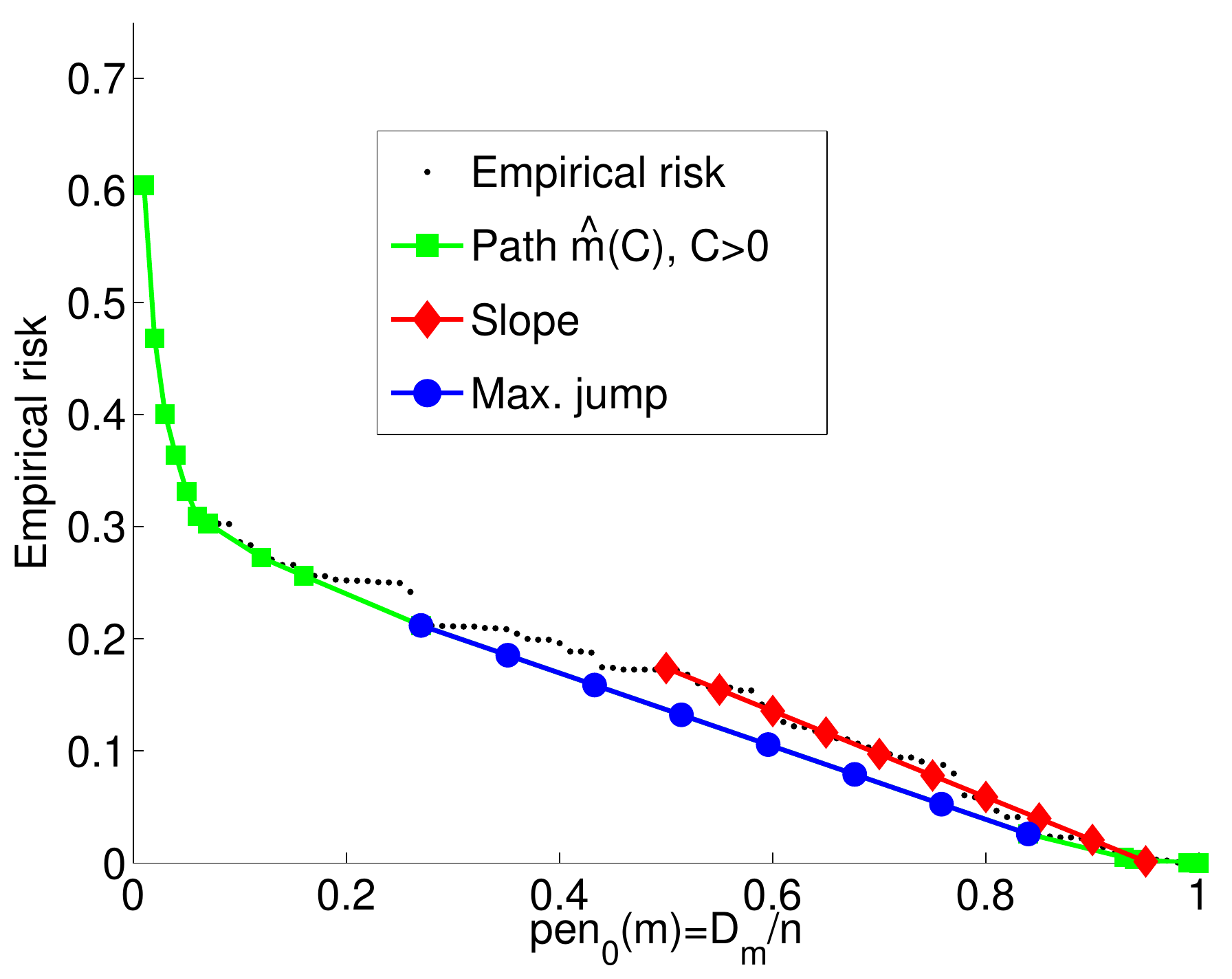} \hfill
\caption{\label{fig.Lcurve.easy.ech6}
Connection between Algorithms \ref{algo.OLS.jump} and~\ref{algo.OLS.slope}\textup{:}
Plot of $D_m \mapsto n^{-1}\norms{Y-\Fh_m}^2$ and visualization of $\Chslope\,$.
Setting called `easy' in Appendix~\ref{app.details-simus}.
}
\end{center}
\end{figure}

\begin{figure}
\begin{center}
\includegraphics[width=.68\textwidth]{\pathfig/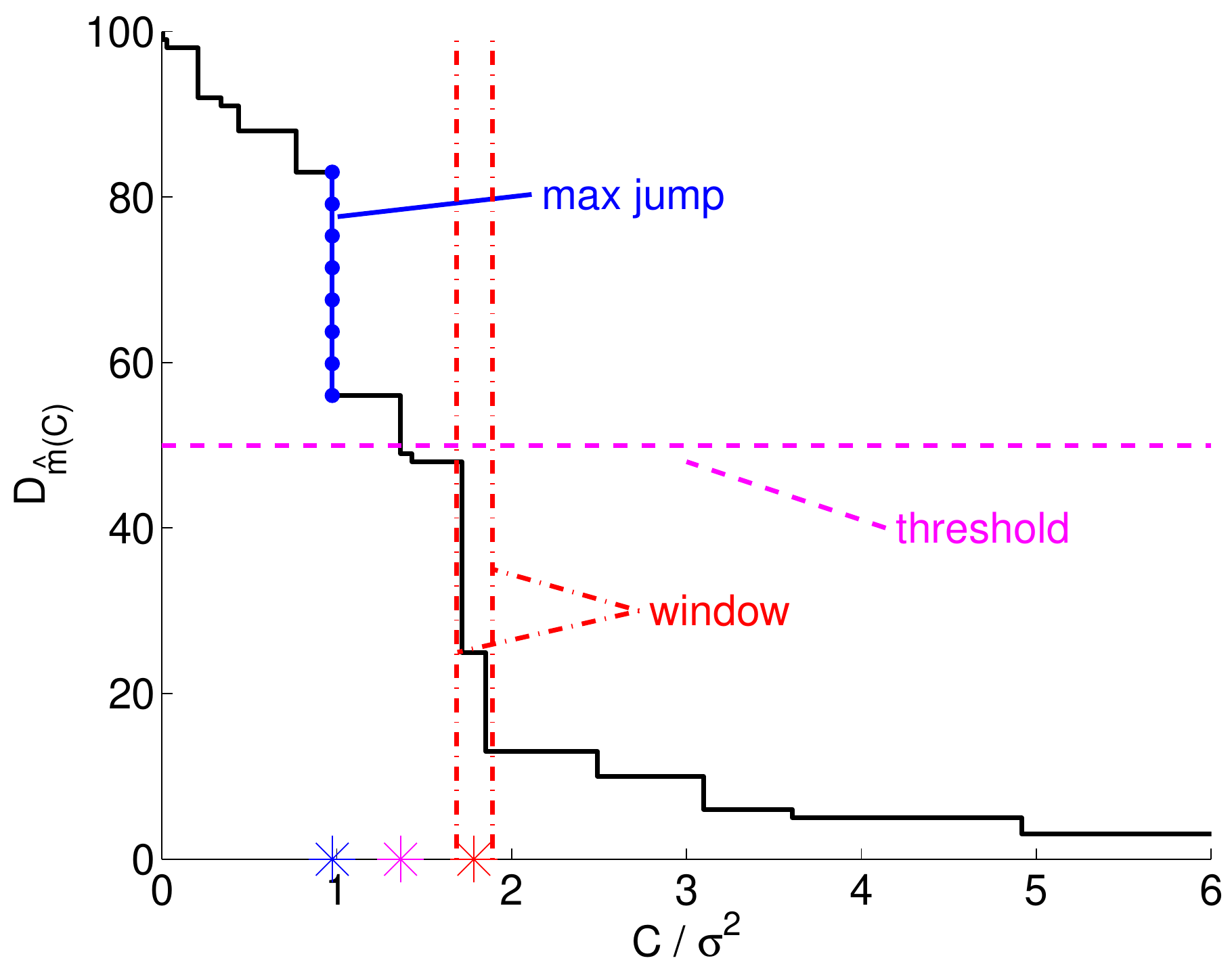}
\caption{\label{fig.DmhC.easy.ech5}
Plot of $C \mapsto D_{\mh(C)}$ \textup{(}black solid lines\textup{)} 
and visualization of the three versions of Algorithm~\ref{algo.OLS.jump}
on the same sample.  
The maximal jump is shown by blue dots. 
The threshold value $T_n = n/2$ is materialized by an horizontal dashed line. 
The largest jump over a window is shown by vertical dash-dot lines. 
The values of $\Chmaxjump\,$, $\Chthr$ and $\Chwindow$ are shown by stars on the $x$-axis \textup{(}in this order\textup{)}. 
The sample chosen here is not typical at all \textup{(}see Table~\ref{tab.stats-comp-mh.easy}\textup{)}  
but illustrates well the differences between $\Chmaxjump\,$, $\Chwindow\,$, and~$\Chthr\,$. 
Setting called `easy' in Appendix~\ref{app.details-simus}.
}
\end{center}
\end{figure}
%%%
%%%%%% Autre mise en page possible
%%%
%%%\begin{figure}
%%%\begin{center}
%%%\begin{minipage}[b]{.46\linewidth}
%%%\includegraphics[width=\textwidth]{\pathfig/figure_Remp_mpath_jump_slope_zoom_xpen__ech6_A140306fig_P1G2_easy_n100_seed3.pdf} \hfill
%%%%
%%%\caption{\label{fig.Lcurve.easy.ech6}
%%%Connection between Algorithms \ref{algo.OLS.jump} and~\ref{algo.OLS.slope}\textup{:}
%%%Plot of $D_m \mapsto n^{-1}\norms{Y-\Fh_m}^2$ and visualization of $\Chslope\,$.
%%%Setting called `easy' in Appendix~\ref{app.details-simus}.
%%%}
%%%\end{minipage}
%%%\hfill
%%%\begin{minipage}[b]{.46\linewidth}
%%%\includegraphics[width=\textwidth]{\pathfig/figure_DmhC_tout__ech5_A140306fig_P1G2_easy_n100_seed3.pdf}
%%%%%
%%%\caption{\label{fig.DmhC.easy.ech5}
%%%Plot of $C \mapsto D_{\mh(C)}$ (black solid lines) 
%%%and visualization of the three versions of Algorithm~\ref{algo.OLS.jump}
%%%on the same sample.  
%%%The maximal jump is shown by blue dots. 
%%%The threshold value $T_n = n/2$ is materialized by an horizontal dashed line. 
%%%The largest jump over a window is shown by vertical dash-dot lines. 
%%%The values of $\Chmaxjump\,$, $\Chthr$ and $\Chwindow$ are shown by stars on the $x$-axis (in this order). 
%%%The simulation setting is called `easy' in Appendix~\ref{app.details-simus}.
%%%%
%%%The sample chosen here is not typical at all \textup{(}see Table~\ref{tab.stats-comp-mh.easy}\textup{)}  
%%%but illustrates well the differences between $\Chmaxjump\,$, $\Chwindow\,$, and $\Chthr\,$. 
%%%}
%%%\end{minipage}
%%%\end{center}
%%%\end{figure}

\paragraph{Theoretical comparison}
%%% 2. Theoretical comparison of the different definitions
%
Let us first compare theoretically the various definitions of $\Ch$.
For $\Chjumpgal\,$, when there is a single large jump in 
\[ C \mapsto \C_{\mhgalzero(C)} \, , \]  
as illustrated by Figure~\ref{fig.OLS.algo} in Section~\ref{sec.slopeOLS.algo}, 
reasonable choices for $T_n$ and $\eta$ 
make $\Chwindow(\eta)$ very close to $\Chmaxjump=\Chthr(T_n)$.
On the contrary, when the phase transition around the minimal penalty yields several jumps of medium size in 
\[ C \mapsto \C_{\mhgalzero(C)} \, , \] 
as in Figure~\ref{fig.DmhC.easy.ech5} for instance, $\Chmaxjump\,$, $\Chthr(T_n)$, and $\Chwindow(\eta)$ can take quite different values and lead to selecting different models.
Theoretical guarantees such as Theorem~\ref{thm.OLS} in Section~\ref{sec.slopeOLS.math} do not exclude such a situation, even asymptotically, so they only apply to $\Chthr (T_n)$ and $\Chwindow (\eta)$ with $T_n$ and $\eta$ of the correct order of magnitude.

Yet, the maximal jump and threshold definitions with 
$T_n = \overline{\C} \egaldef (\max_m \C_m + \min_m \C_m)/2$ 
coincide when the largest jump is of size at least $(\max_m \C_m - \min_m \C_m)/2$. 
This condition always holds true if no $m \in \M$ has a complexity $\C_m \in (\overline{\C} ; \max_m \C_m)$,
which often occurs for computational reasons, 
since estimators with complexity $\C_m > \overline{\C}$ usually are hard to compute and known to be suboptimal. 

For the slope approach, no theoretical guarantee is available, but the linear behavior of $\Remp\parens{\shm}$ as a function of $\pen_0(m)$ is supported theoretically from expectation computations, as detailed in Sections~\ref{sec.slopeOLS.optimal}--\ref{sec.slopeOLS.penmin} and~\ref{sec.penmingal.fails}--\ref{sec.penmingal.linear}.

The jump and slope approaches can seem quite different at first sight, but they actually are the two sides of the same coin.
Section~\ref{sec.slopeOLS} shows that reasoning from the same computations, Eq.~\eqref{eq.EriskFhm}--\eqref{eq.EriskempFhm}, can lead to a heuristic justification of both approaches.
Another argument enlightens the similarity of the jump and slope approaches.
By Proposition~\ref{pro.algo.path} and its proof in Appendix~\ref{app.algos.path},
the path $(\mhgalzero(C))_{C>0}$ is piecewise constant, $\mhgalzero(C)=m_i$ for $C \in [C_i,C_{i+1})$,
and the sequences $(m_i)_{0 \leq i \leq i_{\max}}$ and $(C_i)_{0 \leq i \leq i_{\max}}$ can be visualized on the L-curve $(\pen_0(m) , \Remp\parens{\shm})_{\mM}\,$: 
the angles of the lower convex envelope of the L-curve exactly correspond to 
the $m_i\,$, $0 \leq i \leq i_{\max}\,$, and
\[
C_i = \frac{\Remp\parens{\sh_{m_i}} - \Remp\parens{\sh_{m_{i-1}}}}{ \pen_0(m_{i-1}) - \pen_0(m_i)}
\]
is the opposite of the slope of the segment joining $m_{i-1}$ to $m_i$ on the L-curve.
So, $\Chmaxjump$ can be visualized on the L-curve, 
as illustrated by Figure~\ref{fig.Lcurve.easy.ech6}.  
Given the L-curve (black dots), draw its (piecewise linear) lower convex envelope (green squares), localize the widest segment 
---in terms of values of $\C_m\,$, which is usually proportional to $\pen_0(m)$---:  
its slope is $-\Chmaxjump\,$.
Then, one clearly see why $\Chmaxjump$ is often close to $\Chslope$ in the setting of Figure~\ref{fig.Lcurve.easy.ech6}:
for a random point cloud with a linear trend of slope $\approx - C^{\star}$ for large abscissa values, 
estimating its slope by linear regression is almost equivalent to looking at the slope 
of the longest segment of its lower convex envelope.
Note that $\Chthr$ can be visualized on the L-curve similarly to $\Chmaxjump\,$.

This direct comparison emphasizes the respective drawbacks of $\Chmaxjump$ and $\Chslope\,$.
When the amplitude of the largest jump is small, $\Chmaxjump$ is not a reliable estimation of $C^{\star}$, see Figure~\ref{fig.DmhC-Lcurve.easy.ech540}b in Appendix~\ref{app.morefig}.

When some large models have a significantly positive approximation error,
as in the `hard' setting described in Appendix~\ref{app.details-simus} 
---see the right of Figure~\ref{fig.OLS.slope-vs-residuals.easy-hard}--- 
they pollute the slope estimation and make $\Chslope$ biased, unless only a few such models are present and robust regression is used.
In the latter case, since $\Chjumpgal \in \sets{\Chmaxjump,\Chwindow,\Chthr}$ only depends on the \emph{lower convex envelope} of the L-curve, even a large number of ``polluting'' models will not influence $\Chjumpgal$ at all, making it more robust.

This difference between $\Chjumpgal$ and $\Chslope$ also appears in the assumptions made for their theoretical and heuristic justifications. 
In Section~\ref{sec.slopeOLS}, $\Chslope$ requires the approximation error to be almost constant over \emph{all} large models 
---which makes sense when $(S_m)_{\mM}$ is a family of models with increasing complexity, for instance, but can be violated in some other contexts---, 
whereas $\Chjumpgal$ only assumes that
\emph{two} models exist with a small approximation error, one of moderate complexity 
and one of large complexity. %% in \eqref{hyp.thm.OLS.Id} 

\paragraph{Experimental comparison}
%%% 3. Experimental comparison of the different definitions
%
In addition to the above theoretical comparison, 
we report the results of new simulation experiments for variable selection in least-squares regression.
We consider two settings: 
in the `easy' setting, the order between variables is known, while in the `hard' setting, two possible orders (the correct one and its converse) are considered alternatively, making half of the models very bad. 
The `hard' setting is the archetype of a setting where 
the approximation error is \emph{not} constant over large models; 
it does not aim to be realistic. 
All details about simulation experiments ---data generation, model collection, and exact implementation for each definition of $\Ch$--- are given in Appendix~\ref{app.details-simus}.

\begin{table}
\begin{center}
\begin{tabular}{l@{\hspace{0.050\textwidth}}c@{\hspace{0.025\textwidth}}c@{\hspace{0.025\textwidth}}c@{\hspace{0.025\textwidth}}c@{\hspace{0.050\textwidth}}c@{\hspace{0.025\textwidth}}c}
Configuration & All equal & Exactly $4$ & At least $3$ 
& All different 
& $\Chmaxjump=\Chthr$ & Max, thr, and win 
\\
 & & equal & equal 
 &  
 &  & different 
 \\[0.1cm]
Frequency (`easy') & 0.524 & 0.238 & 0.967 
& $<10^{-3}$ %0.0003
& 0.777 & 0.009 
\\
Frequency (`hard') & 0.134 & 0.118 & 0.894 
& $< 10^{-3}$ %0.0005
& 0.769 & 0.008 
\end{tabular}
\vspace{0.2cm}
\caption{\label{tab.stats-comp-mh.easy} Frequency of various configurations for the set of five models $\mh$ respectively selected by Algorithm~\ref{algo.OLS.jump} with $\Chmaxjump$ \textup{(}`max'\textup{)}, $\Chthr$ \textup{(}`thr'\textup{)}, $\Chwindow$ \textup{(}`win'\textup{)}, by Algorithm~\ref{algo.OLS.slope} \textup{(}$\Chslope$\textup{)} and by $\mhcapushe\,$. 
`Easy' and `hard' settings,  
see Appendix~\ref{app.details-simus} for details.
}
\end{center}
\end{table}

\medbreak

First, since the beginning of this section outlines strong theoretical connections 
between the different definitions of $\Ch$, a natural question is: 
how different are the models finally selected, depending on the definition taken for $\Ch$?
Table~\ref{tab.stats-comp-mh.easy} shows that they all coincide most of the time in the `easy' setting ---with a clear single large jump, as for the sample of Figures~\ref{fig.OLS.algo} and~\ref{fig.Lcurve.easy.ech6}---, and they globally agree at least $90\%$ of the time or more in both settings.
The probability of a total disagreement is very small (less than $0.1\%$) even if it sometimes occurs, as illustrated by Figure~\ref{fig.DmhC.easy.ech5}, 
where $\Chmaxjump\,$, $\Chthr\,$, and $\Chwindow$ respectively lead to selecting $\mh=14$, $11$, and~$7$;
Figure~\ref{fig.DmhC-Lcurve.easy.ech540}b in Appendix~\ref{app.morefig} shows a similar configuration.
Similar conclusions are obtained by \citet[Section~3.3]{Arl_Mas:2009:pente} about $\Chmaxjump$ and $\Chthr\,$.

\begin{figure}
\begin{center}
\begin{minipage}[c]{.49\linewidth}
\includegraphics[width=\textwidth]{\pathfig/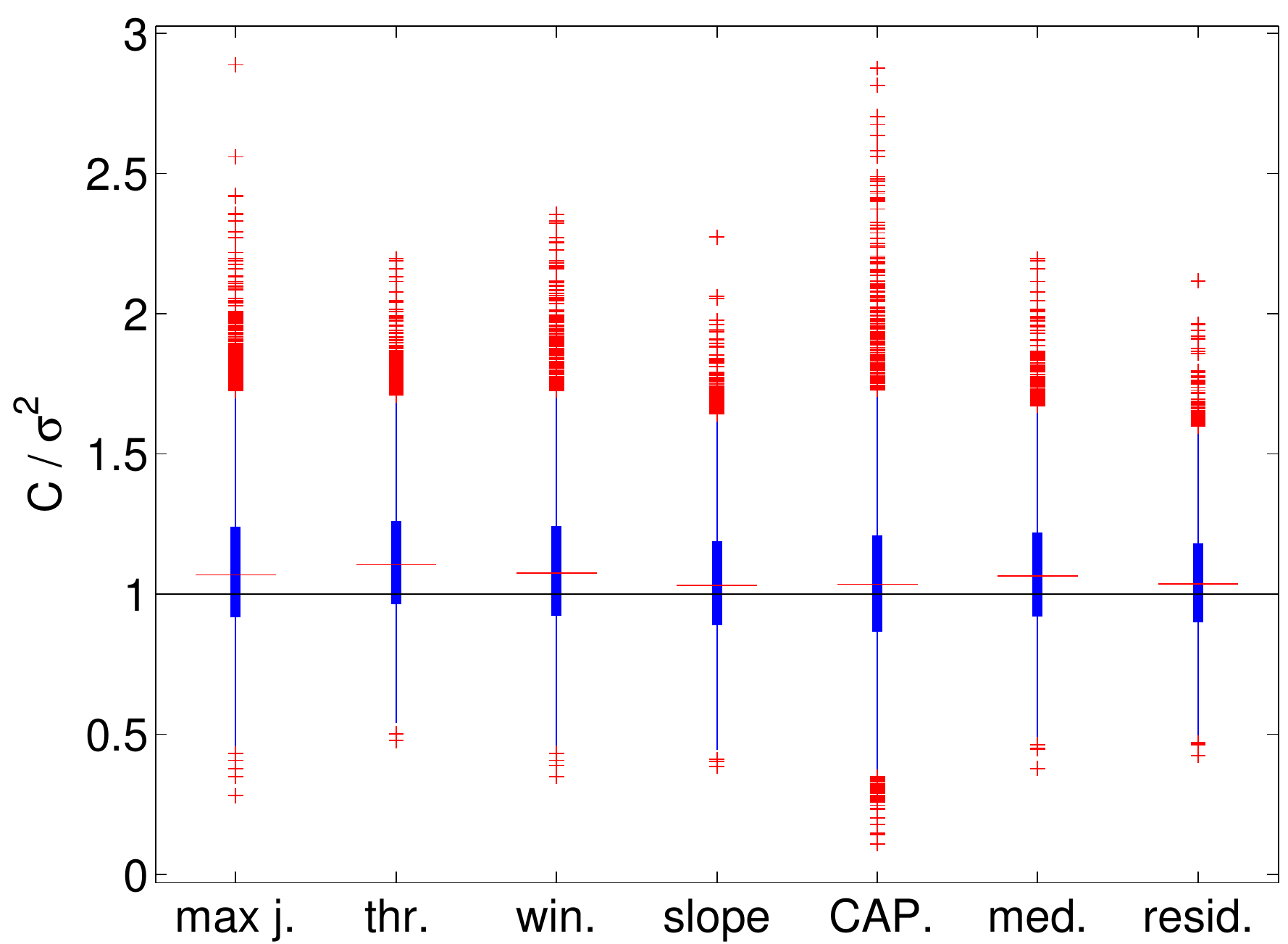}
\\ \centerline{(a) `Easy' setting.}
\end{minipage}
\hfill
\begin{minipage}[c]{.49\linewidth}
\includegraphics[width=\textwidth]{\pathfig/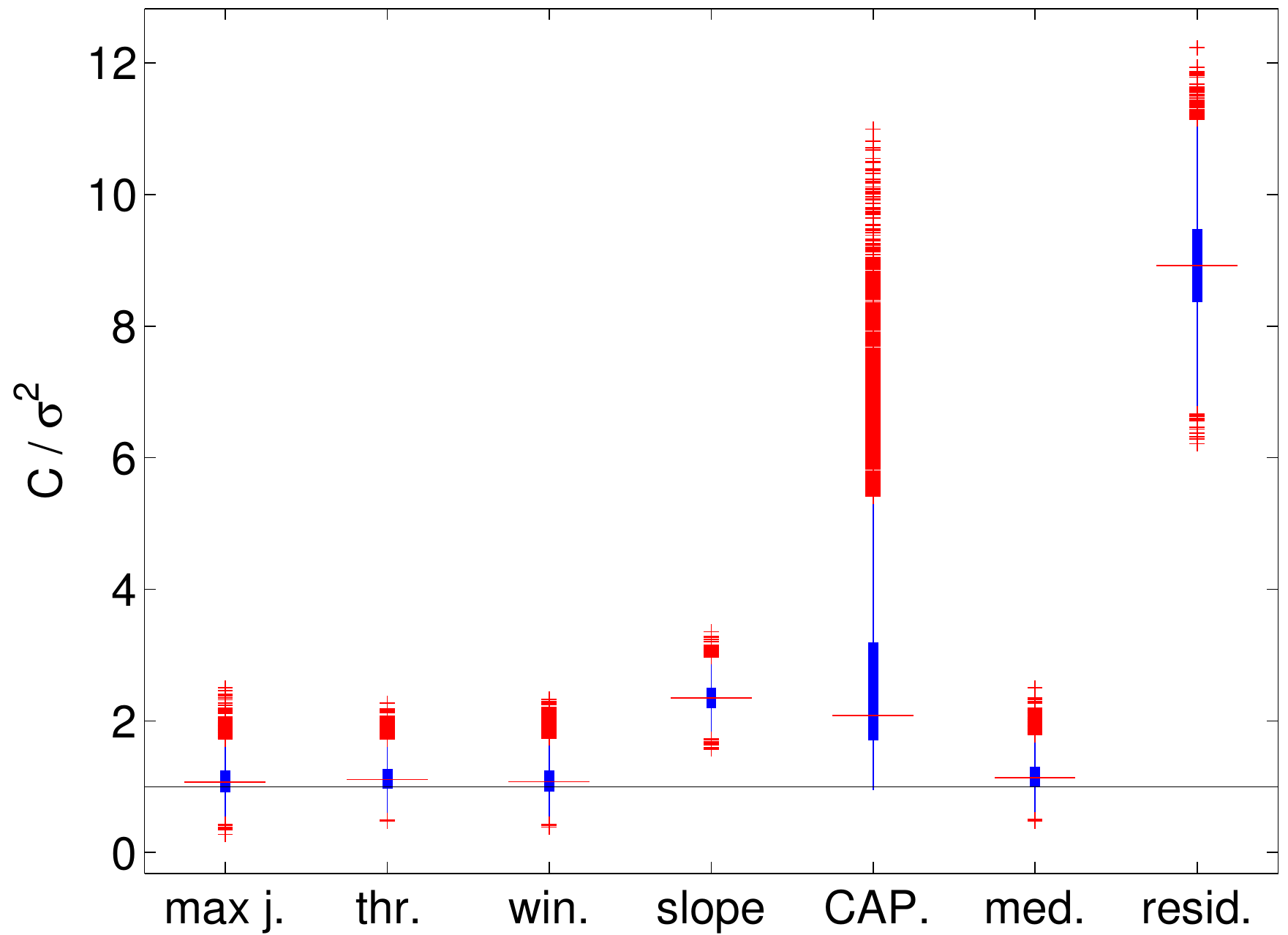}
\\ \centerline{(b) `Hard' setting.}
\end{minipage}
\caption{\label{fig.OLS.dist-Ch} %%Figure~\ref{fig.OLS.dist-Ch} in Section~\ref{sec.practical.jump-vs-slope}
Distribution over $10\,000$ independent samples of $\widehat{C}/\sigma^2$ for seven estimators $\widehat{C}$ of $\sigma^2$\textup{:} 
$\Chmaxjump$ \textup{(}`max j.'\textup{)}, $\Chthr$ \textup{(}`thr.'\textup{)}, $\Chwindow$ \textup{(}`win.'\textup{)}, $\Chslope$ in Algorithm~\ref{algo.OLS.slope} \textup{(}`slope'\textup{)}, $\Chcapushe$ \textup{(}`CAP.'\textup{)}, the median of $\sets{\Chmaxjump,\Chthr,\Chwindow,\Chslope,\Chcapushe}$ \textup{(}`med.'\textup{)}, and $\sigh^2_{m_0} $ defined by Eq.~\eqref{def.sighm0} \textup{(}`resid.'\textup{)}.
See Appendix~\ref{app.details-simus} for details.
}
\end{center}
\end{figure}

\medbreak

Second, since our experiments consider projection estimators in least-squares regression,
all minimal-penalty based $\Ch$ estimate $\sigma^2$, 
so they can be compared to $\sigh^2_{m_0}$ ---defined by Eq.~\eqref{def.sighm0}--- 
as estimators of the residual variance $\sigma^2$.
Results are provided in Figure~\ref{fig.OLS.dist-Ch}, as well as Tables~\ref{tab.dist-Ch.easy}--\ref{tab.dist-Ch.hard} in Appendix~\ref{app.morefig}, where several values of parameters of the $\Ch$ are compared.
In the `easy' setting (Figure~\ref{fig.OLS.dist-Ch}a), all methods behave similarly and as expected from theoretical arguments: the distribution of $\Ch$ is asymmetric around $\sigma^2$, with smaller deviations below $\sigma^2$ than above $\sigma^2$, as in the bounds of Proposition~\ref{pro.variance-estim}.
Such an asymmetry is a good property in terms of model-selection performance, as suggested by Figure~\ref{fig.surpen} in Section~\ref{sec.empirical.overpenalization} for instance.
The order of magnitude of the deviations of $\Chjumpgal/\sigma^2$ from Proposition~\ref{pro.variance-estim} is $\biaismax (c_n) / \sigma^2 +  \sqrt{\log(n)/n}$ with $c_n = n/3$ (for $\Chwindow$) or $T_n/2$ (for $\Chthr$); 
in our experiments, with $n=100$ and $T_n=n/2$, we get $\biaismax (c_n) / \sigma^2 \in [0.04 , 0.08]$, 
%%% Biais pour dimension n/3: 0.01217037 - multiplie par 4 (=1/\sigma^2), on obtient: 0.04868149
%%% Biais pour dimension n/4: 0.01887511 - multiplie par 4 (=1/\sigma^2), on obtient: 0.07550043
and $\sqrt{\log(n)/n} \approx 0.2$, so the constants appearing in Proposition~\ref{pro.variance-estim} here are pessimistic.

The most variable $\Ch$ clearly is $\Chcapushe\,$, but to be completely fair, we must notice that the procedure proposed by \citet{Bau_Mau_Mic:2010} only outputs a selected model $\mhcapushe$ and we make an arbitrary choice for defining some $\Chcapushe$ from the definition of $\mhcapushe$ (see Appendix~\ref{app.details-simus.proc}).
Among other definitions of $\Ch$, $\Chmaxjump\,$, and $\Chwindow$ are slightly more variable than the others but the difference is mild.

Interesting differences occur in the `hard' setting, which is designed as a case example for difficult situations for $\Chslope\,$, $\Chcapushe\,$, and $\sigh^2_{m_0}\,$.
As expected, $\Chslope$ completely fails because of the wide amplitude of the approximation error among large models, and $\sigh^2_{m_0}$ behaves totally differently depending on the parity of $m_0\,$: $\sigh^2_{m_0}$ is worse than $\Chslope$ when $S_{m_0}$ is one of the `bad' models, while it works well when $S_{m_0}$ is one of the `good' models 
(see Figure~\ref{fig.OLS.slope-vs-residuals.easy-hard} and Table~\ref{tab.dist-Ch.hard} in Appendix~\ref{app.morefig}). 
This failure of $\Chslope$ and $\Chcapushe$ ---when $(S_m)_{m \in \M}$ is the union of 
subcollections having different approximation properties--- 
is also reported by \citet[Section~4.5]{Bau:2009:phd}, 
\citet[Figures~3--4]{Dev:2017} and \citet[Figure~6]{Dev_Gou_Pog:2015:journal}  
in realistic settings. 
The nested algorithm presented in Section~\ref{sec.practical.nested} might be a way to fix this issue, 
even if it has not been tested yet in such situations. 

More generally, depending on the setting, choosing the parameter for one definition of $\Ch$ can be a big practical issue.
For instance, Tables~\ref{tab.dist-Ch.easy}--\ref{tab.dist-Ch.hard} in Appendix~\ref{app.morefig} show that $\Chthr (T_n)$ is sensitive to the choice of $T_n\,$.
Even if $T_n=n/2$ works well for the `easy' and `hard' settings,
it is certainly not a universally good choice, and changing $F$, $n$ or $\sigma^2$ could easily make it fail compared to other definitions of $\Ch$.
Similarly, the performance of $\Chslope$ strongly depends on the parameters $p_{\min},p_{\max}$ and choosing them from data is not an easy task, a problem also reported in the change-point detection setting \citep[Chapter~4]{Leb:2002}.
A reasonable option is given by $\mhcapushe$ \citep[Section~4.2]{Bau_Mau_Mic:2010}, and it works reasonably well in the `easy' setting, but it fails in the `hard' setting as expected.

\begin{figure}
\begin{center}
\begin{minipage}[c]{.49\linewidth}
\includegraphics[width=\textwidth]{\pathfig/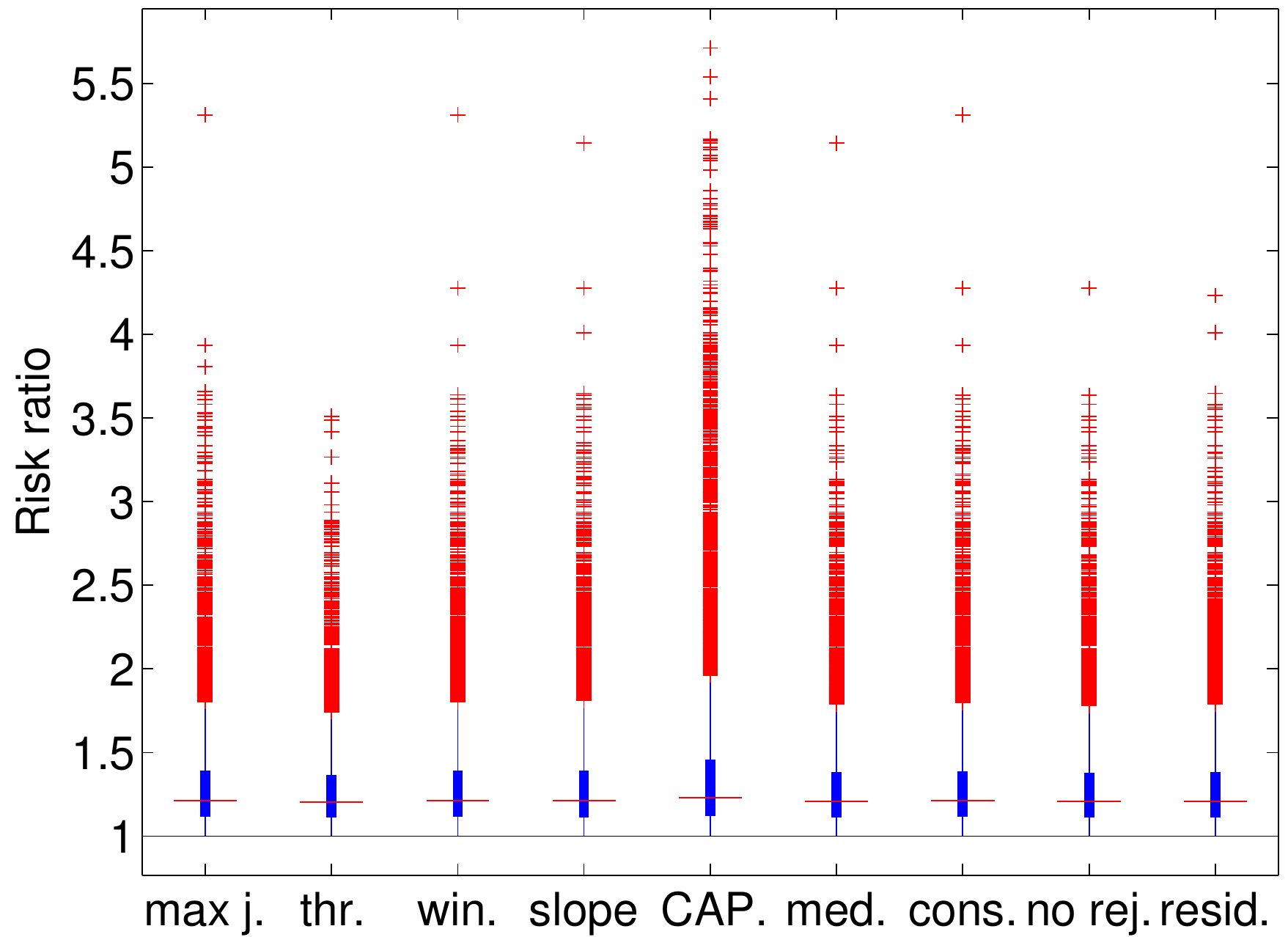}
\\ \centerline{(a) `Easy' setting.}
\end{minipage}
\hfill
\begin{minipage}[c]{.49\linewidth}
\includegraphics[width=\textwidth]{\pathfig/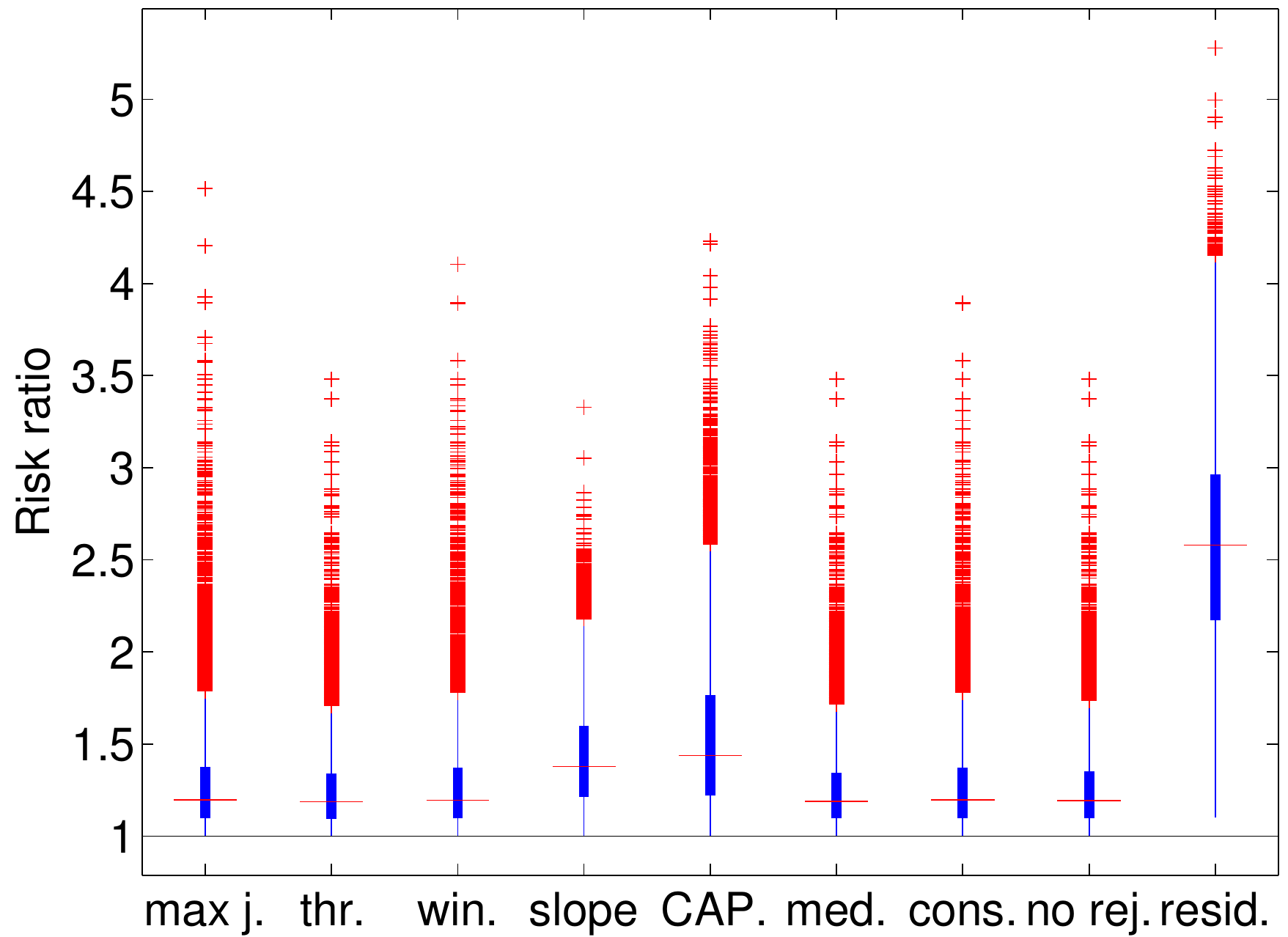}
\\ \centerline{(b) `Hard' setting.}
\end{minipage}
\caption{\label{fig.OLS.dist-risk-ratio}%%Figure~\ref{fig.OLS.dist-risk-ratio} in Section~\ref{sec.practical.jump-vs-slope}
Distribution over $10\,000$ independent samples of $\norms{ \Fh_{\mh} - F}^2 / \inf_{\mM} \norms{ \Fh_m - F}^2 $ for $\mh=\mh(2\Ch)$ with $\Ch$ among the seven estimators $\Ch$ compared 
in Figure~\ref{fig.OLS.dist-Ch}, and for $\mh$ obtained by majority vote among $\sets{\mh(2\Chmaxjump),\mh(2\Chthr),\mh(2\Chwindow),\mh(2\Chslope),\mhcapushe}$ with $\mh(2\Chwindow)$ as a default choice \textup{(}`cons.'\textup{)} or considering only the samples on which such a majority exists \textup{(}`no rej.'\textup{)}.
See Appendix~\ref{app.details-simus} for details.
}
\end{center}
\end{figure}

\medbreak

%%% Model selection performance
%
Third, the model-selection performance of all these
procedures is assessed by Figure~\ref{fig.OLS.dist-risk-ratio} and by Tables~\ref{tab.dist-Ch.easy}--\ref{tab.dist-Ch.hard} in Appendix~\ref{app.morefig}.

At first order, the conclusions are similar to the
ones obtained for estimating $\sigma^2$.
All definitions of $\Ch$ work well in the `easy' setting.
In the `hard' setting, $\sigh^2_{m_0}$ completely fails,
while $\mhcapushe$ and $\Chslope$ do slightly worse than
the other formulations of the slope-heuristics algorithm.

The detailed comparison of the procedures that work well
is a bit different:
the model-selection performance (risk ratios)
are not ordered exactly as the mean-squared errors
in Tables~\ref{tab.dist-Ch.easy}--\ref{tab.dist-Ch.hard}.
The main reason is that risk estimation ---which reduces to estimating $\sigma^2$ in our setting--- is different from model selection \citep{Bre_Spe:1992}.
Figure~\ref{fig.surpen} in Section~\ref{sec.empirical.overpenalization} shows at least one reason for this difference:
overpenalizing slightly, that is, overestimating $\sigma^2$ a bit, improves the model-selection performance.
According to Figure~\ref{fig.surpen}, the best overpenalization factor is $1.12$ in the `easy' setting.
For instance, Table~\ref{tab.dist-Ch.easy} shows that
taking $D_{m_0}=n/10$ for $\sigh^2_{m_0}$
leads to better model-selection performance than
$D_0 = n/2$ in the `easy' setting,
even if $D_0 = n/2$ yields a much better estimator of $\sigma^2$.

Note however that for a given bias (as an estimator of $\sigma^2$),
the best model-selection performance is obtained when the variance is the smallest:
compare for instance $\Chslope$ with $D_0=n/2$ and
\textsc{Capushe} in the `easy' setting (Table~\ref{tab.dist-Ch.easy} in Appendix~\ref{app.morefig}).

\medbreak

%%%%% Experiences dans la litterature
%
Let us finally mention that previous simulation experiments in various settings have compared 
some of the definitions of $\Ch$. 
In short, almost all of them report that $\Chmaxjump$ is less reliable 
---because of the event on which there is not a single large jump 
(\citealp[Figure~8.11]{Mau:2008:phd}; \citealp[Section~5]{Bau_Mau_Mic:2010}), 
which happens more or less often--- compared to 
$\Chthr$ \citep{Arl_Bac:2009:minikernel_long_v2,Sol_Arl_Bac:2011:multitask}, 
$\Chwindow$ \citep{Bon_Tou:2010}, and $\Chslope$ or $\Chcapushe$ 
(\citealp[Section~3.2]{Bau:2009:phd}; \citealp{Mau_Mic:2010,Con:2011:phd,Bau_Mau_Mic:2010};  \citealp[Table~2.1]{Roc:2014}).   
Only \citet{Dev_Gal:2016} report similar performances for $\Chmaxjump$ and $\Chcapushe\,$. 
Nevertheless, $\Chmaxjump$ remains useful for confirming the choice made with another definition of $\Ch$ 
\citep{Con:2011:phd,Bau_Mau_Mic:2010}, 
with a visual check that there is a single large jump. 
The slope approach can also fail for reasons detailed previously in this subsection 
\citep{Bau:2009:phd,Dev:2017,Dev_Gou_Pog:2015:journal}. 
\citet{Leb:2005} even shows that $\Chmaxjump$ and $\Chslope$ can both fail, which motivates a modified algorithm 
---called ``calibrated method''--- for change-point detection; 
note that \citet{Arl_Cel_Har:2012:kernelchpt} fix this precise failure by using the slope heuristics 
with a penalty shape depending on two constants, as detailed in Section~\ref{sec.practical.nappe}.

\paragraph{Conclusion on the choice of $\Ch$}
%%% 4. Conclusion on the choice of \Ch
%
First, it is not surprising to have to choose among several definitions of $\Ch$ 
or to choose some hyperparameter such as $\eta$, $T_n$ or $p_{\min}\,$, because of no free lunch theorems: 
no fully automatic estimation procedure can work uniformly well 
over all statistical problems \citep[Chapter~7]{Dev_Gyo_Lug:1996}. 
An expert advice is always necessary at some point.
For minimal-penalty algorithms, our suggests join the ones of \citet{Bau_Mau_Mic:2010} and \citet{Con:2011:phd}:
never use a single definition of $\Ch$ in a blind way, 
either by considering several definitions for $\Ch$ 
or by checking visually that there is a clear complexity jump 
and/or that the L-curve exhibits a clear linear trend on the data. 
When computing all values of $(\Remp(\shm), \pen_0(m), \pen_1(m), \C_m)_{m \in \M}$ is too expensive, 
one should also take into account the computational cost of the procedure, 
as discussed in Section~\ref{sec.practical.cost}.

We propose the following (semi-automatic) approach for using several definitions 
$\Chmaxjump\,$, $\Chthr\,$, $\Chwindow\,$, $\Chslope\,$, and $\Chcapushe$ of $\Ch$ simultaneously. 
If the goal is to estimate $\sigma^2$, take their median.
If the goal is estimator selection, compute the five corresponding estimator choices $\mh$, and make a majority vote: if at least three over five coincide, take their common value, otherwise, output a warning and ask the user to look at the complexity jump and the L-curve.
When the five methods disagree, using a completely different approach remains a good option, for instance, cross-validation.
The results of using this strategy (in a fully automatic way since our experiments need to be reproducible) are reported in Figures~\ref{fig.OLS.dist-Ch}--\ref{fig.OLS.dist-risk-ratio} above as well as Tables~\ref{tab.dist-Ch.easy}--\ref{tab.dist-Ch.hard} in Appendix~\ref{app.morefig}, showing good performance in all settings.

Finally, the above comparison also points out several risky choices for $\Ch$ 
(in addition to $\sigh^2_{m_0}$): 
$\Chmaxjump$ without checking that there is indeed a single large jump,
$\Chthr$ with a bad choice for $T_n\,$,
the ``naive'' version $\Chslope$ of the slope approach,  
and $\Chslope$ or $\Chcapushe$ when selecting among a union of subcollection of estimators 
that may have different approximation properties.

\subsection{Algorithmic cost} \label{sec.practical.cost}
\paragraph{When all empirical risks can be computed} 
Let us assume that the values of the empirical risk $\Remp\parens{\shm}$, 
the minimal and optimal penalty shapes $\pen_0(m)$ and $\pen_1(m)$, 
and the complexity $\C_m$ for all $\mM$ are stored in memory. 
Then, the computational complexity of minimal-penalty algorithms is the following.

\smallbreak

For Algorithm~\ref{algo.penmingal}, computing the full path $(\mhgalzero(C))_{C \geq 0}$ 
requires at most $\grandO([\card \M]^2)$ operations ---as shown in Appendix~\ref{app.algos.path}--- 
and much less in practice.
Indeed, denoting by $i_{\max}+2$ the cardinality of this path 
---which must be smaller than $\card(\M)$---, it can be computed with $\grandO( i_{\max} \card \M)$ operations.

Furthermore, depending on the definition of $\Chjumpgal\,$, it might not be necessary to compute the full path.
For instance, with the threshold approach, using the notation of Appendix~\ref{app.algos.path}, if $i(T_n)$ is such that $\Chthr(T_n)=C_{i(T_n)}\,$, only $\grandO( i(T_n) \card \M)$ operations are necessary, and usually we have $i(T_n) \ll i_{\max} \ll \card(\M)$.

Computing $\Chwindow$ as defined in Algorithm~\ref{algo.penmingal} might seem costly at first sight.
Appendix~\ref{app.algos.window} shows that given the path $(\mhgalzero(C))_{C \geq 0}\,$, 
of cardinality $i_{\max}+2$, 
computing $\Chwindow$ can be done with at most $\grandO(i_{\max} \log i_{\max} )$ operations.

Finally, step~3 of Algorithm~\ref{algo.penmingal} requires at most $\grandO(\card \M )$ operations.
Overall, Algorithm~\ref{algo.penmingal} always has a complexity $\grandO( i_{\max} \card \M ) \leq \grandO([\card \M ]^2)$.

\smallbreak

For Algorithm~\ref{algo.penmingal.slope}, step~1 is a (robust) linear regression ---hence it has a computational cost $\grandO(\card \M)$--- and step~2 can be done with $\grandO(\card \M )$ operations.
%%% .
%
Note that $\mhcapushe$ has a larger computational cost since it requires to run 
$\grandO(\card \M )$ times 
Algorithm~\ref{algo.penmingal.slope}, hence a total cost of $\grandO([\card \M ]^2)$.

\paragraph{When computing all empirical risks is not tractable} 
In general, most of the computational complexity of computing $\mhAlgE$ or $\mhAlgF$ corresponds to computing $\Remp\parens{\shm}$ for \emph{all} $\mM$.
For instance, for density estimation with Gaussian mixture models \citep{Mau_Mic:2010}, performing maximum-likelihood estimation in several large models involves a large computational cost,
while we know that all corresponding estimators are always bad. 
Can we remove from the collection $(\shm)_{\mM}$ most estimators with $\C_m$ ``large'', 
without degrading too much the performance of Algorithms~\ref{algo.penmingal}--\ref{algo.penmingal.slope}? 

For the jump approach, two estimators having a small approximation error are needed to get a jump, 
as with the assumptions of Theorem~\ref{thm.OLS}: 
one of large complexity, one much less complex. 
If we are not sure of which estimators have a small enough approximation error, 
considering more than two of them can be helpful; 
otherwise, this does not hurt ---and we conjecture that this slightly decreases the 
variance of $\Chjumpgal\,$---, without being mandatory. 

For the slope approach, the picture is different. 
Having only two estimators with a small approximation error implies making a linear regression over 
the corresponding two points, 
which is very close to the residual-based estimator $\sigh^2_{m_0}$ defined by Eq.~\eqref{def.sighm0}, 
as shown by Figure~\ref{fig.OLS.slope-vs-residuals.easy-hard} in Appendix~\ref{app.morefig}.
Therefore, $\Chslope$ with only a few large-complexity estimators faces the risk that some of them have a large approximation error, to which it will be quite sensitive, unlike $\Chjumpgal$ 
(see Figure~\ref{fig.OLS.slope-vs-residuals.easy-hard}b).
Using a robust regression in $\Chslope$ decreases the risk but does not exclude it totally, 
as shown by the poor results of $\mhcapushe$ in our experiments in the `hard' setting in Section~\ref{sec.practical.jump-vs-slope}.

A more reliable strategy for the slope approach ---at least for settings a bit less difficult than our `hard' setting--- 
is to consider only estimators 
of complexity up to $\C_{\max}\,$, and to carefully check that $\C_{\max}$ is large enough 
by visualizing the linear relation between the empirical risk and $\pen_0\,$.
This can be done easily with the \textsc{Capushe} package \citep{Bau_Mau_Mic:2010}. 
The experiments of \citet[Section~5]{Bau_Mau_Mic:2010} show as expected that $\Chslope$ is 
better ---more stable--- when $\C_{\max}$ is large enough. 
Similarly, for change-point detection, 
\citeauthor{Leb:2002} (\citeyear{Leb:2002}, Chapter~4; \citeyear{Leb:2005}, Section~4.2) 
studies the influence of such a bound $\C_{\max}$ on $\Chmaxjump$ and propose a heuristic method 
---called ``calibrated''--- for choosing $\C_{\max}$ from data. 

Note finally that in some frameworks, 
well-chosen large-complexity estimators are easy to compute. 
For instance, in fixed-design regression, the estimator equal to the original data always 
has an empirical risk and an approximation error equal to zero 
---see assumption~\eqref{hyp.thm.OLS.Id} in Theorem~\ref{thm.OLS}.

\subsection{Nested minimal-penalty algorithm} \label{sec.practical.nested}
In a framework where $\M$ is a cartesian product $\M_1 \times \M_2\,$, 
\citet{Dev_Gal_Per:2017} propose a 
``nested slope heuristics'' algorithm, that we here generalize to 
Algorithms~\ref{algo.penmingal}--\ref{algo.penmingal.slope}. 
The idea is to choose $\mh = (\mh_1,\mh_2) \in \M_1 \times \M_2$ in two steps. 
First, for every $m_1 \in \M_1\,$, select one estimator among $( \sh_{(m_1,m_2)} )_{m_2 \in \M_2}$ 
with a minimal-penalty algorithm; 
the selected index is denoted by $\mh_2(m_1)$. 
Then, select one estimator among $(\sh_{( m_1, \mh_2(m_1) )} )_{m_1 \in \M_1}$ 
with a minimal-penalty algorithm. 
The numerical experiments of \citet{Dev_Gal_Per:2017} on some transcriptomic data-analysis problem 
show that 
such a nested algorithm can work, for choosing a number $m_1$ of clusters (of individuals) 
and a partitioning $m_2$ of the features (the genes) used for inferring 
a cluster-dependent gene regulatory network. 

\subsection{Estimation of several unknown constants in the penalty}
\label{sec.practical.nappe}
%
%%% 1. algo "heuristique de nappe": definition et ref historique
%
When the optimal penalty involves several unknown constants, that is,
\begin{equation} \label{eq.practical.nappe}
\forall \mM \, , \qquad 
\penopt(m) = C^{\star}_1 \pen_1^{(1)}(m) + \cdots + C^{\star}_k \pen_1^{(k)}(m)
\end{equation}
for some known $\pen_1^{(1)}, \ldots, \pen_1^{(k)}$,
the slope approach can be generalized, 
using linear regression for estimating simultaneously $C^{\star}_1, \ldots, C^{\star}_k\,$.
The idea has first been proposed with Algorithm~\ref{algo.OLS.slope} 
by \citet[Section~4.3.2]{Leb:2002} 
in the case of change-point detection, 
where the optimal penalty depends on $k=2$ constants. 

%%% 2. references (empiriques) ou l'heuristique de nappe est utilisee 
%
It has since been used ---with good numerical performance--- in several settings: 
change-point detection \citep{Arl_Cel_Har:2012:kernelchpt}, 
joint variable selection and clustering via Gaussian mixture models \citep{Mey_Mau:2012}, 
principal curves estimation \citep{Bia_Fis:2011}, 
and 
unsupervised segmentation of spectral images 
via piecewise-constant Gaussian mixture models \citep{Coh_LeP:2014}.

%%% 3. problemes: 
%
Nevertheless, no theoretical guarantees are currently available for such an algorithm. 
In addition to the practical issues already mentioned for the slope approach, 
this procedure is difficult to apply when there is not a single natural complexity measure $\C_m$ 
but several of them ---$\pen_1^{(1)}(m), \ldots, \pen_1^{(k)}(m)$ can be $k$ complexity measures---, 
which have to be combined wisely for defining what are the ``complex enough'' $m \in \M$ 
over which the (robust) linear regression should be done. 
Another major difficulty is when $\card(\M)$ or $n$ are not large enough to allow 
a good estimation of several constants $C^{\star}_1, \ldots , C^{\star}_k$ simultaneously. 

\medbreak

%%% 4. pas necessaire dans certains cas (forme de penalite simplifiee)
%
%%% 4.1 Presenter l'idee
Another option is to make use of a simplified penalty shape 
---depending on a single multiplicative constant---, 
even when we know that it differs from the optimal shape given by Eq.~\eqref{eq.practical.nappe}. 
%
%%% 4.2 Plein de gens le font, et ca semble marcher
%
Several articles make use of such a simplified penalty, 
instead of trying to calibrate $k=2$ constants, 
with satisfactory numerical results: 
for density estimation / clustering with Gaussian mixture models (\citealp{Mau_Mic:2010}; see also \citealp[App.~C.2]{Mic:2008:phd}) 
or multinomial mixture models \citep{Der_LeP:2017}, 
for choosing a simplicial complex in the computational geometry field \citep{Cai_Mic:2009}, 
and for selecting jointly the rank and a set of variables 
in a high-dimensional finite mixture regression model \citep{Dev:2017:JMVA}.

%%% 4.3 Quand on compare forme simplifiee vs. plusieurs constantes, c'est parfois meilleur, parfois pareil, et parfois pire
%
A numerical comparison between simplified penalty shape and calibration of two constants 
is done in a few other papers, with various conclusions: 
favorable to the simplified shape 
\citep[Section~4.3.2, for change-point detection with Gaussian noise]{Leb:2002}, 
similar for both methods 
\citep[for inference of a high-dimensional Gaussian graphical model]{Dev_Gal:2016}, 
or favorable to the calibration of two constants 
---for change-point detection with Laplace noise and a simplified 
shape derived from experiments with Gaussian noise \citep[Section~4.6.4]{Leb:2002},  
for change-point detection with positive-definite kernels \citep{Arl_Cel_Har:2012:kernelchpt},  
and for curve clustering \citep[Figure~5]{Mey_Mau:2012}. 
%
%%% 4.4 Conclusion: que faire? 
% 
As a conclusion, choosing between these two strategies should be done carefully, 
depending on the framework.

\subsection{Variants for change-point detection} \label{sec.practical.variants-chpt}
For change-point detection seen as a model-selection problem, 
three approaches closely related to minimal penalties have been proposed, 
without being exactly of the form of Algorithm~\ref{algo.penmingal}.
The first is the ``calibrated method'' \citep[Section~4.2]{Leb:2005} mentioned in Section~\ref{sec.practical.cost}.

%% b. Marc Lavielle
%
Second, \citet[Section~2.3]{Lav:2005} remarks that $\Chmaxjump$ often leads to 
underestimating the number of changes.
Then, using the notation of Appendix~\ref{app.algos.path},
it is proposed instead to define $\mh$ as the largest $m_i = \mh(C_i)$
corresponding to a jump whose height $C_i - C_{i-1}$ is much larger than 
the one of the largest subsequent jump, that is, $\max_{j>i} \sets{ C_j - C_{j-1} }$.
This approach might be closer to an elbow heuristics (see Section~\ref{sec.related.elbow}) 
than to a minimal-penalty algorithm.

\paragraph{Statistical Base Jumping} 
%% c. Yves Rozenholc
%
Third, an unpublished idea by \citet{Roz:2012} is the following.
Assume that change-point detection is cast as a model-selection problem in fixed-design regression, 
so that we can use the notation of Section~\ref{sec.slopeOLS}.
Take some $D$ ``large'' but, say, smaller than $n/2$, for instance $D=n/\log(n)$ or $D=\sqrt{n}$.
Compute $\Fh_D$ the empirical risk minimizer over the set of piecewise-constant signals with $D$ pieces 
(which is a union of $\binom{n-1}{D-1}$ vector spaces of dimension~$D$). 
Consider the residual vector $\widetilde{Y} = Y - \Fh_D$ and apply the penalization approach to this pseudo-data, that is, compute
\begin{equation} \label{eq.penmin.Roz}
\forall C >0 \, , \quad
\widetilde{m}(C) \in \argmin_{m \in \M_n^{\mathrm{chpt}}} \set{ \frac{1}{n} \norm{\widetilde{Y} - \Pi_m \widetilde{Y}}^2 + C \pen_0(m) + \pen'_0(m) }
\end{equation}
where the model collection $\M_n^{\mathrm{chpt}}$ is the one adapted to change-point detection,  
and the penalty shape $C \pen_0(m) + \pen'_0(m)$ is a simplified version 
of the penalties proposed by \citet{Com_Roz:2004}, \citet{Lav:2005} and \citet{Leb:2005}, 
such as 
\[ 
C D_m \, , 
\qquad 
C D_m + \log \binom{n-1}{D_m-1} 
\qquad \text{or} \qquad 
C \log \binom{n-1}{D_m-1} + D_m 
\, . 
\]
Define $\ChRoz$ as the minimal value of $C>0$ such that $D_{\widetilde{m}(C)} = 1$, and finally select
\[
\mhRoz \in \argmin_{m \in \M_n^{\mathrm{chpt}}} \set{ \frac{1}{n} \norm{Y - \Pi_m Y}^2 + \frac{2 \ChRoz}{1 + \frac{D}{n}} \pen_0(m)  + \pen'_0(m) }
\, .
\]
Note that $\Fh_D\,$, $(\widetilde{m}(C))_{C \geq 0}$ and $\mhRoz$ can all be computed efficiently, in particular using dynamic programming.
The heuristics behind this method is that if $D$ is large enough to catch 
all true change-points of $F$ in $\Fh_D\,$, 
then $\widetilde{Y}$ does not contain any signal anymore, 
and $\ChRoz \pen_0$ is the minimal penalization level 
needed to recover with Eq.~\eqref{eq.penmin.Roz} the unique model of dimension one (constant signal).
The factor $1 + \frac{D}{n}$ dividing $\ChRoz$ corrects for the variance 
of the pseudo-sample~$\widetilde{Y}$.
Unpublished experiments \citep{Roz:2012} suggest that $\mhRoz$ provides very good segmentations, 
much better than with the original slope heuristics ---that is, 
Algorithm~\ref{algo.OLS.jump} or~\ref{algo.OLS.slope}, as done by \citet{Leb:2005} for instance.

\subsection{Other uses of minimal penalties} \label{sec.practical.other-uses}
Let us finish this section by mentioning two other uses of a minimal-penalty algorithm 
in the literature. 

\paragraph{Choice of a penalty shape} 
%%% 1. Choix entre plusieurs formes de penalites
%
Algorithm~\ref{algo.penmingal} can be used for choosing among several penalty shapes, 
by detecting bad ones, which are the ones that do not lead to a clear dimension jump, 
as illustrated by \citet[Section~3 of supplementary material]{Bau_Mau_Mic:2010} 
in the setting of \citet{Cai_Mic:2009}.

\paragraph{Minimal-penalty assisted experiments} 
%%% 2. Estimation d'une bonne constante de penalite dans des experiences numeriques
%
For estimating a good deterministic constant $\widetilde{C}^{\star}$ to be put in front of the penalty, 
\citet{Cha:2011:esaim} computes on 100 samples the constant $\Ch$ chosen by a slope heuristics algorithm, 
and defines $\widetilde{C}^{\star}$ as the maximal value of $\Ch$ observed over the 100 samples. 
The main interest of this approach is to require less computations than the standard one 
---which would be to compute, for every $C$ in a grid, 
the average over the 100 samples of the risk of the estimator selected with the penalty $C \pen_1\,$, 
and then to take $\widetilde{C}^{\star}$ that minimizes the average risks over $C$ in the grid---,  
even if the value $\widetilde{C}^{\star}$ might not be the optimal one.

%%%%%%%%%%%%%%%%%%%%%%%%%%%%%%%%%%%%%%%%%%%%%%%%%%%%%%%%%%%%%%%%%%%%%%%%%%%%%%%%%%%%
%%%%%%%%%%%%%%%%%%%%%%%%%%%%%%%%%%%%%%%%%%%%%%%%%%%%%%%%%%%%%%%%%%%%%%%%%%%%%%%%%%%%

\section{Conclusion, conjectures, and open problems} \label{sec.empirical}
As a conclusion of this survey, we sketch what is known theoretically for minimal-penalty algorithms, 
as well as several conjectures and open problems of high interest. 
Let us recall that Section~\ref{sec.hints} provides some hints for tackling many of these conjectures 
and open problems. 
\subsection{Settings and losses where minimal-penalty algorithms apply}
Let us sketch the set of frameworks for which minimal-penalty algorithms
are theoretically justified, at least partially
(see Section~\ref{sec.theory}).

The typical situation is a polynomial collection of
minimum-contrast estimators
(also called empirical risk minimizers)
with a regular contrast \citep{Sau:2010:phd}
(for instance, the least-squares contrast), 
the corresponding risk (expected value of the contrast), 
and (almost) i.i.d.\@ data.
Then, all obtained results show that $\penopt \approx 2 \penmin\,$.
In other terms, the slope heuristics holds true in several frameworks ``close to'' 
choosing among a polynomial collection of projection estimators 
with the least-squares risk and i.i.d.\@ data.

The general result proved by \citet[Chapters~7--8]{Sau:2010:phd}
for regular contrasts suggests that the slope heuristics
is probably valid in all frameworks that are close enough
to this ideal situation ---e.g., least-squares estimators in 
regression or (conditional) density estimation with the least-squares risk---, under appropriate assumptions. 

\medbreak

Two results, in regression \citep{Arl_Bac:2009:minikernel_nips} 
and in density estimation \citep{Ler_Mag_Rey:2016}, 
show that linear estimators can be considered
instead of minimum-contrast estimators,
at the price of changing the slope heuristics (Section~\ref{sec.slopeOLS})
into a minimal-penalty heuristics (Section~\ref{sec.penmingal}).
A similar extension can probably be done in other settings for estimators 
that are (close to) linear functions of (part of) the data.

More generally, using the notation introduced in Section~\ref{sec.theory.approach}, 
it is probably often true that $\E[p_2(m)]$ is a minimal penalty 
and $\E[p_1(m) + p_2(m)]$ an optimal penalty.
But unless these expectations are (approximately) known up to a multiplicative constant,
applying such results requires to estimate $\E[p_2(m)]$ and $\E[p_1(m)+p_2(m)]$ 
by resampling (see Remark~\ref{rk.reech} in Section~\ref{sec.theory.complete}),
and we then loose a nice feature of Algorithm~\ref{algo.penmingal} 
which is its small computational cost compared to cross-validation.

\subsection{Unavoidable assumptions}
Even in settings for which a full proof of a minimal-penalty algorithm is known,
a natural question to ask is which assumptions are unavoidable
for this algorithm to work.

Based upon existing proofs ---in particular the one of
Theorem~\ref{thm.OLS}, which is typical---,
we conjecture that at least three assumptions are (almost) needed.

First, a ``complex'' estimator should be present in the collection, similarly to \eqref{hyp.thm.OLS.Id}.
Note that such an estimator can often be added on purpose
to a predefined collection.

Second, one ``less complex'' but ``good'' estimator should also be present in the collection, 
similarly to what Theorem~\ref{thm.OLS} assumes implicitly. 
The exact definition of ``good'' can depend on the context. 
In general, being consistent should suffices; 
note that assuming that the oracle estimator is consistent is a mild assumption for 
estimator selection, since otherwise the problem is not much interesting. 
In the setting of Theorem~\ref{thm.OLS}, it suffices to have a model with a small 
approximation error, even if the corresponding estimator is not consistent. 
Note however that such an assumption can be violated in practice, 
for instance when the estimator collection has not been well chosen, 
so that the approximation error never vanishes. 

Third, it seems reasonable to make some mild moment assumption on the data
so that the key quantities $p_1$ and $p_2$ concentrate around their expectations, 
at least when using deterministic penalty shapes.
Nevertheless, a Gaussian assumption such as \eqref{hyp.thm.OLS.Gauss}
is not necessary  \citep{Arl_Mas:2009:pente,Sau:2010:Reg}; see also Remark~\ref{rk.thm.OLS.subgaussian}  in Section~\ref{sec.slopeOLS.math}.
Independence of data is not necessary either \citep{Ler:2010:mixing,Gar_Ler:2011}.
Risk bounds could be obtained under much weaker moment assumptions 
---for instance, when the noise only has a finite moment of order two---, 
for empirical risk minimizers \citep{Men:2018} 
or for robust estimators \citep[for instance]{Aud_Cat:2011}. 
Nevertheless, we are not aware of any theoretical result on minimal penalties in such a setting. 

\subsection{Other settings, losses, estimators}
\label{sec.empirical.conjectures}
Numerical experiments show that minimal-penalty algorithms can be used fruitfully 
in many other settings such as 
supervised classification \citep{Zwa:2005:phd}, 
model-based clustering \citep{Mau_Mic:2010,Bau:2012:ejs}, 
high-dimensional inference \citep{Dev_Gal:2016}, 
change-point detection \citep{Leb:2005,Bar_Ken_Win:2010}, 
topological data analysis \citep{Cai_Mic:2009}, 
functional linear models \citep{Roc:2014} 
or Hawkes-process intensity estimation \citep{Rey_Sch:2010}, 
with applications in various domains such as 
biology ---genomics \citep{Aka:2011,Rey_Sch:2010}, 
transcriptomics \citep{Rau_Mau_Mag_Cel:2015,Dev_Gal:2016}, 
quantitative trait prediction from genomic data \citep{Dev_Gal_Per:2017}, 
population genetics \citep{Bon_Tou:2010}---, 
energy ---electricity consumption prediction \citep{Dev_Gou_Pog:2015:journal}, 
oil production modelization \citep[Chapter~6]{Mic:2008:phd}---, 
hyperspectral image segmentation \citep{Coh_LeP:2014}, 
text analysis \citep{Der_LeP:2017},  
and bike sharing systems \citep{Bou_Com_Jac:2015,God_Mau_Rau:2018}. 
Combining these numerical experiments 
---especially the ones showing that the empirical risk is indeed close to a linear function of 
some known $\pen_0$ for large-complexity estimators--- 
with the partial theoretical results available (Section~\ref{sec.theory}), 
several settings can be identified where we conjecture that
a minimal-penalty algorithm such as Algorithms~\ref{algo.penmingal}--\ref{algo.penmingal.slope} 
could be used fruitfully. 
Among them, we select below the most challenging ones,
in terms of both practical applications and theoretical interest.

\subsubsection{Supervised classification} 
\label{sec.empirical.conjectures.classif-superv} 
A classical setting where minimal-penalty algorithms would be quite useful is 
supervised classification with the 0--1 loss and corresponding empirical 
risk minimizers. 
No theoretical result is available up to now, 
and it seems tough to prove any because the 0--1 contrast is far from being regular. 
Nevertheless, \citet{Bou_Mas:2004} provide a key ingredient of the proof, 
that is, a concentration inequality for $p_2(m)$ that applies easily to 
0--1 classification, even when fast learning rates are possible. 
Given Proposition~\ref{pro.pbpenmin.gal.below} and 
the general strategy detailed in Section~\ref{sec.hints}, 
what remains is to prove a similar concentration inequality for $p_1(m)$, 
and to be able to estimate (up to the same unknown constant) 
$\E[p_1(m)]$ and $\E[p_2(m)]$.  

A probably easier open problem is to provide theory for the case of 
classification with a convex loss, 
such as the logistic loss ---at the basis of logistic regression--- 
or the hinge loss ---at the basis of support vector machines. 
At least, for the hinge loss, the numerical experiments reported by  
\citet[Section~6.4.3]{Zwa:2005:phd} suggest 
that minimal-penalty algorithms can work with 
$\pen_0(m) = \C_m = D_m$ 
and $\pen_1(m) = 2 D_m\,$. 

\subsubsection{Model-based clustering, choice of the number of clusters} 
\label{sec.empirical.conjectures.clustering} 
Estimating the number of clusters for (unsupervised) clustering 
is another problem where fine tuning of penalties is a major challenge. 
A classical approach ---called model-based clustering--- 
is to estimate the data density by maximum-likelihood on a mixture model, 
and to define clusters by a maximum a posteriori rule. 
Then, the number of clusters can be chosen by maximizing the penalized log-likelihood. 
Minimal-penalty algorithms with $\pen_0(m) = \C_m = D_m$ 
and $\pen_1(m) = 2 D_m$ are shown 
successful by experiments on synthetic and real data, 
for various problems following this strategy (up to modifications that are specified below): 
\begin{itemize}
%%% Experiences dans le cas 'melange simple' 
%
\item clustering with 
Gaussian \citep{Bau:2012:ejs}, 
Poisson \citep{Rau_Mau_Mag_Cel:2015}, 
multinomial \citep{Der_LeP:2017}, 
or some functional \citep{Bou_Com_Jac:2015}
mixture models. 
%
%%% Variante dans le cas de melanges gaussiens: 
%%% ICL au lieu de BIC (Baudry)
%
\item clustering with Gaussian mixtures 
and the \emph{conditional} log-likelihood 
instead of the log-likelihood, 
which leads to slightly different kinds of clusters 
\citep{Bau:2012:ejs}. 
%
%%%%%% Selection de variables
%
\item joint clustering and variable selection ---that is, identifying which features are relevant for clustering--- 
with mixtures of 
Gaussian \citep{Mau_Mic:2010} 
or multinomial \citep{Bon_Tou:2010} 
variables. 
\item efficient joint clustering and high-dimensional variable selection 
with maximum-likelihood estimators trained on \emph{data-driven models} 
obtained by a first step of $L^1$ penalization \citep{Mey_Mau:2012}; 
here, the shape of the minimal penalty seems to be close to a linear combination of 
$D_m$ and $D_m \log \frac{p}{D_m}$ \citep[Figure~6]{Mey_Mau:2012}, 
and the algorithm described in Section~\ref{sec.practical.nappe} can be used. 
\item when an additional feature vector is provided for each observation, 
clustering a mixture of Gaussian regression models, jointly done with 
feature selection \citep{Dev:2017,Dev_Gou_Pog:2015:journal} 
or partitioning \citep{Dev_Gal_Per:2017}, 
via two-steps procedures similar to the one of \citet{Mey_Mau:2012}. 
\end{itemize}

%%% Conjectures que je peux formuler pour ce cadre 
%
We conjecture that minimal-penalty algorithms indeed work in these settings, 
that is, as one can observe on synthetic or real data: 
(i) $p_2(m)$ is (close to) a linear function of the number of parameters $D_m\,$, 
(ii) Algorithm \ref{algo.penmingal} or~\ref{algo.penmingal.slope} with $\pen_0(m)=D_m$ 
and $\pen_1(m) = \alpha D_m$ provides an estimator with a small 
Kullback-Leibler risk, for some $\alpha>1$ to be determined, 
and (iii) the number of clusters selected by this algorithm is equal to the true one $K^{\star}$ with large probability 
when $n$ is large and the data distribution is close to a mixture with $K^{\star}$ components. 
Note that (ii) is a density estimation guarantee 
---hence, slightly different from clustering, but classical for justifying theoretically 
a penalty shape \citep{Mau_Mic:2008,Mey_Mau:2012,Bon_Tou:2010,Der_LeP:2017,Dev:2017:JMVA}--- 
and that (ii) and (iii) may require different values of $\alpha$ 
since estimation and model identification are different goals for model selection, 
see Section~\ref{sec.empirical.conjectures.identif}. 
When variable selection is done jointly with clustering, 
(iii) can be completed by the fact that the true set of relevant variables 
is selected with large probability. 

Up to now, only oracle inequalities with theoretical penalties are available in 
some of these settings. 
Of course, proving the above conjecture would be less difficult for pure model-based clustering 
\citep{Bau:2012:ejs,Bou_Com_Jac:2015,Rau_Mau_Mag_Cel:2015,Der_LeP:2017} 
than for the two-steps algorithm using $L^1$-penalized maximum-likelihood for defining 
data-driven models 
\citep{Mey_Mau:2012,Dev:2017,Dev_Gou_Pog:2015:journal}. 

Note also that variable selection, without any order among variables, 
means that the model collection considered is large 
---with the terminology of Section~\ref{sec.theory.rich}---, 
at least implicitly; 
this fact raises specific issues that are addressed in Section~\ref{sec.empirical.conjectures.rich}. 

\subsubsection{High-dimensional statistics} 
\label{sec.empirical.conjectures.high-dim} 
Hyperparameter tuning is a major issue for high-dimensional statistics \citep[Chapter~5]{Gir:2014},  
which can often be addressed by penalization. 
Despite several positive numerical results, 
providing a full theoretical proof of a minimal-penalty algorithm in this context remains an open problem. 

\paragraph{Clustering} 
Numerical results about minimal-penalty algorithms for joint (model-based) clustering and variable selection 
are reviewed in Section~\ref{sec.empirical.conjectures.clustering}. 
For block-diagonal estimation of the covariance matrix of a high-dimensional Gaussian vector 
(graphical model), \citet{Dev_Gal:2016} provide positive numerical results for a similar algorithm 
---maximum likelihood on data-driven models obtained by thresholding the empirical covariance matrix. 
Nevertheless, given the difficulty of proving an oracle inequality 
\citep[for instance]{Dev_Gal:2016}, 
it seems hard to obtain a full theoretical validation of the minimal-penalty algorithms of 
\citet{Mey_Mau:2012}, \citet{Dev:2017} and \citet{Dev_Gal:2016}. 

\medbreak

\paragraph{Regression: the Lasso and related algorithms} 
One of the most classical high-dimensional statistics problem is variable selection in linear regression, 
for which the (group) Lasso and related algorithms are popular. 
Penalized least-squares can be used for choosing their parameters, 
thanks to covariance penalties \citep{Efr:2004}, 
which have a simple expression of the form $\sigma^2 \mathrm{df} / n$ when the noise is Gaussian, 
where $\sigma^2$ denotes the residual noise-level 
and $\mathrm{df}$ are the degrees of freedom. 
Easy-to-compute estimators of $\mathrm{df}$ exist for 
the Lasso \citep{Tib_Tay:2011b:journal,Dos_etal:2011:journal} 
and group Lasso \citep{Vai_etal:2012}, among others. 
The remaining issue is to estimate $\sigma^2$  
without knowing any small correct model, 
for which minimal penalties are a natural approach 
(see Section~\ref{sec.related.variance}). 

%%% Lasso: experiences (mi figue mi raisin) 
%
The PhD dissertation of 
\citet{Con:2011:phd} provides an extensive numerical study of minimal penalties 
for calibrating either the Lasso or least-squares estimators trained on models selected by the Lasso 
(`Lasso+LS'). 
In short, the major difficulty is that the natural candidate for $\pen_0(m)$, 
that is $\E[p_2(m)]/\sigma^2$, cannot be used because it depends on the (unknown) signal $\beta^{\star}$. 
Simplified penalty shapes ---that is, $\E[p_2(m)]/\sigma^2$ for a zero signal, or for a zero signal and 
an identity design matrix--- often work for Lasso+LS, and sometimes for the Lasso, 
depending on the signal-to-noise ratio and on the sparsity of the signal. 
This is not satisfactory because these minimal-penalty algorithms sometimes fail for the Lasso or Lasso+LS, 
and \citet{Con:2011:phd} proposes ``antidotes'' for detecting the failure 
but nothing for correcting it, except using cross-validation. 

%%% Lasso: cas design orthogonal (conjectures)
%
It nevertheless seems possible to solve the case of an orthogonal design matrix,
when the Lasso is soft thresholding 
and Lasso+LS is hard thresholding. 
Taking the number of selected variables $k$ as tuning parameter, 
\citet[Section~2]{Lou_Mas:2004} conjecture that for soft thresholding, 
$\E[p_2(k)] \approx \frac{3}{2} \frac{\sigma^2 k }{n}$.  
If one could prove this conjecture, a minimal-penalty algorithm could be used for estimating $\sigma^2$ 
---hence for calibrating soft (or hard) thresholding. 
Section~\ref{sec.empirical.conjectures.rich} discusses the case of hard thresholding. 

%%% Lasso: cas general (open pb / conjectures) 
%
For a general design matrix and for other estimators, 
we think that the key question is to find the good parametrization of the 
estimator to be calibrated. 
For instance, the Lasso can be parametrized by the regularization parameter 
or by the number of selected variables, 
and \citet{Con:2011:phd} shows that the theoretical minimal penalty $\E[p_2(m)]$ 
and the performance of minimal-penalty algorithms 
strongly depend on the chosen parametrization. 
We also conjecture that the solution might not come from a direct application of 
Algorithms~\ref{algo.penmingal}--\ref{algo.penmingal.slope}, 
but by the more general approach of identifying an observable phase transition 
---with respect to some well-chosen parameter--- 
that provides the key information for an optimal calibration 
of the algorithm considered ---for instance, an estimation of $\sigma^2$ for the (group) Lasso. 
This idea has already been proposed in a few settings that are detailed in Section~\ref{sec.empirical.outside}. 
For high-dimensional regression, theoretical results prove the existence of 
phase transitions in the risk of the Lasso (\citealp[Section~4]{Bel:2017}; \citealp[Section~4.3]{Bel:2018:v4}) 
and of the constrained formulation of the Lasso \citep[Section~2.1]{Cha:2014}. 
Nevertheless, it is not clear whether these phase transitions are observable, 
so they might not be useful for choosing hyperparameters. 
The approach of Section~\ref{sec.related.thresh-null} seems to be another promising direction 
for tackling this problem.

Let us finally mention that concentration inequalities for $p_1(m)$ are available  
for the Lasso and some related algorithms ---see Section~\ref{sec.theory.hint.penopt} for details. 
They can be useful for validating minimal-penalty algorithms.

\subsubsection{Large collection of models}
\label{sec.empirical.conjectures.rich} 
As recalled in Section~\ref{sec.theory.rich}, the nature of the model-selection problem 
depends on the size of the model collection. 
All full proofs and almost all partial results available for minimal-penalty algorithms are for small 
collections, for which optimal model selection can be obtained with the unbiased risk estimation heuristics. 
For large collections, only a few partial theoretical results are available, 
as reported in Section~\ref{sec.theory.rich}. 
Therefore, the case of large collections remains a widely open problem of major interest.

%%% Variable selection en general, 2 exemples simples en particulier 
% 
The most classical situation is variable selection among $p \gtrsim n$ variables, 
which amounts to select among a collection of $2^p$ models. 
Let us start by focusing on the two settings where minimal-penalty algorithms are best understood: 
(i) variable selection with $p=n$ and an orthonormal design 
---so that penalizing the least-squares criterion by a function of the number of variables 
is equivalent to hard thresholding---, and 
(ii) change-point detection ---finding the locations of abrupt changes 
in the distribution of a sequence of $n$ observations--- 
which can be casted as a variable-selection problem with $p=n-1$ variables 
---the possible breakpoint locations--- and solved by penalized least-squares. 

Let us also recall that the notation 
(\pbpenmin), \eqref{pb.penmin.Cpt-Dgrd}, \eqref{pb.penmin.Cgrd-Dpt}, 
(\pbpenminalt), \eqref{pb.penopt}, and \eqref{pb.penopt.weak} 
refer to partial results about minimal-penalty algorithms; 
they are defined in Sections~\ref{sec.theory.approach}--\ref{sec.theory.pbpenminalt}. 

\medbreak

\paragraph{Orthonormal variable selection by hard thresholding} 
%
%%% 1er exemple "simple": hard thresholding / orthonormal variable selection 
%
For variable selection with an orthonormal design and Gaussian noise, 
\citet[Proposition~2]{Bir_Mas:2006} prove (\pbpenminalt) and \eqref{pb.penopt.weak} 
with a minimal penalty of order 
$2 \sigma^2 \frac{D_m}{n} \log \frac{n}{D_m}$ ---at least for $1 \ll D_m \ll n$. 
We conjecture that (\pbpenmin) holds true in the same setting: 
a proof of \eqref{pb.penmin.Cpt-Dgrd} derives from the proof written by \citet[Proposition~2]{Bir_Mas:2006}, 
and \eqref{pb.penmin.Cgrd-Dpt} seems easy to get given the results 
obtained by \citet{Bir_Mas:2006}. 
Then, (\pbpenmin) and \eqref{pb.penopt.weak} would prove that 
Algorithm~\ref{algo.penmingal}, 
with $\pen_0(m) \approx D_m \log \frac{n}{D_m}$ 
and $\pen_1(m)$ given by \citet[Section~3.1.3]{Bir_Mas:2006},  
provides a good data-driven variable-selection procedure, 
satisfying an oracle inequality close to being optimal. 

%%% Challenge le plus dur pour ce cadre
% 
The main remaining challenge for having a full proof of a first-order optimal procedure 
is problem \eqref{pb.penopt}: 
find a first-order optimal penalty of the form $\sigma^2 \pen_1(m)$ with $\pen_1(m)$ known. 
We think that this is a hard problem, 
whose resolution would have a great impact on model-selection theory in general, 
since even the \emph{value} of the optimal 
excess risk of such a variable-selection procedure is not exactly known at first order. 
We only know by minimax arguments that it should be of order 
$\log \frac{n}{D_{m^{\star}}}$ times the oracle excess risk, up to a constant factor. 

%%% Autre pb ouvert pour hard thresholding
% 
Another open problem for orthonormal variable selection 
is to determine the minimal penalty for non-Gaussian noise. 
Contrary to small model collections, this is not a straightforward extension 
of Gaussian results since the experiments of \citet[Section~5]{Leb:2005} for 
another large collection problem ---change-point detection--- 
suggest that the minimal penalty then is different for Laplace and for Gaussian noise. 

\medbreak

\paragraph{Change-point detection} 
Penalized least-squares is a classical approach to change-point detection, 
for which an oracle inequality \eqref{pb.penopt.weak} holds true for Gaussian noise  
and a penalty of the form $\sigma^2 \frac{D_m}{n} \bigl[ c_1 \log \frac{n}{D_m} + c_2 \bigr]$ 
where $c_1,c_2 >0$ are two numerical constants and $\sigma^2$ is the residual noise-level \citep{Leb:2005}; 
a similar result in a slightly different setting is proved by \citet{Arl_Cel_Har:2012:kernelchpt}. 
The numerical experiments of \citet{Leb:2005}, \citet{Sor:2017}, 
\citet[Section~4.1]{Gar:2017:phd},  
\citet[Figure~4, for instance]{Cab_etal:2018} and  \citet{Arl_Cel_Har:2012:kernelchpt}, 
as well as the partial theoretical results of \citet{Sor:2017} ---see Section~\ref{sec.theory.rich}---, suggest 
that a minimal-penalty algorithm 
with $\C_m=D_m$ and $\pen_0(m)$ proportional to a linear combination of $D_m$ and $D_m \log \frac{n}{D_m}$
(or close to it, according to \citealp{Bir_Mas:2006}, and \citealp{Sor:2017}) should work well in this setting. 
Given the numerical experiments of \citet{Leb:2005} and \citet{Arl_Cel_Har:2012:kernelchpt}, 
we conjecture that the ratio between the optimal and minimal penalty belongs to $(1,2]$;  
it may be model-dependent. 

Proving these conjectures formally would require to solve two open problems: 
(a) prove the existence of a dimension jump for some known $\pen_0$ 
---or, equivalently, prove that step~2 of Algorithm~\ref{algo.penmingal.slope},
or its generalization of Section~\ref{sec.practical.nappe}, works well---, 
and (b) prove an optimal oracle inequality for a penalty $C^{\star} \pen_1$ 
that can be derived from the minimal penalty. 
Problem (b) is very hard, as for orthonormal variable selection. 
Problem (a) seems less difficult: 
\citet[Chapter~8]{Sor:2017} is close to proving it, but there is still a gap between 
``large enough'' and ``too small'' penalties, which leaves open the possibility of having 
no clear dimension jump. 
It remains a challenge, as emphasized by the fact that the shape of the minimal penalty seems to depend 
on the noise distribution not only through its variance \citep[Section~5]{Leb:2005}. 
Note that proving (a) with $C^{\star} = \sigma^2$ would be sufficient to get a 
good data-driven penalty for change-point detection, 
since an oracle inequality ---maybe suboptimal--- is already available 
for a penalty depending on $\sigma^2$ and known quantities. 

%% Autre approche
%
To conclude on change-point detection via penalized least-squares, let us recall that 
\citet{Roz:2012} proposes a related but different approach to penalty tuning 
---see Section~\ref{sec.practical.variants-chpt}--- that might be even more efficient than 
minimal-penalty algorithms for change-point detection with penalized least-squares. 
Justifying it theoretically would therefore be of great interest.

%%% Change-point beyond least-squares (rapide)
%
Beyond penalized least-squares, slope-heuristics algorithms empirically work well 
for change-point detection with dependent data in two settings: 
causal processes with the maximizer of a penalized log-likelihood \citep{Bar_Ken_Win:2010},   
and long-memory processes with the minimizer of a local Whittle contrast \citep{Bar_Gue:2018:preprint}. 
\citet{Bar_Ken_Win:2010} even show the remarkable fact that minimal penalties numerically 
adapt to variations of the optimal constant~$C^{\star}$ ---which can be of order $\log(n)$ or $\sqrt{n}$--- 
when the dependence structure varies. 
A theoretical validation of these results seems quite a challenge, 
since handling small model collections 
in the same settings already is an open problem. 

\medbreak

\paragraph{General setting} 
Understanding minimal-penalty algorithms for more general variable-selection problems seems  
a too high theoretical challenge for the next few years. 
We nevertheless conjecture that minimal-penalty algorithms work well beyond 
orthonormal variable selection and change-point detection, 
given the successful experiments of \citet{Mau_Mic:2010} and \citet{Bon_Tou:2010} 
for joint variable selection and model-based clustering  
(see Section~\ref{sec.empirical.conjectures.clustering}). 

%%% Rich collections implicit in high-dimension 
% 
Minimal-penalty algorithms experimentally work well in several other settings mentioned previously, 
where large model collections are implicitly considered 
through $L^1$ penalization or thresholding: 
model-based clustering \citep{Mey_Mau:2012}, %%with $L^1$ penalization 
multivariate regression with a mixture of linear models \citep{Dev:2017,Dev_Gal_Per:2017},  
and Gaussian graphical model estimation \citep{Dev_Gal:2016}. 

\subsubsection{Infinite estimator collections}
Throughout the article, the estimator collection $(\shm)_{\mM}$ is assumed to be finite. 
Nevertheless, Algorithms~\ref{algo.penmingal}--\ref{algo.penmingal.slope} 
can still be used (at least theoretically, due to computational issues) for some infinite collections 
that behave as small (finite) collections, with the terminology of Sections 
\ref{sec.theory.rich} and~\ref{sec.empirical.conjectures.rich}. 

Indeed, \citet{Arl_Bac:2009:minikernel_long_v2} prove that 
Algorithm~\ref{algo.penmingal} can be used for selecting a tuning parameter 
within a \emph{continuous set}~$\M$. 
The proof of \citet{Arl_Bac:2009:minikernel_long_v2} for kernel ridge regression 
mostly relies on two facts: 
(i) Algorithm~\ref{algo.penmingal} would work for a collection of $L_1 n^{L_2}$ 
such estimators for any fixed $L_1, L_2>0$, 
(ii) the collection of kernel ridge regressors can be well approached by a finite 
collection of $L_1 n^{L_2}$ estimators for some fixed $L_1, L_2>0$. 
We conjecture that a similar approach can be used for proving that 
Algorithm~\ref{algo.penmingal} works with some other continuous collections, 
starting by multiple-kernel ridge regression with a fixed number of kernels.

\subsubsection{Model selection for identification of the true model}
\label{sec.empirical.conjectures.identif} 
Throughout this survey, we assume that the goal is to choose a data-driven $\mh \in \M$ such that 
the risk of $\sh_{\mh}$ is minimal, that is, satisfies a non-asymptotic oracle inequality \eqref{eq.oracle}.
Model selection can target a different goal, which is to identify the smallest true model $\mo \in \M$ 
with probability one asymptotically, assuming that some true model exists; 
a procedure $\mh$ achieving this goal is said ``model-consistent''. 
Then, the exact same procedure cannot achieve the two goals in general \citep{Yan:2005a}. 
Can minimal-penalty algorithms still be useful for identification? 
Some experiments and theoretical arguments suggest a positive answer.

\paragraph{Role of the size of the model collection} 
%
%%% Cas 1: rich collections => both goals simultaneously 
%
For large collections ---see Section~\ref{sec.empirical.conjectures.rich}---, 
it turns out that both estimation and identification require to overpenalize 
compared to the unbiased risk estimation principle. 
Therefore, the minimal-penalty algorithms suggested in Section~\ref{sec.empirical.conjectures.rich} 
should also work for identification. 
This conjecture is supported by the experiments of 
\citet[Figure~2]{Bon_Tou:2010} about variable selection in multinomial mixture models, 
and by those of \citet[Figure~2]{Arl_Cel_Har:2012:kernelchpt} and \citet[Figure~5--6]{Gar_Arl:2015} 
about change-point detection.

%%% Cas 2: small collections => adaptations simples
% 
For small collections, the picture is different. 
A typical example is least-squares fixed-design regression 
with projection estimators, as in Section~\ref{sec.slopeOLS}.  
Let us focus on this setting here for simplicity. 
The slope heuristics then leads to a model-selection procedure 
equivalent to $C_p\,$, which is first-order optimal for estimation  
but inconsistent for identification \citep[Theorem~1]{Sha:1997}. 
A~simple way to fix this failure is to replace the factor $2$ 
between minimal and optimal penalties in the slope heuristics 
by, say, a $\log(n)$ factor, in order to get a BIC-type penalty, 
hence consistent for identification \citep[Theorem~2]{Sha:1997}. 
Such a correction of the slope heuristics may seem unsatisfactory, 
so one may consider to combine it with a procedure choosing from data 
between AIC and BIC-type penalties \citep{Yan:2005a,Erv_Gru_Roo:2008}. 

\paragraph{Change-point detection}
%
%%% Cas 2bis: change-point 
%
For change-point detection with well-chosen small collections of models, 
\citet{Gey_Leb:08}, \citet{Dur_Leb_Toc:2009} and \citet{Aka:2011} propose specific hybrid procedures 
for identification of the change-point locations. 
In short, they consist of two-steps procedures, with minimal-penalty algorithms in both steps. 
The first step selects within a small model collection, providing an oversegmentation of the data sequence. 
The second step removes the unnecessary change-points. 
Proving the consistency of these procedures ---including the minimal-penalty algorithms--- 
is an open problem. 
Another natural question is to generalize such two-steps procedures to other model-selection 
problems with an identification goal.

\paragraph{Minimal penalties for consistent identification} 
Let us finally mention some theoretical results about the minimal level of penalization 
needed for model consistency, that is, for having $\mh = \mo$ a.s. asymptotically.
Even without a corresponding calibration algorithm ---since the minimal penalty is not observable here---, 
this question remains of interest for theory.

In least-squares regression, \citet[Theorem~1 (iii)]{Sha:1997} shows that
$C_p$ is not model consistent, assuming only that some true model $m^{\prime} \in \M$ exists with $D_{\mo} < D_{m^{\prime}} \leq D_{\mo} + \kappa$ for some fixed $\kappa>0$.
Such a result actually holds for any penalty of the form $C D_m / n$ with $C \geq 0$ fixed as $n$ grows, which can be proved from arguments used in the proof of Theorem~\ref{thm.OLS}.
Conversely, using a penalty of the form $\lambda_n D_m / n$ with $\lambda_n \to +\infty$ and $\lambda_n / n \to 0$ as $n$ tends to infinity provides a model-consistent procedure \citep[Theorem~2]{Sha:1997}.
Therefore, the minimal level of penalization for identification is of the form $\lambda_n D_m / n$ with $\lambda_n \to + \infty$.

For maximum-likelihood estimators, at least two results are available.
BIC-type penalties are minimal for estimating the order of a Markov chain 
without any prior upper bound on its order \citep{vHa:2011}. 
For density estimation with i.i.d.\@ data, 
identifying the true model among a nested family 
by minimizing the log-likelihood penalized by $\pen(m) = f(m) g(n)$ 
requires that $g(n) > C^{\star}(\bayes) \log \log n$ for some constant $C^{\star}(\bayes)>0$ \citep{Gas_vHa:2013}.

\subsubsection{Miscellaneous}
\paragraph{Model selection} 
Numerical experiments suggest that minimal-penalty algorithms work well 
for several other model-selection problems. 
We list them below, in order to 
help identifying settings where new theoretical results could be proved: 
\begin{enumerate}
\item \emph{Heteroscedastic regression} when the residual variance is known up to a constant 
---which can occur for inverse problems---, with least-squares risk and estimators 
\citep[Section~2.6.2]{Vil:2007:these}, 
beyond regressograms and strongly localized bases for which theoretical results are 
already known for a random design \citep{Arl_Mas:2009:pente,Sau_Nav:2017}. 
The fixed-design case can be handled similarly to the results of \citet{Arl_Bac:2009:minikernel_long_v2}.  
The random-design case with general models is clearly more challenging.

\item Estimation of two kinds of \emph{geometrical objects}, 
with least-squares risk and estimators: 
simplicial complices \citep{Cai_Mic:2009}   
and principal curves \citep{Bia_Fis:2011}.

\item \emph{Least-squares risk and estimators} for 
clustering of compositional data \citep{God_Mau_Rau:2018},  
Hawkes-process intensity estimation \citep{Rey_Sch:2010}, 
and in a functional linear model \citep[Section~2.4]{Roc:2014}.

\item \emph{Maximum-likelihood estimators with the Kullback-Leibler loss} for 
semiparametric regression with censored data via the Cox model \citep{Let:2000}, 
point-process intensity estimation \citep[Section~6.3 and Appendix~D.1.2]{Mic:2008:phd}, 
regression for counting processes under a proportional-hazard assumption \citep{Oue_Lop:2013},  
and segmentation of spectral images via spatialized Gaussian mixtures \citep{Coh_LeP:2014}. 

\end{enumerate}

\paragraph{Estimator selection} 
We emphasize in this survey that minimal-penalty algorithms can be useful for estimator 
selection in general. 
Beyond the few theoretical results pointed out in Section~\ref{sec.theory.complete}, 
we conjecture that Algorithms~\ref{algo.penmingal}--\ref{algo.penmingal.slope} 
work for several estimator-selection problems. 

First, numerical experiments suggest that they can be used for selecting among 
maximum-likelihood or least-squares estimators trained on 
\emph{data-driven models}, 
obtained by $L^1$ penalization in a variable-selection setting 
\citep{Mey_Mau:2012,Dev:2017}, 
by thresholding the empirical covariance matrix 
\citep{Dev_Gal_Per:2017,Dev_Gal:2016}, 
by $k$-means \citep{Cai_Mic:2009}, 
or by (kernel) PCA in classification \citep[Section~6.4.3]{Zwa:2005:phd}
or functional data analysis \citep[Section~2.4]{Roc:2014}. 
Some of these results are detailed above in Sections~\ref{sec.empirical.conjectures.clustering} 
and~\ref{sec.empirical.conjectures.high-dim}. 

Second, minimal-penalty algorithms can be used successfully for the pruning step of 
CART in regression \citep{Gey_Leb:08} 
and of a spatial variant of CART \citep{Bar_Gey_Pog:2018}, 
according to numerical experiments. 

Finally, Algorithms~\ref{algo.penmingal}--\ref{algo.penmingal.slope} 
certainly cannot succeed for 
any kind of estimator collection. 
Section~\ref{sec.empirical.outside} describes 
a natural and promising way to generalize the minimal-penalty approach 
beyond Algorithms~\ref{algo.penmingal}--\ref{algo.penmingal.slope}.

\subsection{Overpenalization} \label{sec.empirical.overpenalization}
It is known empirically that a better model-selection performance
can be obtained by overpenalizing a bit: 
the penalty $C \penoptgal(m) = C \E\crochb{ \Risk\parens{\shm} - \Remp\parens{\shm} }$ 
%% ---defined by Eq.~\eqref{eq.penopt.gal} in Section~\ref{sec.penmingal.penoptmin}--- 
%
has optimal performance when $C$ is slightly above~$1$,
as shown for instance by \citet{Arl_Bau:2002}, 
\citet[Chapter~11]{Arl:2007:phd} and
\citet[Section~6.3.2]{Arl:2009:RP} in the regression setting, 
and by \citet[Figure~3]{Arl_Ler:2012:penVF:JMLR} in least-squares density estimation. 
A similar phenomenon holds in the experiments of Section~\ref{sec.practical.jump-vs-slope},
as shown by Figure~\ref{fig.surpen}.
\begin{figure}
\begin{center}
\includegraphics[width=0.49\textwidth]{\pathfig/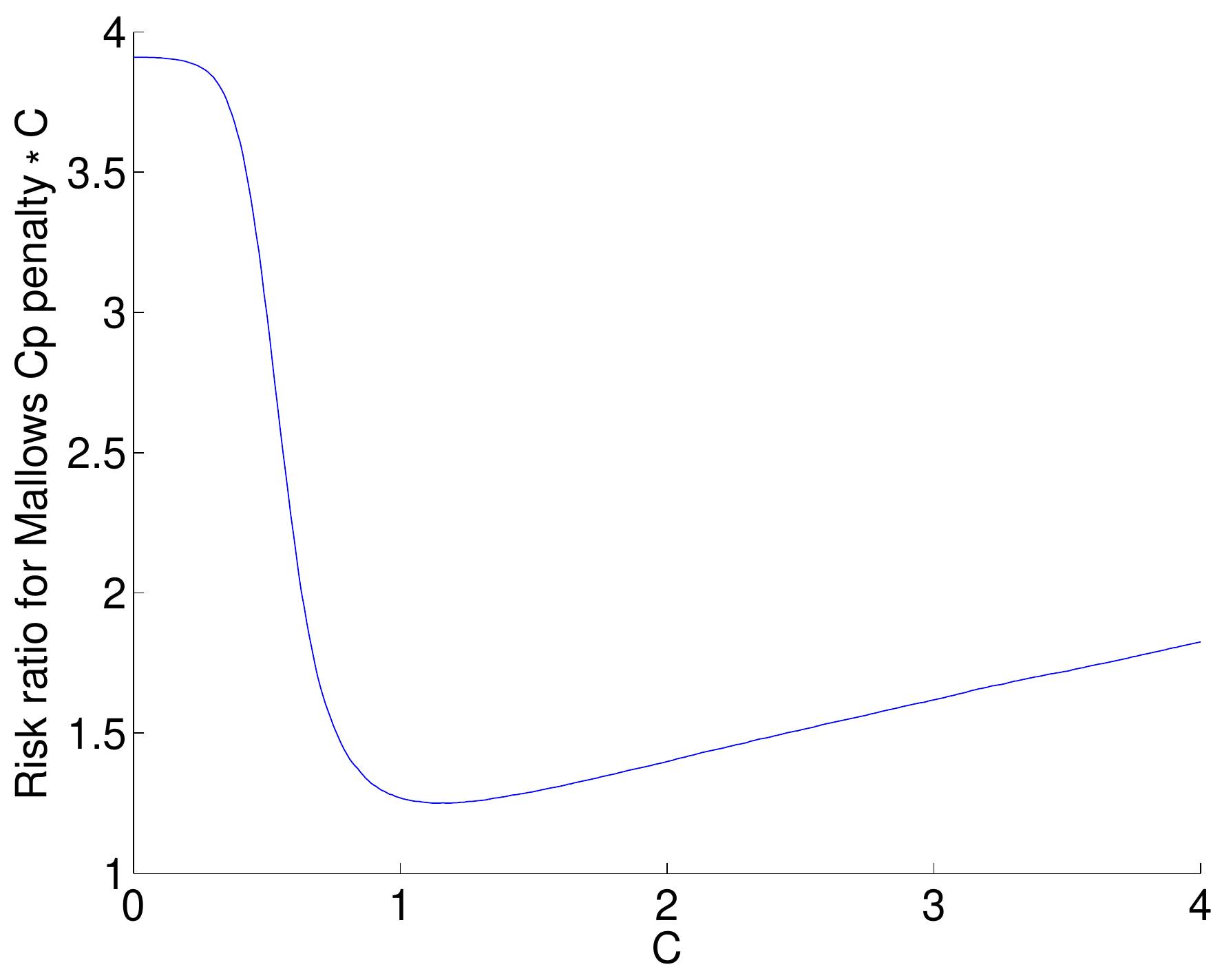}
\end{center}
\caption{\label{fig.surpen} Overpenalization with Mallows' $C_p$ penalty, `easy' setting 
\textup{(}see Appendix~\ref{app.details-simus.proc} for details\textup{)}. 
Error bars are so small that they would not be visible on the graph.
The optimal overpenalization factor is $C=1.12$, leading to an improvement by a factor $1.015$ compared to taking $C=1$.}
\end{figure}
For histogram selection in density estimation, \citet{Sau_Nav:2018:arXivv4} 
propose a natural way to overpenalize automatically, 
which leads to a new corrected version of the AIC criterion. 
Nevertheless, choosing from data an appropriate overpenalization factor remains an open problem.

\medbreak

For reasons detailed below, we conjecture that minimal penalties can help solving this issue.
More precisely, when Algorithms~\ref{algo.penmingal}--\ref{algo.penmingal.slope} 
are known to be first-order optimal, we conjecture that they 
automatically overpenalize ---by a factor close to 1 when $n$ is large---, 
and that this overpenalization decreases the risk of the final estimator 
compared to penalization by $\penoptgal(m)$. 
Another way to formulate this conjecture, following 
\citet{Lac_Mas:2015:journal} and \citet[Section~2.4]{Lac_Mas_Riv:2017}, 
is to state that the optimal constant $C^{\star}$ depends on $n$ differently 
from what first-order asymptotics suggest, 
and that Algorithm~\ref{algo.penmingal} estimates well the 
finite-sample value of~$C^{\star}$. 
Section~\ref{sec.slopeOLS} certainly provides the less difficult setting 
for proving this conjecture, 
even if the challenge is high: it requires to analyze penalization 
procedures at a precision level an order of magnitude higher than ever.

Several results support the above conjecture. 
%%% 1. Experiences qui vont dans ce sens
%
First, several simulation experiments show that minimal-penalty algorithms
overpenalize slightly in most settings:
this is reported by \citet[Section~2.6.2.4]{Vil:2007:these}, 
\citet{Arl_Bac:2009:minikernel_long_v2}  
and \citet[Section~6]{Sol_Arl_Bac:2011:multitask}, 
and this holds for the experiments of 
Section~\ref{sec.practical.jump-vs-slope} 
(see Figure~\ref{fig.OLS.dist-Ch} in Section~\ref{sec.practical.jump-vs-slope} 
and Tables~\ref{tab.dist-Ch.easy}--\ref{tab.dist-Ch.hard} in Appendix~\ref{app.morefig}).

%%% 2. Theoreme 1 asymetrique
%
Second, Theorem~\ref{thm.OLS} is consistent with the fact that $\Chjumpgal$ might overestimate $\sigma^2$.
Taking for instance $T_n = n/2$,
Theorem~\ref{thm.OLS} implies that on a large-probability event,
\[
(1-\etamoins) \sigma^2 \leq \Chthr (T_n) \leq (1+\etaplus) \sigma^2
\]
with $\etaplus > \etamoins\,$; 
Proposition~\ref{pro.variance-estim} in Section~\ref{sec.related.variance} 
provides a precise statement. 
If these bounds are tight, it means that $\Chthr$ is slightly biased 
upwards as an estimator of $\sigma^2$, 
which corresponds to overpenalization.

%%% 3. Lien avec "large" (et ses resultats partiels)
%
Third, minimal-penalty algorithms take into account the full collection of estimators $(\shm)_{\mM}$ in their definition,
and Section~\ref{sec.empirical.conjectures.rich} details why they 
should automatically adapt to the richness of the collection $\M$. 
We claim that the need for overpenalization might be mostly related 
to the richness of $\M$, 
so that the conjectures of Section~\ref{sec.empirical.conjectures.rich} 
could help solving the above overpenalization conjecture. 
Let us explain briefly why, by considering fixed-design regression with projection estimators, using some results and the vocabulary of Section~\ref{sec.theory.rich}.
When $\M$ is small ---say, one model per value of the dimension---, 
$\pen_{C_p}(m) = 2 \sigma^2 D_m / n$ is an (asymptotically) optimal penalty 
and the minimal penalty is $\sigma^2 D_m/n$. 
When $\M$ is large ---say, $\binom{n}{D}$ models of dimension $D$---, 
the minimal amount of penalization required is multiplied by $1 + 2 \log(n/D_m)$, 
which is of order $2\log(n) \gg 1$ (except for the largest models), 
and good performance can be obtained with some penalties of the same order of magnitude.
For a given sample size $n$, between these two extreme settings, 
there is a continuum of collections $\M$ of increasing sizes, 
for which the optimal amount of penalization is $C^{\star}_n(\M) \pen_{C_p}$ 
for some $C^{\star}_n(\M)$ between $1$ and $2 \log(n)$, approximately: 
this is an instance of the overpenalization phenomenon.
So, if a minimal-penalty algorithm adapts to the size of $\M$, 
it would capture the need for overpenalization by $C^{\star}_n(\M)$ in the constant $\Chjumpgal\,$.
At least, $\Chjumpgal$ would be asymptotically of the correct order for both small and large $\M$, 
which cannot be done with some estimator $\sigh^2$ that does not take into account the collection $\M$.

\subsection{Beyond Algorithms~\ref{algo.penmingal}--\ref{algo.penmingal.slope}: phase transitions for estimator selection} \label{sec.empirical.outside}
Most, if not all, nonparametric estimators depend on one or several parameters, 
whose optimal data-driven choice is often a challenge.
In this survey, we focus on a single parameter $C$ that is a multiplicative constant in front of a penalty,
and we show that in several settings:  
\begin{enumerate}
\item[(i)] an observable phase transition occurs 
for the estimator $\sh_{\mhgalzero(C)}$ around $C=C^{\star}$, 
and 
\item[(ii)] $C^{\star}$ can be used for the optimal calibration of $\sh_{\mhoptun(C)}$ 
through Algorithm~\ref{algo.penmingal}. 
\end{enumerate} 
If a similar phenomenon occurs for other types of tuning parameters, 
this would lead to highly interesting 
generalizations of Algorithm~\ref{algo.penmingal}. 
This subsection collects partial theoretical results ---using the notation 
(\pbpenmin), (\pbpenminalt), \eqref{pb.penopt} and \eqref{pb.penopt.weak} 
defined in Sections~\ref{sec.theory.approach}--\ref{sec.theory.pbpenminalt}--- 
and experiments going into this direction, 
as well as several conjectures and open problems. 

\subsubsection{Goldenshluger-Lepski's and related procedures} 
Goldenshluger-Lepski's method \citep{Gol_Lep:2011:AoS,Ber_Lac_Riv:2013} 
is a classical estimator-selection procedure, 
which does not rely on penalization of an empirical risk 
but on pairwise comparisons between estimators. 

\paragraph{Goldenshluger-Lepski's method} 
%%%%%% Goldenshluger-Lepski proprement dit
%
For choosing the bandwidth $h$ of kernel density estimators with a fixed kernel, 
in order to minimize the least-squares risk, 
\citet{Lac_Mas:2015:journal} study a slightly simplified version of 
Goldenshluger-Lepski's method, that depends on a single parameter $a$ that can be 
interpreted as a constant in front of a penalty. 
If we define the complexity by $\C_h = \lVert K \rVert^2 / (nh)$, as usual 
for kernel density estimation, 
\citet[Theorem~3]{Lac_Mas:2015:journal} prove an equivalent of 
\eqref{pb.penmin.Cpt-Dgrd} and \eqref{pb.penmin.Cpt-Rgrd} if $a<1$, 
and \eqref{pb.penopt.weak} if $a>1$. 
Simulation experiments suggest that there is indeed a phase transition for the selected bandwidth 
around some value $a^{\star}$ of $a$ ---hence, \eqref{pb.penmin.Cgrd-Dpt} should also hold true--- 
and that the optimal value of $a$ is slightly above~$a^{\star}$. 
Despite the theoretical results showing that $a^{\star} \to 1$ as $n \to +\infty$, 
for a finite sample size $a^{\star}$ is not necessarily close to~$1$. 
This leads to a minimal-penalty algorithm for estimating~$a^{\star}$, 
hence calibrating (a simplified version of) Goldenshluger-Lepski's method. 
The numerical experiments of two papers show the interest of this algorithm,  
with $\widehat{a}$ defined similarly to~$\Chmaxjump\,$:  
for estimating the stationary distribution of a bifurcating Markov chain on $\R^d$ \citep{Bit_Roc:2017}, 
and for state-by-state inference of the emission densities of a hidden Markov model \citep{Leh:2018}. 
Although \citet{Leh:2018} takes a penalty multiplied by $a=2 \widehat{a}$ for defining the final estimator 
---similarly to the slope heuristics---, 
\citet{Bit_Roc:2017} take $a$ just above $\widehat{a}$, 
by selecting the estimator immediately after the jump of~$\C_h\,$. 
This choice is supported by the fact that the minimal and optimal penalties are almost equal 
in the results of \citet{Lac_Mas:2015:journal}. 

\paragraph{Penalized comparison to overfitting (PCO)} 
%%%%%% Variante = PCO: theorie / experiences en densite L2 (=> pas de calibration ici)
%
In the same framework, with a possibly multivariate bandwidth $h \in \R^d$, 
\citet{Lac_Mas_Riv:2017} propose a new procedure called 
penalized comparison to overfitting (PCO), 
which lies between penalization and Goldenshluger-Lepski's method. 
PCO depends on a parameter $\lambda$ which is a multiplicative factor in front of 
one of the two terms of some kind of penalty. 
Considering again the complexity $\C_h = \lVert K \rVert^2 / (nh)$, 
PCO satisfies an equivalent of 
\eqref{pb.penmin.Cpt-Dgrd} if $\lambda <0$ \citep[Theorems~3--4]{Lac_Mas_Riv:2017}, 
\eqref{pb.penmin.Cgrd-Rpt} if $\lambda >0$, 
and \eqref{pb.penopt} around $\lambda = \lambda^{\star}=1$ 
\citep[Theorems~2 and~5]{Lac_Mas_Riv:2017}. 
The optimality of $\lambda^{\star}=1$ is assessed by numerical experiments on synthetic data \citep{Var_Lac_Mas_Riv:2019}.

%%% Description de l'algo de penalites minimales correspondant a PCO
%
Although PCO does not seem to require a data-driven calibration of $\lambda$  
according to the above result, 
$\lambda = 1$ may not always be a good choice outside the least-squares density estimation setting. 
Therefore, the theoretical results of \citet{Lac_Mas_Riv:2017} suggest the following minimal-penalty 
algorithm for calibrating PCO: 
first, detect $\widehat{\lambda}$ around which $\lambda \mapsto \C_{\widehat{h}(\lambda)}$ jumps, 
then, take 
\[ 
(a) \quad \lambda = \widehat{\lambda}+1 
\qquad \text{or} \qquad 
(b) \quad \lambda = 2 (\widehat{\lambda}+1) - 1 = 2\widehat{\lambda}+1 
\]  
for defining the final estimator. 
Option (a) is suggested by the fact that the difference between the minimal $\lambda$ ---zero--- 
and the optimal $\lambda$ ---one--- is equal to~$1$  \citep[Remark~3.1]{Var_Lac_Mas_Riv:2019}. 
Option (b) is suggested by the slope heuristics, 
since the penalty $\lambda \pen_0$ is equivalent to $C \pen_0 - \pen_0$ with $C=\lambda+1$, 
for which the minimal penalty occurs at $C=1$ and the (asymptotically) optimal penalty occurs for $C=2$ 
---hence a factor~$2$ between the minimal and the optimal penalty.

%%% Experiences de PCO + calibration par penalites minimales
%
The experiments of \citet[Section~5]{Com_Pri_Sam:2017} suggest that a similar way to calibrate PCO 
works well for 
selecting the bandwidth of a kernel estimator of the stationary density of the solution 
of a stochastic differential equation. 
For state-by-state inference of the emission densities of a hidden Markov model, 
\citet[Section~4.3.2]{Leh:2018} proposes a variant of PCO 
that can be well calibrated by a minimal-penalty algorithm according to numerical experiments.

A full theoretical validation of this calibration strategy remains an open problem.  
Providing theoretical guidelines for choosing between options (a) and (b) 
would also be interesting. 
Another natural question is to generalize PCO to other settings 
where Goldenshluger-Lepski's method applies, 
such as density estimation with the $L^p$ risk or regression; 
to the best of our knowledge, this remains an open problem.

\subsubsection{Choice of a threshold} 
For some thresholding estimators, with a threshold depending (non-linearly) on some parameter 
$\gamma \in (0,+\infty)$, an equivalent of (\pbpenminalt) is proved ---for a particular basis 
and assuming that $\bayes$ is equal to $\un_{[0,1]}$---, 
as well as an equivalent of \eqref{pb.penopt.weak} in the general case, in two settings:
density estimation on $\R$ \citep{Rey_Riv_Tul:2009}
and estimation of a Poisson intensity on $\R$ \citep{Rey_Riv:2010}.

For some Dantzig estimator (given some dictionary), with a parameter $\gamma>0$ appearing 
in the Dantzig constraints, \citet{Ber_LeP_Riv:2009} prove an equivalent of (\pbpenminalt) 
---for a particular dictionary and assuming that $\bayes = \un_{[0,1]}$---, 
as well as an equivalent of \eqref{pb.penopt.weak} in the general case.

In all the above results \citep{Rey_Riv:2010,Ber_LeP_Riv:2009,Rey_Riv_Tul:2009}, 
an equivalent of Algorithm~\ref{algo.penmingal} is proposed for a data-driven choice of $\gamma$,
and numerical experiments suggest that
the optimal $\gamma$ is often very close to the minimal $\gamma$.
Therefore, a generalization of the slope heuristics ($\gamma_{\opt} \approx 2 \gamma_{\min}$) probably does not hold here.

\subsubsection{Generalization} 
The above results, obtained for two different kinds of problems, 
suggest that phase transitions could be used much more generally for 
estimator selection, 
including the optimal calibration of learning algorithms. 
Section~\ref{sec.empirical.conjectures.high-dim} 
proposes it for the Lasso and related procedures. 
We conjecture that the same idea can be used fruitfully in several other settings. 

The key question is to find the good parametrization of the estimator collection. 
The successes of minimal-penalty algorithms rely on the parametrization 
by the constant $C$ in front of a well-chosen penalty shape $\pen_0$ 
---that is, chosen following the theoretical guidelines 
of Section~\ref{sec.theory.pbmult}, 
possibly combined with the practical hints referenced 
in Section~\ref{sec.practical.other-uses}. 

Finding an appropriate parametrization for the Lasso, 
for instance, remains an open problem to the best of our knowledge.

\subsection{Related challenges in probability theory}
Addressing the statistical open problems listed above 
mostly relies on a few corresponding open problems in probability theory.
As detailed in Section~\ref{sec.hints},
for each estimator $\shm$ considered, 
given some deterministic $\bayes_m$ 
---which can be the best estimator in the associated model, 
or the expectation of $\shm\,$, for instance---, 
the key theoretical quantities are 
the following: 
\begin{itemize}
\item 
the excess risk 
$p_1(m) = \Risk\parens{\shm} - \Risk\parens{\bayes_m}$, 
\item 
the excess empirical risk 
$p_2(m) = \Remp\parens{\bayes_m} - \Remp\parens{\shm}$, 
\item 
the empirical process at $\bayes_m\,$, 
$\ovdelta(m) = \Risk\parens{\bayes_m} - \Remp\parens{\bayes_m}$.
\end{itemize}
Since the empirical process is well understood in general, 
the main challenges are about $p_1(m)$ and $p_2(m)$. 
One either has to show that $\lvert p_1(m) - p_2(m) \rvert / p_1(m)$ is small 
on a large-probability event 
---using Proposition~\ref{pro.pbpenmin.gal.above} in Section~\ref{sec.theory.hint.penmin.above}---, 
or to show non-asymptotic concentration inequalities for $p_1(m)$ and $p_2(m)$ 
around deterministic quantities that are known up to a multiplicative factor.

We strongly encourage further work on these questions, 
especially on the concentration of the excess risk $p_1(m)$ and the excess empirical risk $p_2(m)$, 
which are difficult theoretical problems of interest for statisticians beyond minimal penalties. 
\\ 
Indeed, concentrating the excess risk provides lower bounds on the risk of the estimator $\shm$ 
for a given statistical problem ---and not in the minimax sense, as most statistical lower bounds---, 
which can be much informative for practicioners. 
This problem has attracted some attention in the last few years, 
and we review the recent work on this topic in Section~\ref{sec.theory.hint.penopt}. 
\\ 
When $\shm$ is the empirical risk minimizer over some model $S_m\,$, 
the excess empirical risk can be rewritten as the supremum of an empirical process 
\[ 
p_2(m) = \sup_{t \in S_m} \bigl\{ \Remp(\bayes_m) - \Remp(t) \bigr\} 
\, , 
\]
which is an object of interest for empirical process theory in general. 
Its concentration can also be seen as a non-asymptotic version of the Wilks phenomenon, 
which is another reason for tackling the theoretical challenge of proving that 
$p_2(m)$ concentrates around some deterministic quantity. 
Section~\ref{sec.theory.hint.penmin.conc} reviews such theoretical results.

Handling large collections of estimators 
---see Sections \ref{sec.theory.rich} and~\ref{sec.empirical.conjectures.rich}--- 
induces additional issues, 
since we cannot expect $p_1(m)$, $p_2(m)$, and $\delta(m)$ to concentrate tightly 
\emph{uniformly} over all $m \in \M$. 
This raises the challenge of understanding precisely their uniform deviations 
among such large collections, 
with high-probability upper \emph{and lower} bounds on these deviations. 
In the case of model selection, by grouping models of the same dimension 
as explained in Section~\ref{sec.theory.rich}, 
this problem reduces to concentrating $p_1(m)$ and $p_2(m)$ for empirical risk minimizers 
over models that are \emph{unions} of a large number of vector spaces of the same dimension. 
Note that the same probabilistic challenge arises in the problem of understanding 
the overpenalization phenomenon ---see Section~\ref{sec.empirical.overpenalization}---, 
for both large and small collections of estimators.

\section*{Acknowledgments}
The author 
%%is also member of the Celeste project-team of Inria Saclay - \^Ile-de-France, 
%%and 
acknowledges the support of the French Agence Nationale de la Recherche (Blanc SIMI 1 2011 projet Calibration). 
Part of this work was done while the author was financed by CNRS and 
member of the Sierra team 
in the D\'epartement d'Informatique de l'\'Ecole normale sup\'erieure 
(CNRS / ENS / Inria UMR 8548), 45 rue d'Ulm, 75005 Paris, France. 

I thank the successive editors of the Journal de la Soci\'et\'e Fran\c{c}aise de Statistique, Philippe Besse and Gilles Celeux, 
for their patience and kindness as they waited for me to submit this invited paper. 

I warmly thank the numerous colleagues who kindly answered 
my questions about their own works related with minimal penalties, 
during the (long) period of preparation of this article or before. 
Special thanks to Matthieu Lerasle and Pascal Massart for many inspiring discussions on the topic, 
to Yves Rozenholc for a long discussion at a CIRM workshop 
about his ``statistical base jumping'' idea that is described in Section~\ref{sec.practical.variants-chpt}, 
and to Gilles Celeux and Adrien Saumard for a careful reading of a preliminary version of this work. 

Finally, I deeply thank Patricia Reynaud-Bouret who introduced me to the topic 
quite early (at the beginning of 2002!) ---I owe her my first numerical experiments on the slope heuristics \citep{Arl_Bau:2002}---, 
and my coauthors on the minimal-penalty related articles I wrote: 
Pascal Massart, 
Francis Bach, 
Matthieu Solnon, 
Alain Celisse, Za\"id Harchaoui, 
and Damien Garreau.

\bibliography{survey_penmin}

\begin{thebibliography}{}

\bibitem[Abramovich et~al., 2006]{Abr_etal:2006}
Abramovich, F., Benjamini, Y., Donoho, D.~L., and Johnstone, I.~M. (2006).
\newblock {Adapting to unknown sparsity by controlling the false discovery
  rate}.
\newblock {\em Ann. Statist.}, 34(2):584--653.

\bibitem[Akaike, 1969]{Aka:1969a}
Akaike, H. (1969).
\newblock Fitting autoregressive models for prediction.
\newblock {\em Ann. Inst. Statist. Math.}, 21:243--247.

\bibitem[Akaike, 1970]{Aka:1969}
Akaike, H. (1970).
\newblock {Statistical predictor identification}.
\newblock {\em Ann. Inst. Statist. Math.}, 22:203--217.

\bibitem[Akaike, 1973]{Aka:1973}
Akaike, H. (1973).
\newblock {Information theory and an extension of the maximum likelihood
  principle}.
\newblock In {\em {Second International Symposium on Information Theory
  (Tsahkadsor, 1971)}}, pages 267--281. Akad{\'e}miai Kiad{\'o}, Budapest.

\bibitem[Akakpo, 2011]{Aka:2011}
Akakpo, N. (2011).
\newblock Estimating a discrete distribution via histogram selection.
\newblock {\em ESAIM: Probability and Statistics}, 15:1--29.

\bibitem[Allen, 1974]{All:1974}
Allen, D.~M. (1974).
\newblock {The relationship between variable selection and data augmentation
  and a method for prediction}.
\newblock {\em Technometrics}, 16:125--127.

\bibitem[Andresen and Spokoiny, 2014]{And_Spo:2013:journal}
Andresen, A. and Spokoiny, V. (2014).
\newblock Critical dimension in profile semiparametric estimation.
\newblock {\em Electron. J. Statist.}, 8(2):3077--3125.

\bibitem[Arlot, 2007]{Arl:2007:phd}
Arlot, S. (2007).
\newblock {\em Resampling and Model Selection}.
\newblock PhD thesis, University Paris-Sud 11.
\newblock Available at \url{https://tel.archives-ouvertes.fr/tel-00198803v1}.

\bibitem[Arlot, 2009]{Arl:2009:RP}
Arlot, S. (2009).
\newblock Model selection by resampling penalization.
\newblock {\em Electron. J. Stat.}, 3:557--624 (electronic).

\bibitem[Arlot, 2011]{Arl:2011:cours_Peccot}
Arlot, S. (2011).
\newblock S\'election de mod\`eles et s\'election d'estimateurs pour
  l'apprentissage statistique.
\newblock Cours Peccot. Coll\`ege de France. Available at
  \url{http://www.di.ens.fr/~arlot/peccot.htm}.

\bibitem[Arlot and Bach, 2009]{Arl_Bac:2009:minikernel_nips}
Arlot, S. and Bach, F. (2009).
\newblock Data-driven calibration of linear estimators with minimal penalties.
\newblock In Ben{g}io, Y., Schuurmans, D., Lafferty, J., Willi{a}ms, C. K.~I.,
  and Culotta, A., editors, {\em Advances in Neural Information Processing
  Systems 22}, pages 46--54.

\bibitem[Arlot and Bach, 2011]{Arl_Bac:2009:minikernel_long_v2}
Arlot, S. and Bach, F. (2011).
\newblock Data-driven calibration of linear estimators with minimal penalties.
\newblock arXiv:0909.1884v2.

\bibitem[Arlot and Baudry, 2002]{Arl_Bau:2002}
Arlot, S. and Baudry, J.-P. (2002).
\newblock S\'election de mod\`eles.
\newblock In French. Master 1 report, ENS Paris. Available at
  \url{https://www.math.u-psud.fr/~arlot/papers/02selection_modeles.pdf}.
  Advisor: Yannick Baraud. Report about the paper ``{G}aussian model
  selection'' by L. Birg\'e \& P. Massart, JEMS 3(3):203--268, 2001.

\bibitem[Arlot and Celisse, 2010]{Arl_Cel:2010:surveyCV}
Arlot, S. and Celisse, A. (2010).
\newblock A survey of cross-validation procedures for model selection.
\newblock {\em Statist. Surv.}, 4:40--79.

\bibitem[Arlot et~al., 2019]{Arl_Cel_Har:2012:kernelchpt}
Arlot, S., Celisse, A., and Harchaoui, Z. (2019).
\newblock A kernel mult{i}ple change-point algorithm via model selection.
\newblock {\em J. Mach. Learn. Res.}
\newblock To appear. Preliminary version available at arXiv:1202.3878.

\bibitem[Arlot and Lerasle, 2016]{Arl_Ler:2012:penVF:JMLR}
Arlot, S. and Lerasle, M. (2016).
\newblock Choice of {$V$} for {$V$}-fold cross-validation in least-squares
  density estimation.
\newblock {\em J. Mach. Learn. Res.}, 17(208):1--50.

\bibitem[Arlot and Massart, 2009]{Arl_Mas:2009:pente}
Arlot, S. and Massart, P. (2009).
\newblock Data-driven calibration of penalties for least-squares regression.
\newblock {\em J. Mach. Learn. Res.}, 10:245--279 (electronic).

\bibitem[Audibert and Catoni, 2011]{Aud_Cat:2011}
Audibert, J.-Y. and Catoni, O. (2011).
\newblock Robust linear least squares regression.
\newblock {\em Ann. Statist.}, 39(5):2766--2794.

\bibitem[Bar-Hen et~al., 2018]{Bar_Gey_Pog:2018}
Bar-Hen, A., Gey, S., and Poggi, J.-M. (2018).
\newblock {Spatial CART Classification Trees}.
\newblock Available at \url{https://hal.archives-ouvertes.fr/hal-01837065v1}.

\bibitem[Baraud, 2000]{Bar:2000}
Baraud, Y. (2000).
\newblock {Model selection for regression on a fixed design}.
\newblock {\em Probab. Theory Related Fields}, 117(4):467--493.

\bibitem[Baraud, 2011]{Bar:2011}
Baraud, Y. (2011).
\newblock Estimator selection with respect to {H}ellinger-type risks.
\newblock {\em Probab. Theory Related Fields}, 151(1-2):353--401.

\bibitem[Baraud et~al., 2009]{Bar_Gir_Hue:2007}
Baraud, Y., Giraud, C., and Huet, S. (2009).
\newblock {{G}aussian model selection with an unknown variance}.
\newblock {\em Ann. Statist.}, 37(2):630--672.

\bibitem[Baraud et~al., 2014]{Bar_Gir_Hue:2010}
Baraud, Y., Giraud, C., and Huet, S. (2014).
\newblock Estimator selection in the {G}aussian setting.
\newblock {\em Ann. Inst. Henri Poincar\'e Probab. Stat.}, 50(3):1092--1119.

\bibitem[Bardet and Guenaizi, 2018]{Bar_Gue:2018:preprint}
Bardet, J.-M. and Guenaizi, A. (2018).
\newblock Semi-parametric detection of multiple changes in long-range dependent
  processes.
\newblock arXiv:1801.02515v2.

\bibitem[Bardet et~al., 2012]{Bar_Ken_Win:2010}
Bardet, J.-M., Kengne, W.~C., and Wintenberger, O. (2012).
\newblock Multiple breaks detection in general causal time series using
  penalized quasi-likelihood.
\newblock {\em Electron. J. Stat.}, 6:435--477 (electronic).

\bibitem[Barron et~al., 1999]{Bar_Bir_Mas:1999}
Barron, A., Birg{\'e}, L., and Massart, P. (1999).
\newblock {Risk bounds for model selection via penalization}.
\newblock {\em Probab. Theory Related Fields}, 113(3):301--413.

\bibitem[Bartlett et~al., 2005]{Bar_Bou_Men:2005}
Bartlett, P.~L., Bousquet, O., and Mendelson, S. (2005).
\newblock {Local {R}ademacher complexities}.
\newblock {\em Ann. Statist.}, 33(4):1497--1537.

\bibitem[Bartlett and Mendelson, 2006]{Bar_Men:2006a}
Bartlett, P.~L. and Mendelson, S. (2006).
\newblock {Empirical minimization}.
\newblock {\em Probability Theory and Related Fields}, 135(3):311--334.

\bibitem[Baudry, 2009]{Bau:2009:phd}
Baudry, J.-P. (2009).
\newblock {\em {Model selection for clustering. Choosing the number of
  classes}}.
\newblock PhD thesis, University Paris-Sud.
\newblock Available at \url{https://tel.archives-ouvertes.fr/tel-00461550v1}.

\bibitem[Baudry, 2015]{Bau:2012:ejs}
Baudry, J.-P. (2015).
\newblock Estimation and model selection for model-based clustering with the
  conditional classification likelihood.
\newblock {\em Electron. J. Statist.}, 9(1):1041--1077.

\bibitem[Baudry et~al., 2012]{Bau_Mau_Mic:2010}
Baudry, J.-P., Maugis, C., and Michel, B. (2012).
\newblock Slope heuristics: overview and implementation.
\newblock {\em Statistics and Computing}, 22(2):455--470.

\bibitem[Bellec, 2017]{Bel:2017}
Bellec, P. (2017).
\newblock Optimistic lower bounds for convex regularized least-squares.
\newblock arXiv:1703.01332v3.

\bibitem[Bellec, 2018]{Bel:2018:v4}
Bellec, P. (2018).
\newblock The noise barrier and the large signal bias of the lasso and other
  convex estimators.
\newblock arXiv:1804.01230v4.

\bibitem[Bellec and Tsybakov, 2017]{Bel_Tsy:2017}
Bellec, P. and Tsybakov, A. (2017).
\newblock Bounds on the prediction error of penalized least squares estimators
  with convex penalty.
\newblock In Panov, V., editor, {\em Modern Problems of Stochastic Analysis and
  Statistics}, pages 315--333, Cham. Springer International Publishing.

\bibitem[Bellec, 2019]{Bel:2019:quadform}
Bellec, P.~C. (2019).
\newblock Concentration of quadratic forms under a {B}ernstein moment
  assumption.
\newblock Technical report, arXiv.
\newblock arXiv:1901.08736v1.

\bibitem[Belloni et~al., 2014]{Bel_Che_Wan:2014}
Belloni, A., Chernozhukov, V., and Wang, L. (2014).
\newblock Pivotal estimation via square-root {L}asso in nonparametric
  regression.
\newblock {\em Ann. Statist.}, 42(2):757--788.

\bibitem[Bertin et~al., 2016]{Ber_Lac_Riv:2013}
Bertin, K., Lacour, C., and Rivoirard, V. (2016).
\newblock Adaptive pointwise estimation of conditional density function.
\newblock {\em Ann. Inst. Henri Poincar\'e Probab. Stat.}, 52(2):939--980.

\bibitem[Bertin et~al., 2011]{Ber_LeP_Riv:2009}
Bertin, K., Le~Pennec, E., and Rivoirard, V. (2011).
\newblock Adaptive {D}antzig density estimation.
\newblock {\em Ann. Inst. H. Poincar{\'e} Probab. Statist.}, 47(1):43--74.

\bibitem[Biau and Fischer, 2012]{Bia_Fis:2011}
Biau, G. and Fischer, A. (2012).
\newblock Parameter selection for principal curves.
\newblock {\em IEEE Transactions on Information Theory}, 58(3):1924--1939.

\bibitem[Birg\'e and Massart, 1997]{Bir_Mas:1997}
Birg\'e, L. and Massart, P. (1997).
\newblock {From model selection to adaptive estimation}.
\newblock In {\em {Festschrift for Lucien Le Cam}}, pages 55--87. Springer, New
  York.

\bibitem[Birg\'e and Massart, 2001a]{Bir_Mas:2002}
Birg\'e, L. and Massart, P. (2001a).
\newblock {{G}aussian model selection}.
\newblock {\em J. Eur. Math. Soc. (JEMS)}, 3(3):203--268.

\bibitem[Birg\'e and Massart, 2001b]{Bir_Mas:2001}
Birg\'e, L. and Massart, P. (2001b).
\newblock A generalized {C}p criterion for {G}aussian model selection.
\newblock Technical report, Universit\'es de Paris 6 et Paris 7.
\newblock Pr\'epublication 647, 39 pages. Available at
  \url{http://massart.pascal.free.fr/Site/publications_files/Cp.pdf}.

\bibitem[Birg\'e and Massart, 2007]{Bir_Mas:2006}
Birg\'e, L. and Massart, P. (2007).
\newblock {Minimal penalties for {G}aussian model selection}.
\newblock {\em Probab. Theory Related Fields}, 138(1-2):33--73.

\bibitem[Bitseki~Penda and Roche, 2017]{Bit_Roc:2017}
Bitseki~Penda, S.~V. and Roche, A. (2017).
\newblock Local bandwidth selection for kernel density estimation in
  bifurcating {M}arkov chain model.
\newblock arXiv:1706.07034v1.

\bibitem[Blanchard and Massart, 2006]{Bla_Mas:2006}
Blanchard, G. and Massart, P. (2006).
\newblock {Discussion: ``{L}ocal {R}ademacher complexities and oracle
  inequalities in risk minimization'' [{A}nn. {S}tatist. {\bf 34} (2006), no.
  6, 2593--2656] by {V}. {K}oltchinskii}.
\newblock {\em Ann. Statist.}, 34(6):2664--2671.

\bibitem[Bontemps and Toussile, 2013]{Bon_Tou:2010}
Bontemps, D. and Toussile, W. (2013).
\newblock Clustering and variable selection for categorical multivariate data.
\newblock {\em Electron. J. Stat.}, 7:2344--2371.

\bibitem[Boucheron et~al., 2013]{Bou_Lug_Mas:2011:livre}
Boucheron, S., Lugosi, G., and Massart, P. (2013).
\newblock {\em Concentration Inequalities: A Nonasymptotic Theory of
  Independence}.
\newblock Oxford University Press, Oxford.

\bibitem[Boucheron and Massart, 2011]{Bou_Mas:2004}
Boucheron, S. and Massart, P. (2011).
\newblock A high dimensional {W}ilks phenomenon.
\newblock {\em Probab. Theory Related Fields}, 150(3-4):405--433.

\bibitem[Bouveyron et~al., 2015a]{Bou_Com_Jac:2015}
Bouveyron, C., C\^ome, E., and Jacques, J. (2015a).
\newblock The discriminative functional mixture model for a comparative
  analysis of bike sharing systems.
\newblock {\em Ann. Appl. Stat.}, 9(4):1726--1760.

\bibitem[Bouveyron et~al., 2015b]{Bou_Fau_Gir:2014}
Bouveyron, C., Fauvel, M., and Girard, S. (2015b).
\newblock Kernel discriminant analysis and clustering with parsimonious
  {G}aussian process models.
\newblock {\em Statistics and Computing}, 25(6):1143--1162.

\bibitem[Breiman and Spector, 1992]{Bre_Spe:1992}
Breiman, L. and Spector, P. (1992).
\newblock {Submodel Selection and Evaluation in Regression. The X-Random Case}.
\newblock {\em International Statistical Review}, 60(3):291--319.

\bibitem[Brown and Levine, 2007]{Bro_Lev:2007}
Brown, L.~D. and Levine, M. (2007).
\newblock Variance estimation in nonparametric regression via the difference
  sequence method.
\newblock {\em Ann. Statist.}, 35(5):2219--2232.

\bibitem[Buckley and Eagleson, 1989]{Buc_Eag:1989}
Buckley, M.~J. and Eagleson, G.~K. (1989).
\newblock A graphical method for estimating the residual variance in
  nonparametric regression.
\newblock {\em Biometrika}, 76(2):203--210.

\bibitem[Burnham and Anderson, 2002]{Bur_And:2002}
Burnham, K.~P. and Anderson, D.~R. (2002).
\newblock {\em {Model Selection and Multimodel Inference}}.
\newblock Springer-Verlag, New York, second edition.
\newblock A practical information-theoretic approach.

\bibitem[Cabrieto et~al., 2018]{Cab_etal:2018}
Cabrieto, J., Tuerlinckx, F., Kuppens, P., Wilhelm, F.~H., Liedlgruber, M., and
  Ceulemans, E. (2018).
\newblock Capturing correlation changes by applying kernel change point
  detection on the running correlations.
\newblock {\em Information Sciences}, 447:117--139.

\bibitem[Caillerie and Michel, 2011]{Cai_Mic:2009}
Caillerie, C. and Michel, B. (2011).
\newblock Model selection for simplicial approximation.
\newblock {\em Foundations of Computational Mathematics}, 11(6):707--731.

\bibitem[Cao and Golubev, 2006]{Cao_Gol:2006}
Cao, Y. and Golubev, Y. (2006).
\newblock {On oracle inequalities related to smoothing splines}.
\newblock {\em Math. Methods Statist.}, 15(4):398--414 (2007).

\bibitem[Carter and Eagleson, 1992]{Car_Eag:1992}
Carter, C.~K. and Eagleson, G.~K. (1992).
\newblock A comparison of variance estimators in nonparametric regression.
\newblock {\em J. Roy. Statist. Soc. Ser. B}, 54(3):773--780.

\bibitem[Castellan, 1999]{Cas:1999}
Castellan, G. (1999).
\newblock Modified {A}kaike's criterion for histogram density estimation.
\newblock Technical Report 1999-61, University Paris-Sud.
\newblock Available at
  \url{https://www.math.u-psud.fr/~biblio/pub/1999/abs/ppo1999_61.html}.

\bibitem[Castellanos et~al., 2002]{Cas_Gom_Gue:2002}
Castellanos, J.~L., G{\'o}mez, S., and Guerra, V. (2002).
\newblock The triangle method for finding the corner of the {L}-curve.
\newblock {\em Appl. Numer. Math.}, 43(4):359--373.

\bibitem[Cattell, 1966]{Cat:1966}
Cattell, R.~B. (1966).
\newblock The scree test for the number of factors.
\newblock {\em Multivariate Behav. Res.}, 1(2):245--276.

\bibitem[Cattell and Vogelmann, 1977]{Cat_Vog:1977}
Cattell, R.~B. and Vogelmann, S. (1977).
\newblock A comprehensive trial of the scree and k.g. criteria for determining
  the number of factors.
\newblock {\em Multivariate Behav. Res.}, 12(3):289--325.

\bibitem[Chagny, 2013]{Cha:2011:esaim}
Chagny, G. (2013).
\newblock Penalization versus {G}oldenshluger-{L}epski strategies in warped
  bases regression.
\newblock {\em ESAIM Probab. Stat.}, 17:328--358.

\bibitem[Chatterjee, 2014]{Cha:2014}
Chatterjee, S. (2014).
\newblock A new perspective on least squares under convex constraint.
\newblock {\em Ann. Statist.}, 42(6):2340--2381.

\bibitem[Chatterjee, 2015]{Cha:2015b}
Chatterjee, S. (2015).
\newblock High dimensional regression and matrix estimation without tuning
  parameters.
\newblock arXiv:1510.07294v3.

\bibitem[Chen et~al., 2017]{Che_Gun_Zha:2017:journal}
Chen, X., Guntuboyina, A., and Zhang, Y. (2017).
\newblock A note on the approximate admissibility of regularized estimators in
  the {G}aussian sequence model.
\newblock {\em Electron. J. Statist.}, 11(2):4746--4768.

\bibitem[Cohen and Le~Pennec, 2014]{Coh_LeP:2014}
Cohen, S.~X. and Le~Pennec, E. (2014).
\newblock Unsupervised segmentation of spectral images with a spatialized
  {G}aussian mixture model and model selection.
\newblock {\em Oil Gas Sci. Technol. -- Rev. IFP Energies nouvelles},
  69(2):245--259.

\bibitem[Comte et~al., 2017]{Com_Pri_Sam:2017}
Comte, F., Prieur, C., and Samson, A. (2017).
\newblock Adaptive estimation for stochastic damping {H}amiltonian systems
  under partial observation.
\newblock {\em Stochastic Process. Appl.}, 127(11):3689--3718.

\bibitem[Comte and Rozenholc, 2004]{Com_Roz:2004}
Comte, F. and Rozenholc, Y. (2004).
\newblock {A new algorithm for fixed design regression and denoising}.
\newblock {\em Ann. Inst. Statist. Math.}, 56(3):449--473.

\bibitem[Connault, 2011]{Con:2011:phd}
Connault, P. (2011).
\newblock {\em Calibration d'algorithmes de type Lasso et analyse statistique
  de donn\'ees m\'etallurgiques en a\'eronautique}.
\newblock PhD thesis, Universit\'e Paris-Sud.

\bibitem[Craven and Wahba, 1978]{Cra_Wah:1979}
Craven, P. and Wahba, G. (1978).
\newblock {Smoothing noisy data with spline functions. {E}stimating the correct
  degree of smoothing by the method of generalized cross-validation}.
\newblock {\em Numer. Math.}, 31(4):377--403.

\bibitem[Derman and Le~Pennec, 2017]{Der_LeP:2017}
Derman, E. and Le~Pennec, E. (2017).
\newblock Clustering and model selection via penalized likelihood for
  different-sized categorical data vectors.
\newblock arXiv:1709.02294v1.

\bibitem[Dette et~al., 1998]{Det_Mun_Wag:1998}
Dette, H., Munk, A., and Wagner, T. (1998).
\newblock Estimating the variance in nonparametric regression---what is a
  reasonable choice?
\newblock {\em J. R. Stat. Soc. Ser. B Stat. Methodol.}, 60(4):751--764.

\bibitem[Devijver, 2017a]{Dev:2017:JMVA}
Devijver, {\'E}. (2017a).
\newblock Joint rank and variable selection for parsimonious estimation in a
  high-dimensional finite mixture regression model.
\newblock {\em Journal of Multivariate Analysis}, 157:1--13.

\bibitem[Devijver, 2017b]{Dev:2017}
Devijver, {\'E}. (2017b).
\newblock Model-based regression clustering for high-dimensional data:
  application to functional data.
\newblock {\em Adv. Data Analysis and Classification}, 11(2):243--279.

\bibitem[Devijver and Gallopin, 2018]{Dev_Gal:2016}
Devijver, {\'E}. and Gallopin, M. (2018).
\newblock {Block-diagonal covariance selection for high-dimensional {G}aussian
  graphical models}.
\newblock {\em {Journal of the American Statistical Association}}, pages
  306--314.

\bibitem[Devijver et~al., 2017]{Dev_Gal_Per:2017}
Devijver, {\'E}., Gallopin, M., and Perthame, E. (2017).
\newblock Nonlinear network-based quantitative trait prediction from
  transcriptomic data.
\newblock arXiv:1701.07899v5.

\bibitem[Devijver et~al., 2019]{Dev_Gou_Pog:2015:journal}
Devijver, {\'E}., Goude, Y., and Poggi, J.-M. (2019).
\newblock Clustering electricity consumers using high-dimensional regression
  mixture models.
\newblock {\em Applied Stochastic Models in Business and Industry}, pages
  1--19.
\newblock \url{https://doi.org/10.1002/asmb.2453}.

\bibitem[Devroye et~al., 1996]{Dev_Gyo_Lug:1996}
Devroye, L., Gy{\"o}rfi, L., and Lugosi, G. (1996).
\newblock {\em {A probabilistic theory of pattern recognition}}, volume~31 of
  {\em {Applications of Mathematics (New York)}}.
\newblock Springer-Verlag, New York.

\bibitem[Donoho et~al., 1995]{Don_etal:1995}
Donoho, D.~L., Johnstone, I.~M., Kerkyacharian, G., and Picard, D. (1995).
\newblock Wavelet shrinkage: asymptopia?
\newblock {\em J. Roy. Statist. Soc. Ser. B}, 57(2):301--369.
\newblock With discussion and a reply by the authors.

\bibitem[Dossal et~al., 2013]{Dos_etal:2011:journal}
Dossal, C., Kachour, M., Fadili, J.~M., Peyr\'e, G., and Chesneau, C. (2013).
\newblock The degrees of freedom of the lasso for general design matrix.
\newblock {\em Statistica Sinica}, 23(2):809--828.

\bibitem[Du and Schick, 2009]{Du_Sch:2009}
Du, J. and Schick, A. (2009).
\newblock A covariate-matched estimator of the error variance in nonparametric
  regression.
\newblock {\em J. Nonparametr. Stat.}, 21(3):263--285.

\bibitem[Durot et~al., 2009]{Dur_Leb_Toc:2009}
Durot, C., Lebarbier, {\'E}., and Tocquet, A.-S. (2009).
\newblock Estimating the joint distribution of independent categorical
  variables via model selection.
\newblock {\em Bernoulli}, 15(2):475--507.

\bibitem[Efron, 1986]{Efr:1986}
Efron, B. (1986).
\newblock {How biased is the apparent error rate of a prediction rule?}
\newblock {\em J. Amer. Statist. Assoc.}, 81(394):461--470.

\bibitem[Efron, 2004]{Efr:2004}
Efron, B. (2004).
\newblock {The estimation of prediction error: covariance penalties and
  cross-validation}.
\newblock {\em J. Amer. Statist. Assoc.}, 99(467):619--642.
\newblock With comments and a rejoinder by the author.

\bibitem[Engl and Grever, 1994]{Eng_Gre:1994}
Engl, H.~W. and Grever, W. (1994).
\newblock Using the {$L$}-curve for determining optimal regularization
  parameters.
\newblock {\em Numer. Math.}, 69(1):25--31.

\bibitem[Frontier, 1976]{Fro:1976}
Frontier, S. (1976).
\newblock {\'E}tude de la d\'ecroissance des valeurs propres dans une analyse
  en composantes principales: Comparaison avec le mod\`ele du b\^aton bris\'e.
\newblock {\em Journal of Experimental Marine Biology and Ecology},
  25(1):67--75.

\bibitem[Garivier and Lerasle, 2011]{Gar_Ler:2011}
Garivier, A. and Lerasle, M. (2011).
\newblock Oracle approach and slope heuristic in context tree estimation.
\newblock arXiv:1111.2191v1.

\bibitem[Garreau, 2017]{Gar:2017:phd}
Garreau, D. (2017).
\newblock {\em Change-point Detection and Kernels Methods}.
\newblock PhD thesis, \'Ecole Normale Sup\'erieure / {PSL} Research University.
\newblock Available at \url{https://tel.archives-ouvertes.fr/tel-01693360v2}.

\bibitem[Garreau and Arlot, 2018]{Gar_Arl:2015}
Garreau, D. and Arlot, S. (2018).
\newblock Consistent change-point detection with kernels.
\newblock {\em Electron. J. Statist.}, 12(2):4440--4486.

\bibitem[Gassiat and Van~Handel, 2013]{Gas_vHa:2013}
Gassiat, E. and Van~Handel, R. (2013).
\newblock Consistent order estimation and minimal penalties.
\newblock {\em IEEE Trans. Inform. Theory}, 59(2):1115--1128.

\bibitem[Gavish and Donoho, 2014]{Gav_Don:2014}
Gavish, M. and Donoho, D.~L. (2014).
\newblock The optimal hard threshold for singular values is $4/\sqrt{3}$.
\newblock {\em IEEE Trans. Inform. Theory}, 60(8):5040--5053.

\bibitem[Gendre, 2008]{Gen:2008}
Gendre, X. (2008).
\newblock {Simultaneous estimation of the mean and the variance in
  heteroscedastic {G}aussian regression}.
\newblock {\em Electron. J. Stat.}, 2:1345--1372.

\bibitem[Gey and Lebarbier, 2008]{Gey_Leb:08}
Gey, S. and Lebarbier, {\'E}. (2008).
\newblock {Using {CART} to Detect Multiple Change Points in the Mean for large
  samples}.
\newblock Technical Report~12, Statistics for Systems Biology.
\newblock Available at \url{https://hal.archives-ouvertes.fr/hal-00327146v1}.

\bibitem[Giacobino et~al., 2017]{Gia_etal:2017}
Giacobino, C., Sardy, S., Diaz-Rodriguez, J., and Hengartner, N. (2017).
\newblock Quantile universal threshold.
\newblock {\em Electron. J. Statist.}, 11(2):4701--4722.

\bibitem[Giraud, 2008]{Gir:2008}
Giraud, C. (2008).
\newblock Estimation of {G}aussian graphs by model selection.
\newblock {\em Electron. J. Stat.}, 2:542--563 (electronic).

\bibitem[Giraud, 2011]{Gir:2010}
Giraud, C. (2011).
\newblock Low rank multivariate regression.
\newblock {\em Electron. J. Stat.}, 5:775--799.

\bibitem[Giraud, 2014]{Gir:2014}
Giraud, C. (2014).
\newblock {\em Introduction to {H}igh-{D}imensional {S}tatistics}, volume 139
  of {\em {Monographs on Statistics and Applied Probability}}.
\newblock Chapman and Hall/CRC, Boca Raton, FL.

\bibitem[Giraud et~al., 2012]{Gir_Hue_Ver:2011}
Giraud, C., Huet, S., and Verzelen, N. (2012).
\newblock High-dimensional regression with unknown variance.
\newblock {\em Statist. Sci.}, 27(4):500--518.

\bibitem[Godichon-Baggioni et~al., 2019]{God_Mau_Rau:2018}
Godichon-Baggioni, A., Maugis-Rabusseau, C., and Rau, A. (2019).
\newblock Clustering transformed compositional data using {K}-means, with
  applications in gene expression and bicycle sharing system data.
\newblock {\em Journal of Applied Statistics}, 46(1):47--65.

\bibitem[Goldenshluger and Lepski, 2011]{Gol_Lep:2011:AoS}
Goldenshluger, A. and Lepski, O. (2011).
\newblock Bandwidth selection in kernel density estimation: oracle inequalities
  and adaptive minimax optimality.
\newblock {\em Ann. Statist.}, 39(3):1608--1632.

\bibitem[Grodzevich and Wolkowicz, 2009]{Gro_Wol:2009}
Grodzevich, O. and Wolkowicz, H. (2009).
\newblock Regularization using a parameterized trust region subproblem.
\newblock {\em Math. Program.}, 116(1-2):193--220.

\bibitem[Hall et~al., 1990]{Hal_Kay_Tit:1990}
Hall, P., Kay, J.~W., and Titterington, D.~M. (1990).
\newblock Asymptotically optimal difference-based estimation of variance in
  nonparametric regression.
\newblock {\em Biometrika}, 77(3):521--528.

\bibitem[Hall and Marron, 1990]{Hal_Mar:1990}
Hall, P. and Marron, J.~S. (1990).
\newblock On variance estimation in nonparametric regression.
\newblock {\em Biometrika}, 77(2):415--419.

\bibitem[Hanke, 1996]{Han:1996}
Hanke, M. (1996).
\newblock Limitations of the {$L$}-curve method in ill-posed problems.
\newblock {\em BIT}, 36(2):287--301.

\bibitem[Hansen, 1992]{Han:1992}
Hansen, P.~C. (1992).
\newblock Analysis of discrete ill-posed problems by means of the {$L$}-curve.
\newblock {\em SIAM Rev.}, 34(4):561--580.

\bibitem[Hansen et~al., 2007]{Han_Jen_Rod:2007}
Hansen, P.~C., Jensen, T.~K., and Rodriguez, G. (2007).
\newblock An adaptive pruning algorithm for the discrete {L}-curve criterion.
\newblock {\em J. Comput. Appl. Math.}, 198(2):483--492.

\bibitem[Hansen and O'Leary, 1993]{Han_OLe:1993}
Hansen, P.~C. and O'Leary, D.~P. (1993).
\newblock The use of the {$L$}-curve in the regularization of discrete
  ill-posed problems.
\newblock {\em SIAM J. Sci. Comput.}, 14(6):1487--1503.

\bibitem[Heng et~al., 2010]{Hen_Lu_Mha_Per:2010}
Heng, Y., Lu, S., Mhamdi, A., and Pereverzev, S.~V. (2010).
\newblock Model functions in the modified {$L$}-curve method---case study: the
  heat flux reconstruction in pool boiling.
\newblock {\em Inverse Problems}, 26(5):055006, 13.

\bibitem[Horn, 1965]{Hor:1965}
Horn, J.~L. (1965).
\newblock A rationale and test for the number of factors in factor analysis.
\newblock {\em Psychometrika}, 30(2):179--185.

\bibitem[Horn and Engstrom, 1979]{Hor_Eng:1979}
Horn, J.~L. and Engstrom, R. (1979).
\newblock Cattell's scree test in relation to {B}artlett's chi-square test and
  other observations on the number of factors problem.
\newblock {\em Multivariate Behavioral Research}, 14(3):283--300.

\bibitem[Jackson, 1993]{Jac:1993}
Jackson, D.~A. (1993).
\newblock Stopping rules in principal components analysis: a comparison of
  heuristical and statistical approaches.
\newblock {\em Ecology}, 74(8).

\bibitem[Koltchinskii, 2001]{Kol:2001}
Koltchinskii, V. (2001).
\newblock {Rademacher penalties and structural risk minimization}.
\newblock {\em IEEE Trans. Inform. Theory}, 47(5):1902--1914.

\bibitem[Koltchinskii, 2006]{Kol:2006}
Koltchinskii, V. (2006).
\newblock {Local {R}ademacher complexities and oracle inequalities in risk
  minimization}.
\newblock {\em Ann. Statist.}, 34(6):2593--2656.

\bibitem[Lacour and Massart, 2016]{Lac_Mas:2015:journal}
Lacour, C. and Massart, P. (2016).
\newblock Minimal penalty for {G}oldenshluger-{L}epski method.
\newblock {\em Stochastic Processes and their Applications},
  126(12):3774--3789.
\newblock In Memoriam: Evarist Gin\'e.

\bibitem[Lacour et~al., 2017]{Lac_Mas_Riv:2017}
Lacour, C., Massart, P., and Rivoirard, V. (2017).
\newblock Estimator selection: a new method with applications to kernel density
  estimation.
\newblock {\em Sankhya A}, 79(2):298--335.

\bibitem[Lavielle, 2005]{Lav:2005}
Lavielle, M. (2005).
\newblock {Using penalized contrasts for the change-point problem}.
\newblock {\em Signal Proces.}, 85(8):1501--1510.

\bibitem[Lawson and Hanson, 1974]{Law_Han:1974}
Lawson, C.~L. and Hanson, R.~J. (1974).
\newblock {\em Solving least squares problems}.
\newblock Prentice-Hall Inc., Englewood Cliffs, N.J.
\newblock Prentice-Hall Series in Automatic Computation.

\bibitem[Lebarbier, 2002]{Leb:2002}
Lebarbier, {\'E}. (2002).
\newblock {\em Quelques approches pour la d\'etection de ruptures \`a horizon
  fini}.
\newblock PhD thesis, Universit\'e Paris-Sud.

\bibitem[Lebarbier, 2005]{Leb:2005}
Lebarbier, {\'E}. (2005).
\newblock {Detecting multiple change-points in the mean of a {G}aussian process
  by model selection}.
\newblock {\em Signal Proces.}, 85:717--736.

\bibitem[Leh{\'e}ricy, 2018]{Leh:2018}
Leh{\'e}ricy, L. (2018).
\newblock State-by-state minimax adaptive estimation for nonparametric hidden
  {M}arkov models.
\newblock {\em Journal of Machine Learning Research}, 19(39):1--46.

\bibitem[Lerasle, 2009]{Ler:2009:phd}
Lerasle, M. (2009).
\newblock {\em R\'e\'echantillonnage et s\'election de mod\`eles optimale pour
  l'estimation de la densit\'e de variables ind\'ependantes ou m\'elangeantes}.
\newblock PhD thesis, INSA de Toulouse.
\newblock Available at \url{http://lerasle.perso.math.cnrs.fr/docs/these.pdf}.

\bibitem[Lerasle, 2010]{Ler:2010:iid:v2}
Lerasle, M. (2010).
\newblock {O}ptimal model selection in density estimation.
\newblock arXiv:0910.1654v2.

\bibitem[Lerasle, 2011]{Ler:2010:mixing}
Lerasle, M. (2011).
\newblock Optimal model selection for stationary data under various mixing
  conditions.
\newblock {\em Ann. Statist.}, 39(4):1852--1877.

\bibitem[Lerasle, 2012]{Ler:2010:iid}
Lerasle, M. (2012).
\newblock Optimal model selection in density estimation.
\newblock {\em Ann. Inst. Henri Poincar\'e Probab. Stat.}, 48(3):884--908.

\bibitem[Lerasle et~al., 2016]{Ler_Mag_Rey:2016}
Lerasle, M., Magalh\~{a}es, N., and Reynaud-Bouret, P. (2016).
\newblock Optimal kernel selection for density estimation.
\newblock In {\em {High Dimensional Probability VII: The Cargese Volume}},
  volume~71 of {\em {Progress in Probability}}, pages 425--460. Springer.
\newblock Preliminary version available at arXiv:1511.02112.

\bibitem[Lerasle and Takahashi, 2011]{Ler_Tak:2010}
Lerasle, M. and Takahashi, D.~Y. (2011).
\newblock An oracle approach for interaction neighborhood estimation in random
  fields.
\newblock {\em Electron. J. Stat.}, 5:534--571 (electronic).

\bibitem[Lerasle and Takahashi, 2016]{Ler_Tak:2011}
Lerasle, M. and Takahashi, D.~Y. (2016).
\newblock Sharp oracle inequalities and slope heuristic for specification
  probabilities estimation in discrete random fields.
\newblock {\em Bernoulli}, 22(1):325--344.

\bibitem[Letu\'e, 2000]{Let:2000}
Letu\'e, F. (2000).
\newblock {\em Mod\`ele de Cox: estimation par s\'election de mod\`ele et
  mod\`ele de chocs bivari\'e}.
\newblock PhD thesis, Universit\'e Paris-Sud.
\newblock Available at
  \url{http://www-ljk.imag.fr/membres/Frederique.Letue/These3.pdf}.

\bibitem[Li, 1985]{KCLi:1985}
Li, K.-C. (1985).
\newblock {From {S}tein's unbiased risk estimates to the method of generalized
  cross validation}.
\newblock {\em Ann. Statist.}, 13(4):1352--1377.

\bibitem[Li, 1986]{KCLi:1986}
Li, K.-C. (1986).
\newblock Asymptotic optimality of {$C\sb L$} and generalized cross-validation
  in ridge regression with application to spline smoothing.
\newblock {\em Ann. Statist.}, 14(3):1101--1112.

\bibitem[Li, 1987]{KCLi:1987}
Li, K.-C. (1987).
\newblock {Asymptotic optimality for {$C\sb p$}, {$C\sb L$}, cross-validation
  and generalized cross-validation: discrete index set}.
\newblock {\em Ann. Statist.}, 15(3):958--975.

\bibitem[Liiti{\"a}inen et~al., 2010]{Lii_Cor_Len:2010}
Liiti{\"a}inen, E., Corona, F., and Lendasse, A. (2010).
\newblock Residual variance estimation using a nearest neighbor statistic.
\newblock {\em J. Multivariate Anal.}, 101(4):811--823.

\bibitem[Liiti\"{a}inen et~al., 2009]{Lii_Ver_Cor_Len:2009}
Liiti\"{a}inen, E., Verleysen, M., Corona, F., and Lendasse, A. (2009).
\newblock Residual variance estimation in machine learning.
\newblock {\em Neurocomputing}, 72(16):3692--3703.
\newblock Financial Engineering Computational and Ambient Intelligence (IWANN
  2007).

\bibitem[Loubes and Massart, 2004]{Lou_Mas:2004}
Loubes, J.-M. and Massart, P. (2004).
\newblock {Discussion: ``Least angle regression'' [{A}nn. {S}tatist. {\bf 32}
  (2004), no. 2, 407--451] by {B.} {E}fron, {T.} {H}astie, {I.} {J}ohnstone and
  {R.} {T}ibshirani}.
\newblock {\em Ann. Statist.}, 32(2):460--465.

\bibitem[Lozano, 2000]{Loz:2000}
Lozano, F. (2000).
\newblock {Model selection using Rademacher penalization}.
\newblock In {\em {Proceedings of the 2nd ICSC Symp. on Neural Computation
  (NC2000). Berlin, Germany}}. ICSC Academic Press.

\bibitem[Lung-Yut-Fong et~al., 2015]{Lun_Lev_Cap:2015}
Lung-Yut-Fong, A., L\'evy-Leduc, C., and Capp\'e, O. (2015).
\newblock Homogeneity and change-point detection tests for multivariate data
  using rank statistics.
\newblock {\em Journal de la SFdS}, 156(4):133--162.

\bibitem[Magalh\~{a}es, 2015]{Mag:2015}
Magalh\~{a}es, N. (2015).
\newblock {\em Cross-Validation and Penalization for Density Estimation}.
\newblock PhD thesis, {Universit{\'e} Paris Sud - Paris XI}.
\newblock Available at \url{https://tel.archives-ouvertes.fr/tel-01164581v1}.

\bibitem[Mallows, 1973]{Mal:1973}
Mallows, C.~L. (1973).
\newblock {Some comments on ${C}_p$}.
\newblock {\em Technometrics}, 15:661--675.

\bibitem[Mammen and Tsybakov, 1999]{Mam_Tsy:1999}
Mammen, E. and Tsybakov, A.~B. (1999).
\newblock {Smooth discrimination analysis}.
\newblock {\em Ann. Statist.}, 27(6):1808--1829.

\bibitem[Massart, 2005]{Mas:2005}
Massart, P. (2005).
\newblock {A non-asymptotic theory for model selection}.
\newblock In {\em {European Congress of Mathematics}}, pages 309--323. Eur.
  Math. Soc., Z{\"u}rich.

\bibitem[Massart, 2007]{Mas:2003:St-Flour}
Massart, P. (2007).
\newblock {\em {Concentration Inequalities and Model Selection}}, volume 1896
  of {\em {Lecture Notes in Mathematics}}.
\newblock Springer, Berlin.
\newblock Lectures from the 33rd Summer School on Probability Theory held in
  Saint-Flour, July 6--23, 2003, With a foreword by Jean Picard.

\bibitem[Massart, 2008]{Mas:2008}
Massart, P. (2008).
\newblock S\'election de mod\`eles: de la th\'eorie \`a la pratique.
\newblock {\em Journal de la SFdS}, 149(4):5--28.

\bibitem[Massart and N{\'e}d{\'e}lec, 2006]{Mas_Ned:2003}
Massart, P. and N{\'e}d{\'e}lec, {\'E}. (2006).
\newblock {Risk bounds for statistical learning}.
\newblock {\em Ann. Statist.}, 34(5):2326--2366.

\bibitem[Matias and Miele, 2017]{Mat_Mie:2017}
Matias, C. and Miele, V. (2017).
\newblock Statistical clustering of temporal networks through a dynamic
  stochastic block model.
\newblock {\em Journal of the Royal Statistical Society: Series B (Statistical
  Methodology)}, 79(4):1119--1141.

\bibitem[Maugis, 2008]{Mau:2008:phd}
Maugis, C. (2008).
\newblock {\em Variable selection for model-based clustering. Application for
  transcriptome data analysis}.
\newblock PhD thesis, Universit\'e Paris-Sud.
\newblock Available at \url{https://tel.archives-ouvertes.fr/tel-00344120v1}.

\bibitem[Maugis and Michel, 2011a]{Mau_Mic:2010}
Maugis, C. and Michel, B. (2011a).
\newblock Data-driven penalty calibration: A case study for {G}aussian model
  selection.
\newblock {\em ESAIM Probab. Stat.}, 15:320--339.

\bibitem[Maugis and Michel, 2011b]{Mau_Mic:2008}
Maugis, C. and Michel, B. (2011b).
\newblock A non asymptotic penalized criterion for gaussian mixture model
  selection.
\newblock {\em ESAIM Probab. Stat.}, 15:41--68.

\bibitem[Mendelson, 2018]{Men:2018}
Mendelson, S. (2018).
\newblock Learning without concentration for general loss functions.
\newblock {\em Probab. Theory Related Fields}, 171(1-2):459--502.

\bibitem[Meynet and Maugis-Rabusseau, 2012]{Mey_Mau:2012}
Meynet, C. and Maugis-Rabusseau, C. (2012).
\newblock {A sparse variable selection procedure in model-based clustering}.
\newblock Available at \url{https://hal.inria.fr/hal-00734316v1}.

\bibitem[Michel, 2008]{Mic:2008:phd}
Michel, B. (2008).
\newblock {\em Mod\'elisation de la production d'hydrocarbures dans un bassin
  p\'etrolier}.
\newblock PhD thesis, Universit\'e Paris-Sud.
\newblock Available at \url{http://tel.archives-ouvertes.fr/tel-00345753v1}.

\bibitem[Miller, 1970]{Mil:1970}
Miller, K. (1970).
\newblock Least squares methods for ill-posed problems with a prescribed bound.
\newblock {\em SIAM J. Math. Anal.}, 1:52--74.

\bibitem[M{\"u}ller et~al., 2003]{Mul_Sch_Wef:2003}
M{\"u}ller, U.~U., Schick, A., and Wefelmeyer, W. (2003).
\newblock Estimating the error variance in nonparametric regression by a
  covariate-matched {$U$}-statistic.
\newblock {\em Statistics}, 37(3):179--188.

\bibitem[Muro and Geer, 2018]{Mur_vdG:2018}
Muro, A. and Geer, S. (2018).
\newblock Concentration behavior of the penalized least squares estimator.
\newblock {\em Statistica Neerlandica}, 72(2):109--125.

\bibitem[Navarro and Saumard, 2017]{Sau_Nav:2017}
Navarro, F. and Saumard, A. (2017).
\newblock Slope heuristics and {V}-fold model selection in heteroscedastic
  regression using strongly localized bases.
\newblock {\em ESAIM Probab. Stat.}, 21:412--451.

\bibitem[Oueslati and Lopez, 2013]{Oue_Lop:2013}
Oueslati, A. and Lopez, O. (2013).
\newblock A proportional hazards regression model with change-points in the
  baseline function.
\newblock {\em Lifetime Data Analysis}, 19(1):59--78.

\bibitem[Ramosaj and Pauly, 2019]{Ram_Pau:2018:journal}
Ramosaj, B. and Pauly, M. (2019).
\newblock Consistent estimation of residual variance with random forest
  out-of-bag errors.
\newblock {\em Statistics \& Probability Letters}, 151:49--57.

\bibitem[Rau et~al., 2015]{Rau_Mau_Mag_Cel:2015}
Rau, A., Maugis-Rabusseau, C., Martin-Magniette, M.-L., and Celeux, G. (2015).
\newblock Co-expression analysis of high-throughput transcriptome sequencing
  data with {P}oisson mixture models.
\newblock {\em Bioinformatics}, 31(9):1420--1427.

\bibitem[Regi{\'n}ska, 1996]{Reg:1996}
Regi{\'n}ska, T. (1996).
\newblock A regularization parameter in discrete ill-posed problems.
\newblock {\em SIAM J. Sci. Comput.}, 17(3):740--749.

\bibitem[Reichel and Rodriguez, 2013]{Rei_Rod:2013}
Reichel, L. and Rodriguez, G. (2013).
\newblock Old and new parameter choice rules for discrete ill-posed problems.
\newblock {\em Numerical Algorithms}, 63(1):65--87.

\bibitem[Reid et~al., 2016]{Rei_Tib_Fri:2013:journal}
Reid, S., Tibshirani, R., and Friedman, J. (2016).
\newblock A study of error variance estimation in {L}asso regression.
\newblock {\em Statist. Sinica}, 26(1):35--67.

\bibitem[Reynaud-Bouret and Rivoirard, 2010]{Rey_Riv:2010}
Reynaud-Bouret, P. and Rivoirard, V. (2010).
\newblock Near optimal thresholding estimation of a {P}oisson intensity on the
  real line.
\newblock {\em Electron. J. Stat.}, 4:172--238 (electronic).

\bibitem[Reynaud-Bouret et~al., 2011]{Rey_Riv_Tul:2009}
Reynaud-Bouret, P., Rivoirard, V., and Tuleau-Malot, C. (2011).
\newblock Adaptive density estimation: a curse of support?
\newblock {\em J. Statist. Plann. Inference}, 141(1):115--139.

\bibitem[Reynaud-Bouret and Schbath, 2010]{Rey_Sch:2010}
Reynaud-Bouret, P. and Schbath, S. (2010).
\newblock Adaptive estimation for {H}awkes processes; application to genome
  analysis.
\newblock {\em Ann. Statist.}, 38(5):2781--2822.

\bibitem[Rice, 1984]{Ric:1984}
Rice, J. (1984).
\newblock Bandwidth choice for nonparametric regression.
\newblock {\em Ann. Statist.}, 12(4):1215--1230.

\bibitem[Roche, 2014]{Roc:2014}
Roche, A. (2014).
\newblock {\em Statistical modeling for functional data: non-asymptotic
  approaches and adaptive methods}.
\newblock PhD thesis, {Universit{\'e} Montpellier II - Sciences et Techniques
  du Languedoc}.
\newblock Available at \url{https://tel.archives-ouvertes.fr/tel-01023919v1}.

\bibitem[Rozenholc, 2012]{Roz:2012}
Rozenholc, Y. (2012).
\newblock Statistical base jumping: A simple and fully data-driven answer to
  penalized model selection.
\newblock S\'eminaire de Statistique du MAP5, February 3rd.

\bibitem[Saumard, 2010a]{Sau:2010:supnorm}
Saumard, A. (2010a).
\newblock {C}onvergence in sup-norm of least-squares estimators in regression
  with random design and nonparametric heteroscedastic noise.
\newblock Available at \url{http://hal.archives-ouvertes.fr/hal-00528539v2}.

\bibitem[Saumard, 2010b]{Sau:2010:phd}
Saumard, A. (2010b).
\newblock {\em Estimation par Minimum de Contraste R\'egulier et Heuristique de
  Pente en S\'election de Mod\`eles}.
\newblock PhD thesis, Universit\'e de Rennes 1.
\newblock Available at \url{http://tel.archives-ouvertes.fr/tel-00569372v1}.

\bibitem[Saumard, 2010c]{Sau:2010:MLE}
Saumard, A. (2010c).
\newblock {N}onasymptotic quasi-optimality of {AIC} and the slope heuristics in
  maximum likelihood estimation of density using histogram models.
\newblock Available at \url{https://hal.archives-ouvertes.fr/hal-00512310v1}.

\bibitem[Saumard, 2012]{Sau:2010:conc-p1-p2}
Saumard, A. (2012).
\newblock Optimal upper and lower bounds for the true and empirical excess
  risks in heteroscedastic least-squares regression.
\newblock {\em Electron. J. Stat.}, 6:579--655.

\bibitem[Saumard, 2013]{Sau:2010:Reg}
Saumard, A. (2013).
\newblock Optimal model selection in heteroscedastic regression using piecewise
  polynomial functions.
\newblock {\em Electron. J. Stat.}, 7:1184--1223.

\bibitem[Saumard, 2017]{Sau:2017}
Saumard, A. (2017).
\newblock A concentration inequality for the excess risk in least-squares
  regression with random design and heteroscedastic noise.
\newblock arXiv:1702.05063v2.

\bibitem[Saumard and Navarro, 2018]{Sau_Nav:2018:arXivv4}
Saumard, A. and Navarro, F. (2018).
\newblock Finite sample improvement of {A}kaike's information criterion.
\newblock arXiv:1803.02078v4.

\bibitem[Schwarz, 1978]{Sch:1978}
Schwarz, G. (1978).
\newblock {Estimating the dimension of a model}.
\newblock {\em Ann. Statist.}, 6(2):461--464.

\bibitem[Shao, 1997]{Sha:1997}
Shao, J. (1997).
\newblock {An asymptotic theory for linear model selection}.
\newblock {\em Statist. Sinica}, 7(2):221--264.
\newblock With comments and a rejoinder by the author.

\bibitem[Solnon, 2013]{Sol:2013:phd}
Solnon, M. (2013).
\newblock {\em Apprentissage statistique multi-t\^aches}.
\newblock PhD thesis, Universit\'e Pierre et Marie Curie - Paris VI.
\newblock Available at \url{https://hal.inria.fr/tel-00911498v1}.

\bibitem[Solnon et~al., 2012]{Sol_Arl_Bac:2011:multitask}
Solnon, M., Arlot, S., and Bach, F. (2012).
\newblock Multi-task regression using minimal penalties.
\newblock {\em J. Mach. Learn. Res.}, 13:2773--2812 (electronic).

\bibitem[Sorba, 2017]{Sor:2017}
Sorba, O. (2017).
\newblock {\em {Minimal penalties for model selection}}.
\newblock PhD thesis, {Universit{\'e} Paris-Saclay}.
\newblock Available at \url{https://tel.archives-ouvertes.fr/tel-01515957v1}.

\bibitem[Spokoiny, 2002]{Spo:2002}
Spokoiny, V. (2002).
\newblock Variance estimation for high-dimensional regression models.
\newblock {\em J. Multivariate Anal.}, 82(1):111--133.

\bibitem[Spokoiny, 2012]{Spo:2012a}
Spokoiny, V. (2012).
\newblock Parametric estimation. finite sample theory.
\newblock {\em Ann. Statist.}, 40(6):2877--2909.

\bibitem[Spokoiny, 2017]{Spo:2012b:journal}
Spokoiny, V. (2017).
\newblock Penalized maximum likelihood estimation and effective dimension.
\newblock {\em Ann. Inst. Henri Poincar\'e Probab. Stat.}, 53(1):389--429.

\bibitem[Stein, 1981]{Ste:1981}
Stein, C.~M. (1981).
\newblock {Estimation of the mean of a multivariate normal distribution}.
\newblock {\em Ann. Statist.}, 9(6):1135--1151.

\bibitem[Stone, 1974]{Sto:1974}
Stone, M. (1974).
\newblock {Cross-validatory choice and assessment of statistical predictions}.
\newblock {\em J. Roy. Statist. Soc. Ser. B}, 36:111--147.
\newblock With discussion by G. A. Barnard, A. C. Atkinson, L. K. Chan, A. P.
  Dawid, F. Downton, J. Dickey, A. G. Baker, O. Barndorff-Nielsen, D. R. Cox,
  S. Giesser, D. Hinkley, R. R. Hocking, and A. S. Young, and with a reply by
  the authors.

\bibitem[Sugar and James, 2003]{Sug_Jam:2003}
Sugar, C.~A. and James, G.~M. (2003).
\newblock Finding the number of clusters in a dataset: an information-theoretic
  approach.
\newblock {\em J. Amer. Statist. Assoc.}, 98(463):750--763.

\bibitem[Tibshirani et~al., 2001]{Tib_Wal_Has:2001}
Tibshirani, R., Walther, G., and Hastie, T. (2001).
\newblock Estimating the number of clusters in a data set via the gap
  statistic.
\newblock {\em J. R. Stat. Soc. Ser. B Stat. Methodol.}, 63(2):411--423.

\bibitem[Tibshirani and Taylor, 2012]{Tib_Tay:2011b:journal}
Tibshirani, R.~J. and Taylor, J. (2012).
\newblock Degrees of freedom in lasso problems.
\newblock {\em Ann. Statist.}, 40(2):1198--1232.

\bibitem[Tong et~al., 2013]{Ton_Ma_Wan:2013}
Tong, T., Ma, Y., and Wang, Y. (2013).
\newblock Optimal variance estimation without estimating the mean function.
\newblock {\em Bernoulli}, 19(5A):1839--1854.

\bibitem[Ullah and Zinde-Walsh, 1992]{Ull_Zin:1992}
Ullah, A. and Zinde-Walsh, V. (1992).
\newblock On the estimation of residual variance in nonparametric regression.
\newblock {\em J. Nonparametr. Statist.}, 1(3):263--265.

\bibitem[Vaiter et~al., 2012]{Vai_etal:2012}
Vaiter, S., Deledalle, C., Peyr{\'e}, G., Fadili, J.~M., and Dossal, C. (2012).
\newblock {The Degrees of Freedom of the Group Lasso}.
\newblock In {\em {International Conference on Machine Learning Workshop
  (ICML)}}, Edinburgh, United Kingdom.
\newblock Available at \url{https://hal.archives-ouvertes.fr/hal-00695292}.

\bibitem[van~de Geer and Wainwright, 2017]{vdG_Wai:2017}
van~de Geer, S. and Wainwright, M.~J. (2017).
\newblock On concentration for (regularized) empirical risk minimization.
\newblock {\em Sankhya A}, 79(2):159--200.

\bibitem[{van Erven} et~al., 2012]{Erv_Gru_Roo:2008}
{van Erven}, T., Gr{\"u}nwald, P.~D., and {de Rooij}, S. (2012).
\newblock Catching up faster by switching sooner: a predictive approach to
  adaptive estimation with an application to the {AIC-BIC} dilemma.
\newblock {\em Journal of the Royal Statistical Society: Series B (Statistical
  Methodology)}, 74(3):361--417.

\bibitem[van Handel, 2011]{vHa:2011}
van Handel, R. (2011).
\newblock On the minimal penalty for {M}arkov order estimation.
\newblock {\em Probability Theory and Related Fields}, 150(3):709--738.

\bibitem[Varet et~al., 2019]{Var_Lac_Mas_Riv:2019}
Varet, S., Lacour, C., Massart, P., and Rivoirard, V. (2019).
\newblock Numerical performance of penalized comparison to overfitting for
  multivariate kernel density estimation.
\newblock Technical report, arXiv.
\newblock arXiv:1902.01075v1.

\bibitem[Vert, 2006]{Vert:2006:phd}
Vert, R. (2006).
\newblock {\em Theoretical Insights on Density Level Set Estimation,
  Application to Anomaly Detection}.
\newblock PhD thesis, Universit\'e Paris Sud.
\newblock Available at
  \url{http://sites.google.com/site/regisvert/Home/publications/files-1/thesis.pdf}.

\bibitem[Verzelen, 2010]{Ver:2010}
Verzelen, N. (2010).
\newblock Data-driven neighborhood selection of a {G}aussian field.
\newblock {\em Comput. Statist. Data Anal.}, 54(5):1355--1371.

\bibitem[Villers, 2007]{Vil:2007:these}
Villers, F. (2007).
\newblock {\em {Tests et S{\'e}lection de Mod{\`e}les pour l'Analyse de
  Donn{\'e}es Prot{\'e}omiques et Transcriptomiques}}.
\newblock PhD thesis, University Paris XI.
\newblock Available at
  \url{http://www.proba.jussieu.fr/~villers/manuscript.pdf}.

\bibitem[Vogel, 1996]{Vog:1996}
Vogel, C.~R. (1996).
\newblock Non-convergence of the {$L$}-curve regularization parameter selection
  method.
\newblock {\em Inverse Problems}, 12(4):535--547.

\bibitem[Wahba, 1977]{Wah:1977}
Wahba, G. (1977).
\newblock A survey of some smoothing problems and the method of generalized
  cross-validation for solving them.
\newblock In {\em Applications of statistics ({P}roc. {S}ympos., {W}right
  {S}tate {U}niv., {D}ayton, {O}hio, 1976)}, pages 507--523. North-Holland,
  Amsterdam.

\bibitem[Wilks, 1938]{Wil:1938}
Wilks, S.~S. (1938).
\newblock The large-sample distribution of the likelihood ratio for testing
  composite hypotheses.
\newblock {\em Ann. Math. Statistics}, 9:60--62.

\bibitem[Yang, 2005]{Yan:2005a}
Yang, Y. (2005).
\newblock {Can the strengths of {AIC} and {BIC} be shared? {A} conflict between
  model indentification and regression estimation}.
\newblock {\em Biometrika}, 92(4):937--950.

\bibitem[Zwald, 2005]{Zwa:2005:phd}
Zwald, L. (2005).
\newblock {\em Statistical performances of learning algorithm : Kernel
  Projection Machine and Kernel Principal Component Analysis}.
\newblock PhD thesis, Universit\'e Paris Sud.
\newblock Available at \url{http://tel.archives-ouvertes.fr/tel-00012011v1}.

\end{thebibliography}

%%%%%%%%%%%%%%%%%%%%%%%%%%%%%%%%%%%%%%%%%%%%%%%%%%%%%%%%%%%%%%%%%%%%%%%%%%%%%%%%%%%%
%%%%%%%%%%%%%%%%%%%%%%%%%%%%%%%%%%%%%%%%%%%%%%%%%%%%%%%%%%%%%%%%%%%%%%%%%%%%%%%%%%%%
%%%%%%%%%%%%%%%%%%%%%%%%%%%%%%%%%%%%%%%%%%%%%%%%%%%%%%%%%%%%%%%%%%%%%%%%%%%%%%%%%%%%
%%%%%%%%%%%%%%%%%%%%%%%%%%%%%%%%%%%%%%%%%%%%%%%%%%%%%%%%%%%%%%%%%%%%%%%%%%%%%%%%%%%%

\appendix

%%%%%%%%%%%%%%%%%%%%%%%%%%%%%%%%%%%%%%%%%%%%%%%%%%%%%%%%%%%%%%%%%%%%%%%%%%%%%%%%%%%%
%%%%%%%%%%%%%%%%%%%%%%%%%%%%%%%%%%%%%%%%%%%%%%%%%%%%%%%%%%%%%%%%%%%%%%%%%%%%%%%%%%%%

\section{Some proofs}

\subsection{Proof of Proposition~\ref{pro.pbpenmin.gal.below}}
\label{app.proof.pro.pbpenmin.gal.below}
\paragraph{Proof of Eq.~\eqref{eq.proofalt.penmin.Cpt-Dgrd} and~\eqref{eq.proofalt.penmin.Cpt-Dgrd.2}}
By definition of $\mhgalzero(C)$, for every $\mM$,
\begin{align*}
\Risk\parenB{\bayes_{\mhgalzero(C)}}  - \ovdelta\parenB{ \mhgalzero(C) } + (C-1) p_2\parenB{ \mhgalzero(C) }
&\leq
\Risk\parens{\bayes_{m}} - \ovdelta(m) + (C-1) p_2(m)
\end{align*}
hence
\begin{align*}
(1-C) p_2\parenB{ \mhgalzero(C) }
&\geq
\Risk\parenB{\bayes_{\mhgalzero(C)}}-\Risk\parens{\bayes}
- \crochb{ \Risk\parens{\bayes_{m}}-\Risk\parens{\bayes} }
\\
&\qquad
+ \ovdelta(m) - \ovdelta\parenB{ \mhgalzero(C) }
+ (1-C) p_2(m)
\end{align*}
which implies, using Eq.~\eqref{eq.altproof.controldelta}, that
\begin{align*}
(1-C) p_2\parenB{ \mhgalzero(C) }
\geq
- 2 \crochb{ \Risk\parens{\bayes_{m}}-\Risk\parens{\bayes} }
+ (1-C) p_2(m)
- \varepsilon^{\prime}_{\delta}
\end{align*}
hence Eq.~\eqref{eq.proofalt.penmin.Cpt-Dgrd} by dividing by $(1-C)$.
Eq.~\eqref{eq.proofalt.penmin.Cpt-Dgrd.2} is a straightforward consequence of Eq.~\eqref{eq.proofalt.penmin.Cpt-Dgrd}
since $p_2(m_1)>0$.

\paragraph{General proof of Eq.~\eqref{eq.altproof.controldelta}}
Let us assume that $\xi_1, \ldots, \xi_n \in \X$ are i.i.d.\@ and some contrast function $\gamma: \Xi \times \Set \to \R $ and constants $A,L>0$ exist such that Eq.~\eqref{hyp.gal-controldelta.contrast}--\eqref{hyp.gal-controldelta.margin} hold true. 
Then, for every $x \geq 0$, with probability at least $1 - 2 \card(\M) \mathrm{e}^{-x}$,
for every $\theta >0$,
Eq.~\eqref{eq.altproof.controldelta} holds with
\[
\varepsilon_{\delta} = \theta
\quad \text{and} \quad
\varepsilon^{\prime}_{\delta} = 2 \paren{ \frac{L}{2 \theta} + \frac{2 A}{3}} \frac{x}{n}
\, .
\]

Indeed, for every fixed $\mM$,
\[
\ovdelta(m) - \crochb{ \Risk\parens{\bayes} - \Remp\parens{\bayes} }
= \frac{1}{n} \sum_{i=1}^n \parenb{ X_{i,m} - \E\crochs{X_{i,m}} }
\quad \text{where} \quad
X_{i,m} = \gamma\paren{\xi_i , \bayes} - \gamma\paren{\xi_i , \bayes_m}
\]
are i.i.d.\@ random variables satisfying
$\absj{X_{i,m}} \leq 2 A$ almost surely. 
Therefore, by Bernstein's inequality \citep[Theorem~2.10]{Bou_Lug_Mas:2011:livre}, for every $x \geq 0$, with probability at least $1-2 \mathrm{e}^{-x}$,
\begin{align*}
\absB{\ovdelta(m) - \crochb{ \Risk\parens{\bayes} - \Remp\parens{\bayes} } }
&\leq
\sqrt{ \frac{2 x \var(X_{i,m})}{n} } + \frac{2 A x}{3 n}
\leq \theta \parenb{ \Risk\parens{\bayes_m} - \Risk\parens{\bayes}} + \paren{ \frac{L}{2 \theta} + \frac{2 A}{3}} \frac{x}{n}
\end{align*}
and the result follows by the union bound.
\qed

\medbreak

\paragraph{Extension} 
Note that assuming only that $(1-\varepsilon_0) p_2(m) \leq \pen_0(m) \leq (1+\varepsilon_0) p_2(m)$ 
for some $\varepsilon_0 \in [0,1)$ with $C(1-\varepsilon_0) < 1$ 
---which implies $p_2(m) \geq 0$---, 
instead of Eq.~\eqref{eq.proofalt.penmin.Cpt-Dgrd} we get that
\begin{align*}
p_2\parenB{ \mhgalzero(C) } \geq \sup_{\mM} \setj{ \frac{ 1 - (1+\varepsilon_0) C}{ 1 - (1-\varepsilon_0) C} p_2(m) - \frac{2 \crochb{ \Risk\paren{\bayes_{m}}-\Risk\parens{\bayes} }}{ 1 - (1-\varepsilon_0) C}  } - \frac{\varepsilon^{\prime}_{\delta}}{ 1 - (1-\varepsilon_0) C}
\, .
\end{align*}
If in addition some $m_1 \in \M$ exists such that
$\Risk\parens{\bayes_{m_1}} = \Risk\parens{\bayes}$ and $p_2(m_1)>0$,
we get that for any $\alpha \in (0,1)$, 
\begin{align*}
p_2\parenB{ \mhgalzero(C) }
&\geq (1-\alpha) p_2(m_1)
\end{align*}
if
\[
C \leq 1 - \eta_{\alpha}
\qquad \text{where} \qquad
 \eta_{\alpha} \egaldef
1 - \parenj{ 1 - \frac{\varepsilon^{\prime}_{\delta}}{\alpha p_2(m_1)} }
\mathopen{}\croch{ 1 + \varepsilon_0 \parenj{ \frac{2}{\alpha} - 1 }}^{-1} \mathclose{}
\, .
\]

\subsection{Proof of Proposition~\ref{pro.pbpenmin.gal.above}}
\label{app.proof.pro.pbpenmin.gal.above}
By definition of $\mhgalzero(C)$, for every $\mM$ and $C>1$,
\begin{align*}
& \quad
\Risk\parenB{\sh_{\mhgalzero(C)}} - p_1\parenB{\mhgalzero(C)} - \ovdelta\parenB{\mhgalzero(C)} + (C-1) p_2\parenB{\mhgalzero(C)}
\\
&\leq \Risk\parens{\sh_m} - p_1(m) - \ovdelta(m) + (C-1) p_2(m)
\end{align*}
hence, using Eq.~\eqref{eq.altproof.controlp1p2}, 
\begin{align*}
& \quad
\Risk\parenB{\sh_{\mhgalzero(C)}} - \Risk\parens{\bayes}
+ \crochb{ (C-1) (1-\varepsilon_p) - 1} p_1\parenB{\mhgalzero(C)}
\\
&\leq
\Risk\parens{\sh_m} - \Risk\parens{\bayes}
+ \crochb{ (C-1) (1+\varepsilon_p) - 1 } p_1(m)
+ \ovdelta\parenB{\mhgalzero(C)}- \ovdelta(m)
\, .
\end{align*}
By Eq.~\eqref{eq.altproof.controldelta},
we get that for every $\mM$ and $C>1$, 
\begin{align*}
& \quad
\Risk\parenB{\sh_{\mhgalzero(C)}} - \Risk\parens{\bayes}
+ \crochb{ (C-1) (1-\varepsilon_p) - 1} p_1\parenB{\mhgalzero(C)}
- \varepsilon_{\delta} \crochbb{ \Risk\parenB{\bayes_{\mhgalzero(C)}} - \Risk\parens{\bayes}  }
\\
&\leq
\Risk\parens{\sh_m} - \Risk\parens{\bayes}
+ \crochb{ (C-1) (1+\varepsilon_p) - 1 } p_1(m)
+ \varepsilon_{\delta} \crochb{ \Risk\parens{\bayes_{m}} - \Risk\parens{\bayes}  }
+ \varepsilon^{\prime}_{\delta}
\end{align*}
that is,
\begin{align*}
& \quad
\parenB{ 1 - \max\setb{  1 - (C-1) (1-\varepsilon_p)  \, , \, \varepsilon_{\delta} }}
\crochbb{ \Risk\parenB{\sh_{\mhgalzero(C)}} - \Risk\parens{\bayes} }
\\
&\leq
\parenB{ 1 + \max\setb{ (C-1) (1+\varepsilon_p) - 1 \, , \, \varepsilon_{\delta} }}
\inf_{\mM} \setb{ \Risk\parens{\sh_m} - \Risk\parens{\bayes} }
+ \varepsilon^{\prime}_{\delta}
\end{align*}
which proves Eq.~\eqref{eq.proofalt.penmin.Cgrd-Rpt}.
Using again Eq.~\eqref{eq.altproof.controlp1p2}, we get
\begin{align*}
p_2\parenB{\mhgalzero(C)}
&\leq (1+\varepsilon_p) p_1\parenB{\mhgalzero(C)}
\leq (1+\varepsilon_p)  \crochbb{ \Risk\parenB{\sh_{\mhgalzero(C)}} - \Risk\parens{\bayes} }
\\
&\leq
\cteProPenminAbove(C) (1+\varepsilon_p)
\inf_{\mM} \setb{ \Risk\parens{\sh_m} - \Risk\parens{\bayes} }
+ \cteProPenminAbove^{\prime}(C) (1+\varepsilon_p)
\, .
\qed
\end{align*}

\paragraph{Extension}
Note that if $\pen_0(m)=p_2(m)$ is replaced by 
\[ 
(1-\varepsilon_0) p_2(m) \leq \pen_0(m) \leq (1+\varepsilon_0) p_2(m) 
\] 
for some $\varepsilon_0 \geq 0$ with $C(1-\varepsilon_0)>1$,
and if Eq.~\eqref{eq.altproof.controlp1p2} is replaced by
\begin{equation}
\label{eq.altproof.controlp1p2.alt}
\forall m \in \M \, , \quad
-\varepsilon_p^{\prime} + \varepsilon_p^- p_1(m) \leq p_2(m) \leq  \varepsilon_p^+ p_1(m) + \varepsilon_p^{\prime}
\end{equation}
for some $\varepsilon_p^-,\varepsilon_p^+>0$ and $\varepsilon_p^{\prime} \geq 0$, the same proof shows that  Eq.~\eqref{eq.proofalt.penmin.Cgrd-Rpt} holds true with
\begin{align*}
\cteProPenminAbove(C)
&\egaldef \frac{ \max\set{ [ C (1+\varepsilon_0) -1 ] \varepsilon_p^+  \, , \, 1 + \varepsilon_{\delta} } }{ \min\set{  [ C (1-\varepsilon_0) -1 ] \varepsilon_p^-  \, , \, 1 - \varepsilon_{\delta} }   }
\notag
\\
\text{and} \qquad
\cteProPenminAbove^{\prime}(C)
&\egaldef \frac{ \varepsilon^{\prime}_{\delta} + 2(C-1) \varepsilon_p^{\prime} }{ \min\set{  [ C (1-\varepsilon_0) -1 ] \varepsilon_p^-  \, , \, 1 - \varepsilon_{\delta} }  }
\, ,
\notag
\end{align*}
and Eq.~\eqref{eq.proofalt.penmin.Cgrd-Dpt} replaced by
\begin{equation}
\label{eq.proofalt.penmin.Cgrd-Dpt.alt}
p_2\parenB{\mhgalzero(C)}
\leq
\cteProPenminAbove(C) \varepsilon_p^+
\inf_{\mM} \setb{ \Risk\parens{\sh_m} - \Risk\parens{\bayes} }
+ \cteProPenminAbove^{\prime}(C) \varepsilon_p^+ + \varepsilon_p^{\prime}
\, .
\end{equation}

\subsection{Proof of Proposition~\ref{pro.variance-estim}}% mean squared error of $\Chthr$
\label{app.MSE-Chthr}
We first state two general lemmas for $\Chthr$ and $\Chwindow\,$, 
as defined by Eq. \eqref{def.Chthr} and~\eqref{def.Chwin} in Section~\ref{sec.slopeOLS.math}, respectively. 
These lemmas do not assume a specific definition for $\mh(C)$ and $D_m\,$, 
so they apply to Algorithms \ref{algo.OLS.jump}, \ref{algo.gal.slope.naif}, \ref{algo.penmin.linear}, 
and~\ref{algo.penmingal} (possibly up to a rescaling of $\C_m$ for Lemma~\ref{le.Chwindow} 
and Algorithm~\ref{algo.penmingal}). 
\begin{lemma}\label{le.Chthr}
Let $T_n \in \R$, 
$\mh: [0,+\infty) \to \M$ some function, 
and $\Chthr(T_n) \egaldef \inf \setj{ C \geq 0 \, / \, D_{\mh(C)} \leq T_n }$. 
\\ 
If $\Gamma_1 \geq 0$ and $D_{\mh(C)} > T_n$ for all $C < \Gamma_1\,$, then, $\Chthr(T_n) \geq \Gamma_1\,$.
\\ 
If $\Gamma_2 \geq 0$ and $D_{\mh(\Gamma_2)} \leq T_n\,$, then, $\Chthr(T_n) \leq \Gamma_2\,$.
\end{lemma}
\begin{proof}
The proof is straightforward from the definition of $\Chthr\,$.
\end{proof}
\begin{lemma}\label{le.Chwindow} 
Let $\eta>0$, $\mh: [0,+\infty) \to \M$ some function, 
and 
\[ 
\Chwindow(\eta) \in \argmax_{C \geq 0} \set{ D_{\mh(C/[1+\eta])} - D_{\mh(C[1+\eta])} }
\, . 
\]
Assume that $0 \leq D_m \leq n$ for every $\mM$. 
Assume in addition that $a_n, b_n \in \R$ and $\Gamma_2>\Gamma_1>0$ exist such that 
$a_n - b_n > \max\sets{ n-a_n, b_n }$, 
\[
\forall C \leq \Gamma_1, \quad D_{\mh(C)} \geq a_n
\, ,
\qquad
\forall C \geq \Gamma_2, \quad D_{\mh(C)} \leq b_n
\, ,
\qquad \text{and} \qquad
(1+\eta)^2 \geq \frac{\Gamma_2}{\Gamma_1}
\, .
\]
Then, we have 
\[
\frac{\Gamma_1}{1+\eta} < \Chwindow(\eta) < \Gamma_2 (1+\eta)
\, .
\]
\end{lemma}
\begin{proof}
First, for $C = \sqrt{\Gamma_1 \Gamma_2}$, we have 
$C/(1+\eta) \leq \Gamma_1$ and $C(1+\eta)\geq \Gamma_2\,$,
hence we obtain $D_{\mh \parens{C/[1+\eta] }} - D_{\mh\parens{ C[1+\eta] }} \geq a_n - b_n\,$.
Second, for every $C \leq \Gamma_1 / (1+\eta)$,
we have the inequality  $D_{\mh \parens{C/[1+\eta] }} - D_{\mh\parens{ C[1+\eta] }} \leq n-a_n\,$.
Third, for any $C \geq \Gamma_2 (1+\eta) $,
$D_{\mh \parens{C/[1+\eta] }} - D_{\mh\parens{ C[1+\eta] }} \leq b_n\,$.
The result follows since $a_n - b_n > \max\sets{ n-a_n, b_n }$. 
\end{proof}

\medbreak

Let us now prove Proposition~\ref{pro.variance-estim}.

By Eq.~\eqref{eq.pr.thm.OLS.Cpt-Dgrd.4} and Eq.~\eqref{eq.pr.thm.OLS.Cpt-Dpt.4} 
in the proof of Theorem~\ref{thm.OLS}, 
for all $x \geq 0$, on the event $\Omega_x$ which has a probability larger than
$1-4\card(\M) \mathrm{e}^{-x}$,
we have
\begin{equation} 
\label{eq.pr.pro.variance-estim.1}
\forall C \leq C_1(x ; a_n), \quad D_{\mh(C)} \geq a_n
\qquad \text{and} \qquad
\forall C \geq C_2(x ; b_n ; c_n), \quad D_{\mh(C)} \leq b_n
\, ,
\end{equation} 
whatever $0 \leq c_n < b_n \leq n$ and $0 \leq a_n < n$, 
provided that $\M$ contains at least one model of dimension at most $c_n\,$. 

\paragraph{Proof of Eq.~\eqref{eq.bound-Chwindow}}
First, note that 
\[ 
C_1 \parenj{ x;\frac{2n}{3} }
= \sigma^2 \crochj{ 1 - \parenj{12 \sqrt{\frac{x}{n}} + 18 \frac{x}{n} } }
> 0 
\]
since $\sqrt{x/n} \leq (\sqrt{6}-2)/6$. 
Therefore, by continuity of $C_1$ and using that $c_n < n/3$, 
there exists some $\epsilon_1 \in (0, \min\{ 1/2 , 1 - 3 c_n/ n \} )$ such that 
$C_1( x ; \frac{2n}{3} (1+\epsilon))>0$ 
for any $\epsilon \in [0,\epsilon_1]$. 
Let us take any $\epsilon \in (0,\epsilon_1]$, $a_n = \frac{2n}{3} (1+\epsilon) \in (2n/3,n)$, and $b_n = \frac{n}{3} (1-\epsilon) \in (c_n,n/3)$. 
By Lemma~\ref{le.Chwindow} with $a_n,b_n$ as above, 
$\Gamma_1 = C_1(x;a_n)$, and $\Gamma_2 = C_2(x ; b_n ; c_n)$, 
on $\Omega_x\,$, 
since Eq.~\eqref{eq.pr.pro.variance-estim.1} holds true, 
we get that 
\[ 
\forall \eta \geq \sqrt{ \frac{C_2\parenj{ x ; \frac{n}{3} (1-\epsilon) ; c_n}}{C_1\parenj{ x; \frac{2n}{3} (1+\epsilon) }} } - 1 \, , 
\quad 
\frac{C_1 \parenj{ x ; \frac{2n}{3} (1+\epsilon) }}{1+\eta} 
< \Chwindow (\eta)
< C_2 \parenj{ x ; \frac{n}{3} (1-\epsilon) ; c_n } (1+\eta) 
\, . 
\]
Now, since $z \mapsto C_1(x;z)$ is continuous and different from zero at $z=2n/3$,  
and since the mapping $z \mapsto C_2(x;z;c_n)$ is continuous at $z=n/3$ 
(using that $c_n < n/3$), 
for any 
\[ 
\eta > \sqrt{ \frac{C_2\parenj{ x ; \frac{n}{3} ; c_n}}{C_1\parenj{ x; \frac{2n}{3} }} } - 1
\, , 
\]
some $\epsilon_2 \in (0,\epsilon_1]$ exists such that 
\[ 
\forall \epsilon \in (0,\epsilon_2] \, , \qquad 
\eta \geq \sqrt{ \frac{C_2\parenj{ x ; \frac{n}{3} (1-\epsilon) ; c_n}}{C_1\parenj{ x; \frac{2n}{3} (1+\epsilon) }} } - 1
\, . 
\]
So, for such $\eta$, on $\Omega_x\,$, for every $\epsilon \in (0,\epsilon_2]$, 
\[ 
\frac{C_1 \parenj{ x ; \frac{2n}{3} (1+\epsilon) }}{1+\eta} 
< \Chwindow(\eta) 
< C_2 \parenj{ x ; \frac{n}{3} (1-\epsilon) ; c_n } (1+\eta) 
\, . 
\]
Making $\epsilon$ tend to zero in the above inequality yields the result.

\paragraph{Proof of Eq.~\eqref{eq.bound-Chthr}}
For any $a_n \in (T_n,n)$, by Eq.~\eqref{eq.pr.pro.variance-estim.1}, 
on $\Omega_x\,$, we have $D_{\mh(C)} \geq a_n > T_n$ for every $C \leq C_1(x;a_n)$, 
hence $\Chthr(T_n) \geq C_1(x;a_n)$ by Lemma~\ref{le.Chthr} with $\Gamma_1 = C_1(x;a_n)$. 
So, on $\Omega_x\,$, 
\[ 
\Chthr(T_n) \geq \sup_{a_n \in (T_n,n)} \setb{ C_1(x;a_n) } = C_1(x;T_n) 
\, . 
\]
By Eq.~\eqref{eq.pr.pro.variance-estim.1} with $b_n = T_n > c_n\,$, on $\Omega_x$ we have 
$D_{\mh(C)} \leq T_n$ for every $C \geq C_2(x ; T_n ; c_n)$, 
hence $\Chthr(T_n) \leq C_2(x ; T_n ; c_n)$ by Lemma~\ref{le.Chthr} with $\Gamma_2 = C_2(x;T_n;c_n)$.

\paragraph{Proof of Eq.~\eqref{eq.MSE-Chthr}} 
Let $\alpha = \log(4 \card \M ) \geq \log(8)$,
since $\card(\M) \geq 2$ under the assumptions of Proposition~\ref{pro.variance-estim}.
For every $z \geq 0$, by Eq.~\eqref{eq.bound-Chthr} with $x = z+\alpha$ and $c_n$ replaced by $T_n/2$, 
\begin{equation} \label{eq.dev-Chthr}
\begin{split}
\Proba 
\Biggl(
\paren{\Chthr - \sigma^2}^2 
&\geq 4 \max\setj{ \paren{1-\frac{T_n}{n}}^{-2} \, , \, \paren{\frac{T_n}{2n}}^{-2} } 
\\ 
&\qquad \times  
\left.
\croch{ \biaismax \parenj{\frac{T_n}{2}} + 2 \sigma^2 \parenj{ \sqrt{\frac{z + \alpha}{n}} + \frac{2 (z + \alpha)}{n} } }^2 
\right)
\leq \mathrm{e}^{-z}
\, . 
\end{split}
\end{equation}
Then, integrating Eq.~\eqref{eq.dev-Chthr}
with respect to $z$ ---that is, using Lemma~\ref{le.dev=>boundE} below--- 
we get that Eq.~\eqref{eq.MSE-Chthr} holds true.
Note that much smaller constants can be obtained by assuming that $\card(\M)$ is large enough, 
or that $\log( \card \M)/n$ is small enough. 
For instance, under the assumption that  
%%$\log(400) \leq \alpha=\log(4\card(\M)) \leq n/25$, 
$100 \leq \card(\M) \leq \exp(n/100)$, 
we get 
\begin{align*}
\E\croch{ \paren{\Chthr - \sigma^2}^2 }
&\leq
\max\setj{ \paren{1-\frac{T_n}{n}}^{-2} \, , \, \paren{\frac{T_n}{2n}}^{-2} } 
\crochj{ 12 \biaismax  \paren{\frac{T_n}{2}}^2 + \frac{102 \sigma^4 \log\parens{ \card \M }}{n} }
\, . 
\end{align*}
%%12 \times 6.473821 \log\parenb{ 4\card(\M) } \leq 12 \times 8.79297 \log\parenb{ \card(\M) } \leq 101.0716 \log\parenb{ \card(\M) }
%
\qed

\begin{lemma} \label{le.dev=>boundE}
For a real-valued random variable $Z$, 
if some $a,b,c \geq 0$ exist such that for every $z \geq 0$,
\begin{align*}
&\Proba\parenj{ Z \geq a + b z + c z^2} \leq \mathrm{e}^{-z} \, ,
\\
\text{then,} \qquad
&\E\crochs{ Z } \leq a + 2 b + 4 c 
\, .
\end{align*}
\end{lemma}
Lemma~\ref{le.dev=>boundE} is a classical integration exercise.

\subsection{Computations about \texorpdfstring{$\sigh^2_{m_0}$}{hat(sigma2)(m0)}}
\label{app.var-sighm0}
The following proposition gives a general formula for the variance and MSE of the residual-variance estimator $\sigh^2_{m_0}$ defined by Eq.~\eqref{def.sighm0} in Section~\ref{sec.related.variance}.
Note that Proposition~\ref{pro.var-sighm0} and
Lemma~\ref{le.var-sighm0} below are classical results 
(see for instance \citealp[Eq.~(5)]{Ull_Zin:1992}, 
or \citealp[Eq.~(6)]{Det_Mun_Wag:1998}).
We state and prove them here for completeness.

In this subsection, for any matrix $M \in \mathcal{M}_n(\R)$, $\diag(M)$ denotes the diagonal matrix of the diagonal elements of $M$ and $\vectun = (1, \ldots, 1) \in \R^n$.

\begin{proposition}
\label{pro.var-sighm0} 
%% Proposition~\ref{pro.var-sighm0} in Appendix~\ref{app.var-sighm0}
Let $\Pi \in \mathcal{M}_n(\R)$ some orthogonal projection matrix 
such that $D= \tr(\Pi) < n$, 
$\varepsilon \in \R^n$ some random vector with independent components, 
and $F \in \R^n$. 
Assume that for all $i \in \sets{1, \ldots, n}$, 
\[ \E\crochs{\varepsilon_i} = 0 
\, ,
\qquad 
\E\crochj{\varepsilon_i^2} = \sigma^2
\, ,
\qquad 
\E\crochj{\varepsilon_i^3} = m_3 
\, ,
\qquad \text{and} \qquad \E\crochj{\varepsilon_i^4} = m_4
\, .  \]
Let
\[ \sigh^2 \egaldef \frac{1}{n-D} \normb{ (\Id_n - \Pi) (F + \varepsilon)}^2 \, . \]
Then,
\begin{align}
\label{eq.pro.var-sighm0.var.gal}
\var\parenj{\sigh^2}
&= V
+ \frac{4 \normb{(\Id_n - \Pi)F}^2 \sigma^2}{(n-D)^2}
+ \frac{4 \prodscalb{F}{(\Id_n - \Pi) \diag(\Id_n - \Pi) \vectun}}{(n-D)^2} m_3
\\
\notag \text{where} \quad
V &\egaldef \frac{1}{(n-D)^2} \parenj{ \sum_{i=1}^n (1-\Pi_{i,i})^2 } \paren{ m_4 - 3 \sigma^4}
+ \frac{2}{n-D} \sigma^4 
\, , 
\\
\label{eq.pro.var-sighm0.risquequad.gal}
\text{and} \quad
\E\crochj{ \parenb{\sigh^2 - \sigma^2}^2 } &=
V
+ \frac{4 \normb{(\Id_n - \Pi)F}^2 \sigma^2}{(n-D)^2}
+ \frac{\normb{ (\Id_n - \Pi) F}^4 }{(n-D)^2}
\\
\notag
&\qquad
+ \frac{4 \prodscalb{F}{(\Id_n - \Pi) \diag(\Id_n - \Pi) \vectun}}{(n-D)^2} m_3
\, . 
\end{align}
In particular, if the $\varepsilon_i$ are Gaussian,
\begin{align}
\label{eq.pro.var-sighm0.var.gauss}
\var\paren{\sigh^2} &= \frac{2 \sigma^4}{n-D} + \frac{4 \normb{(\Id_n - \Pi)F}^2}{(n-D)^2}  \sigma^2
\\
\label{eq.pro.var-sighm0.risquequad.gauss}
\E\crochj{ \parenb{\sigh^2 - \sigma^2}^2 }
&= \frac{2 \sigma^4}{n-D} + \frac{4\normb{(\Id_n - \Pi)F}^2}{(n-D)^2}  \sigma^2 + \frac{\normb{ (\Id_n - \Pi) F}^4}{(n-D)^2}
\, .
\end{align}
\end{proposition}
\begin{proof}[Proof of Proposition~\ref{pro.var-sighm0}]
Applying Lemma~\ref{le.var-sighm0} below with $M=\Id_n - \Pi$ yields Eq.~\eqref{eq.pro.var-sighm0.var.gal}, 
since we have $\sigh^2 = Z / (n-D) $ and $M$ is an orthogonal projection matrix with $\tr(M)=n-D$.
Eq.~\eqref{eq.pro.var-sighm0.risquequad.gal} follows, in combination with Eq.~\eqref{eq.sighm0.biais}, since
\[ \E\crochB{ \parenb{\sigh^2 - \sigma^2}^2 } = \parenB{ \E\crochb{ \sigh^2 } - \sigma^2}^2 + \var\parenb{\sigh^2} \, . \]
In the Gaussian case, $m_3 = 0$ and $m_4 = 3 \sigma^4$,
hence
\[ V = \frac{2 \sigma^4}{n-D} \, , \]
which leads to Eq.~\eqref{eq.pro.var-sighm0.var.gauss}
and~\eqref{eq.pro.var-sighm0.risquequad.gauss}.
\end{proof}

\begin{lemma}
\label{le.var-sighm0}
Let $F \in \R^n$, $M \in \mathcal{M}_n(\R)$ a symmetric matrix, $\varepsilon \in \R^n$ some random vector with independent components such that for all $i \in \set{1, \ldots, n}$,
\[ 
\E\crochs{\varepsilon_i} = 0 \, ,
\qquad 
\E\crochj{\varepsilon_i^2} = \sigma^2
\, ,
\qquad 
\E\crochj{\varepsilon_i^3} = m_3 
\, ,
\qquad \text{and} \qquad
\E\crochj{\varepsilon_i^4} = m_4
\, .  \]
Then, if $Z = \prodscalb{F+\varepsilon}{M (F+\varepsilon)}$,
\begin{align}
\label{eq.le.var-sighm0.var.gal}
\var\paren{Z}
&= W + 4 \norm{ M F }^2 \sigma^2
+ 4 \prodscalb{F}{M \diag(M) \vectun} m_3
\\
\notag
\text{where} \qquad W &\egaldef
\tr\paren{\diag(M)^2}
\paren{ m_4 - 3 \sigma^4}
+ 2 \tr(M^2) \sigma^4
\\
\label{eq.le.var-sighm0.var.gal.encadrement}
\text{satisfies} \qquad
0 \leq W
&\leq
\parenj{ 2 \sigma^4 + m_4 } \tr(M^2)
\, .
\end{align}
\end{lemma}
\begin{proof}[Proof of Lemma~\ref{le.var-sighm0}]
First note that
\[ 
Z 
= \prodscals{F}{MF} + \prodscals{\varepsilon}{M \varepsilon}  + 2 \prodscals{M F}{\varepsilon}
\, . \]
Then,
\[ 
\E\crochs{Z} 
= \prodscals{F}{MF} + \E\crochb{\prodscals{\varepsilon}{M \varepsilon}} 
= \prodscals{F}{MF} + \sigma^2 \tr(M) 
\, . \]
Furthermore,
\begin{align*}
\E\crochj{Z^2} &= \prodscals{F}{MF}^2 + \E\crochj{ \prodscals{\varepsilon}{M \varepsilon}^2 } + 4 \E\crochj{\prodscals{M F}{\varepsilon}^2}
\\
&\qquad + 2 \prodscals{F}{MF} \sigma^2 \tr(M) + 4 \E\crochb{ \prodscals{\varepsilon}{M \varepsilon} \prodscals{M F}{\varepsilon}}
\\
&= \prodscals{F}{MF}^2
+ \parenj{ \sum_{i=1}^n M_{i,i}^2 } \parenj{ m_4 - 3 \sigma^4}
+ \crochj{ \tr(M)^2 + 2 \tr(M^{\top} M) } \sigma^4
\\
&\qquad
+ 4 \norms{ M F }^2 \sigma^2
+ 2 \prodscals{F}{MF} \sigma^2 \tr(M)
+ 4 \parenj{ \sum_{i,j=1}^n M_{i,i} M_{i,j} F_j } m_3
\end{align*}
where we used that
\begin{align*}
\E\crochj{ \prodscals{\varepsilon}{M \varepsilon}^2 }
&= \E\crochj{ \paren{ \sum_{i,j} \varepsilon_i M_{i,j} \varepsilon_j }^2 }
= \sum_{i,j,k,\ell} M_{i,j} M_{k,\ell} \E\crochj{ \varepsilon_i \varepsilon_j \varepsilon_k \varepsilon_{\ell}}
\\
&= \parenj{ \sum_{i=1}^n M_{i,i}^2 } \parenj{ m_4 - 3 \sigma^4}
+ \parenj{ \tr(M)^2 + \tr(M^2) + \tr(M^{\top} M) } \sigma^4
\end{align*}
and for any $G \in \R^n$ (here, $G=M F$),
\begin{align*}
\E\crochj{ \prodscals{G}{\varepsilon}^2 } 
&= \norms{G}^2 \sigma^2 
\\
\text{and} \qquad 
\E\crochb{ \prodscals{\varepsilon}{M\varepsilon} \prodscals{G}{\varepsilon} }
&= \E\crochj{ \sum_{i,j,k} \varepsilon_i M_{i,j} \varepsilon_j G_k \varepsilon_k }
= \parenj{ \sum_{i=1}^n M_{i,i} G_i } m_3
\, .
\end{align*}
Eq.~\eqref{eq.le.var-sighm0.var.gal} follows since $M$ is symmetric, 
$\var(Z) = \E[Z^2]-\E[Z]^2$,
\[
\sum_{i,j=1}^n M_{i,i} M_{i,j} F_j  
= \prodscalj{F}{M^{\top} \diag(M) \vectun}
\qquad \text{and} \qquad
\sum_{i=1}^n M_{i,i}^2  
= \tr\paren{\diag(M)^2} 
\, .
\]
For proving Eq.~\eqref{eq.le.var-sighm0.var.gal.encadrement},
we remark that
\begin{gather*}
W = \tr\parenj{\diag(M)^2} \parenj{ m_4 - \sigma^4 }
+ 2 \crochB{ \tr(M^2) - \tr\parenb{\diag(M)^2}} \sigma^4
\, ,
\\ 
 m_4 \geq \sigma^4
\, ,
\quad \text{and} \quad
0 \leq \tr\paren{\diag(M)^2} \leq \tr(M^2) 
\end{gather*}
since $M$ is symmetric. 
\end{proof}

\section{Algorithms}
\label{app.algos}

\subsection{Computation of the full path \texorpdfstring{$(\mh(C))_{C \geq 0}$}{( hat(m)(C), C geq 0)} in Algorithms~\ref{algo.OLS.jump}, \ref{algo.gal.slope.naif}, \ref{algo.penmin.linear}, and~\ref{algo.penmingal}}
\label{app.algos.path}
One can formulate the first step in Algorithms~\ref{algo.OLS.jump}, \ref{algo.gal.slope.naif}, \ref{algo.penmin.linear}, and~\ref{algo.penmingal} as computing, for every $C \geq 0$,
\begin{equation} \label{eq.mhC.abstrait.algo}
\mh(C) 
\in \argmin_{\mM} \setb{ f(m) + C g(m) }
\end{equation}
for some functions $f,g: \M \to \R$,
where $\M$ is assumed to be finite.
In the most general case (Algorithm~\ref{algo.penmingal}), $f(m) = \Remp\parens{\shm}$ and $g(m) = \pen_0(m)$.
Particular cases (Algorithms~\ref{algo.OLS.jump}, \ref{algo.gal.slope.naif}, and~\ref{algo.penmin.linear}) follow.

This subsection explains how to compute the full path $(\mh(C))_{C \geq 0}$ defined 
by Eq.~\eqref{eq.mhC.abstrait.algo}, given $\parens{f(m)}_{\mM}$ and $\parens{g(m)}_{\mM}\,$, 
with at most $\grandO([\card \M]^2)$ operations, and much less in practice.
The material presented here is adapted from \citet[Section~3.2]{Arl_Mas:2009:pente}. 
Similar results ---with a bit less details and formulated in specific frameworks where $\M \subset \mathbb{N}$--- 
have been proved earlier by \citet[Lemma~4.4.1]{Leb:2002}, 
\citet[Proposition~2.1]{Lav:2005} and \citet[Section~6.4.3]{Zwa:2005:phd}. 

\medbreak

First, remark that the definition \eqref{eq.mhC.abstrait.algo} of $\mh(C)$ can be ambiguous.
Let us choose a strict total order $\prec$ on $\M$ such that $g$ is non-decreasing, which is always possible since $\M$ is finite.
Then, by convention, for every $C \geq 0$, $\mh(C)$ is defined as
\begin{equation} \label{eq.EC.abstrait.algo}
\mh(C) = \min_{\prec} \cE(C)
\quad \text{where} \quad
\cE(C) \egaldef \argmin_{\mM} \setb{ f(m) + C g(m) }
\, .
\end{equation}
The main reason why the whole trajectory $\parens{\mh(C)}_{C\geq 0}$ can be computed efficiently is its particular shape.
Indeed, the proof of Proposition~\ref{pro.algo.path} below shows that $C \flapp \mh(C)$ is piecewise constant and non-increasing for $\prec$.
Then, the whole trajectory $\parens{\mh(C)}_{C\geq 0}$ can be written as
\begin{equation}
\label{eq.algo.path.description}
\forall i \in \set{0, \ldots, i_{\max}}, \quad \forall C \in \left[ C_i , C_{i+1} \right) , \qquad \mh(C) = m_i
\end{equation}
where
$i_{\max} \in \set{0, \ldots, \card(\M) - 1}$ is the number of jumps,
$(C_i)_{0 \leq i \leq i_{\max}+1}$ is an increasing sequence of non-negative reals (the location of the jumps)
with $C_0=0$ and $C_{i_{\max}+1}=+\infty$, and
$(m_i)_{0 \leq i \leq i_{\max}}$ is a non-increasing sequence of elements of $\M$.
\begin{algo}
\label{algo.path}
\algoinput\textup{:} $\parens{f(m)}_{\mM}\,$, $\parens{g(m)}_{\mM}\,$, and $\prec$ some strict total order on $\M$ such that $g$ is non-decreasing.
\\
Initialization\textup{:}
$C_0 \egaldef 0$ and $m_0 \egaldef  \min_{\prec} \argmin_{\mM} \setb{ f(m) }$.
\\
Step $i$, $i \geq 1$\textup{:}
Let
\[ \cG(m_{i-1}) \egaldef \setb{\mM \telque f(m) > f(m_{i-1}) \quad \text{ and } \quad g(m) < g(m_{i-1})}
\, . \]
If $\cG(m_{i-1})=\emptyset$, then put $C_i = + \infty$, $i_{\max}=i-1$ and stop. 
Otherwise, define 
\begin{gather} \label{def.algo.path.Ci}
C_{i} \egaldef \min \setj{\frac{f(m) - f(m_{i-1})} {g(m_{i-1}) - g(m)} \telque m \in \cG(m_{i-1}) } \\
\text{and} \qquad
m_i \egaldef \min_{\prec} \cF_i
\qquad \text{with} \quad
\cF_i \egaldef \argmin_{m \in \cG(m_{i-1})} \setj{\frac{f(m) - f(m_{i-1})} {g(m_{i-1}) - g(m)} }
\, . \notag
\end{gather}
\\
\algooutput\textup{:} $(C_i)_{0 \leq i \leq i_{\max}+1}$ and $(m_i)_{0 \leq i \leq i_{\max}}$, which describe according to Eq.~\eqref{eq.algo.path.description} the full trajectory $(\mh(C))_{C \geq 0}$ defined by Eq.~\eqref{eq.EC.abstrait.algo}.
\end{algo}
%
%%% Proposition~\ref{pro.algo.path} in Appendix~\ref{app.algos.path}
%
\begin{proposition}[Correctness of Algorithm~\ref{algo.path}] \label{pro.algo.path}
For every $C \geq 0$, let $\mh(C)$ be defined by Eq.~\eqref{eq.EC.abstrait.algo}.
Assume $\M$ is finite.
Then, Algorithm~\ref{algo.path} terminates and $i_{\max} \leq \card(\M) - 1$.
Furthermore, Algorithm~\ref{algo.path} is correct, that is, $(C_i)_{0 \leq i \leq i_{\max}+1}$ is increasing and $\forall i \in \set{0, \ldots, i_{\max}-1}$, $\forall C \in [C_i, C_{i+1})$, $\mh(C) = m_i\,$.
\end{proposition}
Proposition~\ref{pro.algo.path} also gives an upper bound on the computational complexity of Algorithm~\ref{algo.path}:
since the complexity of each step is $\grandO(\card \M)$, the complexity of Algorithm~\ref{algo.path} is upper-bounded by $\grandO(i_{\max} \card \M) \leq \grandO([\card \M]^2)$.
In general, this upper bound is pessimistic since we usually have $i_{\max} \ll \card(\M)$ in practice.
\begin{proof}[Proof of Proposition~\ref{pro.algo.path}]
First, since $\M$ is finite, $\cG(m_{i-1})$ is also finite 
and $m_i$ is well-defined as soon as $\cG(m_{i-1}) \neq \emptyset$, 
which holds for every $i \leq i_{\max}\,$.
Moreover, by construction, $g(m_i)$ decreases with $i$, so that all the $m_i \in \M$ are different; 
hence, Algorithm~\ref{algo.path} terminates and we have the inequality $i_{\max} +1 \leq \card(\M)$.
Notice also that $C_i$ can always be defined by Eq.~\eqref{def.algo.path.Ci} with the convention $\min \emptyset=+\infty$.
\\
We now prove by induction that the following property holds true for every $i \in \set{0, \ldots, i_{\max}  }$,
which implies that Proposition~\ref{pro.algo.path} holds true:
\[ \mathcal{P}_i : \qquad C_i < C_{i+1} \quad \text{ and } \quad \forall C \in [C_i, C_{i+1}), \quad \mh(C) = m_i \, . \]

\paragraph{$\mathcal{P}_0$ holds true}
By definition of $C_1\,$, since $\M$ is finite, $C_1>0$. 
Note that $C_1$ may be equal to $+\infty$ if $\cG(m_0)=\emptyset$. 
For $C=C_0=0$, the definition of $m_0$ is the one of $\mh(0)$, so that $\mh(C) = m_0\,$.
For $C \in (0,C_1)$, Lemma~\ref{le.algo.path} below shows that either $\mh(C)=\mh(0)=m_0$ or $\mh(C) \in \cG(m_0)$.
In the latter case, by definition of $C_1\,$,
\[ \frac{f\parenb{ \mh(C) } - f(m_0)} {g(m_0) - g\parenb{ \mh(C) }} \geq C_1 > C \]
hence
\[ f\parenb{ \mh(C) } + C g\parenb{ \mh(C) } > f(m_0) + C g(m_0) \] which contradicts the definition of $\mh(C)$.
Therefore, $\mathcal{P}_0$ holds true.

\paragraph{$\mathcal{P}_i \Rightarrow \mathcal{P}_{i+1}$ for every $i \in \set{0, \ldots, i_{\max}-1}$} 
Assume that $\mathcal{P}_i$ holds true.
First, we have to prove that $C_{i+2}>C_{i+1}\,$. 
If $i=i_{\max}-1$, this is clear since $C_{i_{\max} + 1}=+\infty$. 
Otherwise, $C_{i+2}<+\infty$ and $m_{i+2}$ exists. Then, by definition of $m_{i+2}$ and $C_{i+2}$ (resp. $m_{i+1}$ and $C_{i+1}$), we have
\begin{align}
\label{eq.pro.algo.path.1}
f(m_{i+2}) - f(m_{i+1}) = C_{i+2} \crochb{ g(m_{i+1}) - g(m_{i+2}) }
\\
\label{eq.pro.algo.path.2}
f(m_{i+1}) - f(m_{i}) = C_{i+1} \crochb{ g(m_{i}) - g(m_{i+1}) } \, .
\end{align}
Moreover, $m_{i+2} \in \cG(m_{i+1}) \subset \cG(m_i)$ and $m_{i+2} \prec m_{i+1}$ (because $g$ is non-decreasing). Using again the definition of $C_{i+1}\,$, we have 
\begin{equation}
\label{eq.pro.algo.path.3}
f(m_{i+2}) - f(m_i) > C_{i+1} \crochb{ g(m_i) - g(m_{i+2}) }
\end{equation}
(the inequality is strict: otherwise, we would have $m_{i+2} \in \cF_{i+1}$ and $m_{i+2} \prec m_{i+1} = \min_{\prec} \cF_{i+1}\,$, 
which is not possible).
The difference of Eq.~\eqref{eq.pro.algo.path.3} and \eqref{eq.pro.algo.path.2} yields 
\[ 
f(m_{i+2}) - f(m_{i+1}) 
> C_{i+1} \crochb{ g(m_{i+1}) - g(m_{i+2}) } 
\, . 
\]
By Eq.~\eqref{eq.pro.algo.path.1}, we deduce that 
\[ 
C_{i+2} \crochb{ g(m_{i+1}) - g(m_{i+2}) } 
> C_{i+1} \crochb{ g(m_{i+1}) - g(m_{i+2}) } 
\, , \]
hence $C_{i+2} > C_{i+1}$ since $g(m_{i+1}) > g(m_{i+2})$.

Second, we prove that $\mh(C_{i+1})=m_{i+1}\,$.
From $\mathcal{P}_i\,$, we know that for every $\mM$, for every $C\in[C_i, C_{i+1})$, $f(m_i) + C g(m_i) \leq f(m) + C g(m)$.
Taking the limit when $C$ tends to $C_{i+1}\,$, it follows that $m_i \in \cE(C_{i+1})$.
By Eq.~\eqref{eq.pro.algo.path.2}, we then have $m_{i+1} \in \cE(C_{i+1})$.
Now, let $m'$ be any element of $\cE(C_{i+1})$.
By Lemma~\ref{le.algo.path} with $C=C_i\,$, $m=m_i = \mh(C_i) \in \cE(C_i)$ and $C'=C_{i+1}>C_i\,$,
we have either 
(a) $f(m') = f(m_i)$ and $g(m') = g(m_i)$ 
or (b) $m' \in \cG(m_i)$; 
case (c) is excluded since $m_i = \mh(C_i)$. 
In case (a), $g(m')= g(m_i) > g(m_{i+1})$, hence $m_{i+1} \prec m'$  
because $g$ is non-decreasing. 
In case (b), notice that $m_i , m' \in \cE(C_{i+1})$ implies 
$f(m') + C_{i+1} g(m') = f(m_i) + C_{i+1} g(m_i)$. 
Since $m' \in \cG(m_i)$, we get that $m' \in \cF_{i+1}\,$. 
Then, by definition of $m_{i+1}\,$, we have $m_{i+1} \preceq m'$. 
Overall, we have proved that $m_{i+1}$ belongs to $\cE(C_{i+1})$ and is smaller than any element $m'$ of $\cE(C_{i+1})$, 
which proves that $m_{i+1} = \min_{\prec} \cE(C_{i+1}) = \mh(C_{i+1})$.

Let $C' \in (C_{i+1},C_{i+2})$. 
It remains to prove $\mh(C')=m_{i+1}\,$.
From the last statement of Lemma~\ref{le.algo.path} with $C=C_{i+1}\,$,
we have either $\mh(C')=\mh(C_{i+1})=m_{i+1}$ or $\mh(C') \in \cG(\mh(C_{i+1})) = \cG(m_{i+1})$.
In the latter case (in which $\cG(m_{i+1}) \neq \emptyset$ hence $C_{i+2} < \infty$), by definition of $C_{i+2}\,$,
\[ \frac{f\parenb{ \mh(C') } - f(m_{i+1})} {g(m_{i+1}) - g\parenb{ \mh(C') }} \geq C_{i+2} > C' \]
so that
\[ f\parenb{ \mh(C') } + C' g\parenb{ \mh(C') } > f(m_{i+1}) + C' g(m_{i+1}) \] 
which contradicts the definition of $\mh(C')$. 
Therefore, $\mh(C') = m_{i+1}\,$, which ends proving $\mathcal{P}_{i+1}\,$. 
\end{proof}
The following lemma is used in the proof of Proposition~\ref{pro.algo.path} above.
\begin{lemma} \label{le.algo.path}
With the notation of Proposition~\ref{pro.algo.path} and its proof, if we have $0  \leq C < C^{\prime}$, $m \in \cE(C)$,  and $m^{\prime} \in \cE(C^{\prime})$, then one of the following statements holds true\textup{:}
\begin{enumerate}
\item[\textup{(}a\textup{)}] $f(m) = f(m^{\prime})$ and $g(m) = g(m^{\prime})$.
\item[\textup{(}b\textup{)}] $f(m) < f(m^{\prime})$ and $g(m)>g(m^{\prime})$.
\item[\textup{(}c\textup{)}] $C=0$, $f(m) = f(m^{\prime})$ and $g(m) > g(m^{\prime})$, hence $m \neq \mh(0)$.
\end{enumerate}
In particular, for any $0 \leq C < C'$, we have 
either $\mh(C)=\mh(C^{\prime})$ or $\mh(C^{\prime}) \in \cG(\mh(C))$.
\end{lemma}
\begin{proof}[Proof of Lemma~\ref{le.algo.path}]
By definition of $\cE(C)$ and $\cE(C^{\prime})$,
\begin{align}
\label{eq.le.algo.path.1}
f(m) + C g(m) 
&\leq f(m^{\prime}) + C g(m^{\prime}) 
\\
\text{and} \qquad 
f(m^{\prime}) + C^{\prime} g(m^{\prime}) 
&\leq f(m) + C^{\prime} g(m) 
\, .
\label{eq.le.algo.path.2}
\end{align}
Summing Eq.~\eqref{eq.le.algo.path.1} and \eqref{eq.le.algo.path.2} gives $(C^{\prime} - C) g(m^{\prime}) \leq (C^{\prime} - C) g(m)$ so that
\begin{equation}
\label{eq.le.algo.path.3} 
g(m^{\prime}) 
\leq g(m) 
\, . 
\end{equation}
Since $C \geq 0$, Eq.~\eqref{eq.le.algo.path.1} and \eqref{eq.le.algo.path.3} give $f(m) + C g(m) \leq f(m^{\prime}) + C g(m)$, that is
\begin{equation}
\label{eq.le.algo.path.4} 
f(m) \leq f(m^{\prime}) 
\, .
\end{equation}

If $g(m) = g(m^{\prime})$, Eq.~\eqref{eq.le.algo.path.2} and \eqref{eq.le.algo.path.4} imply 
$f(m^{\prime})=f(m)$ hence (a) is satisfied.
Otherwise, $g(m) > g(m^{\prime})$ by Eq.~\eqref{eq.le.algo.path.3}, 
and Eq.~\eqref{eq.le.algo.path.1} implies $f(m) < f(m^{\prime})$ or $C=0$.
If $f(m) < f(m')$, (b) holds true. 
Otherwise, $f(m)=f(m')$ and $C=0$. 
Since $g(m^{\prime}) < g(m)$, we get $m^{\prime} \prec m$ hence $m \neq \mh(0)$.

The last statement follows by taking $m=\mh(C)$ and $m^{\prime}=\mh(C^{\prime})$, which excludes case (c).
In case (a), $\cE(C) = \cE(C^{\prime})$ hence $\mh(C) = \mh(C^{\prime})$.
In case (b), $\mh(C^{\prime}) \in \cG(\mh(C))$.
\end{proof}

\subsection{Computation of \texorpdfstring{$\Chwindow$}{hat(C)(window)} in step~2 of Algorithms~\ref{algo.OLS.jump}, \ref{algo.gal.slope.naif}, \ref{algo.penmin.linear}, and~\ref{algo.penmingal}}
\label{app.algos.window}
Step~2 of Algorithms~\ref{algo.OLS.jump}, \ref{algo.gal.slope.naif}, \ref{algo.penmin.linear}, and~\ref{algo.penmingal} require to localize a jump in the trajectory $(\C_{\mh(C)})_{C \geq 0}\,$, given the path $(\mh(C))_{C \geq 0}$ and some complexity measure $(\C_m)_{\mM}\,$.
Although the maximal jump is straightforward to localize, Theorem~\ref{thm.OLS} suggests to look for the largest jump over a geometrical window of values of $C$, that is, $\Chwindow$ as defined by Eq.~\eqref{def.Chwin} in Section~\ref{sec.slopeOLS.math}.
This section explains how $\Chwindow$ can be computed efficiently given the path $(\C_{\mh(C)})_{C \geq 0}\,$, 
with a complexity $\grandO(i_{\max} \log i_{\max} ) = \grandO( \card \M  \log (\card \M ) )$.

\medbreak

Let us consider a slightly more general problem:
given some $\alpha > \beta >0$, compute
\begin{equation}
\label{eq.Chwindow.general}
\Chwindowgal \egaldef \argmax_{C \geq 0} \set{ \C_{\mh( \beta C)} - \C_{\mh( \alpha C)} }
\, .
\end{equation}
Note that $\Chwindowgal$ is usually not reduced to a singleton, but can be an interval or a finite union of intervals.

From Eq.~\eqref{eq.algo.path.description} in Appendix~\ref{app.algos.path}, the path $(\C_{\mh(C)})_{C \geq 0}$ is piecewise constant and can be fully described with a small number of parameters: writing $\C_i = \C_{m_i}\,$,
\begin{equation}
\label{eq.algo.path.description.C}
%\text{with} \qquad
\forall i \in \set{0, \ldots, i_{\max}}, \quad \forall C \in \left[ C_i , C_{i+1} \right) , \qquad \C_{\mh(C)} = \C_i
\, .
\end{equation}
Given this description of $(\C_{\mh(C)})_{C \geq 0}\,$, Algorithm~\ref{algo.window} below determines the set $\Chwindowgal\,$, as proved by Proposition~\ref{pro.algo.window} .
\begin{algo}
\label{algo.window}
\algoinput\textup{:} $(C_i)_{0 \leq i \leq i_{\max}+1}$ an increasing sequence of non-negative reals with $C_0=0$ and $C_{i_{\max}+1}=+\infty$, and $(\C_i)_{0 \leq i \leq i_{\max}}$ a sequence of real numbers.
\begin{enumerate}
\item If $i_{\max} = 0$, define $\Chwindowgal \egaldef [0,+\infty)$ and stop.
\item Otherwise, proceed and compute 
$\overline{C} = ( C_1 / \beta , \ldots, C_{i_{\max}}/\beta , C_1 / \alpha , \ldots, C_{i_{\max}} / \alpha ) \in \R^{2 i_{\max}}$ 
and $\overline{\Delta} = ( \C_1 - \C_0 , \ldots, \C_{i_{\max}} - \C_{i_{\max}-1} , \C_0 - \C_1, \ldots, \C_{i_{\max}-1} - \C_{i_{\max}} ) \in \R^{2 i_{\max}}$. 
\item Sort $\overline{C}$ and $\overline{\Delta}$ according to $\overline{C}$, that is, 
find some permutation $\sigma$ of $\sets{1 , \ldots, 2 i_{\max}}$ such that
$\overline{C}_{\sigma(1)} \leq \dots \leq \overline{C}_{\sigma(2 i_{\max})}$ 
and compute
$\overline{C^{\sigma}} = (\overline{C}_{\sigma(i)})_{1 \leq i \leq 2 i_{\max}}$ and
$\overline{\Delta^{\sigma}} = (\overline{\Delta}_{\sigma(i)})_{1 \leq i \leq 2 i_{\max}}$. 
%%%%%%
%%%compute
%%%$\overline{C^{\sigma}} = (\overline{C}_{\sigma(i)})_{1 \leq i \leq 2 i_{\max}}$ and
%%%$\overline{\Delta^{\sigma}} = (\overline{\Delta}_{\sigma(i)})_{1 \leq i \leq 2 i_{\max}}$ where
%%%$\sigma$ is some permutation of $\sets{1 , \ldots, 2 i_{\max}}$ such that
%%%$\overline{C}_{\sigma(1)} \leq \dots \leq \overline{C}_{\sigma(2 i_{\max})}$.
%
\item Compute $W \egaldef \mathrm{cumsum}(\overline{\Delta^{\sigma}}) \in \R^{2 i_{\max}}$, that is, for every $i \in \sets{1 , \ldots, 2 i_{\max}}$,
\[ W_i = \sum_{j=1}^i \overline{\Delta^{\sigma}_j} \, . \]
\item Compute $V\in \R^{2 i_{\max}}$ such that,
for every $i \in \sets{1 , \ldots, 2 i_{\max}}$,
$V_i = W_i$ if $\overline{C^{\sigma}_{i}}<\overline{C^{\sigma}_{i+1}}$,
and otherwise $V_i = - \infty$.
\item Determine $\mathcal{K} \egaldef \argmax_{i \in \sets{1 , \ldots, 2 i_{\max}}} V_i\,$.
\item Define $\Chwindowgal \egaldef \bigcup_{k \in \mathcal{K}} [\overline{C^{\sigma}_k} , \overline{C^{\sigma}_{k+1}} )$ with $\overline{C^{\sigma}_{2i_{\max}+1}} = + \infty$.
\end{enumerate}
\algooutput\textup{:} $\Chwindowgal\,$.
\end{algo}
\begin{proposition}[Correctness of Algorithm~\ref{algo.window}] \label{pro.algo.window}
Algorithm~\ref{algo.window} is correct, that is, it terminates and its output $\Chwindowgal$ actually satisfies Eq.~\eqref{eq.Chwindow.general} provided Eq.~\eqref{eq.algo.path.description.C} holds true.
\end{proposition}
\begin{proof}[Proof of Proposition~\ref{pro.algo.window}]
If $i_{\max} = 0$, $C \mapsto \C_{\mh(C)}$ is constant over $[0,+\infty)$ so Algorithm~\ref{algo.window} is correct.
Otherwise, Eq.~\eqref{eq.algo.path.description.C} can be rewritten as
\[
\forall C \geq 0 \, , \quad \C_{\mh(C)} = \C_0 + \sum_{i=1}^{i_{\max}} (\C_i - \C_{i-1}) \un_{C \geq C_i}
\]
hence, for every $C \geq 0$, using the notation of Algorithm~\ref{algo.window},
\begin{align}
\notag
\C_{\mh( \beta C)} - \C_{\mh( \alpha C)}
&=
\sum_{i=1}^{i_{\max}} (\C_i - \C_{i-1}) \un_{\beta C \geq C_i}
- \sum_{i=1}^{i_{\max}} (\C_i - \C_{i-1}) \un_{\alpha C \geq C_i}
\\
&=
\sum_{i=1}^{2 i_{\max}} \overline{\Delta}_i \un_{C \geq \overline{C}_i}
=
\sum_{i=1}^{2 i_{\max}} \overline{\Delta^{\sigma}_i} \un_{C \geq \overline{C^{\sigma}_i}}
=
\sum_{i=1}^{2 i_{\max}} W_i  \un_{C \in [\overline{C^{\sigma}_i} , \overline{C^{\sigma}_{i+1}} )}
\notag
\\
&=
\sum_{i=1}^{2 i_{\max}} V_i  \un_{C \in [\overline{C^{\sigma}_i} , \overline{C^{\sigma}_{i+1}} )}
\label{eq.pr.pro.algo.window.1}
\end{align}
with the conventions $\overline{C^{\sigma}_{2 i_{\max}+1}} = +\infty$ and $\infty \un_{C \in \emptyset} = 0$.
For the last equality, we use the fact that when
$[\overline{C^{\sigma}_i} , \overline{C^{\sigma}_{i+1}} )$ is empty ---which corresponds to
values of $\overline{C^{\sigma}_i}$ that are equal to $C_j / \beta = C_k / \alpha$ for some $j,k \in \sets{1 , \ldots, i_{\max}}$---,
the value of $W_i  \un_{C \in [\overline{C^{\sigma}_i} , \overline{C^{\sigma}_{i+1}} )}$ is zero whatever $W_i\,$, hence $W_i$ can be changed into $V_i\,$.

By Eq.~\eqref{eq.pr.pro.algo.window.1},
\[ 
\sup_{C \geq 0} \set{ \C_{\mh( \beta C)} - \C_{\mh( \alpha C)} } 
= \max_{1 \leq i \leq 2i_{\max}} V_i 
\]
and the supremum is attained exactly at the values of $C$ belonging to some interval $[\overline{C^{\sigma}_k} , \overline{C^{\sigma}_{k+1}} )$ with $k \in \mathcal{K} = \argmax_i V_i\,$.
In other words, Algorithm~\ref{algo.window} is correct.
\end{proof}

\section{More figures and experimental results} \label{app.morefig}

\begin{figure}
\begin{center}
\begin{minipage}[b]{.49\linewidth}
\includegraphics[width=\textwidth]{\pathfig/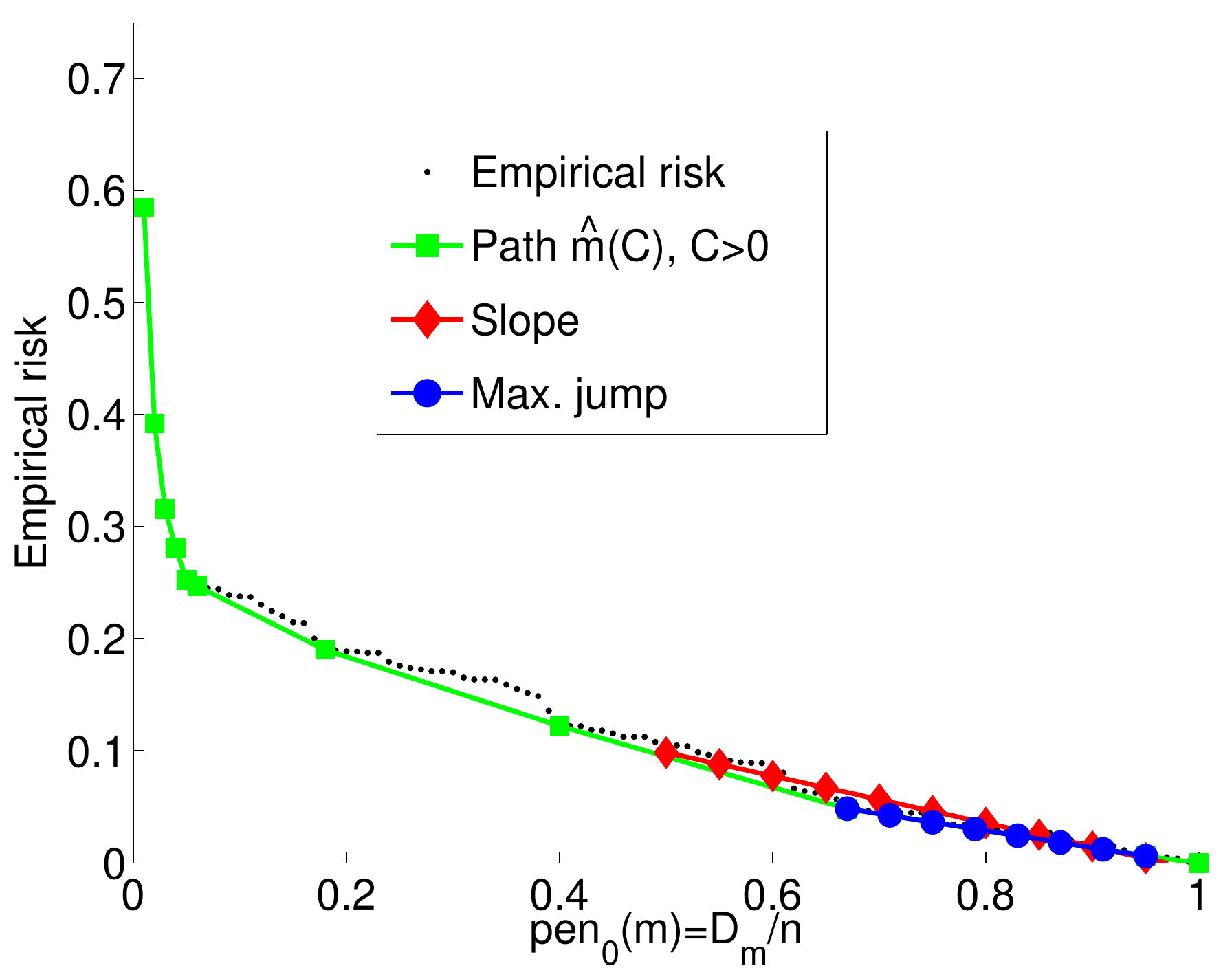} 
\\ \centerline{(a) L-curve (as in Figure~\ref{fig.Lcurve.easy.ech6}). }
\end{minipage}
\begin{minipage}[b]{.49\linewidth}
\includegraphics[width=\textwidth]{\pathfig/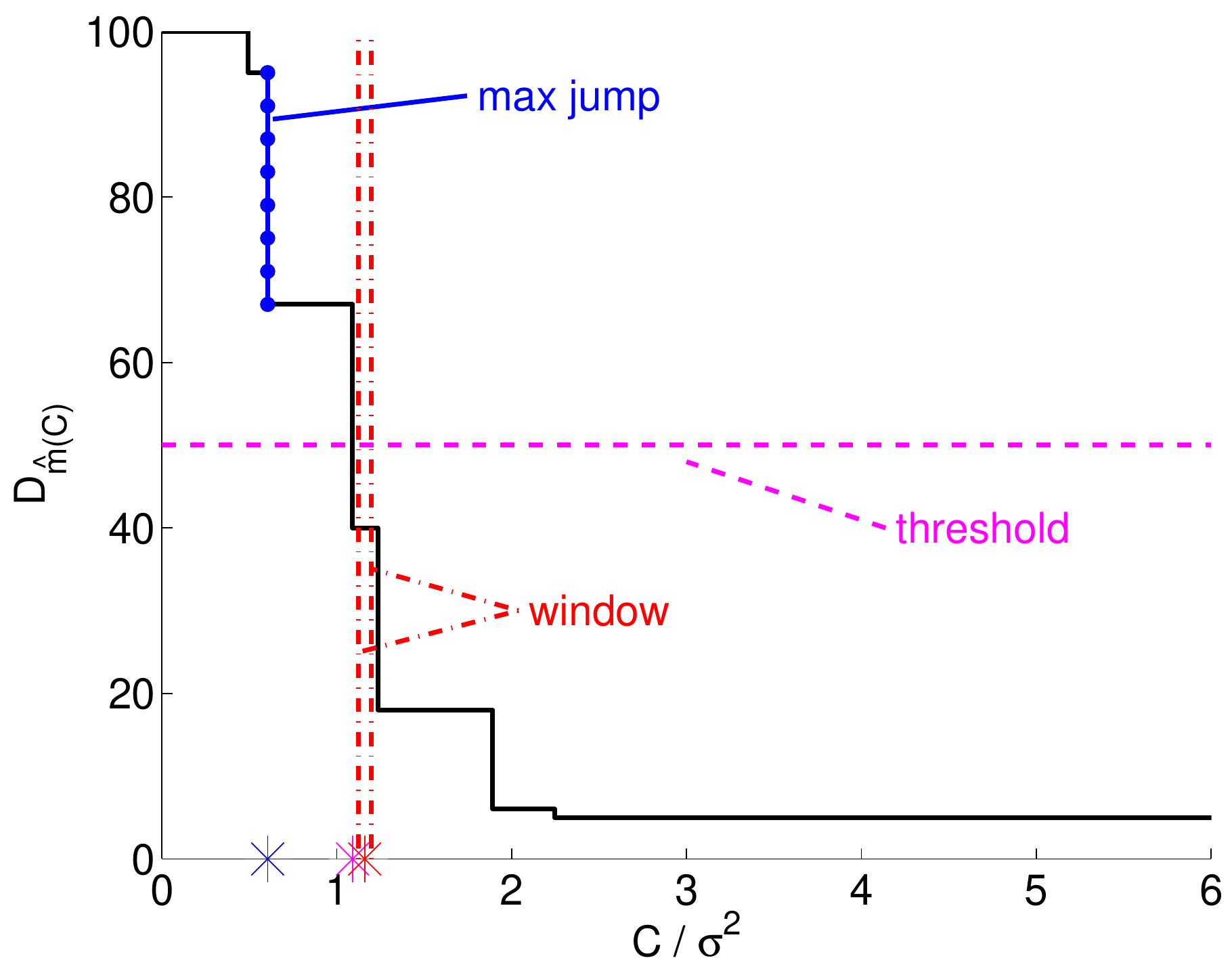}
\\ \centerline{(b) Plot of $C \mapsto D_{\mh(C)}\,$. }
\end{minipage}
\caption{\label{fig.DmhC-Lcurve.easy.ech540}
On the same sample, visualization of the three versions of Algorithm~\ref{algo.OLS.jump}
and of $\Chslope\,$. 
`Easy' setting,
see Appendix~\ref{app.details-simus} for details.
}
\end{center}
\end{figure}

\begin{figure}
\begin{center}
\begin{minipage}[c]{.49\linewidth}
\includegraphics[width=\textwidth]{\pathfig/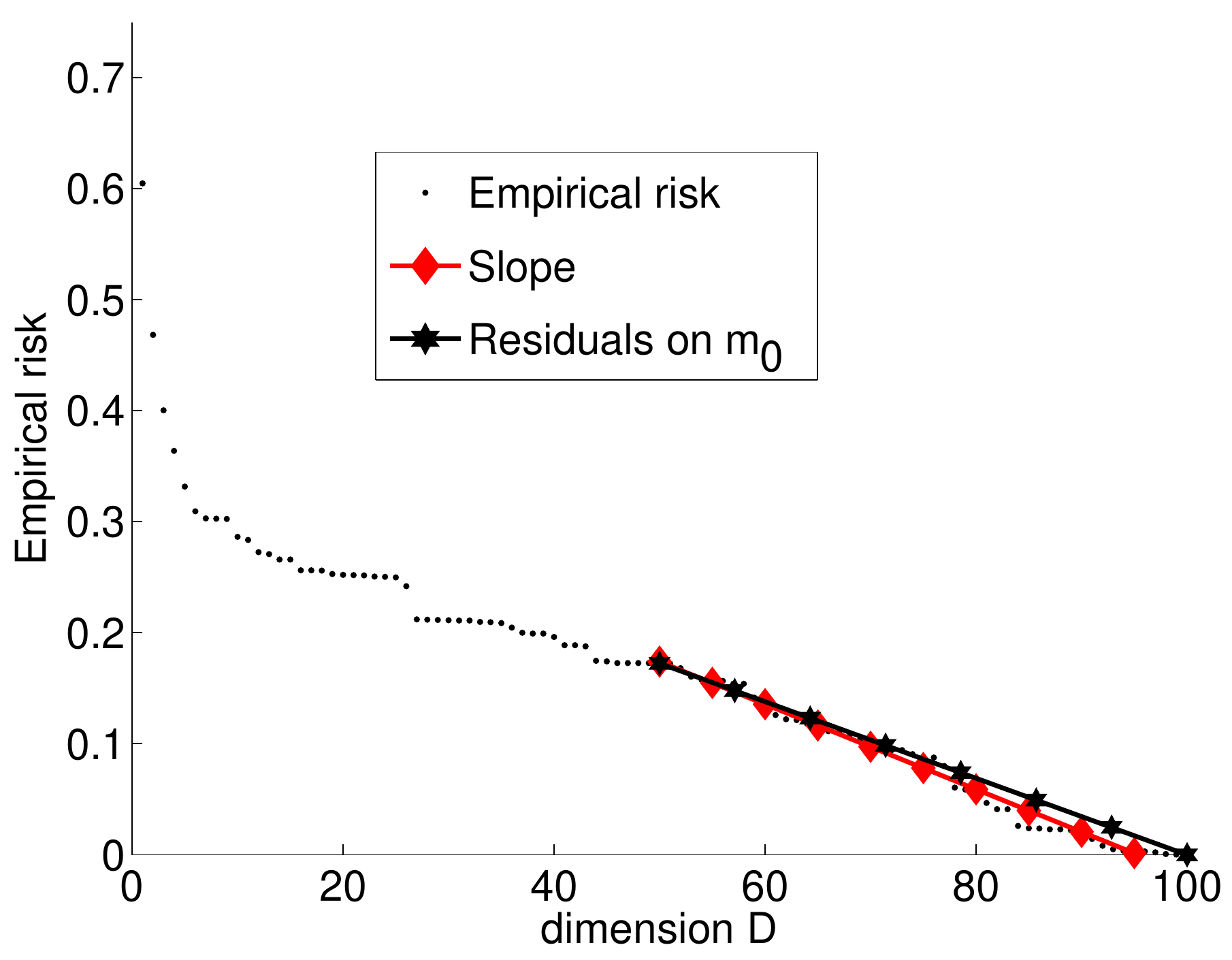}
\\ \centerline{(a) `Easy' setting.}
\end{minipage}
\hfill
\begin{minipage}[c]{.49\linewidth}
\includegraphics[width=\textwidth]{\pathfig/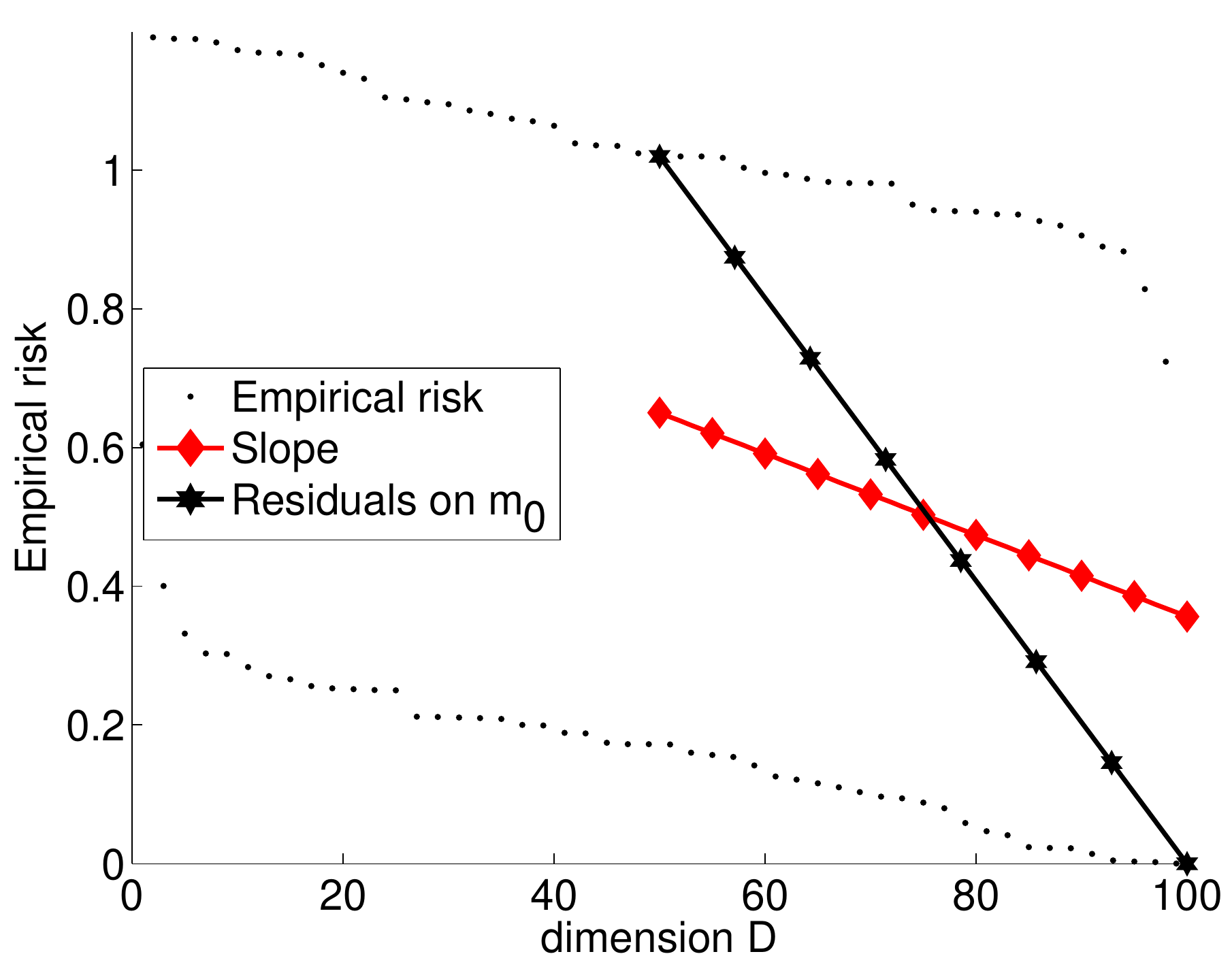}
\\ \centerline{(b) `Hard' setting.}
\end{minipage}
\caption{\label{fig.OLS.slope-vs-residuals.easy-hard}
Slope estimation $\Chslope$ vs. residual-based variance estimator $\sigh^2_{m_0}\,$.
See Appendix~\ref{app.details-simus} for details.
}
\end{center}
\end{figure}

\begin{table}
\begin{center}
\begin{tabular}{lcccc}
$\Ch$        & $\E \crochb{ \Ch/\sigma^2 }$ & $\sqrt{\var(\Ch)}/\sigma^2$ & $\E\crochb{ (\Ch - \sigma^2)^2 }/\sigma^4$ & risk ratio\\
\tabespvert
\hline
\tabespvert
$\Chmaxjump$                           & 1.09  & 0.257 & 0.0749 & 1.309 $\pm$ 0.003 \\
\tabespvert
\hline
\tabespvert
$\Chthr\,$, $T_n = n/10$                 & 3.12  & 1.281 & 6.140\ch & 1.647 $\pm$ 0.004 \\ %T=10
$\Chthr\,$, $T_n = n/\log(n)$            & 1.60  & 0.469 & 0.584\ch & 1.310 $\pm$ 0.002 \\ %T=21.7
$\Chthr\,$, $\mathbf{T_n = n/2}$         & 1.13  & 0.229 & 0.0683 & 1.278 $\pm$ 0.003 \\ %T=50
$\Chthr\,$, $T_n = 9n/10$                & 0.84  & 0.239 & 0.0826 & 1.621 $\pm$ 0.083 \\ %T=90
\tabespvert
\hline
\tabespvert
$\Chwindow\,$, $\eta = 1/n$              & 1.09  & 0.257 & 0.0745 & 1.309 $\pm$ 0.003 \\ %eta = 0.01
$\Chwindow\,$, $\mathbf{\eta = 1/\sqrt{n}}$& 1.10& 0.256 & 0.0752 & 1.308 $\pm$ 0.003 \\ %eta = 0.1
$\Chwindow\,$, $\eta = 1.5/\sqrt{n}$     & 1.10  & 0.258 & 0.0776 & 1.307 $\pm$ 0.003 \\ %eta = 0.15
$\Chwindow\,$, $\eta = \sqrt{\log(n)/n}$ & 1.12  & 0.263 & 0.0829 & 1.304 $\pm$ 0.003 \\ %eta = 0.215
$\Chwindow\,$, $\eta = 2\sqrt{\log(n)/n}$& 1.17  & 0.286 & 0.110\ch & 1.294 $\pm$ 0.003 \\ %eta = 0.429
\tabespvert
\hline
\tabespvert
$\Chslope\,$, $D_{0} = n/10$             & 1.15  & 0.181 & 0.0544 & 1.243 $\pm$ 0.002 \\ %D_0=10
$\Chslope\,$, $D_{0} = n/\log(n)$        & 1.09  & 0.188 & 0.0437 & 1.260 $\pm$ 0.002 \\ %D_0=21.7
$\Chslope\,$, $\mathbf{D_{0} = n/2}$     & 1.05  & 0.228 & 0.0543 & 1.313 $\pm$ 0.003 \\ %D_0=50
$\Chslope\,$, $D_{0} = 9n/10$            & 1.02  & 0.478 & 0.229\ch & 1.672 $\pm$ 0.009 \\ %D_0=90
\tabespvert
\hline
\tabespvert
\textsc{Capushe}                          & 1.05  & 0.291 & 0.0873 & 1.410 $\pm$ 0.005 \\
\tabespvert
\hline
\tabespvert
median                                 & 1.08  & 0.229 & 0.0588 & 1.301 $\pm$ 0.003 \\
\tabespvert
\hline
\tabespvert
consensus                              & --    & --    & --     & 1.306 $\pm$ 0.003 \\
consensus when no reject               & --    & --    & --     & 1.298 $\pm$ 0.003 \\
\tabespvert
\hline
\tabespvert
$\sigh^2_{m_0}\,$, $D_{m_0} = n/10$      & 1.23  & 0.180 & 0.0862 & 1.237 $\pm$ 0.002 \\ %D_m0=10
$\sigh^2_{m_0}\,$, $D_{m_0} = n/\log(n)$ & 1.12  & 0.176 & 0.0443 & 1.241 $\pm$ 0.002 \\ %D_m0=21.7
$\sigh^2_{m_0}\,$, $D_{m_0} = n/2$       & 1.05  & 0.211 & 0.0469 & 1.304 $\pm$ 0.003 \\ %D_m0=50
$\sigh^2_{m_0}\,$, $D_{m_0} = n/2+1$     & 1.05  & 0.213 & 0.0478 & 1.305 $\pm$ 0.003 \\ %D_m0=51
$\sigh^2_{m_0}\,$, $D_{m_0} = 9n/10$     & 1.02  & 0.455 & 0.2080 & 1.641 $\pm$ 0.008 \\ %D_m0=90
\tabespvert
\hline
\tabespvert
$C_p$ (known $\sigma^2$)               & --    & --    & --     & 1.269 $\pm$ 0.003 \\
$C_p$ $\times 1.12$ (known $\sigma^2$) & --    & --    & --     & 1.251 $\pm$ 0.002 \\
\end{tabular}
\medbreak
\caption{\label{tab.dist-Ch.easy}
Algorithms~\ref{algo.OLS.jump}--\ref{algo.OLS.slope},
`easy' setting\textup{:} distribution of $\Ch$ and model-selection performance,
with various definitions for $\Ch$ and various parameters for each definition.
The risk ratio is $\E\crochs{\norms{\Fh_{\mh} - F}^2 / \inf_{\mM} \norms{\Fh_{m} - F}^2}$.
Reported values are empirical estimates obtained from $N=10\,000$ independent samples.
For the risk ratio, error bars are equal to the standard deviation of the ratio $\norms{\Fh_{\mh} - F}^2 / \inf_{\mM} \norms{\Fh_{m} - F}^2$ divided by $\sqrt{N}$.
}
\end{center}
\end{table}

\begin{figure}
\begin{center}
\begin{minipage}[c]{.47\linewidth}%% pas plus que .47 pour que ca tienne sur la meme page que Tab 3
\includegraphics[width=\textwidth]{\pathfig/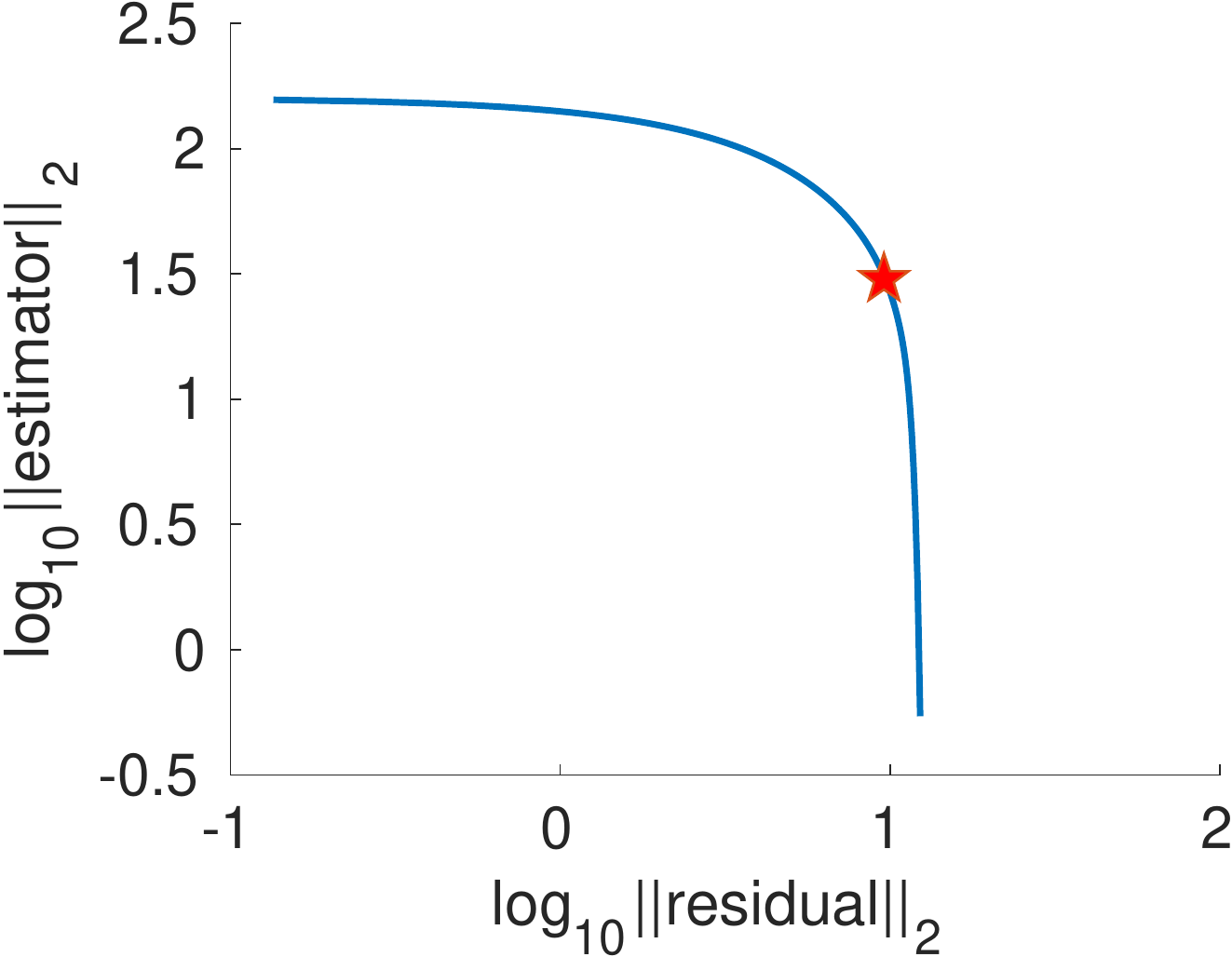}
\\ \centerline{(a) $n=100$}
\end{minipage}
\hfill
\begin{minipage}[c]{.47\linewidth}%% pas plus que .47 pour que ca tienne sur la meme page que Tab 3
\includegraphics[width=\textwidth]{\pathfig/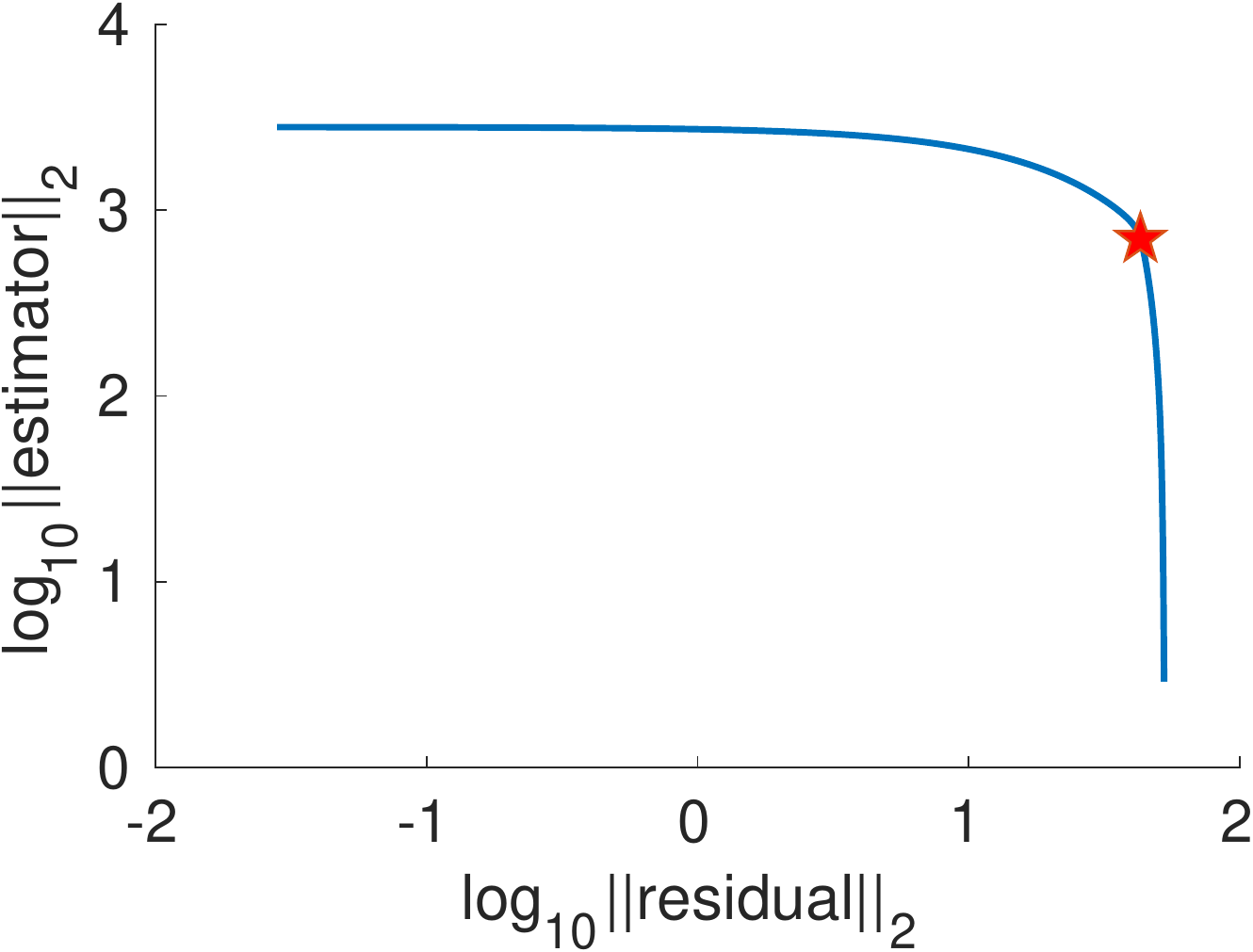}
\\ \centerline{(b) $n=2000$}
\end{minipage}
\caption{\label{fig.ridge.Lcurve} 
L-curve $(\log_{10} \norms{Y - \Fh_{\lambda}} , \log_{10} \norms{\Fh_{\lambda}})_{\lambda>0}$ 
in the `kernel ridge' framework. 
See Appendix~\ref{app.details-simus} for details. 
The red star shows the position of the oracle \textup{(}minimum of the risk $n^{-1} \norms{F - \Fh_{\lambda}}^2$\textup{)}.
The ``elbow'' is clearly localized for $n=2\,000$ \textup{(}and close to the oracle\textup{)}, 
but not for $n=100$.
}
\end{center}
\end{figure}

\begin{table}
\begin{center}
\begin{tabular}{lcccc}
$\Ch$        & $\E \crochb{ \Ch/\sigma^2 }$ & $\sqrt{\var(\Ch)}/\sigma^2$ & $\E\crochb{ (\Ch - \sigma^2)^2 }/\sigma^4$ & risk ratio\\
\tabespvert
\hline
\tabespvert
$\Chmaxjump$                           & 1.10  & 0.259   & 0.076\ch & 1.291 $\pm$ 0.003 \\
\tabespvert
\hline
\tabespvert
$\Chthr\,$, $T_n = n/10$                 & 3.38  & 1.390   & 7.57\ch\ch & 1.661 $\pm$ 0.004 \\ %T=10
$\Chthr\,$, $T_n = n/\log(n)$            & 1.59  & 0.462   & 0.563\ch & 1.285 $\pm$ 0.002 \\ %T=21.7
$\Chthr\,$, $\mathbf{T_n = n/2}$            & 1.13  & 0.231   & 0.0703 & 1.258 $\pm$ 0.002 \\ %T=50
$\Chthr\,$, $T_n = 9n/10$                & 0.86  & 0.236   & 0.077\ch & 1.566 $\pm$ 0.008 \\ %T=90
\tabespvert
\hline
\tabespvert
$\Chwindow\,$, $\eta = 1/n$              & 1.09  & 0.257   & 0.0746 & 1.292 $\pm$ 0.003 \\ %eta = 0.01
$\Chwindow\,$, $\mathbf{\eta = 1/\sqrt{n}}$& 1.10& 0.257   & 0.0762 & 1.288 $\pm$ 0.003 \\ %eta = 0.1
$\Chwindow\,$, $\eta = 1.5/\sqrt{n}$     & 1.11  & 0.258   & 0.078\ch & 1.288 $\pm$ 0.003 \\ %eta = 0.15
$\Chwindow\,$, $\eta = \sqrt{\log(n)/n}$ & 1.12  & 0.263   & 0.0827 & 1.287 $\pm$ 0.003 \\ %eta = 0.215
$\Chwindow\,$, $\eta = 2\sqrt{\log(n)/n}$& 1.17  & 0.285   & 0.109\ch & 1.275 $\pm$ 0.003 \\ %eta = 0.429
\tabespvert
\hline
\tabespvert
$\Chslope\,$, $D_{0} = n/10$             &  1.54 & 0.188   & 0.328\ch & 1.268 $\pm$ 0.002 \\ %D_0=10
$\Chslope\,$, $D_{0} = n/\log(n)$        &  1.65 & 0.193   & 0.46\ch\ch & 1.291 $\pm$ 0.002 \\ %D_0=21.7
$\Chslope\,$, $\mathbf{D_{0} = n/2}$     &  2.36 & 0.231   & 1.89\ch\ch & 1.437 $\pm$ 0.003 \\ %D_0=50
$\Chslope\,$, $D_{0} = 9n/10$       & 20.2\ch\ch & 2.07\ch & 374\phantom{.000000} & 3.68\ch{} $\pm$ 0.016 \\ %D_0=90
\tabespvert
\hline
\tabespvert
\textsc{Capushe}                          & 2.77  & 1.66\ch & 5.87\ch\ch & 1.562 $\pm$ 0.005 \\
\tabespvert
\hline
\tabespvert
median                                 & 1.16  & 0.253   & 0.0911 & 1.260 $\pm$ 0.002 \\
\tabespvert
\hline
\tabespvert
consensus                              & --    & --      & --     & 1.285 $\pm$ 0.003 \\
consensus when no reject               & --    & --      & --     & 1.266 $\pm$ 0.003 \\
\tabespvert
\hline
\tabespvert
$\sigh^2_{m_0}\,$, $D_{m_0} = n/10$      &  5.44 & 0.473   & 19.9\phantom{0000} & 2.055 $\pm$ 0.006 \\ %D_m0=10
$\sigh^2_{m_0}\,$, $D_{m_0} = n/\log(n)$ &  1.12 & 0.176   & 0.0443 & 1.223 $\pm$ 0.002 \\ %D_m0=21.7
$\sigh^2_{m_0}\,$, $D_{m_0} = n/2$       &  8.94 & 0.828   & 63.7\phantom{0000} & 2.577 $\pm$ 0.006 \\ %D_m0=50
$\sigh^2_{m_0}\,$, $D_{m_0} = n/2+1$     &  1.05 & 0.213   & 0.0478 & 1.285 $\pm$ 0.003 \\ %D_m0=51
$\sigh^2_{m_0}\,$, $D_{m_0} = 9n/10$& 38.9\ch\ch & 3.95\ch & 1450\phantom{.0000000} & 6.11\ch{} $\pm$ 0.011 \\ %D_m0=90
\tabespvert
\hline
\tabespvert
$C_p$ (known $\sigma^2$)               & --    & --      & --     & 1.252 $\pm$ 0.003 \\
$C_p$ $\times 1.12$ (known $\sigma^2$) & --    & --      & --     & 1.232 $\pm$ 0.002 \\
\end{tabular}
\medbreak
\caption{\label{tab.dist-Ch.hard}
Same as Table~\ref{tab.dist-Ch.easy} for the `hard' setting.
}
\end{center}
\end{table}

\clearpage

\section{Detailed information about figures and simulation experiments}
\label{app.details-simus}

This section provides all details necessary to reproduce the figures and simulation experiments reported throughout the article.

\subsection{Data and estimators} \label{app.details-simus.data}

All experiments are made within the fixed-design regression framework described in Section~\ref{sec.slopeOLS.framework}, with two main kinds of estimator collections and data.

\paragraph{Least-squares framework (`easy'/`hard')} 
All figures and tables, except Figures~\ref{fig.linear.jump} and \ref{fig.ridge.Lcurve}, consider data and estimators as follows.
Data satisfy
\[ Y = F + \varepsilon \in \R^n \]
with independent Gaussian noise $\varepsilon \sim \mathcal{N}(0,\sigma^2 \Id_n)$, $\sigma^2=1/4$, $n=100$,
\begin{gather}
\notag %\label{setting.P1G2.n100}
F_i = \frac{C_n}{i}
\quad \text{and} \quad
C_n = \paren{ \sum_{i=1}^n \frac{1}{i^2} }^{-1/2}
\, .
\end{gather}
The choice of $C_n$ ensures that $n^{-1} \norm{F}^2 = 1$.

The estimators considered are least-squares (projection) estimators with one among the following two collections of models $(S_m)_{1 \leq m \leq n}\,$:
\begin{itemize}
\item `easy' setting: for every $m \in \set{1, \ldots, n}$, $S_m = \Seasy_m$ is the linear span of the first $m$ vectors of the canonical basis of $\R^n$.
\item `hard' setting: for every $m \in \set{1, \ldots, n}$, $S_m = \Shard_m$ is the linear span of the first $m$ vectors of the canonical basis of $\R^n$ if $m$ is odd, and $S_m = \Shard_m$ is the linear span of the last $m$ vectors of the canonical basis of $\R^n$ if $m$ is even.
\end{itemize}
Both settings correspond to (ordered) variable selection with an orthogonal design, after having transformed the data conveniently according to the design matrix.
In the easy case, the variables are ordered by decreasing order of magnitude.
In the hard case, some uncertainty remains about the correct order (ascending or descending), and the two options are considered alternatively (depending on the parity of $n$).
Of course, models $\Shard_m$ with $m$ even are very poor, but this can be unknown before seeing the data.

\paragraph{Kernel ridge framework} 
Figures~\ref{fig.linear.jump} and \ref{fig.ridge.Lcurve} consider data and estimators as follows.
Data satisfy
\[ Y = F + \varepsilon \in \R^n \]
with independent Gaussian noise $\varepsilon \sim \mathcal{N}(0,\sigma^2 \Id_n)$, $\sigma^2=1$,
\[ F_1 = \frac{1}{2}
\qquad \text{and} \qquad
\forall i \in \set{2, \ldots, n} \, , \qquad
F_i = \sin\paren{ 25 \pi x_i^3}
\quad \text{with} \quad
x_i = \frac{i-1}{n-1} \, .
\]
The family of estimators considered is the family of kernel ridge estimators
$(\Fh_{\lambda})_{\lambda > 0}$
where for every $\lambda > 0$,
\[ \Fh_{\lambda} = K (K + n\lambda \Id_n)^{-1} Y
\, ,
\quad
K = \parenb{ k(x_i,x_j) }_{1 \leq i,j \leq n}
\, ,
\quad \text{and} \quad
\forall x,x^{\prime} \in \R , \,
k(x,x^{\prime}) = \exp \parenb{- \alpha \abss{ x - x^{\prime} } }
\]
the Laplace kernel, with $\alpha = 8$.
In the experiments, only a finite set $\sets{\lambda_0 , \ldots, \lambda_n}$ of values of $\lambda$ is considered, chosen such that the degrees of freedom $\tr( K (K + n\lambda_i \Id_n)^{-1})$ are equal to $i$ for every $i=0, \ldots, n$.

\subsection{Procedures} \label{app.details-simus.proc}

The exact definitions of all procedures considered in the experiments for computing some $\Ch$ or choosing some model $\mh$ are the following.
For the procedures depending on some parameter, its default value is used everywhere except in Tables~\ref{tab.dist-Ch.easy}--\ref{tab.dist-Ch.hard}.
Note that the choice of the default values was made \emph{prior to the simulations}: 
we can check afterwards on Tables~\ref{tab.dist-Ch.easy}--\ref{tab.dist-Ch.hard}
that these choices provide reasonably good results (which fortunately happened), 
so that results using only the default values of the parameters (for instance, Table~\ref{tab.stats-comp-mh.easy}) are meaningful.

\paragraph{Maximal jump  ($\Chmaxjump\,$, `Max. jump', `max j.' or `max')} 
In Section~\ref{sec.practical.jump-vs-slope}, we define
\[
\Chmaxjump \in \argmax_{C \geq 0} \set{ D_{\mhgalzero(C^-)}  - D_{\mhgalzero(C^+)}}
\, ,
\]
that is, the location of the maximal jump of $C \mapsto D_{\mhgalzero(C)}\,$, assuming it is unique.
In our experiments, when the argmax contains several values of $C$, we choose the largest one, that is, the last largest jump; 
this choice is natural, since it means taking the less complex model 
among those corresponding to a maximal jump, 
and it matches the choice made by \citet{Ler_Tak:2010}. 

Note that for change-point detection, \citet[Section~4.2]{Leb:2005} suggests an opposite convention ---taking the smallest value of $C$ in the $\argmax$---, arguing from simulation experiments that otherwise too small models are selected.
Nevertheless, \citet[Section~4.2]{Leb:2005} also reports that the latter convention can lead to taking $\Chmaxjump$ too small, so a rather complicated method is suggested for choosing some threshold $\alpha_{\mathrm{thr}}$ and imposing $\Chmaxjump \geq \alpha_{\mathrm{thr}}\,$.

\paragraph{Threshold ($\Chthr$ or `thr')} 
Eq.~\eqref{def.Chthr} in Section~\ref{sec.slopeOLS.math} defines 
\[ 
\Chthr (T_n)
\egaldef \min \set{ C \geq 0 \, / \, D_{\mh(C)} \leq T_n } 
\, ,
\]
which depends on some parameter $T_n\,$.
The default value of $T_n$ is $n/2$.

Note that Theorem~\ref{thm.OLS} suggests that $T_n = \rho n$ works for any $\rho \in (0,1)$,
and previous theoretical results \citep[Section~3.3]{Arl_Mas:2009:pente} suggest to take
$T_n \propto n / \log(n)$ or $n / (\log n)^2$.
Nevertheless, all these theoretical results involve pessimistic constants (as shown by the simulation experiments), so they cannot be used for a fine tuning of $T_n\,$.
It turns out that $n/2$ does very good in the experiments of Tables~\ref{tab.dist-Ch.easy}--\ref{tab.dist-Ch.hard}, while other choices lead to much worse performance.

\paragraph{Window ($\Chwindow$ or `win')} 
Eq.~\eqref{def.Chwin} in Section~\ref{sec.slopeOLS.math} defines 
\[ 
\Chwindow (\eta)
\in \argmax_{C \geq 0} \setj{ D_{\mh(C/[1+\eta])} - D_{\mh(C[1+\eta])} } 
\, , 
\]
which depends on some parameter $\eta >0$.
Similarly to $\Chmaxjump\,$, the $\argmax$ is usually not reduced to a single point, so a more precise definition must be given for $\Chwindow\,$.
Actually, when $\eta>0$, Appendix~\ref{app.algos.window} shows that
$\argmax_{C \geq 0} \sets{ D_{\mh(C/[1+\eta])} - D_{\mh(C[1+\eta])} }$ is a finite union of intervals.
Denoting by $[\Chwindow^{(1)} , \Chwindow^{(2)})$ the last of these intervals ---that is, the one corresponding to the largest values of $C$--- we define
\[ \Chwindow = \sqrt{ \Chwindow^{(1)} \Chwindow^{(2)} } \, . \]
Of course, other choices could be possible and we do not claim that our (arbitrary) choice is the best one.

In Figures~\ref{fig.DmhC.easy.ech5} and~\ref{fig.DmhC-Lcurve.easy.ech540}b, the interval represented by the two red vertical lines is $[\Chwindow^{(1)} , \Chwindow^{(2)})$.
Note that this interval often looks like $[ \Chwindow / (1+\eta) , \Chwindow (1+\eta) )$ but it can also be quite different.

Taking the limit $\eta \to 0^+$ in the definition of $\Chwindow\,$, we recover $\Chmaxjump\,$.
Theorem~\ref{thm.OLS} suggests to take $\eta \propto \etaplus \geq \sqrt{\log(n)/n}$, that we consider in our experiments (see Tables~\ref{tab.dist-Ch.easy}--\ref{tab.dist-Ch.hard}).
In our experiments, the default value for $\eta$ is $n^{-1/2}$, a choice made to get a slightly smaller value than $\sqrt{\log(n)/n} \approx 0.22$ (we recall that $n=100$ in the least-squares framework).

\paragraph{Slope ($\Chslope$ or `slope')} 
In Algorithm~\ref{algo.OLS.slope}, the definition of $\Chslope$ is rather vague;
it is a bit more precise in Section~\ref{sec.practical} where the range of models considered in the regression is defined by $\pen_0(m) \in [p_{\min} , p_{\max}]$ for some $p_{\min}<p_{\max}$ to be chosen.
In the experiments, since $\pen_0(m)$ is equal to $D_m/n$ in the least-squares framework, we choose $p_{\min} = D_0/n$ and $p_{\max}=1$ for some parameter $D_0 \in [1,n)$.

In other words, given $D_0 \in [1,n)$, we consider only models of dimension $D_m \geq D_0$ and we perform a (standard) linear regression of the empirical risk $n^{-1} \norms{\Fh_m - F}^2$ against $- D_m / n$, that is, we solve
\[
(\widehat{a},\widehat{b}) \in \argmin_{(a,b) \in \R^2} \sum_{\mM \, / \, D_m \geq D_0} \paren{a - b \frac{D_m}{n} - \frac{1}{n} \norm{\Fh_m - F}^2}^2
\, ,
\]
and we define $\Chslope$ as the resulting slope~$\widehat{b}$.
The default value of $D_0$ is $n/2$.

\paragraph{Capushe ($\Chcapushe\,$, $\mhcapushe$ or `CAP')} 
The procedure called `\textsc{Capushe}' throughout this article is the one proposed by \citet[Section~4.2]{Bau_Mau_Mic:2010} and implemented in the \textsc{Capushe} package for Matlab and~R.
For completeness, let us recall its definition ---which depends on some parameter $pct \in (0,1)$--- in the least-squares framework.
\begin{itemize}
\item Step 1: If several models have the same dimension $D$, keep only the one with the smallest empirical risk.
This step does not change anything in our experimental setting since there is exactly one model per dimension.
\item Step 2: for all $D \in [1,n-2]$, compute by robust linear regression the slope $\Chslopecapushe(D)$ of 
the empirical risk $n^{-1} \norms{\Fh_m - F}^2$ against $- D_m / n$, 
among models of dimension $D_m \geq D$.

\item Step 3: for all $D \in [1,n-2]$, compute the corresponding selected model
\[
\mh_D = \mh\parenb{ 2 \Chslopecapushe(D) }
= \argmin_{\mM} \set{ \frac{1}{n} \norm{\Fh_m - Y}^2 + \frac{2 \Chslopecapushe(D) D_m}{n} } \, .
\]
Then, $\mh_1, \ldots, \mh_{n-2}$ is piecewise constant, and some $1 = D_1 < \ldots < D_{I+1}=n-1$ and $m_1, \ldots, m_{I+1} \in \M$ exist such that
\[
\forall i \in \set{1, \ldots , I} \, , \,
\forall D \in [ D_i , D_{i+1} - 1] \, , \,
\quad
\mh_D = m_i
\]
with $m_1 \neq m_2\,$, $\ldots$, $m_{I} \neq m_{I+1}\,$.
The intervals $[D_i, D_{i+1} - 1]$ are called ``plateau'' (platforms) by \citet[Section~4.2]{Bau_Mau_Mic:2010} and their size is denoted by $N_i = D_{i+1}-D_i\,$.
\item Step 4: Keep only the platforms of size $N_i$ larger than $pct$ times the total size $\sum_{\ell} N_{\ell} = n-2$, and among these, define $\ih$ the last platform, that is,
\[
\ih = \max \setb{ i \in \set{1, \ldots , I} \, / \, N_i > pct \times (n-2) }
\]
and select
\[
\mhcapushe = m_{\ih} \, .
\]
\end{itemize}
Note that at step 4, it can happen that no platform is large enough.
In such cases, we consider the last platform among the ones of largest size, that is,
\[
\ih = \max \set{ i \in \set{1, \ldots , I} \, / \, N_i = \max_j N_j }
\, .
\]
We always take $pct=0.15$ in our experiments, that is, the default value proposed by \citet{Bau_Mau_Mic:2010}.

Note that \citet{Bau_Mau_Mic:2010} only provide a model-selection procedure $\mhcapushe\,$,
and not a value $\Ch$ of the constant in front of the penalty.
In order to help understanding better $\mhcapushe\,$, we also report in our experiments the distribution of $\Chcapushe$ that we define as some median of
\[
\setb{ \Chslopecapushe(D_{\ih}), \ldots, \Chslopecapushe(D_{\ih+1}-1) }
\, .
\]
This choice is arbitrary among many others that all lead to having $\mhcapushe = \mh(2 \Chcapushe) $.

\paragraph{Median (`med')} 
As defined in the caption of Figure~\ref{fig.OLS.dist-Ch}, `median' refers to taking $\Ch$ as the median of
\[
\set{\Chmaxjump,\Chthr,\Chwindow,\Chslope,\Chcapushe}
\]
(with their default parameter values for $\Chthr\,$, $\Chwindow\,$, and $\Chslope\,$: 
$T_n=n/2$, $\eta=n^{-1/2}$, $D_0=n/2$),
and
$\mh = \mh(2 \Ch)$.

Remark that the set of procedures considered is arbitrary, and other choices could be made. 
Nevertheless, it seems wise to keep an equilibrium between 
the jump and slope formulations; 
here, the jump approach is slightly favored, 
but the slope definitions come into play when 
$\Chmaxjump\,$, $\Chthr\,$, and $\Chwindow$ do not exactly coincide. 

The idea of considering some median of several values of $\Ch$ could also be used when there is some uncertainty about the parameter of some procedure (say, $T_n$ for $\Chthr$), by considering the median of the set of values obtained on a grid of values of the parameter.

\paragraph{Residuals ($\sigh^2_{m_0}\,$, `Residuals on $m_0$' or `resid')} 
The residual-based variance estimator $\sigh^2_{m_0}$ is defined by Eq.~\eqref{def.sighm0} in Section~\ref{sec.related.variance}:
\begin{equation}
\notag
\sigh^2_{m_0} \egaldef \frac{1}{n-D_{m_0}} \norm{Y-\Fh_{m_0}}^2
\end{equation}
for some model~$m_0\,$.
In the experiments, there is one model per dimension, 
so $m_0$ is given by the value of its dimension $D_{m_0}$ and the default choice is $D_{m_0} = n/2$.
Since $n=100$, the default choice is $D_{m_0} = 50$ which is even, so the definition of $\Shard_m$ 
---in which models of odd dimension are good and models of even dimension are very poor--- 
is made on purpose.

Note that in Tables~\ref{tab.dist-Ch.easy}--\ref{tab.dist-Ch.hard}, the line ``$D_{m_0} = n / \log(n)$'' means ``$D_{m_0} = 21$'' (hence, for the `hard' setting, it is a reasonably good model).

\paragraph{Consensus (`cons')} 
As defined in the caption of Figure~\ref{fig.OLS.dist-risk-ratio}, 
the ``consensus'' procedure performs a majority vote among
\[
\set{\mh(2\Chmaxjump),\mh(2\Chthr),\mh(2\Chwindow),\mh(2\Chslope),\mhcapushe}
\]
with their default parameters values.
If no majority emerges (that is, if we do not have at least three of these procedures that agree), the default choice is $\mh(2\Chwindow)$.
Remark that Table~\ref{tab.stats-comp-mh.easy} shows that an agreement occurs for more than $96 \%$ of the samples in the `easy' setting, and for more than $89 \%$ of the samples in the `hard' setting.

\paragraph{Consensus when no reject (`no rej')} 
This actually refers to the same procedure as `consensus', but showing results (a boxplot or an estimation of the expectation of the loss ratio) only for the samples for which a majority emerged.
Again, Table~\ref{tab.stats-comp-mh.easy} shows that this only removes a small fraction of the $N=10^4$ independent samples generated in our experiments.

\paragraph{Mallows' $C_p$}
When the variance $\sigma^2$ is known, a natural model-selection procedure for the framework of Section~\ref{sec.slopeOLS} 
is Mallows' $C_p$ \citep{Mal:1973}, that is, selecting
\[
\mh = \mh(2\sigma^2) = \argmin_{\mM} \set{ \frac{1}{n} \norm{\Fh_m - Y}^2 + \frac{2 \sigma^2 D_m}{n} }
\, .
\]
Its performance is shown in Tables~\ref{tab.dist-Ch.easy}--\ref{tab.dist-Ch.hard} for comparison.

Mallows' $C_p$ is also considered for illustrating the overpenalization phenomenon in Figure~\ref{fig.surpen} in Section~\ref{sec.empirical.overpenalization}.
On the graph of Figure~\ref{fig.surpen}, what is plotted is, 
for $C \in [0,4]$, the estimated value (from $N=10^4$ independent samples) of the expected risk ratio
\[
\E\croch{ \frac{\norm{\Fh_{\mh(2 C \sigma^2)} - F}^2}{\inf_{\mM} \norm{\Fh_{m} - F}^2} }
\]
when using Mallows' $C_p$ penalty multiplied by~$C$; 
we recall that $\mh(\cdot)$ is defined 
by Eq.~\eqref{eq.mhC} in Section~\ref{sec.slopeOLS.penmin}. 
For plotting the graph of Figure~\ref{fig.surpen}, a linear grid of values of $C$ with stepsize $1/100$ is considered.
The optimal performance is obtained for $C=1.12$ in the `easy' and `hard' settings, and it is also included in Tables~\ref{tab.dist-Ch.easy}--\ref{tab.dist-Ch.hard}.

\subsection{Additional remarks}

Repeated experiments show results obtained from $N=10^4$ independent samples.

Illustrations made on a single sample in the least-squares framework are showed in Figures~\ref{fig.OLS.algo}, \ref{fig.Lcurve.easy.ech6}, \ref{fig.DmhC.easy.ech5}, \ref{fig.DmhC-Lcurve.easy.ech540}, \ref{fig.OLS.slope-vs-residuals.easy-hard}.
The samples considered have been chosen manually in order to illustrate 
either typical or rare (but still possible) configurations. 
The graphs of Figure~\ref{fig.OLS.algo}, Figure~\ref{fig.Lcurve.easy.ech6}, 
and Figure~\ref{fig.OLS.slope-vs-residuals.easy-hard}a are made on the same sample (they correspond to a ``typical'' situation).
The graph of Figure~\ref{fig.DmhC.easy.ech5} is made on a second sample (corresponding to a ``rare'' situation).
The two graphs of Figure~\ref{fig.DmhC-Lcurve.easy.ech540} are made on a third sample (also corresponding to a ``rare'' situation, similar to the one shown in Figure~\ref{fig.DmhC.easy.ech5}).

Figure~\ref{fig.linear.jump} is taken from the article by \citet[top right graph of Figure~2]{Arl_Bac:2009:minikernel_long_v2}. 
It is made from a single sample generated as in the kernel-ridge framework 
(see Appendix~\ref{app.details-simus.data}), with a sample size $n=200$.

\end{document}